%% file: ms.tex
\documentclass[12pt]{amsart}
\pdfoutput=1
\input{defsv4}

\begin{document}

\title{Quantum traces for $SL_n$-skein algebras}

\author[Thang T. Q. L\^e]{Thang T. Q. L\^e}
\address{School of Mathematics, 686 Cherry Street, Georgia Tech, Atlanta, GA 30332, USA}
\email{letu@math.gatech.edu}

\author[Tao Yu]{Tao Yu}
\address{Shenzhen International Center for Mathematics,
Southern University of Science and Technology, Shenzhen, China 
}
\email{yut6@sustech.edu.cn}

\thanks{
Supported in part by National Science Foundation. \\
2010 \textbf{Mathematics Classification:} Primary 57N10. Secondary 57M25.
}

\begin{abstract}
We establish the existence of several quantum trace maps. The simplest one is an algebra homomorphism between two quantizations of the algebra of regular functions on the $SL_n$-character variety of a punctured bordered surface $ \fS$ equipped with an ideal triangulation $\lambda$.
The first quantization is the (stated) $SL_n$-skein algebra $\cS(\fS)$, defined using tangle diagrams on the surface. The second quantized algebra $\bXS$ is the Fock and Goncharov's quantization of their $X$-moduli space, which belongs to a simple class of noncommutative algebras known as quantum tori. The quantum trace is an algebra homomorphism
\[\btr^X_\lambda: \bSS \to \bXS,\]
from the reduced skein algebra $\bSS$, a quotient of $\SS$, to $\bXS$. When the quantum parameter is 1, the quantum trace $\btr^X_\lambda$ coincides with the classical Fock-Goncharov homomorphism. This is a generalization of the famous Bonahon-Wong quantum trace map for the case $n=2$.

We will define the extended Fock-Goncharov algebra $\FG(\fS, \lambda)$ and show that $\btr^X_\lambda$ can be lifted to an extended quantum trace
\[\tr^X_\lambda: \SS \to \FG(\fS, \lambda).\]
We show that both $\btr^X_\lambda$ and $\tr^X_\lambda$ are natural with respect to the change of triangulations.

When each connected component of $\fS$ has non-empty boundary and no interior ideal point, we define a quantum torus $\bA(\fS, \lambda)$, which is a quantization of the Fock-Goncharov $A$-moduli space, and its extension $\cA(\fS, \lambda)$, also a quantum torus. We then show that there exist the $A$-versions of the quantum traces
\[\btr^A: \bSS \to \bA(\fS, \lambda), \qquad
\tr^A: \SS \embed \cA(\fS, \lambda)\]
where the second map is injective, while the first is injective at least when $\fS$ is a polygon. Moreover the image $\tr^A(\SS)$ is sandwiched between the quantum space $\cA_+(\fS, \lambda)$ and the quantum torus $\cA(\fS, \lambda)$. Similar fact holds for the image $\btr^A(\SS)$. The transitions from $\tr^X_\lambda$ to $\tr^A_\lambda$ and from $\btr^X_\lambda$ to $\btr^A_\lambda$ are given by multiplicatively linear maps.
\end{abstract}

\maketitle
\tableofcontents{}

\section{Introduction}

Throughout the paper, the ground ring $R$ is a commutative domain with a distinguished invertible element $\hq$. All modules and algebras are over $R$. For the reader's reference, the usual quantum parameter in the theory of quantized universal algebra of $\mathfrak{sl}_n$ is $q= (\hq)^{2n^2}$.

\subsection{Punctured surfaces} \label{sec.001}

Assume that $\fS$ is a \term{punctured surface}, i.e. it is the result of removing a finite number of points, called \term{ideal points}, from a closed oriented surface. We will consider two quantizations of the $SL_n$ character variety of $\fS$.

The first quantization $\bSS$ is a twisted version of skein algebra introduced by A. Sikora \cite{Sikora}. As an $R$-module $\bSS$ is freely spanned by link diagrams on $\fS$ subject to certain relations which are local relations of the $SL_n$ Reshetikhin-Turaev link invariants \cite{RT}, see Section \ref{sec.skein}. The product of two link diagrams is obtained by stacking the first above the second.
Sikora showed that $\bSS$ is a quantization of the $SL_n$-character variety along the Atiyah-Bott-Goldman Poisson bracket. In other words, if $\bSS_{\hq=1}$ denotes the algebra $\bSS$ with $R=\BC$ and $\hq =1$, then $\bSS_{\hq=1}$ is isomorphic to the ring
of regular functions on the $SL_n$-character variety of $\fS$. In addition, the semiclassical limit of the non-commutativity gives the Atiah-Bott-Goldman Poisson structure. For $n=2$ the skein algebra $\bSS$ is isomorphic to the Kauffman bracket skein modules \cite{Prz,Turaev}, and for $n=3$ it is isomorphic to the Kuperberg skein algebra \cite{Kuperberg}. If the quantum integers are invertible in $R$ then $\bSS$ is isomorphic to the skein algebra defined using MOY graphs \cite{MOY,CKM}.

The second quantization is Fock-Goncharov's algebra $\bXS$ which depends on an ideal triangulation $\lambda$ of $\fS$ and quantizes the $X$-variety in Fock and Goncharov's theory of higher Teichm\"uller spaces, see ~\cite{FG,FG2}.
The algebraic structure of $\bXS$ is very simple as it is a \term{quantum torus}, which by definition is an algebra of the form
\begin{equation}\label{eq.TTT}
\BT(Q) := R \la x_1^{\pm 1}, \dots, x_r^{\pm 1} \ra /(x_i x_j = \hq^{2Q_{ij}} x_j x_i),
\end{equation}
where $Q$ is an antisymmetric integer $r\times r$ matrix.
In a sense, a quantum torus is a simplest possible non-commutative algebra. Its algebraic structures and representations, etc, are known. For example, $\BT(Q)$ is a domain, and hence has a well-defined division algebra of fractions $\Fr(\BT(Q))$. Besides, the Gelfand-Kirillov dimension of $\BT(Q)$ is $r$, the size of the matrix $Q$.

For an ideal triangulation $\lambda$ of $\fS$, Fock and Goncharov define an integer antisymmetric matrix $\bQ(\fS, \lambda)$, and $\bXS$ is the quantum torus $\BT(\bQ(\fS, \lambda))$. Actually the original Fock-Goncharov algebra is the subalgebra of $\bSS$ generated by $x_i^{\pm n}$ in the presentation \eqref{eq.TTT}.

The classical Fock-Goncharov algebra $\bXS_{\hq =1}$, which is $\bXS$ with $R=\BC$ and $\hq =1$, is a Poisson algebra and is a chart of the $X$-variety, a cousin of the character variety. Fock and Goncharov \cite[Chapter 9]{FG} showed that there is a Poisson algebra homomorphism
\begin{equation}\label{eq.tr0}
\tTr_\lambda: \bSS_{\hq=1} \to \bXS_{\hq =1}.
\end{equation}

There was an important question of whether the algebra homomorphism $\tTr_\lambda$ can be lifted to the quantum level. For $n=2$ this was asked by Chekhov-Fock \cite{CF}. For general $n$ the question is formulated as a conjecture by D. Douglas \cite{Douglas}.
One main result of this paper is to answer this question in affirmative.

\begin{thm}[Part of Theorems \ref{thm-rdtrX}, \ref{thm-co-chg-X}, and \ref{thm.n3}] \label{thm.01}
Assume $\fS$ is a punctured surface with an ideal triangulation $\lambda$. There exists an algebra homomorphism
\begin{equation}
\btr^X_\lambda: \bSS \to \bXS
\end{equation}
with the following properties:
\begin{enumerate}[(i)]
\item If $R=\BC$ and $\hq =1$, then $\btr^X_\lambda$ is the Fock-Goncharov homomorphism $\tTr_\lambda$, and
\item The map $\btr^X_\lambda$ is natural with respect to the change of triangulations.
\end{enumerate}
Besides $\btr^X_\lambda$ is injective if $ n\le 3$.
\end{thm}

When $n=2$ the theorem was first proved by Bonahon and Wong \cite{BW}. See also \cite{Le:qtrace,CL,KLS} for other approaches to the $SL_2$-quantum trace.
For $n=3$ the theorem is also obtained independently by H. Kim \cite{Kim1, Kim2}.
Besides, D. Douglas \cite{Douglas} gave a definition of $\btr^X_\lambda$ for $n=3$ and showed that it satisfies some, but not all, defining relations of the skein algebra. Douglas also suggested a definition of $\btr^X_\lambda$ for all $n$.


Let us explain the naturality in part (ii) of Theorem \ref{thm.01}.
We will define the balanced subalgebra $\bsX^\bal(\fS, \lambda) \subset \bXS$, which is also a quantum torus of the same dimension. We will show that the image of $\btr^X_\lambda$ is in $\bsX^\bal(\fS, \lambda)$, and for another ideal triangulation $\lambda'$ there is an algebra isomorphism of division algebras
\begin{equation}
\rd{\Psi}^X_{\lambda, \lambda'}: \Fr(\bsX^\bal(\fS,\lambda') ) \xrightarrow{\cong}\Fr(\bsX^\bal(\fS,\lambda) )
\end{equation}
which intertwines $\btr^X_\lambda$ in the sense that
\begin{equation}
\rd{\Psi}^X_{\lambda, \lambda'} \circ \btr^X_{\lambda'} = \btr^X_{\lambda}.
\end{equation}
In addition, $\rd{\Psi}^X_{\lambda, \lambda} = \id$ and $\rd{\Psi}^X_{\lambda, \lambda'} \circ \rd{\Psi}^X_{\lambda', \lambda''} = \rd{\Psi}^X_{\lambda, \lambda''}$. This is the naturality of the map $\btr^X_\lambda$.


\subsection{Punctured bordered surfaces}

Theorem \ref{thm.01} is proved by cutting the surface $\fS$ into triangles and reducing the proof to the triangle case. For this purpose we need to consider surfaces with boundary.

A \term{punctured bordered surface} (pb surface for short) $\fS$ is the result of removing a finite number of points, called \term{ideal points}, from a compact surface $\bfS$ such that every boundary component of $\bfS$ contains at least one ideal point.

The first author and A. Sikora extended the notion of skein algebra to pb surfaces in \cite{LS}, where it is called the \term{stated skein algebra}. We will use the notation $\SS$ for this stated skein algebra, which is the same as the ordinary skein algebra $\bSS$ when $\pfS=\emptyset$. Naturally, we have to use tangle diagrams with endpoints on the boundary $\pfS$, and impose further boundary conditions which also come from the local identities of the $SL_n$ Reshetikhin-Turaev invariant. When $n=2$ this type of stated skein algebra was introduced by the first author in \cite{Le:triangulation}, where it was used to give a simple proof of the existence of the Bonahon-Wong quantum trace. When $n=3$, the stated skein algebra was introduced by Higgins \cite{Higgins}. The stated skein algebra is closely related to Alekseev-Grosse-Schomerus' moduli algebra \cite{AGS} and factorization homology \cite{BBJ}.

The obvious extension of the quantum trace map to the case when $\pfS\neq \emptyset$ has a big kernel. For this reason, we introduce the \term{reduced skein algebra} $\bSS$, which is a quotient algebra of $\SS$, factored out by certain elements near the boundary. If $\fS$ has empty boundary, then $\bSS$ is the same skein algebra considered in Subsection \ref{sec.001}. For $n=2$ the reduced version was defined by Costantino and the first author \cite{CL}.

For an ideal triangulation $\lambda$ of $\fS$ the Fock-Goncharov algebra $\bsX(\fS, \lambda)$ can also be defined. The Fock-Goncharov classical trace can be defined so that if $\al$ is a closed immersed curve on $\fS$ then $\tTr_\lambda(\al) \in \bsX(\fS, \lambda)_{ \hq=1 }$, with some favorable properties. The first result on quantum trace map for pb surfaces is following theorem, similar to Theorem \ref{thm.01}.

\begin{thm}[Part of Theorems \ref{thm-rdtrX}, \ref{thm-co-chg-X}, and \ref{thm.n3}] \label{thm.02}
Assume $\fS$ is a punctured bordered surface with an ideal triangulation $\lambda$.
There exists an algebra homomorphism
\begin{equation}
\btr^X_\lambda: \bSS \to \bsX(\fS, \lambda)
\end{equation}
such that the followings hold.
\begin{enumerate}[(i)]
\item If $R=\BC$, $\hq =1$, and $\al$ is a closed immersed curve on $\fS$ then $\btr^X_\lambda(\al)=\tTr_\lambda(\al)$.
\item The map $\btr^X_\lambda$ is natural with respect to the change of triangulations.
\item The map $\btr^X_\lambda$ is compatible with splitting of surface $\fS$ by edges of $\lambda$.
\end{enumerate}
Moreover $\btr^X_\lambda$ is injective for $n \le 3$.
\end{thm}

Part (iii) of Theorem \ref{thm.02}, the compatibility with the splitting, is explained in Section \ref{sec.Xtrace}.


\begin{conjecture}\label{conj.inj2}
The quantum trace $\btr^X_\lambda$ of Theorem \ref{thm.02} is injective.
\end{conjecture}

This was proved for $n=2$ by Costantino and the first author \cite{CL}. When $n=2$ and $\fS$ does not have boundary the conjecture was proved by Bonahon and Wong \cite{BW}. In this paper we give a proof for the case $n=3$.
When $n=3$ and the surface does not have boundary H. Kim also has an independent proof \cite{Kim1}. For surfaces with non-empty boundary we do have some injectivity results. See Subsection \ref{sec.14} below.

\subsection{The extended Fock-Goncharov algebra and the extended quantum trace}

Let $\fS$ be a punctured bordered surface with an ideal triangulation $\lambda$. As the Gelfand-Kirillov dimension of the Fock-Goncharov algebra $\bsX(\fS, \lambda)$ is less than that of $\SS$ if $\pfS \neq \emptyset$, there cannot be any embedding of $\SS$ into the quantum torus $\bsX(\fS, \lambda)$.

To have a potentially {\em injective} quantum trace map for the full skein algebra $\SS$, we introduce the \term{extended Fock-Goncharov algebra} $\FG(\fS,\lambda)$, which is also a quantum torus, by using an extension of the surface $\fS$. If $\pfS= \emptyset$ then $\FG(\fS,\lambda)= \bsX(\fS,\lambda)$. When $n=2$, the extended algebra $\FG(\fS,\lambda)$ was defined by the authors in \cite{LY2}.

We show that the quantum trace map of Theorem \ref{thm.02} can be lifted to an extended quantum trace.

\begin{thm}[Part of Theorems \ref{thm.full} and \ref{thm-co-chg-X}] \label{thm.02e}
Let $\fS$ be a punctured bordered surface with an ideal triangulation $\lambda$. Suppose each connected component of $\fS$ has non-empty boundary.
\begin{enuma}

\item There exists an algebra homomorphism, called the \term{extended quantum trace map},
\begin{equation}
\tr^X_\lambda: \SS \to \FG(\fS, \lambda)
\end{equation}
which is natural with respect to the change of triangulations.
\item The extended quantum trace $\tr^X_\lambda$ is a lift of $\btr^X_\lambda$ in the following sense: The image of $\tr^X_\lambda$ lies in a subalgebra $\FG'(\fS,\lambda)$ which comes with a surjective algebra homomorphism $p:\FG'(\fS,\lambda) \onto \bXS$ such that the following diagram is commutative
\begin{equation}\label{eq.diag01}
\begin{tikzcd}
\SS \arrow[r,"\tr^X_\lambda"] \arrow[d,two heads] & \FG'(\fS, \lambda) \arrow[d,two heads,"p"]\\
\bSS \arrow[r,"\rdtr^X_\lambda"] & \rd{\FG}(\fS, \lambda)
\end{tikzcd}
\end{equation}
\end{enuma}
\end{thm}

The algebra $\FG'(\fS,\lambda)$ is of very simple type, as it is linearly spanned by monomials in the generators of the quantum torus $\FG(\fS,\lambda)$, and the projection $\FG'(\fS,\lambda) \onto \bXS$ is also of a very simple form, as it sends certain monomials to zero while leaves other monomials alone. In Theorem \ref{thm.full}, the algebra $\FG'(\fS,\lambda)$ is the monomial subalgebra $\BT(\mQ_\lambda,B_\lambda)$.

It should be noted that in general there is no algebra homomorphism $\FG(\fS,\lambda) \to \rd\FG(\fS,\lambda)$ such that $\FG'(\fS, \lambda)$ can be replaced by $\FG(\fS,\lambda)$ in Diagram~\eqref{eq.diag01}. Hence the construction of $\FG'(\fS, \lambda)$ is quite non-trivial.

\begin{conjecture}
The extended quantum trace homomorphism $\tr^X_\lambda$ is injective.
\end{conjecture}

For $n=2$ the conjecture, as well as Theorem~\ref{thm.02e}, was proved in \cite{LY2}. For general $n$ it is confirmed in special cases, see Theorem~\ref{thm.04}.

\subsection{No interior ideal points case: $A$-versions of quantum traces} \label{sec.14}

Assume that each connected component of the pb surface $\fS$ has non-empty boundary and at the same time no interior ideal point. In this case we get the strongest results. Let $\lambda$ be an ideal triangulation of $\fS$. We will show that other versions of quantum traces, called the $A$-versions, exist. They have properties more favorable than the $X$-versions quantum traces and help us to prove many results of the paper, even those concerning more general surfaces.

First we will introduce the $A$-version quantum tori
\[ \bAS= \BT(\bmP(\fS, \lambda) ), \quad \AS= \BT(\mP(\fS, \lambda)), \]
where the matrix $\bmP(\fS, \lambda)$ has the size of $\bmQ(\fS, \lambda)$ and $\mP(\fS, \lambda)$ has the size of $\mQ(\fS, \lambda)$. Theorem \ref{thm.dual} will show that $\bmP(\fS, \lambda)$ is \term{compatible with $\bmQ(\fS, \lambda)$} in the sense of \cite{BZ}.
The algebra $\bAS$ can be thought of as the quantization of the torus chart of the $A$-moduli space of Fock and Goncharov. \\
Lemma \ref{lemma-HK} shows that there are non-degenerate integer square matrices $\bmK$ and $\mK$ such that
\[ \bmK \, \matsp \bmQ(\fS, \lambda) \, \matsp \bmK^t = \bmP(\fS, \lambda), \qquad
\mK \matsp \mQ(\fS, \lambda) \matsp \mK^t = \mP(\fS, \lambda).\]
Consequently there are algebra embeddings given by multiplicatively linear maps
\begin{equation}\label{eq.psi1}
\rd{\psi}_\lambda: \bAS \embed \bXS, \qquad {\psi}_\lambda: \AS \embed \XS
\end{equation}
whose images are called the \term{balanced parts} $\bsX^\bal(\fS, \lambda)$ and $\FG^\bal(\fS,\lambda)$ of $\bXS$ and $\XS$ respectively.


The quantum torus $\AS$ is the ring of Laurent polynomials in $q$-commuting variables $x_1, \dots, x_r$ as described in the presentation \eqref{eq.TTT}, where $r= r(\fS)$. The subalgebra generated by $x_1, \dots, x_r$ is denoted by $\lenT_+(\fS, \lambda)$, and is known as a quantum space. One defines the quantum space $\bA_+(\fS, \lambda)$ similarly.

Now we can formulate the main result concerning the $A$-version of quantum trace.

\begin{thm}[Part of Theorem \ref{thm-Atr}]\label{thm.04} Suppose each connected component of a punctured bordered surface $\fS$ has non-empty boundary and no interior ideal point, and
$\lambda$ is an ideal triangulation of $\fS$.
There exists an algebra embedding
\begin{equation}\label{eq.Atr0}
\tr_\lambda^A: \SS \embed \cA(\fS,\lambda),
\end{equation}
called the $A$-version quantum trace homomorphism, such that
\begin{equation}
\cA_+(\fS,\lambda) \subset \tr_\lambda^A (\SS ) \subset \cA(\fS,\lambda)
\label{eq.sand}
\end{equation}
Moreover, $\tr^A_\lambda$ and $\tr^X_\lambda$ are intertwined by $\psi_\lambda$, so that the following diagram is commutative
\[\begin{tikzcd}[row sep=tiny]
&\cA(\fS, \lambda) \arrow[dd,"\psi_\lambda","\cong"'] \\
\SS \arrow[ru,hook,"\tr^A_\lambda"]
\arrow[rd,hook,"\tr^X_\lambda"']&\\
&\FG^\bal(\surface, \lambda)
\end{tikzcd}\]
\end{thm}

The fact \eqref{eq.sand} that $\SS$ is sandwiched between the quantum space $\cA_+(\fS,\lambda)$ and quantum torus $\cA(\fS,\lambda)$ is an advantage of the $A$-version quantum trace over the $X$-version.
For example, from here it is easy to calculate the center of $\SS$ and study its representation theory, especially when $\hq$ is a root of 1. It also follows that the $A$-version quantum trace $\tr^A_\lambda$ induces an isomorphism of the division algebras
\begin{equation}
\Fr(\tr^A_\lambda): \Fr (\SS ) \xrightarrow{\cong } \Fr \cA(\fS,\lambda). \label{eq.Frtr1}
\end{equation}
There is a similar result for the reduced version, though the injectivity result is weaker.

\begin{thm}[Part of Theorem \ref{thm-Atr}] \label{thm.04a}
With the same assumption of Theorem \ref{thm.04},
there exists an algebra homomorphism
\begin{equation}\label{eq.bAtr0}
\rdtr_\lambda^A: \bSS \to \bA(\fS,\lambda)
\end{equation}
such that
\begin{equation}\label{eq-Aincl-rd0}
\bA_+(\fS,\lambda) \subset \rdtr_\lambda^A(\bSS) \subset \bA(\fS,\lambda).
\end{equation}
Moreover, $\tr^A_\lambda$ and $\tr^X_\lambda$ are intertwined by $\psi_\lambda$:
\begin{equation}
\btr^X_\lambda = \bar \psi_\lambda\circ \btr^A_\lambda
\end{equation}
In addition, if $\fS$ is a polygon, then $\rdtr_\lambda^A$ is injective.
\end{thm}

Even though so far the injectivity of $\btr^A_\lambda$ is established only for polygons, this case is very important for us. We will use the injectivity for the case of quadrilateral and pentagon to prove the naturality properties of the $X$-version trace quantum maps of Theorems \ref{thm.01} and \ref{thm.02}. We also conjecture that $\rdtr_\lambda^A$ is always injective.

\begin{remark}
When $n=2$, the existence of a matrix compatible with $\bmQ(\fS, \lambda)$ was proved by G. Muller \cite{Muller}. A. Goncharov kindly informed us that the same fact for general $n$ (under the assumption of Theorem \ref{thm.04a}) can be derived from the results of \cite[Section 12]{GS}, even for groups more general than $SL_n$. Our approach gives an explicit, combinatorial description of a compatible matrix of $\bmQ(\fS, \lambda)$, see Subsection \ref{seb.trig1} below and Section \ref{sec.qtori}. As compatible matrices are not unique, one might ask if our $\bmP(\fS, \lambda)$ is equal to the one coming from \cite{GS}. A further question is the relationship between our $\bsX_{\hq=1}$ and the space $\mathcal{P}_{SL_n, \fS}$ of \cite{GS}, which, a priori, look different even though they have the same dimension. Note that $\bsX$ is defined as a quantum space, so that our $\bsX_{\hq=1}$ has an obvious quantization.
\end{remark}

\begin{remark}
For $n=2$, Theorem \ref{thm.04} was proved in \cite{LY2}, based on earlier work of Muller \cite{Muller}. In fact Muller constructed a skein algebra, equal to a subalgebra $\cS_+(\fS)$ of $\SS$. Then he defined the quantum cluster algebra as a localization of $\cS_+(\fS)$. In \cite{LY2} we proved that our reduced algebra $\bSS$ is equal to Muller's quantum cluster algebra. One might ask if our reduced skein algebra $\bSS$ has connection to the quantum cluster algebra, and we plan to return to this question in a future work. For a partial generalization of Muller's result to the case $n=3$ see \cite{IY}.
\end{remark}

\subsection{Triangle case} \label{seb.trig1}

There are three main steps in the proof of the existence of the quantum trace maps. First, given an ideal triangulation $\lambda$, we cut $\fS$ along edges of the triangulation and the result is a disjoint union of ideal triangles. The cutting homomorphism of the stated skein algebra \cite{LS} gives an algebra map
\begin{equation}
\Theta: \SS \to \bigotimes_{\stdT} \cS(\stdT),
\end{equation}
where the tensor product is over all ideal triangles $\stdT$ which are faces of the triangulation.

The second step is to show that for each ideal triangle $\stdT$ there are $A$-version and $X$-version quantum traces. The third step is to show that the we can patch the quantum traces from triangle to get quantum traces for the whole surface.

Let us discuss the second step, of how to construct quantum traces for the triangle. At the same time we illustrate the results of Theorem~\ref{thm.04a} by the example of the triangle. Present the triangle $\stdT$ as the simplex
\begin{equation}
\stdT = \{ (x,y,z)\in \BR^3_{ \ge 0} \mid x+y+z=n \}.
\end{equation}

\begin{figure}
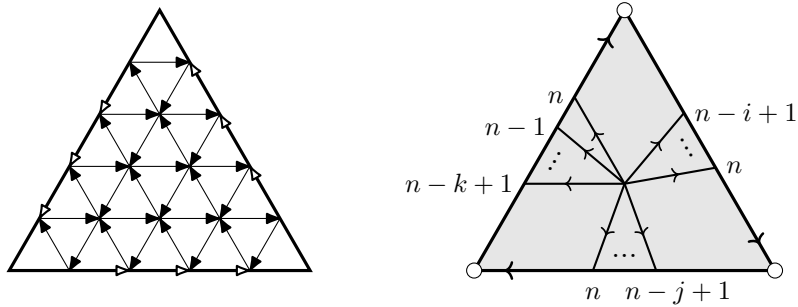

\centering
\input{def-Q} \qquad
\input{trigen-res-ordered}
\caption{Left: The quiver $\Gamma$ with $n=5$. Right: Elements $\gaa_{ijk}$}\label{f.n-triang01}
\end{figure}

Let $\Gamma$ be the quiver where the set of vertices $\bV$ consists of all points in $\stdT$ with integer coordinates, excluding the three vertices of $\stdT$. The elements of $\bV$ are connected by weighted arrows as in Figure \ref{f.n-triang01}, where a black arrow has weight 2, and a white arrow has weight 1. The Fock-Goncharov matrix $\bmQ$ is the Poisson matrix of the quiver $\Gamma$, i.e. it is the anti-symmetric map
$ \bmQ: \bV \times \bV \to \BZ$
given by
\[\bmQ(x, y) = \begin{cases} w, \qquad &\text{if there is an arrow of weight $w$ from $x$ to $y$},\\
0, &\text{otherwise}.\end{cases}\]
The Fock-Goncharov algebra is $\rd\FG(\stdT)= \BT(\bmQ)$.

We define $\bmP: \bV \times \bV \to \BZ$ as the unique anti-symmetric map satisfying
\begin{itemize}
\item $\bmP$ is invariant under rotation of the triangle $\stdT$ by $2\pi /3$, and
\item If $v=(i,j,k), v'=(i', j', k') \in \bV$ with $i \le i' , j \le j'$, then
\begin{equation}
\bmP(v,v') = n \det \begin{pmatrix} i & i' \\ j & j'
\end{pmatrix} = n(ij' - j i').
\end{equation}
\end{itemize}

Our quantization of the $A$-moduli space of Fock and Goncharov is the quantum torus
$$\bA(\stdT):=\BT(\bmP )= R \la a_v^{\pm 1}, v \in \bV \ra / (a_v a_{v'} = \hq^{2 \bmP(v,v')} a_{v'} a_v).$$

For $v=(i,j,k) \in \bV$, let $\gaa_v \in \reduceS(\stdT)$ be the element represented by the diagram in Figure~\ref{f.n-triang01} with some normalization constant, which is explained in detail in Section~\ref{sec.qtr3}. The collection $\{ \gaa_v\mid v \in \bV\}$ realizes a quantum space in $\reduceS(\stdT)$, as follows.

\begin{thm}[Parts of Theorems \ref{thm.domainr}, \ref{thm-tdual-tri} and \ref{thm.bP3a}] We have the following.
\begin{enumerate}[(i)]
\item The reduced skein algebra $\reduceS(\stdT)$ is a domain.
\item For $v,v'\in \bV$, we have $\gaa_v \gaa_{v'}= \hq^{2 \bmP(v,v')} \gaa_{v'} \gaa_v$. Consequently, there is an algebra homomorphism $f: \BT_+(\bmP) \to \reduceS(\stdT)$ given by $f(a_v) = \gaa_v$.
\item The algebra homomorphism $f$ is injective.
\item For $a\in \reduceS(\stdT)$ there is a monomial $\fm= \prod_{v\in \bV} \gaa_v^{k_v}\in \BT_+(\bmP)$ such that $a \fm \in \BT_+(\bmP)$.
\item The matrix $\bmP$ is compatible with $\bmQ$, with interior vertices being the exchangeable indices. (See Section~\ref{sec-ntri} for an explanation).
\end{enumerate}
\end{thm}

From (i)-(iv) it is not difficult to show that $\reduceS(\stdT)$ embeds into $\BT(\bmP)$ such that
$$ \BT_+(\bmP) \subset \reduceS(\stdT) \subset \BT(\bmP)= \bA(\stdT).$$
 The second embedding is the $A$-version quantum trace $\rdtr^A: \reduceS(\stdT)\embed \bA(\stdT) =\BT(\bmP) $.
By composing with the multiplicatively linear isomorphism $\rd{\lenT}(\stdT) \xrightarrow{\cong} \rdbl(\stdT)$ we get the $X$-version quantum trace for the ideal triangle:
\begin{equation}\label{eq.trTri}
\rdtr^X: \reduceS(\stdT)\embed \bsX(\stdT).
\end{equation}

For a connected surface $\fS$ with non-empty boundary and no interior punctures, a similar collection $\{\gaa_v\}$, realizing a quantum space in $\SS$, can be constructed, see Section \ref{sec.Atrace}.

\subsection{On naturality with respect to triangulation changes}

Given two triangulations $\lambda$ and $\lambda'$, to show that there is a natural transition isomorphism
\[\rd{\Psi}^X_{\lambda', \lambda}: \Fr( \rdbl(\fS, \lambda) ) \xrightarrow{\cong } \Fr( \rdbl(\fS, \lambda') )\]
intertwining the quantum traces $\rdtr^X_\lambda$ and $\rdtr^X_{\lambda'}$ is not easy, even in the case $n=2$ \cite{BW} and $n=3$ \cite{Kim2}. In the cited works, the transition isomorphisms are constructed explicitly by hand, and there are many cases involved and a lot of identities to prove. Here we use another approach, via the $A$-version quantum traces.

First assume the surface $\fS$ is connected, with non-empty boundary and no interior point. The construction of $A$-version transition isomorphism is easy. We define
\[\Psi^A_{\lambda', \lambda} := \Fr(\tr^A_{\lambda'} ) \circ \Fr(\tr^A_\lambda )^{-1},\]
where $\Fr(\tr^A_\lambda )$ is the isomorphism given by \eqref{eq.Frtr1}. Clearly $\Psi^A_{\lambda', \lambda}$ intertwines $\tr^A_\lambda$ and $\tr^A_{\lambda'}$. Using the linear isomorphism $\psi_\lambda$ of \eqref{eq.psi1} (with restriction onto $\FGbl(\fS, \lambda)$), we get the $X$-version natural transition isomorphism for this special type of surfaces.

For the reduced map $\rdtr^A_\lambda$ the above simple argument does not work since we do not know if $\rdtr^A_\lambda$ is injective. However, since $\rdtr^A_\lambda$ is injective for polygons, the above argument shows that we have the natural transition isomorphisms, both the $A$- and $X$- versions, for polygons.

Now assume $\fS$ is an arbitrary punctured bordered surface having two different ideal triangulations $\lambda$ and $\lambda'$. There is a sequence of flips connecting $\lambda$ and $\lambda'$, where a flip replaces a diagonal of a quadrilateral by the other diagonal. Using the transition isomorphism for the involved quadrilaterals and the local nature of the $X$-algebra $\bsX(\fS, \lambda)$ we can define a transition isomorphism $\rd{\Psi}^X_{\lambda', \lambda}$, which a priori might depend on the sequence of flips connecting $\lambda$ and $\lambda$. Two sequences of flips leading $\lambda$ to $\lambda'$ differ by the so-called pentagonal moves, and using the transition isomorphism for pentagons we will show that the transition map $\rd{\Psi}^X_{\lambda', \lambda}$ does not depend on the choice of the flip sequence.

\subsection{Integrality, GK dimension, orderly finite generation}

Recall that a not necessarily commutative ring $A$ is a domain if $ab=0$ implies $a=0$ or $b=0$.

In order to establish the existence of the quantum trace maps we need to prove that the stated skein algebra $\SS$ and its reduced quotient $\bSS$ for many surfaces are domains. This fact has its own independent interest, and is part of the following.

\begin{thm} [Parts of Theorems \ref{thm.domain} and \ref{thm.domainr}]
\label{th.05}
Assume the ground ring $R$ is a domain, and each connected component of a punctured bordered surface $\fS$ has non-empty boundary.
\begin{enuma}
\item The skein algebra $\SS$ is domain.
\item As $R$-modules both the domain $\SS$ and the target space $\FG(\fS, \lambda)$ of the extended quantum trace are free and have the same Gelfand-Kirillov dimension $r(\fS)$ given by
\begin{equation}\label{eq.rfS0}
r(\fS) = (n^2-1) \left( \# \pfS - \chi(\fS) \right),
\end{equation}
where $\# \pfS$ is the number of connected components of $\pfS$, and $\chi(\fS)$ is Euler characteristic of $\fS$.
\item The algebra $\SS$ is orderly finitely generated, i.e. it has elements $g_1, \dots, g_k$ such that the set $\{ g_1^{m_1}\dots g_k^{m_k}\mid m_i \in \BN\}$ spans $\SS$ over $R$.
\item If $\fS$ is a polygon, then the reduced skein algebra $\bSS$ is a domain.
\end{enuma}
\end{thm}

We conjecture that both $\SS$ and $\bSS$ are domains for any punctured bordered surface. For $n=2$ and $\fS$ a closed surface (without ideal points) the conjecture was proved by Przytycki and Sikora \cite{PS2}. For $n=2$ and $\fS$ has empty boundary but not a closed surface the conjecture was proved by Bonahon and Wong \cite{BW}. For $n=2$ and other surfaces the conjecture was proved in \cite{Le:triangulation} for $\SS$ and in \cite{CL} for $\bSS$. When $n=3$ and each connected component of the surface has at least one ideal point, the conjecture for $\SS$ is proved by Higgins \cite{Higgins}. In Theorem~\ref{thm.n3} we prove the conjecture for $\bSS$ and $n=3$. For further results see \cite{CKL}.

When we are finishing our paper Baseilhac, Faigt, and Roche sent out a preprint \cite{BFR} where Theorem \ref{th.05}(a) is proved for a slightly smaller class of surfaces, but for all simple Lie algebras whereas we consider only $\mathfrak{sl}_n$. On the other hand, for $\mathfrak{sl}_n$ our result is more general in that our ground ring is an arbitrary integral domain $R$, while in \cite{BFR} the ground ring is the field of rational function $\BQ(q)$. Note that if $A$ is a $\BZ[q,q^{-1}]$-domain, then the change of ground ring algebra $A\ot_{\BZ[q,q^{-1}]} R$ might not be a domain, where $R$ is commutative $\BZ[q,q^{-1}]$-domain. In \cite{BFR} it is also proved that under the same assumption (and the over the field $\BQ(q)$) the algebra $\SS$ is finitely generated.

The result of Theorem \ref{th.05}(d), even though applied to polygons only, will be crucial in our paper. In general, the integrality of the reduced skein algebra is more difficult to establish, as it is a quotient algebra. In fact the proof of Theorem \ref{th.05}(d) is one of the most difficult one of the paper.

\subsection{Another approach to the reduced quantum trace for triangle}

The theory of the stated skein algebra allows us to write down a presentation of the reduced skein algebra $\bS(\stdT)$ of the triangle in terms of generators and relations. Thus in order to define, say the $X$-version quantum trace $\btr^X$ in \eqref{eq.trTri}, one can first try to define it for generators and then check that all relations are satisfied. The latter is not easy, as demonstrated in the case $n=2$ in the original proof of the existence of quantum trace by Bonahon and Wong \cite{BW}. For $n=3$ Douglas used computer to check a few, but not all, relations. In the old version of our paper we were able to check all the relations by using
the main results of Chekhov and Shapiro \cite{CS}, which are certain identities for quantum holonomy. This would give a shorter proof of the existence of the reduced $X$-version quantum trace. But with this approach we could not have the injectivity of $\btr^X$, nor the $A$-versions of the quantum trace. Further we could not have the embedding and the naturality for the quadrilateral and pentagon, which are crucial for the proof of the naturality for general surfaces. The new approach in the current paper via quantum torus frame allows us to prove not only the existence of the reduced $X$-version quantum trace, but also many other related facts. Besides, the geometric picture of the quantum torus frame gives a more satisfactory explanation of the nature of the quantum trace maps. The holonomy of paths in \cite{CS} (or the one suggested in \cite{Douglas}), properly normalized, is actually equal to our reduced quantum trace. However, in both \cite{CS} and \cite{Douglas} the $SL_n$ skein algebras are not considered. In particular, there are no sinks and sources and the extra relations coming from them.

Going backwards, with the new approach in the current paper, we can recover the main results of \cite{CS}, see Subsection \ref{ss.CS}.

\subsection{Applications} The quantum traces will have applications in the study of the skein algebras, in particular, the representation theory of $\SS$ and $\bSS$ at roots of unity, and the corresponding TQFTs. We will address these questions in future work.

G. Scharder and A. Shapiro \cite{SS0} (see also \cite{Shen}) showed that there is an embedding of the quantized universal algebra $U_q(\mathfrak{sl}_n)$ into a quantum torus.
In the upcoming work \cite{LS2} the first author and S. Sikora show how to use the quantum trace map for the once-punctured bigon $ \PP_{2,1}$ to recover this result, over the integral ring $\BZ[q, q^{-1}]$. The target space is the quantum torus $ \rd \FG( \PP_{2,1}, \lambda )$, where $\lambda$ is the ``obvious" triangulation.

\subsection{Organization of the paper} Sect.~\ref{sec.alg} contains algebraic background materials. Sect.~\ref{sec.Oqsln} recalls and studies the quantized algebras of regular functions on $SL_n$ and its Borel subgroups. In Sect.~\ref{sec.skein} we define the stated skein algebras and prove a few auxiliary results. Sect.~\ref{sec.pMon} studies the punctured monogon. In Sect.~\ref{sec.Int} we prove the integrality and calculate the GK dimension of stated skein algebra in many cases. Sect.~\ref{sec.Rd} introduces the reduced skein algebra, which is proved to be a domain in an important case of the polygon in Sect.~\ref{sec.Rd2}. In Sect.~\ref{sec-ntri} we recall the Fock-Goncharov $X$-space of the triangle and introduce the $A$-space. Sect.~\ref{sec.qtr3} proves the existence of the $A$- and $X$- quantum traces for the triangle. Sect.~\ref{sec.qtori} recalls the Fock-Goncharov $X$-space of a triangulated surfaces and introduces its extended version as well as the $A$-spaces. Sect.~\ref{sec.Xtrace} proves the existence of the $X$-quantum trace and its extended version. Sect.~\ref{sec.Atrace} proves the existence of the $A$-quantum traces. Sect.~\ref{sec.Nat} proves the naturality of the quantum traces with respect to the change of triangulations. Sec.~\ref{sec.sl3} treats the case of $SL_3$. In Appendix we prove certain matrix identities of surfaces and Theorem~\ref{thm-trX-CS}.

\subsection{Acknowledgments}

The authors would like to thank D.~Allegretti, F. Bonahon, L.~Chekhov, F.~Costantino,
D.~Douglas, V.~Fock, A.~Goncharov, D.~Jordan, H.~Kim, J.~Korinmann, M.~Shapiro,
A.~Sikora, Z. Wang, and H. Wong for helpful discussions. The first author is partially support by NSF grant DMS-2203255.

The authors presented the results of this paper at many seminars and conferences, including Topology seminars at Georgie Washington University (September 2020), Michigan stated University (March 2021), Conference “Quantum Topology and Geometry”, IHP Paris (June 2022), and Conference “Geometric Representation Theory \& Quantum Topology”, Univercite Paris cite (December 2022), and would like to thank the organizers for the opportunities to present their work.

\section{Notations, algebraic preliminaries}\label{sec.alg}

We fix notations and review the theory of quantum tori, the Gelfand-Kirillov dimension, and the skew-Laurent extension. We will introduce the notions \term{quantum torus frame, tensor product factorization}, and \term{quasi-monomial basis} and prove basic facts about them, which will be used extensively in the paper.

\subsection{Notations, conventions} \label{ss.notation}

We denote by $\BN, \BZ, \BC$ respectively the set of non-negative integers, the set of integers, and the set of complex numbers. We emphasize that $0\in\BN$.

Throughout the number $n$ in $SL_n$ is fixed. Let $\JJ= \{1,2,\dots, n\}$. The \term{conjugate} of $i\in \JJ$ is $\bi:= n+1-i$. Let $\Sym_n$ be the group of permutations of $\JJ$. As usual for $\sigma\in \Sym_n$ the \term{length} $\ell(\sigma)$ is the number of inversions of $\sigma$.

We use Kronecker's delta notation and its sibling:
\begin{equation*}
\delta_{i,j} =\begin{cases} 1, & \text{if } j=i, \\0, & \text{if } j \neq i,\end{cases}, \qquad
\delta_{j>i} =\begin{cases} 1, & \text{if } j>i, \\0, & \text{if } j \le i.\end{cases}
\end{equation*}

All rings are associative and unital, and ring homomorphisms preserve the unit. For a subset $S$ of a ring $A$, denote by $A/(S)$ the quotient $A/I$ where $I\lhd A$ is the two-sided ideal generated by $S$.
For a positive integer $r$ and let $\Mat_r(A)$ be the ring of all $r\times r$ matrices with entries in $A$.

\subsection{Ground ring}\label{ss.ground}

The ground ring $R$ is a commutative domain with a distinguished invertible element $\hq$. An example is $R=\Zq$, the ring of Laurent polynomials in the free variable $\hq$ with integer coefficients. Denote $\Fr(R)$ the field of fractions of $R$. All algebras are $R$-algebras where $1\in R$ acts as the identity, and tensor products are over $R$ unless otherwise stated.

The element $q= \hq^{2n^2}$ is the usual quantum parameter appeared in quantum group theory. For a non-negative integer $m$ we define the quantum integer $[m]$ and its factorials by
\[ [m] = \frac{q^m - q^{-m}}{q - q^{-1}}, \quad
[m]!= \prod_{i=1}^m [i], \quad [0]!=1.\]

We will often use the following constants in $R$:
\begin{equation}
\ttt= (-1)^{n-1} q^{\frac{n^2-1}{n}}, \qquad
\aaa= q^{(1-n)(2n+1)/4}, \qquad
\ccc_i= q^{\frac{n-1}{2n}}(-q)^{n-i}, \quad i \in \JJ.
\end{equation}

\subsection{Monomials, Ore sets} \label{sec.mono}

In the remaining part of this section we fix an $R$-algebra $A$. An element $a\in A$ is \term{regular} if it is not a zero divisor, i.e. if $ax =0$ or $xa=0$ then $x=0$. If every non-zero element of $A$ is regular then we call $A$ a \term{domain}, or \term{$R$-domain}.

For a subset $S\subset A$ let $\Mon_m(S)$ be the set of all products of $\le m$ elements of $S$. Let $\Mon(S) =\bigcup_{m=1}^\infty \Mon_m(S)$, whose elements are called \term{$S$-monomials}. The $R$-spans of $\Mon_m(S)$ and $\Mon(S)$ are denoted respectively by $\Pol_m(S)$ and $\Pol(S)$. Note that $\Mon(S)$ is the multiplicative subset generated by $S$ (containing 1) and $\Pol(S)$ is the $R$-subalgebra of $A$ generated by $S$.

The multiplicative subset $\Mon(S)$ is a \term{right Ore set} if for every $s\in \Mon(S)$ and $a\in A$ we have $s A \cap a\Mon(S) \neq \emptyset$ and $s$ is regular. When $\Mon(S)$ is a right Ore set one can define the right quotient algebra $A S^{-1}$ which contains $A$ as a subalgebra, in which every element of $S$ is invertible, and every its element can be presented by $a s^{-1}$ with $a\in A$ and $s\in \Mon(S)$.

If the set of nonzero elements in a domain $A$ is a right Ore set, then $A$ is called an \term{Ore domain}, and $\Fr(A)$ denotes its division ring of fractions.

\subsection{$q$-commuting elements}

For $x,y\in A$ we write $x \eqq y$ if $x = \hq^{2k} y$ for $k\in \BZ$. We say $x,y\in A$ are \term{$q$-commuting} if $xy \eqq yx$. A set $S= \{x_1,x_2,\dots,x_m\}$ is \term{$q$-commuting} if any two its elements are $q$-commuting, i.e. $x_i x_j = \hq^{2k_{ij}} x_j x_i$ for $k_{ij} \in \BZ$. For such a set define the
\term{Weyl-normalization} of $x_1x_2\dots x_m$ by
\[[x_1x_2\dots x_m]_\Weyl =\hq^{-\sum_{i<j}k_{ij}}x_1x_2\dots x_m.\]
It is easy to check that if $\sigma$ is a permutation of $\{1,2,\dots,m\}$, then
\[[x_1x_2\dots x_m]_\Weyl=[x_{\sigma(1)}x_{\sigma(2)}\dots x_{\sigma(m)}]_\Weyl.\]

\subsection{Normal elements} \label{ss.normal}

Suppose $B$ is an $R$-subalgebra of $A$ and $S\subset A$ is a subset. Let $SB$ (respectively $BS$) be the $R$-span of elements of the form $sb$ (respectively $bs$) where $s\in S$ and $b\in B$.

We say $S$ is \term{$B$-normal} if $SB= BS$. In case $S=\{s\}$, we say $s$ is $B$-normal. If $s\in A$ is $A$-normal and regular, then there is an algebra automorphism $\tau_s: A \to A$ such that $sa = \tau_s(a) s$ for all $a\in A$.

An algebra automorphism $f: A \to A$ is \term{diagonal} if there is a set of algebra generators of $A$ consisting of eigenvectors of $f$.

An element $s\in A$ is \term{$q$-commuting} with $B$ if there is a set $G$ of algebra generators of $B$ such that $s$ is $q$-commuting with each element of $G$. Clearly if a regular $s\in A$ is $q$-commuting with $A$ then $s$ is $A$-normal, and $\tau_s$ is a diagonal automorphism.

\subsection{Orderly finitely generated algebra} An $R$-algebra $A$ is {\bf orderly finitely generated} if has elements $g_1, \dots, g_k$ such that the set $\{ g_1^{m_1} \dots g_k^{m_k} \mid m_i \in \BN\}$ spans $A$ over $R$. In that case we say that the sequence $(g_1, \dots, g_k)$ {\bf orderly generates} $A$.

\begin{lemma}\label{r.ogen}
Suppose $A= A_1 A_2$ where each $A_i$ is a subalgebra of $A$ and is orderly finitely generated. Then $A$ is orderly finitely generated.
\end{lemma}

\begin{proof}
If sequences $G_1$ and $G_2$ orderly generate $A_1$ and $A_2$ respectively then the concatenation $G_1 G_2$ orderly generates $A$.
\end{proof}

\subsection{Gelfand-Kirillov dimension}

The Gelfand-Kirillov (GK) dimension is a noncommutative analog of the Krull dimension. It is usually defined when the ground ring is a field. Since our ground ring $R$ is not a field, we will change $R$ to its field of fraction $\Fr(R)$ before defining the GK dimension. Thus, for an $R$-module $M$ define
\[\dim_R M = \dim_{\Fr(R)} (M \ot _R \Fr(R)).\]

Let $A$ be a finitely generated $R$-algebra. Choose a finite set $S$ of $R$-algebra generators. The \term{Gelfand-Kirillov dimension} is defined as
\[\GKdim A = \limsup_{m\to\infty}\frac{\log \dim_R (\Pol_m(S))}{\log m}.\]
It is known that the GK dimension is independent of the choice of $S$.

If a finite set $G$ orderly generates $A$ then it is easy to show that $\GKdim (A) \le |G|$. This is a good intuition about the GK dimension.

\begin{lemma}\label{r.GKdim}
Let $A$ and $B$ be finitely generated $R$-algebras.
\begin{enuma}
\item If $B$ is a subalgebra or a quotient of $A$, then $\GKdim B\le\GKdim A$.
\item Suppose $A$ is a torsion-free $R$-module and a domain, $f: A \to B$ is a surjective algebra homomorphism, and $\GKdim(A)\le\GKdim(B)$, then $f$ is an algebra isomorphism.
\item Suppose $s\in A$ is regular and $q$-commuting with $A$. Then $\{s^k \mid k \in \BN\}$ is a right Ore set of $A$ and $ \GKdim (A\{s\} ^{-1}) = \GKdim (A)$.
\end{enuma}
\end{lemma}

\begin{proof}
(a) is well known \cite[Proposition 8.2.2]{MR}.

(b) Since $A$ is torsion free, the natural map $A \to A \ot_R \Fr(R)$ is injective. The statement is reduced to the case when $R$ is a field, which is assumed now.

By assumption and part (a),
\[\GKdim (A) \le \GKdim(B) = \GKdim(A/\ker f) \le \GKdim(A),\]
which implies $\GKdim (A) = \GKdim(A/\ker f)$. By \cite[Proposition 3.15]{KL}, if an ideal $I \lhd A$ contains a regular element, then $\GKdim(A/I) < \GKdim (A)$. Since $A$ is a domain, any non-zero element of $A$ is regular. Hence $\ker f=\{0\}$. This shows $f$ is injective.

(c) Assume $s$ is $q$-commuting with each $g\in G$, a set of algebra generators of $A$. Let us show that the algebra automorphism $\tau_s: A \to A$, given by $as = \tau_s (a)s$, is \term{locally algebraic} in the sense that any $a\in A$ is contained in a finitely generated $R$-submodule of $A$ which is invariant under $\tau_s$. In fact, since $a$ is a finite $R$-linear combination of $G$-monomials, the $R$-span of the involved monomials is invariant under $\tau_s$ and contains $a$.

By \cite[Theorem 2]{Leroy}, since $\tau_s$ is locally algebraic, $\GKdim (A\{s^{-1}\}) = \GKdim (A)$.
\end{proof}

\subsection{Algebra with Reflection} \label{ss.reflection}

Suppose $R=\BZ[\hq^{\pm 1}]$. An \term{$R$-algebra with reflection} is an $R$-algebra $A$ equipped with a $\BZ$-linear anti-involution $\omega$, called the \term{reflection}, such that $\omega(\hq)=\hq^{-1}$. In other words, $\omega : A \to A$ is a $\BZ$-linear map such that for all $x,y \in A$,
\[\omega(xy)=\omega(y)\omega(x),\qquad \omega(\hq x)=\hq^{-1} \omega(x),\qquad \omega^2=\id.\]

An element $z\in A$ is called \term{reflection invariant} if $\omega(z)=z$. If $B$ is another $R$-algebra with reflection $\omega'$, then a map $f:A\to B$ is \term{reflection invariant} if $f\circ \omega=\omega'\circ f$.

In some calculations, reflection invariance allows us to ignore overall scalars and recover them later.

\subsection{Quantum tori}

The \term{quantum space} and \term{quantum torus} associated to an antisymmetric matrix $Q\in \Mat_r(\BZ)$ are the algebras
\begin{align}
\mathbb{T}_+(Q)& := R\langle x_1,\dots,x_r\rangle/\langle x_ix_j=\hq^{2Q_{ij}}x_jx_i\rangle \label{eq.qspace1} \\
\mathbb{T}(Q)& := R\langle x_1^{\pm1},\dots,x_r^{\pm1}\rangle/\langle x_ix_j=\hq^{2Q_{ij}}x_jx_i\rangle. \label{eq.qtorus1}
\end{align}
We say $A$ is a \term{quantum space (or quantum torus)} on the variables $x_1, \dots, x_r$ if $A= \BT_+(Q)$ (respectively $A= \BT(Q)$) for a certain anti-symmetric $Q$ with the above presentation. All quantum tori and quantum spaces are domains, see e.g. \cite{GW}.

For $\bk=(k_1,\dots, k_r)\in \BZ^r$, let
\[ x^\bk := [ x_1 ^{k_1} x_2^{k_2} \dots x_r ^{k_r} ]_\Weyl = \hq^{- \sum_{i<j} Q_{ij} k_i k_j} x_1 ^{k_1} x_2^{k_2} \dots x_r ^{k_r}\]
be the Weyl normalized monomial. Then $\{x^\bk \mid \bk \in \BZ^r\}$ is a free $R$-basis of $\bT(Q)$, and
\begin{equation}\label{eq.prod}
x^\bk x ^{\bk'} = \hq^{\la \bk, \bk'\ra_Q} x^{\bk + \bk'},\quad
\text{where } \la \bk, \bk'\ra_Q := \sum_{1\le i, j \le r} Q_{ij} k_i k'_j .
\end{equation}
Hence we have the following $\BZ^r$-grading of the algebra $\bT(Q)$:
\begin{equation}\label{eq.grad}
\bT(Q) = \bigoplus_{\bk\in \BZ^r} R x^\bk
\end{equation}

Suppose $Q'\in \Mat_{r'}(\BZ)$ is another antisymmetric matrix such that $HQ' H^T= Q$, where $H$ is an $r\times r'$ matrix and $H^T$ is its transpose. Then the $R$-linear map
\begin{equation}\label{eq.psiH}
\psi_H:\bT(Q)\to \bT(Q'), \quad \psi_H( x^\bk)= x^{\bk H}
\end{equation}
is an algebra homomorphism, called a \term{multiplicatively linear homomorphism}. Here $\bk H$ is the product of the row vector $\bk$ and the matrix $H$.

When $R=\BZ[\hq^{\pm 1}]$, the algebra $\qtorus(Q)$ has a reflection anti-involution $\omega:\qtorus(Q)\to\qtorus(Q)$ defined by
\[\omega(x_i)=x_i.\]
All normalized monomials $x^\bk$ are reflection invariant, and all multiplicatively linear homomorphisms are reflection invariant.

\subsection{Monomial subalgebra}

If $\Lambda\subset \BZ^r$ is a submonoid, then the $R$-submodule $\bT(Q;\Lambda)\subset \bT(Q)$ spanned by $\{x^\bk\mid \bk \in \Lambda \}$ is an $R$-subalgebra of $\bT(Q)$, called a \term{monomial subalgebra}. By \cite[Lemma 2.3]{LY2}, the monomial algebra $\bT(Q;\Lambda)$ is a domain, and its GK dimension is the rank of the abelian group generated by $\Lambda$. In particular, if $Q$ has size $r\times r$, then
\begin{equation}
\GKdim(\mathbb{T}(Q)) = \GKdim(\mathbb{T}_+(Q))=r.
\end{equation}

If $\Lambda$ is $\BN$-spanned by a finite set $G$, then $G$, with any total order, is orderly generating $\bT(Q;\Lambda)$.

\subsection{Embedding into quantum tori}

We introduce the notion of quantum torus frame and show how to use it to embed certain algebras into quantum tori. This approach was first used in \cite{Muller}.

Assume $A$ is an $R$-domain and $S\subset A$ is a subset of non-zero elements. Recall that $\Pol(S)$ is the $R$-subalgebra of $A$ generated by $S$.
Let $\LPol(S)$ be the set of all $a \in A$ for which there is an $S$-monomial $\fm$ such that $a\fm \in \Pol(S)$. In a sense, such an $a$ would be a Laurent polynomial in $S$. If $A= \LPol(S)$ we say $S$ \term{weakly generates} $A$.

\begin{definition}\label{def.qtf}
Let $A$ be an $R$-domain. A finite set $S=\{a_1, \dots, a_r\} \subset A$ is a \term{quantum torus frame} for $A$ if the following conditions are satisfied.
\begin{enumerate}
\item The set $S$ is $q$-commuting and each element of $S$ is non-zero.
\item The set $S$ weakly generates $A$.
\item the set $\{a_1 ^{n_1} \dots a_r^{n_r} \mid n_i \in \BN \}$ is $R$-linearly independent.
\end{enumerate}
\end{definition}

\begin{proposition}\label{r.qtframe}
Let $A$ be an $R$-domain and $S=\{a_1, \dots, a_r\}\subset A$.
\begin{enuma}
\item Suppose $S$ is a quantum torus frame of $A$ with $a_i a_j = \hq^{2Q_{ij}} a_j a_i$, where $Q\in \Mat_r(\BZ)$ is an anti-symmetric matrix. Then there is an $R$-algebra embedding $f: A \embed \bT(Q)$ such that
\begin{equation}\label{eq.23}
\bT_+(Q) \subset f(A) \subset \bT(Q).
\end{equation}
Besides $A$ is an Ore domain and $f$ induces an isomorphism of division algebras
\[\Fr(A) \xrightarrow{\cong} \Fr(\bT(Q)).\]
\item In addition, suppose $R=\ints[\hq^{\pm1}]$ and $A$ has a reflection $\omega'$ such that elements of $S$ are reflection invariant, then the embedding $f$ is reflection invariant.
\item If condition $(3)$ in the definition of a quantum torus frame is replaced with
\begin{center}
$(3')$ the GK dimension of $A$ is $r$,
\end{center}
then $S$ is still a quantum torus frame for $A$.
\end{enuma}
\end{proposition}

\begin{proof}
The $q$-commutation of $S$ implies there is an algebra homomorphism
\[ h: \BT_+(Q) \to A, \quad h(x_i)= a_i.\]
The image of $h$ is $\Pol(S)$, and $h$ maps monomials to monomials.

(a) Condition (3) is equivalent to $h$ is injective, so that $h: \BT_+(Q) \to \Pol(S)$ is bijective. Then (a) is \cite[Proposition 2.2]{LY2}, and it is proved there that $f$ is the unique extension of $h^{-1}:\Pol(S)\to\qplane(Q)$.

(b) By weak generation, for any $a\in A$, there exists normalized monomials $x^\vec{k},x^{\vec{k}_i}\in\qplane(Q)$ and $c_i\in R$ such that
\[ah(x^\vec{k})=\sum_{i=1}^k c_ih(x^{\vec{k}_i}),\qquad
f(a)=\sum_{i=1}^k c_ix^{\vec{k}_i}x^{-\vec{k}}.\]
Clearly, $h$ is reflection invariant. Thus,
\[h(x^\vec{k})\omega'(a)=\sum_{i=1}^k\omega'(c_i)h(x^{\vec{k}_i}).\]
Note $\omega(c_i)=\omega'(c_i)$. Using the definition of $f$,
\[f(\omega'(a))=x^{-\vec{k}}\sum_{i=1}^k\omega'(c_i)x^{\vec{k}_i}=\omega(f(a))\]

(c) In the proof of \cite[Proposition 2.2]{LY2} it is shown that $\Mon(S)$ is a right Ore set for and $A$. Here is a proof: For $b\in \Pol(S)$ and $s\in \Mon(S)$, the $q$-commutation show that there is an element $b^\ast= b^\ast(b,s)\in \Pol(S)$ such that $bs = s b^\ast$. Let $a\in A$ and $s\in \Mon(S)$. By weak generation there is $s'\in \Mon(S)$ such that $b= as'\in \Pol(S)$.
We have
\[ a\Mon(S) \ni as' s = bs = s b^\ast \in s A.\]
This shows $a\Mon(S) \cap s A\neq \emptyset$. Hence $\Mon(S)$ is a right Ore set for $A$.

The embedding $A \embed AS^{-1}$ shows that $\GKdim (AS^{-1}) \ge \GKdim (A)=r$. The universality of the right quotients implies the composition
\[ g: \BT_+(Q) \xrightarrow{h} A \embed AS^{-1}\]
can be extended to an algebra homomorphism $\tilde g: \BT(Q) \to AS^{-1}$. The weak generation implies $\tilde g$ is surjective. As $\GKdim(\BT(Q))= r \le \GKdim(AS^{-1})$ and $\BT_+(Q)$ is a free $R$-module, by Lemma \ref{r.GKdim}(b) the map $\tilde h$ is bijective. It follows that $h$ is bijective, which implies Condition (3).
\end{proof}

\begin{lemma} \label{r.qframe3}
Suppose $S$ is a $q$-commuting set of of non-zero elements of an $R$-domain $A$.
\begin{enuma}
\item If $ a\fm \in \LPol(S)$ where $a\in A$ and $\fm$ is an $S$-monomial, then $a\in \LPol(S)$.
\item The set $\LPol(S)$ is a subalgebra of $A$.
\end{enuma}
\end{lemma}

\begin{proof}
(a) As $a\fm \in \LPol(S)$, there is an $S$-monomial $\fm'$ such that $a \fm \fm' \in \Pol(S)$. Since $S$ is $q$-commuting we have $\fm \fm'\eqq \fm''$ for an $S$-monomial $\fm''$. Since $a \fm''\in \Pol(S)$, we have $a \in \LPol(S)$.

(b) We need to show that if $x, x'\in \LPol(S)$ then $xx'\in \LPol(S)$. There are $S$-monomials $\fm, \fm'$ such that $x \fm, x' \fm'\in \Pol(S)$. Thus $x' \fm' = \sum c_i \fm_i$ where $c_i\in R$ and each $\fm_i$ is an $S$-monomial. As $S$-monomials are $q$-commuting we have $c_i \fm_i \fm = c'_i \fm \fm_i$ with $c_i' \eqq c_i$. Now
\[(x x') (\fm' \fm) = x \Big( \sum c_i \fm_i\Big) \fm = \sum_i c'_i (x \fm) \fm_i \in \Pol(S).\]
This shows $xx'\in \LPol(S)$.
\end{proof}

\subsection{Tensor product factorization} \label{ss.tproduct}

We introduce the notion of tensor product factorization, which will play an important role in the paper.

\begin{definition}\label{def.22}
A \term{tensor product factorization} of an $R$-algebra $A$ is a collection $A_1,\dots, A_k$ of $R$-subalgebras of $A$ such that
\begin{enumerate}[(i)]
\item the $R$-linear map $A_1 \ot \dots \ot A_k \to A$ given by $a_1 \ot \dots \ot a_k \to a_1 \dots a_k$ is bijective,
\item each $A_i$ has a finite set $G_i$ of $R$-algebra generators such that for any two indices $i,j$,
\begin{equation}\label{eq.quad}
\Pol_1(G_i) \Pol_1(G_j) = \Pol_1(G_j) \Pol_1(G_i).
\end{equation}
\end{enumerate}
\end{definition}

If $A_1, \dots, A_k$ form a tensor product factorization of $A$, we will write
\[ A = A_1 \boxtimes \dots \boxtimes A_k.\]

Condition \eqref{eq.quad}, called the \term{quadratic exchange law}, implies
\begin{itemize}
\item [(ii')] $A_i A_j = A_j A_i$.
\end{itemize}
If in Definition \ref{def.22} Condition (ii) is replaced by the weaker (ii'), then we say that $A_1, \dots, A_k$ form a \term{weak tensor product decomposition} of $A$. This notion is equivalent to the notion of ``twisted tensor product" \cite{CSV}.

An example of a tensor product factorization is the $R$-algebra $A_1 \otst \dots \otst A_k$, which is the $R$-module tensor product $A_1 \ot \dots \ot A_k$ equipped with the \term{standard product}, i.e.
\[(a_1 \ot \dots \ot a_k) (a'_1 \ot \dots \ot a'_k)= a_1 a'_1 \ot \dots \ot a_k a'_k.\]
Even for the standard tensor product the GK dimension is not additive. In general,
\[\GKdim (A_1 \otst A_2) \le \GKdim(A_1) + \GKdim(A_2),\]
but we don't have the equality. However, the equality holds under a mild condition, see~\cite{KL}. This mild condition can be easily adapted to the case of tensor product decomposition. For this, we say an $R$-algebra $A$ has \term{uniform GK dimension} if it has a \term{uniform GK set}, which by definition is a finite set $S$ of generators such that
\begin{equation*}
\GKdim A = \lim_{m\to\infty}\frac{\log \dim_R (\Pol_m(S))}{\log m}.
\end{equation*}
Note that on the right-hand side is the ordinary limit, not the superior limit.

\begin{proposition}\label{r.tensorP}
Let $A_1,\dots, A_k$ be a tensor product decomposition of an $R$-algebra $A$. Assume each $A_i$ is finitely generated as an $R$-algebra.
\begin{enuma}
\item If $S_i \subset A_i$ and $S_i$ is $A_j$-normal for all $1\le i, j \le k$, then $A_1/(S_1), \dots, A_k/(S_k)$ form a tensor product factorization of $A/ (S_1 \cup \dots \cup S_k)$.
\item One has $\GKdim (A) \le\sum_{i=1}^k \GKdim(A_i)$.
\item If the each $A_i$ has uniform GK dimension then $\GKdim (A) =\sum_{i=1}^k \GKdim(A_i).$
\end{enuma}
\end{proposition}

\begin{proof}
Let $G_i\subset A_i$ be a finite set of generators for which the quadratic exchange law \eqref{eq.quad} holds.

(a) To simplify the notation we assume $k=2$. The proof for general $k$ is similar.

Let $I_i=S_i A_i = A_i S_i \lhd A_i$ be the ideal generated by $S_i$, and $I\lhd A$ be the ideal generated by $S_1 \cup S_2$. Since $S_iA_j = A_j S_i$ and $A_1 A_2 = A_2 A_1 = A$, we have $I_i A_j = A_j I_i$ and
\[I = (S_1 \cup S_2) A ( S_1 \cup S_2) = \sum _{1\le i,j\le 2} S_i A_1 A_2 S_j = I_1 A_2 + I_2 A_1.\]
Let $\tI_1$ be the image of $ I_1 \ot A_2 \to A_1 \ot A_2$ and $\tI_2$ be the image of $ A_1 \ot I_2 \to A_1 \ot A_2$. We have the following $R$-linear isomorphism
\[A_1/I_1 \ot A_2/I_2 \xrightarrow{\cong} (A_1 \ot A_2)/(\tI_1 + \tI_2) \xrightarrow{f} A/(I_1 A_2 + I_2 A_1) = A/I,\]
where the first map is a known isomorphism in linear algebra and $f$ is the descendant of the isomorphism $A_1 \ot A_2\to A$ given by $a_1 \ot a_2 \to a_1 a_2$. Thus the map
\[A_1/I_2 \ot a_2/I_2 \to A/I, \quad a_1 \ot a_2 \mapsto a_1 a_2,\]
is an $R$-isomorphism.

Let $\bar G_i \subset A_i/I_i$ be the image of $G_i$. The quadratic exchange law for $G_1, G_2$ descends to a quadratic exchange law for $\bar G_1, \bar G_2$. Thus $A_1/I_1$ and $A_2/I_2$ form a tensor product factorization of $A/I$.

(b) Let $G= \bigcup_{i=1}^k G_i$. From the quadratic exchange law \eqref{eq.quad} one has
\[\Pol_m(G) \subset \Pol_m(G_1) \dots \Pol_m(G_k).\]
It follows that
\begin{align*}
\GKdim(A) &=\limsup_{m\to \infty}\frac{\log \dim_R \Pol_m(G)}{\log m} \\
&\le \sum_{i=1}^k \limsup_{m\to \infty}\frac{\log \dim_R \Pol_m(G_i)}{\log m} = \sum_{i=1}^k \GKdim(A_i).
\end{align*}

(c) Let $T_i$ be a uniform GK set for $A_i$, and $T= \bigcup_{i=1}^k T_i$. From
\[\Pol_{km}(T) \supset \Pol_{m}(T_1) \dots \Pol_{m}(T_k),\]
we get
\begin{align*}
\GKdim(A) &=\limsup_{m\to \infty}\frac{\log \dim_R \Pol_{km}(T)}{\log m} \\
&\ge \sum_{i=1}^k \lim_{m\to \infty}\frac{\log \dim_R \Pol_{m}(T_i)}{\log m} = \sum_{i=1}^k \GKdim(A_i).\qedhere
\end{align*}
\end{proof}

\subsection{Skew-Laurent extension} \label{sec.skew}

Suppose $\tau: A \to A$ is an algebra automorphism. The \term{skew-Laurent extension} $A[x^{\pm 1};\tau]$ is an $R$-algebra containing $A$ as a subalgebra and an invertible element $x$ such that
\begin{itemize}
\item as a left $A$-module $A[x;\tau]$ is free with basis $\{x^k \mid k\in \BZ\}$, and
\item for all $a\in A$ we have $a x = \tau(x) a$.
\end{itemize}
Such an algebra exists uniquely. The subalgebra $A[x;\tau]= \bigoplus_{k\in \BN} A x^k$ is called a \term{skew-polynomial extension} of $A$.


\begin{lemma}\label{r.Ore4}
Let $\tau: A \to A$ be an algebra automorphism.
\begin{enuma}
\item If $A$ is a domain then $A[x^{\pm 1};\tau]$ and $A[x;\tau]$ are domain.
\item Suppose $I \lhd A$ is an ideal such that $\tau(I)=I$ where $\tau$ is an automorphism of $A$. Then $A[x^{\pm 1};\tau]/(I) \cong (A/I)[x^{\pm 1};\tau]$.
\item If $\tau$ is locally algebraic then $\GKdim (A[x^{\pm 1};\tau]) = \GKdim (A) +1$.
\end{enuma}
\end{lemma}

\begin{proof}
For (a) see \cite[Corollary I.7.4]{Kass}. Part (b) follows easily from the definition, while (c) is \cite[Proposition 1]{Leroy}.
\end{proof}

\begin{example} \label{exa.001}
Suppose $\tau_1, \dots, \tau_r$ are pairwise commuting algebra automorphisms of $A$. Define the iterated skew-Laurent extensions
\[ A[ x_1^{\pm 1},\dots, x_r^{\pm 1}; \tau_1, \dots, \tau_r]:=A[x_1^{\pm 1};\tilde \tau_1] \dots [x_r^{\pm 1}; \tilde\tau_r],\]
where $\tilde \tau_{i}$ is the algebra automorphism of $A[x_1^{\pm 1};\tilde \tau_1] \dots [x_{i-1}^{\pm 1}; \tilde\tau_{i-1}]$ which is $\tau_i$ on $A$, and $\tilde \tau_{i} (x_k) = x_k$ for $k < i$. It is easy to check that $\tilde \tau_i$ is a well-defined algebra automorphism.

If $S$ is a quantum torus from of $A$ then clearly $S \cup \{x_1, \dots, x_r\}$ is a quantum torus frame of $A[ x_1^{\pm 1},\dots, x_r^{\pm 1}; \tau_1, \dots, \tau_r]$.
\end{example}

\subsection{Quasi-monomial basis}

We introduce the notion of quasi-monomial basis and use it to show that many algebras are domains.

\begin{definition}
\begin{enuma}
\item An \term{enhanced monoid} is a submonoid $\Lambda$ of a free abelian group equipped with a monoid homomorphism $\oo: \Lambda\to \BZ^r$.
\item A set $E$ is a \term{quasimonomial $R$-basis} of an $R$-algebra $A$ if it is a free $R$-basis of $A$ and can be parameterized by an enhanced monoid $(\Lambda,\oo)$, i.e. $E=\{e(m) \mid m \in \Lambda\}$, such that
\begin{equation}\label{eq.ord5}
e(m) e(m') \eqq e(m+m') + A(\oo<m+m'),
\end{equation}
where $A(\oo<k)$ is the $R$-span of $e(k')$ with $\oo(k') \llex \oo(k)$. Here $\llex$ is the lexicographic order on $\BZ^r$.
\end{enuma}
\end{definition}

\begin{proposition}\label{r.domain9}
If an $R$-algebra $A$ has a quasi-monomial basis then $A$ is a domain.
\end{proposition}

\begin{proof}
This follows from a lead term argument, or the theory of filtered algebras.

First assume $\oo: \Lambda\to \BZ^r$ is injective. A non-zero $x \in A$ has a unique presentation $x = \sum _{m\in J} c_m e(m)$, where $J \subset \Lambda$ is a finite non-empty set and $ 0\neq c_m\in R$. Define the lead term $\LT(x)= c_{m_0} e(m_0)$, where $m_0\in J$ has maximum value of $\oo$. From \eqref{eq.ord5} it follows that for non-zero $x,x'\in A$, with $\LT(x) = ce(m)$ and $\LT(x')= c'e(m')$ we have
\[ x x' = cc' e(m+m') + A(\oo<m+m').\]
The right-hand side is an $R$-linear combination of elements of the basis $E$, in which the coefficient of $e(m+m')$ is non-zero. Hence the right-hand side is non-zero, which means $x x'\neq 0$. Thus $A$ is a domain.

Now assume $\oo$ is not injective. Recall that $\Lambda\subset \BZ^k$ for certain $k$. Consider $\oo': \Lambda\embed \BZ^r \times \BZ^k$ given by $\oo'(m) = (\oo(m), m)$. Since $\oo(m) \llex \oo(m')$ implies $\oo'(m) \llex \oo'(m')$, Identity \eqref{eq.ord5} still holds true if $\oo$ is replaced with $\oo'$. As $\oo'$ is injective, by the above case $A$ is a domain.
\end{proof}

\begin{lemma} \label{r.domain2}
Let $A_1$ and $A_2$ form a weak tensor product factorization of an $R$-algebra $A$. Suppose for $i=1,2$ the $R$-algebra $A_i$ has a quasimonomial basis $\{e (m) \mid m \in \Lambda_i \}$, parameterized by the enhanced monoid $(\Lambda_i, \oo_i)$. Assume for $m\in \Lambda_1, r\in \Lambda_2$ we have
\begin{equation}
e(r) e(m) \eqq e(m) e(r) + A_1( \oo_1 < m) A_2. \label{eq.ord7}
\end{equation}
Then $A$ has a quasimonomial basis and hence is a domain. More precisely, the set
\[B=\{e(m) e(r) \mid (m,r) \in \Lambda_1 \times \Lambda_2 \},\]
with the enhancement $\oo= \oo_1\times \oo_2$, is a quasimonomial basis of $A$.
\end{lemma}

\begin{proof}
By the weak tensor product factorization, the set $B$ is a free $R$-basis of $A$. Using \eqref{eq.ord7} and then \eqref{eq.ord5} we have, for $t,m\in \Lambda_1$, $r, s\in \Lambda_2$,
\begin{align*}
(e(t) e(r)) ( e(m) e(s)) &\eqq e(t) e(m) e(r) e(s) + e(t) A_1(\oo_1 < m) A_2\\
&\eqq e(t+m) e(r+s) + A_1(\oo_1 < t+m) A_2\\
& \eqq e(t+m) e(r+s) + A ( \oo < (t+m, r+s)),
\end{align*}
which proves \eqref{eq.ord5} and hence the statement.
\end{proof}

\begin{lemma} \label{r.domain3}
Assume an $R$-algebra $A$ has a quasimonomial basis $\{e(m) \mid m \in \Lambda\}$ parameterized by an enhanced monoid $(\Lambda, \oo)$. Assume an ideal $I \lhd A$ is the $R$-span of $\{e(m) \mid m \in \Lambda \setminus \bar \Lambda\}$, where $\bar \Lambda$ is a submonoid of $\Lambda$. Then the quotient $A/I$ has a quasimonomial basis parameterized by $(\bar \Lambda, \oo)$ and hence is a domain.
\end{lemma}

\begin{proof}
Let $p: A \to A/I$ be the natural projection. Clearly the set $\bar B= \{p(e(m)) \mid m \in \bar \Lambda\}$ is a free $R$-basis of $A/I$.
Apply $p$ to both sides of \eqref{eq.ord5} we get that, for $m,m'\in \bar \Lambda$,
\[ p(e(m)) p(e(m')) \eqq p(e(m+m') + (A/I)(\bar \oo < m+ m').\]
This proves $\bar B$ is a quasimonomial basis of $A/I$.
\end{proof}

\section{Quantized algebras of regular functions on $SL_n$ and its Borel subgroup} \label{sec.Oqsln}

In this section we review the quantized algebra $\Oq$ of regular functions on $SL_n$, which is usually denoted by $\mathcal O_q(SL_n)$ in many texts. We also consider the quotient $\LO$, the quantized algebra of regular functions on the Borel subgroup of $SL_n$. These algebras will be the building blocks for (reduced) stated skein algebra of surfaces. We will show that both $\Oq$ and $\LO$ have quasi-monomial bases, a frequently used fact. We present a quantum torus frame for $\LO$, which will be used for the construction of quantum trace maps later.

Recall that the ground ring $R$ is a commutative domain with a distinguished invertible element $\hq$. Also $\JJ = \{1,2, \dots, n\}$ and $\llex$ is the lexicographic order on $\BZ^r$.

\subsection{Quantum matrices}

\begin{definition}
\begin{enuma}
\item A $k \times m$ matrix with entries in a ring is a called a \term{$q$-quantum matrix} if any $2\times 2$ submatrix $\begin{pmatrix}a & b \\ c& d\end{pmatrix}$ of it satisfies the relations
\begin{equation}\label{eq.Mq2}
ab = q ba, \quad ac = q ca, \quad bd = q db, \quad cd = q dc,\quad
bc = cb, \quad ad - da = (q-q^{-1}) bc.
\end{equation}
\item The \term{$q$-quantum matrix algebra} $\Mqn$ is the $R$-algebra generated by entries $u_{ij}$ of the matrix $\buu= (u_{ij})_{i,j=1}^n$ subject to the relations \eqref{eq.Mq2} for any $2\times 2$ submatrix $\begin{pmatrix}a & b \\ c& d\end{pmatrix}$.
\end{enuma}
\end{definition}

The algebra $\Mqn$ is also known as the quantized algebra of coordinate functions on $n\times n$ matrices. The defining relations of $\Mqn$ can be written by one matrix equation
\begin{equation}
(\buu \ot \buu) \cR = \cR (\buu \ot \buu), \label{e.Mq}
\end{equation}
where $\buu \ot \buu$ is the $n^2\times n^2$ matrix with entries $(\buu \ot \buu)^{ik}_{jl} := u_{ij} u_{kl}$ for $i,j,k,l\in \JJ $ and $\cR$ is the $n^2 \times n^2$ matrix given by
\begin{equation}\label{e.R}
\cR^{ij}_{lk} = q^{-\frac 1n} \left( q^{\delta_{i,j}} \delta_{j,k} \delta_{i,l} + (q-q^{-1})
\delta_{j<k} \delta_{j,l} \delta_{i,k}\right).
\end{equation}
This is the $R$-matrix of the fundamental representation of $\mathfrak{sl}_n$, cf. \cite[Equ. 8.4.2(60) and Section 9.2]{KS}.

The defining relations can also be rewritten as follows. For $i,j,k,l\in \JJ$ let
\[C_{ij,kl} := \delta_{ik} + \delta_{i<k} \delta_{jl}.\] Then the defining relation \eqref{e.Mq} is equivalent to: for $(i,j) <_{\text{lex}} (k,l)\in \JJ^2$,
\begin{equation}\label{eq.M110}
u_{ij} u_{kl}- q^{C_{ij,kl}} u_{kl} u_{ij}= \delta_{i<k} \delta_{j<l}(q-q^{-1}) u_{il} u_{kj}.
\end{equation}

The \term{quantum determinant} of the $q$-quantum matrix $\buu$, defined by
\begin{equation}\label{eq.det}
\detq(\buu)
:=\sum_{\sigma\in \Sym_n} (-q)^{\ell(\sigma)}u_{1\sigma(1)}\cdots u_{n\sigma(n)}
=\sum_{\sigma\in \Sym_n} (-q)^{\ell(\sigma)}u_{\sigma(1)1}\cdots u_{\sigma(n)n},
\end{equation}
is a central element in $\Mqn$, cf. \cite[9.2.2]{KS}.

The \term{adjugate} $\buu^!\in \Mat_n(\Mqn)$ is the $n\times n$ matrix with entries
\begin{equation}
(\buu^!)_{ij} = (-q)^{i-j} {\det}_q(\buu^{ji}),
\end{equation}
where $\buu^{ji}$ is the result of removing the $j$-th row and the $i$-column from $\buu$. Then $\buu^!$ is $q^{-1}$-quantum and
\begin{equation}\label{eq.adjugate}
\buu^! \buu = \buu \buu^! = ({\det}_q \buu) \id.
\end{equation}

\subsection{Cramer's rule}

We use \eqref{eq.adjugate} to solve linear equations.

\begin{proposition}\label{e.Cramer}
Let $M'=[c|M]$ be an $n \times (n+1)$ $q$-quantum matrix with entries in a ring $A$, with the last $n$ columns forming a submatrix $M$ and the first column being $c$. Let $M_i$ be the result of removing the $(i+1)$-th column from $M'$. Assume ${\det}_q(M) $ is invertible in $A$. Suppose $x= (x_1,x_2, \dots, x_n)^T$ is a column of elements of $A$. Then $Mx =c$ if and only if
\[ x_i = (-q)^{i-1} ({\det}_q(M))^{-1} {\det}_q(M_i)\quad \text{for all }i=1,\dots, n.\]
\end{proposition}

\begin{proof}
For a square matrix $X$ let $X^{ji}$ be the result of removing the $j$-th row and the $i$-th column from $X$. Note that $M^{ji} = (M_i)^{j1}$. From \eqref{eq.adjugate},
\begin{align*}
M x = c &\Longleftrightarrow \id x = ({\det}_q(M))^{-1} M^! c \\
& \Longleftrightarrow x_i = ({\det}_q(M))^{-1} \sum_{j=1}^n (-q)^{i-j} {\det}_q (M^{ji}) c_j \\
& \Longleftrightarrow x_i = (-q)^{i-1} ({\det}_q(M))^{-1} \sum_{j=1}^n (-q)^{1-j} {\det}_q ((M_i)^{j1}) c_j \\
& \Longleftrightarrow x_i = (-q)^{i-1} ({\det}_q(M))^{-1} {\det}_q(M_i).\qedhere
\end{align*}
\end{proof}

\subsection{The quantized algebra of regular functions on $SL_n$}

The quotient
\[\Oq := \Mqn/({\det}_q \buu-1)\]
is known as the \term{quantized algebra of regular functions on $SL_n$}. By abuse of notation, we denote the image of $u_{ij}\in \Mqn$ under the natural projection $\Mqn \to \Oq$ also by $u_{ij}$.

It is known that $\Oq$ is Hopf algebra (see e.g. \cite[9.2.2]{KS}) where the comultiplication, the counit, and the antipode are given by
\begin{align}
\Delta(u_{ij}) & = \sum_{k=1}^n u_{ik} \ot u_{kj}, \quad \ve(u_{ij})= \delta_{ij}.\label{eq.Deltave}\\
S({u}_{ij} )&= (\buu^!)_{ij} = (-q)^{i-j} {\det}_q(\buu^{ji}).\label{e-Oq-ops}
\end{align}
Here $\buu^{ji}$ is the result of removing the $j$-th row and $i$-th column from $\buu$.

\subsection{Degrees and filtrations}

Define three degrees $d_0, d_1$, and $d_2$ by
\begin{equation}
d_0(u_{ij})=1, \ d_1(u_{ij})=i-j, \ d_2 (u_{ij}) = 6ij - (n+1) (2n+1).
\end{equation}
Then for each $i=1,2,3$ and a word $w$ in the letters $\{ u_{ij}\}$ define $d_i(w)$ additively, i.e. if $w= u_{i_1 j_1} \dots u_{i_k j_k}$ then $d_i(w) = d_i( u_{i_1 j_1}) + \dots d_i( u_{i_k j_k})$. Note that $d_0(w)\in \BN$ is the length of the word $w$.

\begin{proposition}\label{r.d1}
 The degree $d_1$ descends to a $\BZ$-grading of the algebra $\Oq$. That is,
\begin{equation}
\Oq = \bigoplus_{k\in \BZ} \Oq_{d_1=k}, \quad \Oq_{d_1=k}\, \Oq_{d_1=k'} \subset
 \Oq_{d_1=k+k'},
\end{equation}
where $\Oq_{d_1=k}:= R\text{-span of} \ \{[w] \mid d_1(w)=k)\}$.
\end{proposition}

\begin{proof}
One easily sees that $d_1$ respects the defining relations \eqref{eq.M110} and the relation $\det_q(\buu)=1$.
\end{proof}

\subsection{Quasi-monomial bases of $\Oq$}

We present now a quasi-monomial basis for $\Oq$. To parameterize a basis of $\Oq$, consider the monoid
\begin{equation}\label{eq.Gamma}
\Gamma = \Mat_n(\BN)/ (\Id).
\end{equation}
Here $\Mat_n(\BN)=\BN^{n\times n}$ is an additive monoid, and $(\Id)$ is the submonoid generated by the identity matrix. Two matrices $m,m'\in \Mat_n(\BN)$ determine the same element in $\Gamma$ if and only if $m-m'= k \Id$ for $k\in \BZ$. Each $m\in \Gamma$ has a unique lift $\hat m\in \Mat_n(\BN)$, called the \term{minimal representative}, such that $\min_{i} \hat m_{ii}=0$. Note that $\Gamma\cong \BN^{n^2-n}\oplus \BZ^{n-1}$, hence it is a submonoid of a free abelian group.

\begin{proposition}[Theorem 2.1 of \cite{Gavarini}] \label{r.basisOq}
For any linear order $\ord$ on $\JJ^2$, the set
\begin{equation}\label{eq.Bord}
B^\ord:=\{b(m):= \prod_{(i,j)\in \JJ ^2} u_{ij}^{\hat m _{ij}} \mid m \in \Gamma = \Mat_n(\BN)/(\Id)\}
\end{equation}
where the product is taken with respect to the order $\ord$, is a free $R$-basis of $\Oq$. Consequently $\Oq$ is orderly finitely generated.
\end{proposition}

Let $d_2: \Mat_n(\BN)\to \BZ$ be the $\BN$-linear map defined by: For $m= (m_{ij})_{i,j=1}^n \in \Mat_n(\BN)$ let
\begin{equation}
d_2(m):= d_2 (u_{ij}^{m _{ij}}) = \sum_{i,j} [6ij - (n+1) (2n+1)] m_{ij}. \label{eq.d2}
\end{equation}
The term $-(n+1)(2n+1)$ was added in the definition of $d_2$ so that $d_2(\Id)=0$. Hence $d_2$ descends to a monoid homomorphism, also denoted by $d_2: \Gamma \to \BZ$. Consider the enhanced monoid $(\Gamma, d_2)$. Recall that the ground ring $R$ is a commutative domain.

\begin{proposition}\label{r.Oq} Let $\ord$ be a linear order on $\JJ ^2$.

\begin{enuma}
\item The algebra $\Oq$ has uniform GK dimension $n^2-1$, with uniform GK set
$$G=\{u_{ij} \mid i,j \in \JJ^2\}.$$
\item The set $B^\ord$, parameterized by the enhanced monoid $(\Gamma,d_2)$, is a quasi-monomial basis of $\Oq$.
\end{enuma}
\end{proposition}

\begin{proof}
(a) As $B^\ord$ is a free $R$-basis of $\Oq$, the set $\Pol_k(G)$ is the $R$-module freely spanned by monomials $\prod u_{ij}^{m_{ij}}$ of total degree $\le k$, with one of $m_{ii}$ equal to 0. Hence,
\[ \dim_R(\Pol_k(G)) = \Big| \{m \in \Mat_n(\BN) \mid \sum m_{ij} \le k, \min_i m_{ii} =0\} \Big|\]
Then $ f_k \le \dim_R(\Pol_k(G)) \le n f_k$, where
\[ f_k = \Big| \{m \in \Mat_n(\BN) \mid \sum m_{ij} \le k, m_{11} =0\} \Big|\]
Since $f_k$ is the dimension of the space of polynomials in $n^2-1$ commutative variables of totals degrees $\le k$, we have
\[ \lim_{k \to \infty} \frac{\log f_k}{k}= n^2-1.\]
By sandwich limit theorem, we also have
\[ \lim_{k\to \infty} \frac{\log \dim_R (\Pol_k(S))}{k}= n^2-1.\]
This show $G$ is a uniform GK set for $\Oq$, and that $\GKdim(\Oq))= n^2-1$.

(b) We need to prove that for $m,m'\in \Gamma$,
\begin{equation}\label{eq.filter9a}
b(m) b(m') \eqq b(m+m') \mod \Oq (d_2 <m+m').
\end{equation}
For a word $w$ in $u_{ij}$ let $\tilde m(w)_{ij}\in \BN$ be the number of times $u_{ij}$ appears in $w$. Let $m(w)\in \Gamma$ be the element determined by $\tilde m(w)\in \Mat_n(\BN)$.

\begin{lemma} \label{r.ord001}
For a word $w$ in the letters $\{u_{ij}\}$ one has
\begin{align}
\label{eq.boo2}
[w] &\eqq b(m(w)) \mod \Oq(d_2 < m(w)) \\
[w] &\in R\text{-span of} \ \{ b(m) \mid d_0 (b(m)) \le d_0(w) \}. \label{eq.boo5}
\end{align}
\end{lemma}

\begin{proof}
Let $d_{02}(w)= (d_0(w), d_2(w))\in \BN \times \BZ$. We prove \eqref{eq.boo2} and \eqref{eq.boo5} by induction on $d_{02}(w)$, using the partial order on $\BN \times \BZ$ defined by $(k,l) \le _\parti (k',l')$ if $k\le k'$ and $l\le l'$. Note that, because of the presence of $d_0$, there are only a finite number of words $w'$ such that $d_{02} (w') <_\parti d_{02}(w)$. The base case, when $d_0(w)=0$, is trivial since $w$ is the empty word.

By \eqref{eq.M110} the defining relations of $\Oq$ are, for $(i,j) <_{\text{lex}} (k,l)$ and $C_{ij,kl} := \delta_{ik} + \delta_{i<k} \delta_{jl}$,
\begin{align}
u_{ij} u_{kl}- q^{C_{ij,kl}} u_{kl} u_{ij}&= \delta_{i<k} \delta_{j<l}(q-q^{-1}) u_{il} u_{kj}, \label{eq.M11} \\
1- u_{11}\dots u_{nn}&= \sum_{\id \neq \sigma \in \Sym_n} (-q)^{\ell(\sigma)} u_{1 \sigma(1)} \dots u_{n \sigma(n)}. \label{eq.M12}
\end{align}
The main property of $d_2$ is that in each equation, all the monomials in the left-hand side have the same $d_{2}$, which is higher than $d_2$ of any monomial in the right-hand side. For \eqref{eq.M12} this is true due to the Cauchy-Schwarz inequality.

Relation \eqref{eq.M11} shows that if $w'$ is a permutation of $w$, then
$$[w]\eqq [w'] + \boo_{02}(w),$$ where $\boo_{02}(w)$ stand for an $R$-linear combination of $[w']$ with $d_{02}(w') <_\parti d_{02} (w)$. Permutations and Relation \eqref{eq.M12} shows that if $\tilde m(w)_{ii} \ge 1$ for all $i$ then $[w]\eqq [w'] + \boo_{02}(w)$, where $\tilde m(w')= \tilde m - \Id$. Combining the two operations we get
\begin{equation}\label{eq.36}
[w]\eqq b(m(w)) + \boo_{02}(w).
\end{equation}
Induction on $d_{02}(w)$ we get both \eqref{eq.boo2} and \eqref{eq.boo5}.
\end{proof}

Return to the proposition. Let $m,m'\in \Gamma$. Assume $b(m)$ and $b(m')$ are represented by words $w, w'$ respectively. Since $b(m) b(m')= [ww']$ from \eqref{eq.boo2} we have
\[b(m) b(m')\eqq b(m+m') \mod \Oq(d_2 < m+m').\]
This completes the proof of the proposition.
\end{proof}

As a corollary, we get the following well-known result.
\begin{corollary}\label{r.Oq000}
Over any ground ring $R$ which is a domain, the algebra $\Oq$ is a domain.
\end{corollary}

\begin{remark}
As far as we know, for arbitrary domain $R$, this result was first proved in \cite{Levasseur}.
\end{remark}

Let us record here variations of several facts we just proved. Define $d_1: \Gamma \to \BZ$ by
\begin{equation}\label{eq.d1}
d_1(m):= d_1( \prod u_{ij}^{\hat m_{ij}} )= \sum_{ij} (i-j) \hat m_{ij}
\end{equation}
where $\hat m\in \Mat_n(\BN)$ is a lift of $m$. Clearly $d_1$ is well-defined. Let $d_{12}=(d_1, d_2): \Gamma\to \BZ^2$ and $d_{01}(w)=(d_0(w), d_1(w))\in \BN \times \BZ$.

\begin{corollary} \label{r.boo12}
Suppose $m,m'\in \Gamma$ and $w$ is a word in the letters $u_{ij}$. Then
\begin{align}
\label{eq.boo12}
b(m) b(m') & \eqq b(m+m') + \Oq (d_{12} \llex m+m')\\
[w] &\in R\text{-span of} \ \{ b(m) \mid d_{01} (b(m)) \lelex d_{01}(w) \}. \label{eq.boo5a}
\end{align}
\end{corollary}

\begin{proof} Since $d_1$ gives a $\BZ$-grading on $\Oq$ (by Proposition \ref{r.d1}), all the terms in \eqref{eq.filter9a} can be assumed to have the same $d_1$, in which case it implies \eqref{eq.boo12}. Similarly \eqref{eq.boo5} implies~\eqref{eq.boo5a}.
\end{proof}

\subsection{The quantized algebra \LOsec of regular functions on the Borel subgroup} \label{sec.Bq}

Let $G^-= \{u_{ij}\in \Oq \mid i<j, \, i, j \in \JJ \}$ and $\cI^-\lhd \Oq$ be the 2-sided ideal generated by $G^-$. Then
\[\LO:= \Oq/\cI^-\]
is known as the quantized algebra of regular functions on the Borel subgroup of $SL_n$. Let $\bar u_{ij}\in \LO$ be the image of $u_{ij}$. Since $\bar u_{ij}=0$ if $i<j$, the $q$-quantum matrix $\bar \buu= ( \bar u_{ij})_{i,j=1}^n$ is lower triangular.

\begin{proposition}\label{r.redOq0}
The following holds in $\Oq$.
\begin{enuma}
\item For $i,j,k \in \JJ =\{1,\dots, n \}$,
\begin{align}
\bar u_{ii} \, \bar u_{jj} & = \bar u_{jj} \, \bar u_{ii} \label{eq.iijj}\\
\bar u_{ii} \, \bar u_{jk} &\eqq \bar u_{jk} \, \bar u_{ii} \label{eq.iijk}\\
\prod_{i=1}^n \bar u_{ii} &=1. \label{eq.diag0}
\end{align}
Consequently each $\bar u_{ii}$ is invertible.
\item The ideal $\cI^-$ is a Hopf-ideal of $\Oq$, i.e.
\begin{align}
\epsilon(\cI^-)&=0,\quad
\Delta(\cI^-) \subset \Oq \ot \cI^- + \cI^-\ot \Oq, \label{eq.Hopfideal1}\\
S(\cI^-)&= \cI^-. \label{eq.HopfS}
\end{align}
Consequently $\LO$ inherits a Hopf algebra structure from $\Oq$.
\item The set $G^-$ is $\Oq$-normal. In other words $G^- \Oq = \Oq G^- = \cI^-$.
\end{enuma}
\end{proposition}

\begin{proof}
(a) Identities \eqref{eq.iijj} and \eqref{eq.iijk} follow from Relation \eqref{eq.M11}, while Identity \eqref{eq.diag0} follows from Relation \eqref{eq.M12}, taking into account $\bar u_{ij} = 0$ for $i<j$.

(b) Assume $i<j$. By definition $\ve(u_{ij}) = \delta_{ij}=0$. Hence $\epsilon(\cI^-)=0$.

For any $k$, either $i<k$ or $k<j$, hence
\[ \Delta(u_{ij}) = \sum_k u_{ik} \ot u_{kj} \in \Oq \ot \cI^- + \cI^- \ot \Oq.\]
This proves \eqref{eq.Hopfideal1}. By \eqref{e-Oq-ops},
\[ S(u_{ij})= (-q)^{i-j} {\det}_q (\buu^{ji}).\]
The fact $ i < j$ implies the submatrix $\bar \buu^{ji}$ is lower triangular and having a 0 on its diagonal. Hence ${\det}_q(\bar \buu^{ji})=0$. This shows ${\det}_q (\buu^{ji})\in \cI^-$, proving \eqref{eq.HopfS}.

(c) Let $x= u_{ij}$ with $i-j<0$. We need to show $x \Oq\subset \Oq G^-$ and $\Oq x \subset G^- \Oq$. Since $\{u_{kl}\}$ is a set of generators, it is enough to show that for arbitrary $y= u_{kl}$ we have
\begin{equation}\label{eq.35}
xy \in \Oq G^-, \quad yx \in G^- \Oq.
\end{equation}
Let $M$ be a $2\times 2$ submatrix of $\buu$ containing $x$ and $y$. If one of $x,y$ is not on the diagonal of $M$ then they $q$-commute, and \eqref{eq.35} is true. Assume $x,y$ are on the diagonal of $M$. Let $z$ be the top right corner entry and $v$ be the bottom left corner of $M$. Since $x\in G^-$ we must have $z\in G^-$. By the 5-th identity of \eqref{eq.Mq2} we have $zv= vz \in G^- \Oq \cap \Oq G^-$. By the 6-th identity of \eqref{eq.Mq2} we have
\[ xy -yx = \pm (q - q^{-1}) zv,\]
from which we have \eqref{eq.35}.
\end{proof}

\subsection{Quasi-monomial basis of \LOsec}

We show that a subset of the quasimonomial basis $B^\ord$ of $\Oq$ given by Proposition \ref{r.Oq} descends to a quasimonomial basis of $\LO$. In particular this will imply that \LOsec is a free $R$-module, and is a domain.

Recall that $B^\ord$, where $\ord$ is a linear order of $\JJ^2$, is parameterized by $\Gamma$,
\[ B^\ord = \{b(m)
\mid m \in \Gamma= \Mat_n (\BN)/(\Id)\}.\]
Let $\bar b(m)$ be the image of $b(m)$ under the projection $\Oq \to \LO$.

Consider the submonoid $\bG\subset \Gamma$ consisting of upper triangular matrices
\begin{equation}
\bG = \{m \in \Mat_n (\BN) \mid m_{ij} =0 \text{ if } i< j\}/(\Id) \subset \Gamma.
\end{equation}
We enhance $\bG$ by $d_2: \bG \to \BZ$, which is the restriction of $d_2: \Gamma \to \BZ$.

\begin{proposition}\label{r.redOq}
Let $\ord$ be a linear order of $\JJ^2$.
\begin{enuma}
\item The set $\bar B^\ord=\{\bar b(m) \mid m \in \bar \Gamma \}$ is a quasimonomial basis of $\LO$ parameterized by the enhanced monoid $(\bG, d_2)$. Consequently $\LO$ is a domain.
\item The algebra $\LO$ has uniform GK dimension $(n-1)(n+2)/2$.
\end{enuma}
\end{proposition}

\begin{proof}
(a) We will show that $B^- := \{b(m) \mid m \in \Gamma \setminus \bG\}$ spans $\cI^-$ over $R$. Then Lemma \ref{r.domain3} proves part (a).

Since $m\in \Gamma \setminus \bG$ if and only if $m_{ij} = 0$ for some $i<j$, we have $ B^- \subset \cI^-$.

We need to revisit the proof of Proposition \ref{r.Oq}(c), and use the notations therein. Let $W^-$ be the set of all words $w$ in letters $u_{ij}$ containing at least one letter in $G^-$, i.e. $\tilde m(w)_{ij} >0$ for some pair $(i,j)$ with $ i<j$. By definition $\cI^-$ is spanned by $\{[w] \mid w\in W^-\}$. Let us look at the process of bringing $w$ to $b(m(w))$ using Relation \eqref{eq.M11} and \eqref{eq.M12}. Each monomial of the right-hand side of \eqref{eq.M12} is in $W^-$, while if a monomial in the left-hand side of \eqref{eq.M11} is in $W^-$ then so is the monomial in the right-hand side. Hence the proof of Identity \eqref{eq.36} shows that if $w\in W^-$, then
\[ [w] \eqq b(m(w)) + \boo_{02}(w),\]
where $\boo_{02}(w)$ is an $R$-linear combination of $w'\in W^-$ with $d_{02}(w') <_\parti d_{02}(w)$. Clearly for $w\in W^-$ we have $b(m(w)) \in B^-$. Hence by induction on $d_{02}(w)$ we can express $w$ as an $R$-linear combination of elements of $B^-$. Thus $\cI^-$ is spanned by $B^-$.

(b) The proof is identical to that of Proposition \ref{r.Oq}(b), except for the number of variables: Let $\bar G=\{\bar u_{ij} \mid 1 \le j \le i \le n\}$. As $\bar B^\ord$ is a free $R$-basis of $\LO$, the set $\Pol_k(\bar G )$ is the $R$-module freely spanned by monomials $\prod \bar u_{ij}^{m_{ij}}$ of total degree $\le k$, with one of $m_{ii}$ equal to 0, and $m_{ij} =0$ for $i<j$. By considering cases $m_{ii}=0$, we get
\[ f _k \le \dim_R \Pol_k(\bar G ) \le n f_k,\]
where $f_k$ is the dimension of space of polynomials in $(n-1)(n+2)/2$ variables of totals degrees $\le k$. Hence
\[ \lim_{k\to \infty} \frac{\log \dim_R (\Pol_k(\bar G))}{k}= (n-1)(n+2)/2.\]
This show $\bar G$ is a uniform GK set for $\LO$, and that $\GKdim(\Oq))= (n-1)(n+2)/2$.
\end{proof}

\subsection{Quantum torus frame for \LOsec}

\begin{theorem}\label{thm.qtorus1}
For $j \le i \in \JJ= \{1, \dots, n\}$ let
\begin{equation}\label{eq.bvij}
\bar v_{ij} = M ^{[i; n]}_{[j; j+\bar i -1]}(\bbuu),
\end{equation}
where $[k;l]= \{m \in \JJ \mid k \le m \le l \}$, and $M^I_J(\bbuu)$ is the quantum determinant of the $I \times J$ submatrix of $\bbuu$. Then $\TT=\{\bar v_{ij} \mid 1 \le j \le i \le n, i\ne1 \}$ is a quantum torus frame for $\LO$.
\end{theorem}

\begin{proof}
First we prove that $\TT$ is a $q$-commuting set. This follows immediately from a known criterion for the $q$-commutation of two quantum minors. More precisely, from
\cite[Identity 3.13]{Goodearl}, we have: If $i \le i'$,
\begin{equation}\label{eq-v-comm}
\bar v_{ij} \bar v_{i',j'} = q^{\sign(j-j')\, \abs{J'\setminus J}} \bar v_{i',j'} \bar v_{ij}.
\end{equation}
where $\sign(x)=1,0,$ or $-1$ according as $x >0, x=0$, or $x <0$ respectively, and $J= [j; j+\bar i -1], J'= [j'; j'+\bar {i'} -1]$.

Let us prove $\LPol(\TT) = \LO$. By Lemma \ref{r.qframe3} it is enough to show that each generator $\bar u_{ij}$, with $i\ge j$, is in $\LPol(\TT)$. We use induction on the lexicographic order of $(i,j)$, beginning with $(i,j)= (n,n)$ and going down. Since $\bar u_{nj} =\bar v_{nj} \in \LPol(\TT)$ we will assume $i<n$. Let $\LO_{> ij}$ be the subalgebra generated by $\bar u_{i'j'}$ with $(i,j)\llex(i',j')$. By Laplace's expansion along the first row, see \eqref{eq.adjugate}, of the quantum determinant in \eqref{eq.bvij}, we have
\[ \bar v_{ij} = \bar u_{ij} \bar v_{i+1, j+1} \mod \LO_{> ij}.\]
By induction hypothesis we have $\LO_{> ij} \subset \LPol(\TT)$. Hence $ \bar u_{ij} \bar v_{i+1, j+1} \in \LPol(\TT)$. Then by Lemma~\ref{r.qframe3}, we have $ \bar u _{ij} \in \LPol(\TT)$. By induction all $\bar u_{ij}$ with $1 \le j \le i \le n$ and $(i,j) \neq 1$ is in $\LPol(\TT)$. Note that $\bar v_{2,2}= \bar u_{22} \dots \bar u_{nn}$. Since $\bar u_{11} \bar v_{2,2}=1\in \LPol(\TT)$, we also have $\bar u_{11} \in \LPol(\TT)$. This completes the proof that $\LPol(\TT) = \LO$.

Let us now prove each $\bar v_{ij}\in \TT$ is not $0$. In the determinant formula \eqref{eq.det} for $\bar v_{ij}$, each monomial in the right-hand side is either 0 or an element of the basis $\bar B^\ord$, where $\ord$ is the lexicographic order on $\BN^2$. Moreover, all the non-zero monomials are distinct elements of $\bar B^\ord$. One of them is non-zero, for example the diagonal monomial. Hence $\bar v_{ij}\neq 0$.

The set $\bar G$ is $q$-commuting, consisting of non-zero elements, and weakly generating $\LO$. Besides $|\bar G|= (n-1)(n+2)/2 =\GKdim (\LO)$. By Proposition \ref{r.qtframe}, the set $\bar G$ is a quantum torus frame for $\LO$.
\end{proof}

\section{Stated $SL_n$-skein algebra} \label{sec.skein}

In this section we recall the definition of the stated $SL_n$ skein algebra \cite{LS} and survey its main properties. To each boundary edge of the surface we introduce a $\BZ^{n-1}$-grading which will play an important role later. We also establish several $q$-commutation results which are consequences of the upper triangular nature of the braiding (or $R$-matrix).

Recall that the ground ring $R$ is a commutative domain with a distinguished invertible element $\hq$.

\subsection{Punctured bordered surface and $n$-web}

\begin{definition}
A \term{punctured border (pb) surface $\fS$} is a surface of the form $\fS =\bfS\setminus \cV$, where $\bfS$ is a compact oriented 2-dimensional manifold with (possibly empty) boundary $\partial \bfS$, and $\marked\subset \bfS $ is a finite set such that every component of $\partial \bfS $ intersects $\cV$. Each connected component of $\partial \fS =\partial \bfS \setminus \marked$ is diffeomorphic to the open interval $(0,1)$ and is called a \term{boundary edge}. A point $x\in \cV$ is called an \term{ideal point}, or a \term{puncture}, of $\fS$. A puncture on $\partial \bfS$ is called a \term{vertex}.

A pb surface $\fS$ is \term{essentially bordered} if every connected component of it has non-empty boundary.
\end{definition}

An \term{ideal arc} in $\fS$ is an embedding $c:(0,1)\embed\fS$ which can be extended to an immersion $\bar c : [0,1] \to \bfS$ such that $\bar c(0), \bar c(1) \in \marked$. An ideal arc $c$ is \term{trivial} if the extended map $\bar c$ can be homotoped relative its boundary to a point.

A closed interval properly embedded in $\fS$ is called a \term{$\pfS$-arc}. A $\pfS$-arc is \term{trivial} if it is homotopic relative its boundary points to a subinterval of $\pfS$.

The thickening of $\fS$ is the oriented 3-manifold $\tfS := \fS \times (-1,1)$. We often identify $\fS$ as the subset $\fS \times \{0\}$ of $\tfS $. For a point $(x,t)\in \tfS = \fS \times (-1,1)$, its height is $t$. A vector at $(x,t)$ is \term{upward vertical} if it is along the positive direction of the component $(-1,1)$. We denote by $\pr: \tfS \to \fS$ the projection onto the first component. If $b$ is a boundary edge of $\fS$ then $\tilde b:= b \times (-1,1)$ is called a \term{boundary wall} of $\tfS $. The boundary $\partial \tfS$ of $\tfS$ is the union of all the boundary walls.

\begin{definition}\label{d.n-web}
An \term{$n$-web} over $\fS$ is a set $\al\subset \tfS= \fS \times (-1,1)$ each connected component of which is either a properly embedded oriented circles and or a finite directed graph satisfying
\begin{enumerate}
\item Every vertex is either $1$-valent or $n$-valent. Each $n$-valent vertex is a sink or a source. We denote set of 1-valent vertices, called \term{endpoints} of $\al$, by $\pal$.
\item Each edge of the graph is a smooth embedding of the closed interval $[0,1]$ into $\tfS$.
\item $\alpha$ is equipped with a \term{framing}, which is a continuous non-vanishing vector field transversal to $\alpha$. In particular, the framing at a vertex is transversal to all incident edges.
\item The set of half-edges at every $n$-valent vertex is cyclically ordered.
\item $\al\cap \partial \tfS= \pal$, the framing at an endpoint is upward vertical, and on each boundary wall the endpoints of $\al$ have distinct heights.
\end{enumerate}
\end{definition}

The points of $\pal$ over a boundary edge $b$ is ordered by their heights. Together they give a partial order on $\pal$, where two points are comparable if and only if they are in the same boundary wall.

We consider $n$-webs up to \term{isotopy} which are continuous deformations of $n$-webs in their class. By convention, the empty set is considered as an $n$-web which is isotopic only to itself. Any isotopy preserves the height order.

Every $n$-web can be isotoped to a \term{vertical position}, where
\begin{itemize}
\item the framing is upward vertical everywhere,
\item $\al$ is in general position with respect to the projection $\pr: \tfS \to \fS$, and
\item at every $n$-valent vertex, the cyclic order of half edges, after projected onto $\fS$, is the positive orientation of $\fS$ (counterclockwise if drawn on the pages of the paper).
\end{itemize}

\begin{definition}
Suppose $\al$ is an $n$-web in vertical position.
The projection $D=\pr(\al)$, together with the usual over/underpassing at each double point, and the partial order on $\partial D= \pr(\partial \al)$ induced from the height order, is called the \term{diagram} of $\al$.

An $n$-web diagram is the diagram of an $n$-web.
\end{definition}

The orientation of a boundary edge $e$ of $\fS$ is \term{positive} if it is induced from the orientation of $\fS$. In picture the convention is that the positive orientation of a boundary edge is the counterclockwise one. If the height order of an $n$-web diagram $\al$ is given by the positive orientation, i.e. the height order increases when following the positive direction on each boundary edge, then we say $\al$ has \term{positive order}. One define \term{negative order} similarly, using the \term{negative orientation}, which is the opposite of the positive orientation.

\subsection{Defining Relations}

Let $\fS$ be a pb surface. Recall that for $i\in \JJ=\{1,2, \dots,n\}$ its conjugate $\bi$ is $n+1-i$. Also $\Sym_n$ is the symmetric group of $\JJ$.

A \term{state} of an $n$-web $\al$ is a map $s: \partial \al \to \JJ$. Let $\SS$ be the $R$-module freely spanned by stated $n$-webs over $\fS$ modulo the following defining relations using the constants $\ttt, \aaa, \ccc_i$ defined in Subsection \ref{ss.ground}.
\begin{gather}
q^{\frac 1n} \cross{}{}{p}{>}{>} - q^{-\frac 1n}\cross{}{}{n}{>}{>}
= (q-q^{-1})\walltwowall{}{}{>}{>}, \label{e.pm} \\
\kink = \ttt \horizontaledge{>}, \label{e.twist}\\
\circlediag{<} = (-1)^{n-1} [n] \TanglePic{0.9}{0.9}{}{}{}, \label{e.unknot}\\
\sinksourcethree{>}=(-q)^{\binom{n}{2}}\cdot \sum_{\sigma\in S_n}
(-q^{(1-n)/n})^{\ell(\sigma)} \coupon{$\sigma_+$}{>}. \label{e.sinksource}
\end{gather}
where the ellipse enclosing $\sigma_+$ is the minimum crossing positive braid representing a permutation $\sigma\in S_n$ and $\ell(\sigma)$ is the length of $\sigma\in \Sym_n$.
\begin{align}
\vertexnearwall{white} & = \aaa \sum_{\sigma \in S_n} (-q)^{\ell(\sigma)} \nedgewall{<-}{white}{$\sigma(n)$}{$\sigma(2)$}{$\sigma(1)$}
\label{e.vertexnearwall}\\
\capwall{<-}{right}{white}{$i$}{$j$} & = \delta_{\bar j,i} \ccc_i,\label{e.capwall}\\
\capnearwall{white} &= \sum_{i=1}^n (\ccc_{\bar i})^{-1} \twowall{<-}{white}{black}{$i$}{$\bar{i}$}
\label{e.capnearwall}\\
\crosswall{<-}{p}{white}{white}{$i$}{$j$}
&=q^{-\frac{1}{n}}\left(\delta_{{j<i}}(q-q^{-1})\twowall{<-}{white}{white}{$i$}{$j$}+q^{\delta_{i,j}}\twowall{<-}{white}{white}{$j$}{$i$}\right),
\label{e.crossp-wall}
\end{align}
where small white circles represent an arbitrary orientation (left-to-right or right-to-left) of the edges, consistent for the entire equation. The black circle represents the opposite orientation. When a boundary edge of a shaded area is directed, the direction indicates the height order of the endpoints of the diagrams on that directed line, where going along the direction increases the height, and the involved endpoints are consecutive in the height order. The height order outside the drawn part can be arbitrary.

For two $n$-webs $\al, \beta$ its product $\al \beta\in \SS$ is the result of stacking $\al$ above $\beta$. This means, we first isotope so that $\al \subset \fS \times (0,1)$ and $\beta \subset \fS\times(-1,0)$, then $\al \beta = \al \cup \beta$.

\subsection{Edge grading by weight lattice}\label{sec-grading}

Recall that the weight lattice $\LL$ of the Lie algebra $\mathfrak{sl}_n(\BC)$ is the abelian group generated by $\ww_1, \ww_2, \dots, \ww_n$, modulo the relation
\begin{equation}
\ww_1 + \ww_2 + \dots +\ww_n=0.
\end{equation}
Then $\LL \cong \BZ^{n-1}$. Let $\ror: \LL \to \LL$ be the involution given by $\ww_i \mapsto \wwr_i:= - \ww_{\bi}$. There is a standard symmetric bilinear form on $L$ with values in $\frac{1}{n}\ints$, where
\begin{equation}\label{eq-boundary-pairing}
\langle \ww_i ,\ww_j\rangle=\delta_{ij}-1/n.
\end{equation}

In the standard setting, the fundamental weights $\varpi_i$ are
\begin{equation}\label{eq.rw}
\varpi_i = \ww_1 + \dots + \ww_i, \quad i =1, \dots, n-1.
\end{equation}
Then one check easily that $\ror(\varpi_i)= \varpi_{n-i}$, and
\begin{align}
\la \varpi_i, \varpi_{i'} \ra &= \, \min \{i,i'\} - ii'/n, \label{eq.weights} \\
\la \ror(u) , \varpi_{i} \ra &= \la u , \varpi_{n-i} \ra. \label{eq.rorbra}
\end{align}
For convenience, define $\varpi_0=\varpi_n=0$ so that the equalities above hold.

As described in \cite[Section 8.4.1]{KS}, the simple $\mathfrak{sl}_n(\BC)$-module $V=\BC^n$ with highest weight $\varpi_1$ has basis $\{ v_1, \dots, v_n \}$, where $v_n$ is the highest weight vector. The dual space $V^\ast$, with basis $\{ v^1, \dots, v^n\}$ dual to $\{ v_1, \dots, v_n \}$, is the simple $\mathfrak{sl}_n$-module of highest weight $\varpi_{n-1}$. Then $v_i$ has weight $\ww_{\bi }$ and $v^i$ has weight $-\ww_i$.

Fix a boundary edge $e$ of a pb surface $\fS$. For a stated web diagram $\alpha$ over $\fS$ define
\begin{equation}
\dd_e(\al) = \sum _{x\in \partial \al \cap e} \ww^\ast_{\overline{s(x)}}\in \LL,
\end{equation}
where $s(x)$ is the state of $x$ and $\ww^\ast=\ww$ or $\wwr$ according as $\al$ points out of the surface at $x$ or $\al$ points into the surface at $x$. Let $\Gr^e_{\bk}(\SS)\subset \SS$ be the $R$-span of all elements represented by stated web diagrams $\al$ with $\dd_e(\al) = \bk$.

\begin{proposition} \label{r.grade1}
Suppose $e$ is a boundary edge of a pb surface $\fS$. We have
\begin{equation}\label{eq.decomp00}
\SS = \bigoplus_{\ww \in \LL} \Gr^e_{\bk}(\SS),
\end{equation}
which gives an $\LL$-grading of the algebra $\SS$. This means
\begin{equation}\label{eq.decomp01}
\Gr^e_{\bk}(\SS) \Gr^e_{\bk'}(\SS) \subset \Gr^e_{\bk+\bk'}(\SS).
\end{equation}
\end{proposition}

\begin{proof}
It is easy to check that $\dd_e(\al)$ is preserved by all the defining relations and hence we have \eqref{eq.decomp00}. From the definition we also have $\dd_e(\al \al') = \dd_e(\al) + \dd_e(\al')$, proving \eqref{eq.decomp01}.
\end{proof}

The degree $\dd_e(\al)$ can be understood as the total weight of $\al$ on edge $e$ as follows. Each stated endpoint $x$ of $\al$ stands for a vector in $V$ or $V^\ast$: if $x$ is outgoing endpoint the vector is $v_{s(x)}$, and if $x$ is an incoming endpoint, the vector is $v^{\overline{s(x)}} $. Then $\deg_e(\al)$ is the total weight of all endpoints in $\al \cap e$. Proposition~\ref{r.grade1} holds true because all the defining relations, being relations of the Reshetikhin-Turaev operator invariants, preserves the total weight.

\subsection{Edge weight isomorphisms}\label{ss.mark}

Let $R^\times$ be the multiplicative group of invertible elements in $R$. Recall that we define diagonal automorphism in Subsection \ref{ss.normal}.

\begin{proposition}\label{r.edgeweight}
Assume $\eta: \JJ \to R^\times$ is a map such that $\prod_{i=1}^n \eta(i)=1$ and $e$ is a boundary edge of a pb surface $\fS$.
\begin{enuma}
\item There exists a unique $R$-algebra isomorphism $\phi_{e,\eta}:\cS(\fS) \to \cS(\fS)$ such that if $D$ is a stated $n$-web diagram on $\fS$ then
\begin{equation} \phi_{e,\eta} (D) = \eta^\ast(\dd_e(D)) D, \label{eq.edgeiso}
\end{equation}
where $\eta^\ast: \LL \to R^\times$ is the group homomorphism defined by $\eta^\ast(\ww_i) = \eta(i)$.
\item The algebra automorphism $\phi_{e,\eta} $ is diagonal.
\item Any two such automorphisms $\phi_{e,\eta_1}$ and $\phi_{e,\eta_2}$ commute.
\end{enuma}
\end{proposition}

\begin{proof}
(a) follows from a general and easy fact of graded algebras: For any group homomorphism $\eta^\ast$ from the grading group $\LL$ to $R^\times$ the map $\phi_{e,\eta}$ of \eqref{eq.edgeiso} is an algebra homomorphism.
Its inverse is $\phi_{e, \mu}$, where $\mu(i) = \eta(i)^{-1}$.

(b) Since $\SS$ is spanned by stated $n$-webs, which are eigenvectors of $\phi_{e,\eta}$, the latter is diagonal.

(c) is obvious from the definition.
\end{proof}

\begin{remark}
If $\eta(i) \eta(\bi) =1$ then $\phi_{e,\eta}$ is the marking automorphism of \cite[Section 4.10]{LS}.
\end{remark}

\subsection{Reversing orientation} \label{sec.reverse}

\begin{proposition}[Corollary 4.8 of \cite{LS}]\label{r.reverse}
Assume $\fS$ is a pb surface. There is a unique $R$-algebra automorphism $\rOr: \SS \to \SS$ such that if $\al$ is a stated $\pfS$-web diagram then $\rOr(\al)$ is the result of reversing the orientation of $\al$.
\end{proposition}

\subsection{Cutting homomorphism}

We now present a main feature of the stated skein algebra: the cutting homomorphism.

Let $c$ be an ideal arc in the interior of a pb surface $\fS $. The cutting $\Cut_c(\fS)$ is a pb surface having two boundary edges $c_1, c_2$ such that $\fS= \Cut_c(\fS)/(c_1=c_2)$, with $c=c_1=c_2$.

An $n$-web diagram $D$ is \term{$c$-transverse} if the $n$-valent vertices of $D$ are not in $c$ and $D$ is transverse to $c$. Assume $D$ is a stated $c$-transverse $n$-web diagram. Let $h$ be a linear order on the set $D \cap c$. Let $p: \Cut_c(\fS) \to \fS$ be the natural projection map.
For a map $s: D \cap c \to \JJ$, let $(D,h,s)$ be the stated $n$-web diagram over $\Cut_c(\fS)$ which is $p^{-1}(D)$ where the height order on $c_1 \cup c_2$ is induced (via $p$) from $h$, and the states on $c_1 \cup c_2$ are induced (via $p$) from $s$.

\begin{theorem}[Theorem 5.2 and Proposition 7.11 of \cite{LS}]\label{t.splitting2}
Suppose $c$ is an interior ideal arc of a punctured bordered surface $\fS$. There is a unique $R$-algebra homomorphism $\Theta_c: \skein(\fS) \to \skein(\Cut_c(\fS))$ such that if $D$ is a stated $c$-transverse diagram of a stated $n$-web $\al$ over $\fS$ and $h$ is any linear order on $D \cap c$, then
\begin{equation}\label{eq.cut00}
\Theta_c(\al) =\sum_{s: D \cap c \to \JJ} (D, h, s).
\end{equation}
If in addition $\fS$ is essentially bordered, then $\Theta_c$ is injective.
\end{theorem}

If $\Cut_c\surface=\surface_1\sqcup\surface_2$ is disconnected, then there is a natural isomorphism
\begin{equation}
\skein(\Cut_c\surface)\cong\skein(\surface_1)\otimes\skein(\surface_2).
\end{equation}
In this case, $\Theta_c$ has an alternative form $\skein(\surface)\to\skein(\surface_1)\otimes\skein(\surface_2)$.

\subsection{Polygons}

We will define polygons and explain the relation between the bigon and $\Oq$, the quantized algebra of functions on $SL_n$ defined in Section \ref{sec.Oqsln}.

An \term{ideal $k$-gon}, or simply a \term{$k$-gon}, is the result of removing $k$ points on the boundary of the standard closed disk.
A \term{based $k$-gon} is a $k$-gon with one distinguished vertex, called the \term{based vertex}. Given two based $k$-gons there is a unique, up to isotopies, orientation preserving diffeomorphism between them, preserving the base. In this sense the based $k$-gon is unique, and we denote it by $\PP_k$.

Thus $\PP_1$ is the monogon. By \cite[Theorem 6.1]{LS}, we have an isomorphism $R\cong \skein(\PP_1)$, given by $x \to x\cdot \emptyset$. We will often identify $\skein(\poly_1)\equiv R$.

The bigon will play an important role. In picture the based bigon $\PP_2$ is depicted with the based vertex at the top, and we can define the \term{left edge $e_l$} the \term{right edge $e_r$}, as in Figure \ref{fig-bigon}(a). We often depict $\PP_2 $ as the square $[-1,1] \times (-1,1)$, as in Figure \ref{fig-bigon}(b).

By \cite[Theorem 6.3]{LS}, we have an isomorphism of $R$-algebra $ \Oq \cong \cS(\PP_2)$, which maps the generator $u_{ij}$ to the stated $\partial \PP_2$-arc described in Figure \ref{fig-bigon}(c). We will identify $\Oq \equiv \cS(\PP_2)$, and abusing notations, also use $u_{ij}$ to denote the stated $\partial\PP_2$-arc in Figure \ref{fig-bigon}(c). Let $\cev{u}_{ij}$ be the same arc $u_{ij}$ with reverse orientation.

\begin{figure}
\centering
\input{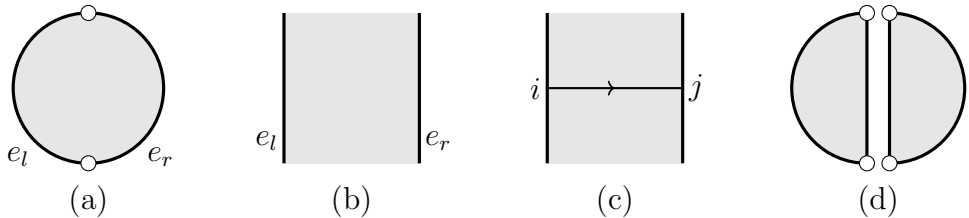}
\caption{(a) \& (b) Bigon $\PP_2$. (c) Stated arc $u_{ij}$, (d) splitting of $\PP_2$}\label{fig-bigon}
\end{figure}

In \cite{LS} it is shown that the counit, comultiplication, and antipode all have simple geometric description. In particular, by cutting $\PP_2 $ along the an interior ideal arc connecting the two vertices, we get two copies of $\PP_2$. See Figure \ref{fig-bigon}(d). The cutting homomorphism
\[ \skein(\PP_2) \to \skein(\PP_2 ) \ot \skein(\PP_2 )\]
is the coproduct under the identification $\skein(\PP_2) \equiv \Oq$.

Geometrically the antipode is given by
\begin{equation}\label{eq.Saij}
S(u_{ij}) = (-q)^{i-j}\cev{u}_{\bj \bi}.
\end{equation}

Let us discuss the counit $\epsilon$. Recall that in any coalgebra
\[ \epsilon(x) = \sum \epsilon(x_1) \epsilon(x_2), \quad \text{where } \Delta(x) = \sum x_1 \ot x_2.\]
Hence the calculation of $\epsilon(\al)$, where $\al$ is a stated $n$-web diagram over $\PP_2$, is reduced to the cases when $\al$ is one of the stated $n$-webs given in the following \cite[Section 6]{LS}:
\begin{align}
\epsilon( u_{ij}) &= \epsilon ( \cev u_{ij}) = \delta_{i,j} \label{eq.epsuij}\\
\epsilon\left(\crossSt{<-}{<-}{p}{white}{white}{$i'$}{$j'$}{$j$}{$i$}\right)
& = q^{-\frac{1}{n}} \left( q^{\delta_{i,j}} \delta_{i,i'}\delta_{j,j'}
+ (q-q^{-1}) \delta_{i<j}\delta_{i,j'}\delta_{j,i'}\right). \label{eq.epsR}\\
\epsilon\left(\crossSt{<-}{<-}{n}{white}{white}{$i'$}{$j'$}{$j$}{$i$}\right)
& = q^{\frac{1}{n}} \left( q^{-\delta_{i,j}} \delta_{i,i'}\delta_{j,j'}
- (q-q^{-1}) \delta_{j<i}\delta_{i,j'}\delta_{j,i'}\right). \label{eq.epsRn}\\
\epsilon\left(\crossSt{<-}{<-}{n}{white}{black}{$i'$}{$j'$}{$j$}{$i$}\right)
& = q^{\frac{1}{n}} \left( q^{-\delta_{i,\bar j}} \delta_{i,i'}\delta_{j,j'}
- (-q)^{j'-j}(q-q^{-1}) \delta_{i<i'}\delta_{i,\bar j}\delta_{i', \bar{j'}}\right). \label{eq.epsRr}\\
\epsilon\left(\crossSt{<-}{<-}{p}{white}{black}{$i'$}{$j'$}{$j$}{$i$}\right)
& = q^{-\frac{1}{n}} \left( q^{\delta_{i,\bar j}} \delta_{i,i'}\delta_{j,j'}
+ (-q)^{j'-j}(q-q^{-1}) \delta_{j<j'}\delta_{i,\bar j}\delta_{i', \bar{j'}}\right). \label{eq.epsRrp}
\end{align}
The right-hand side of \eqref{eq.epsR} is the $R$-matrix $\cR _{ij}^{i'j'}$ defined in \eqref{e.R}.
The last three identities \eqref{eq.epsRn}--\eqref{eq.epsRrp} follow from the first two \eqref{eq.epsuij}--\eqref{eq.epsR} and the isotopy invariance of $n$-webs in $\cS(\PP_2)$. Moreover, for a stated $n$-web $\al$ over $\PP_2$, the value $\epsilon(\al)$ is equal to a specific matrix element of the Reshetikhin-Turaev operator of a tangle associated to $\al$, see \cite[Proposition 6.6]{LS} for details.

\subsection{Coaction of $\Oq$ on $\skein(\surface)$}

Suppose $\fS$ is a punctured bordered surface and $b$ is a boundary edge. Let $c$ be an interior ideal arc isotopic to $b$. Then $b$ and $c$ cobound a bigon. By cutting $\fS$ along $c$ we get a surface $\fS'$ and a based bigon with $b$ considered its right edge. As $\fS'$ is diffeomorphic to $\fS$ via a unique up to isotopy diffeomorphism, we identify $\cS(\fS') = \cS(\fS)$. The cutting homomorphism gives an algebra homomorphism
\begin{equation}\label{e.coact}
\Delta_b: \cS(\fS) \to \cS(\fS) \ot \Oq,
\end{equation}
which gives a right coaction of the Hopf algebra $\Oq$ on $\cS(\fS)$, see \cite[Section 7]{LS}. The right coactions at different boundary edges commute. Since $\Delta_b$ is an algebra homomorphism, $\cS(\fS)$ is a \term{right comodule-algebra} over $\Oq$, as defined in \cite[Section III.7]{Kass}.

One frequently used basic property of a coaction is the following. For $x\in \cS(\fS)$, we have
\begin{equation}\label{eq.coact}
x = \sum x_1 \epsilon(x_2),\quad
\text{where } \Delta(x) = \sum x_1 \ot x_2.
\end{equation}

By making a different identification, we also obtain a left $\Oq$-comodule structure.

As an application, we derive the following generalization of \eqref{e.vertexnearwall}.

\begin{lemma}\label{lemma-vertex-rev}
Let $J=\{j_1,\dots,j_k\}\subset\JJ$. Define $\vec{j}:[1;k]\to\JJ$ by $\vec{j}(i)=j_i$. Then
\begin{gather}
\input{vertex-rev}\label{eq-vertex-rev}\\
\input{vertex-rev-left}\label{eq-vertex-rev-left}
\end{gather}
Here the sums are over bijections $\sigma_2:[1;n-k]\to\bar{J}^c$.

The equalities also hold when the states have repetition, where the sum is empty and interpreted as zero.
\end{lemma}

\begin{proof}
We start with \eqref{eq-vertex-rev}. Move the vertex upward and then toward the right edge. Then using the defining relations, we have
\[\input{vertex-rev-calc}\]
To make all returning arcs nonzero, we must have
\[\sigma(t)=\bar{j}_{k+1-t},\quad t=1,\dots,k.\]
This is only possible if the states are distinct. Thus if the states repeat, all terms are zero. When the states are distinct, define the restrictions
\[\sigma_1=\sigma|_{[1,k]}:[1,k]\to\bar{J},\qquad
\sigma_2:[1;n-k]\to\bar{J}^c\quad \sigma_2(i)=\sigma(k+i).\]
Then the diagram evaluates to
\begin{equation}\label{eq-vrev-calc}
\input{vertex-rev-calc2}.
\end{equation}
We can decompose the length $\ell(\sigma)$ as
\begin{align*}
\ell(\sigma)&=\ell(\sigma_1)+\ell(\sigma_2)+\abs{\{(a,b)\in\bar{J}\times \bar{J}^c\mid a>b\}}\\
&=\ell(\vec{j})+\ell(\sigma_2)+\left(\sum_{s\in\bar{J}}s\right)-\frac{k(k+1)}{2}.
\end{align*}
Thus the coefficient in \eqref{eq-vrev-calc} is
\begin{align*}
\vertexa(-q)^{\ell(\sigma)}\prod_{t=1}^k\returnc_{\sigma(t)}
&=q^{\frac{(1-n)(2n+1)}{4}}(-q)^{\ell(\vec{j})+\ell(\sigma_2)+\left(\sum_{s\in\bar{J}}s\right)-\frac{k(k+1)}{2}}\prod_{s\in\bar{J}}\left(q^{\frac{n-1}{2n}}(-q)^{n-s}\right)\\
&=(-1)^{\binom{n}{2}}q^{\frac{1}{2n}\left(\binom{k}{2}-\binom{n-k}{2}\right)}(-q)^{\ell(\vec{j})-\binom{n-k}{2}}(-q)^{\ell(\sigma_2)}.
\end{align*}

This proves \eqref{eq-vertex-rev}. As a corollary,
\begin{equation}
\epsilon\left(\input{minor-vertex}\right)=\begin{cases}
(-1)^{\binom{n}{2}}q^{\frac{1}{2n}\left(\binom{k}{2}-\binom{n-k}{2}\right)}(-q)^{\ell(\vec{j})-\ell(\vec{s})},&J\cup S=\JJ,\\
0,&\text{otherwise},
\end{cases}
\end{equation}
where $S=\{s_1,\dots,s_{n-k}\}\subset\JJ$, and $\vec{s}:[1;n-k]\to\JJ$ is given by $\vec{s}(i)=s_i$.

To prove \eqref{eq-vertex-rev-left}, split off the vertex and use the coaction on the left edge.
\end{proof}

\subsection{Upper triangular nature of the $R$-matrix}

The upper triangular nature of the $R$-matrix allows us to write down the top degree part of certain products in $\SS$.

For two sequences $\vec{i}=(i_1,\dots,i_k)$, $\vec{i}'=(i'_1,\dots,i'_k)$, we write $\vec{i}'\gg \vec{i}$ if $\vec{i}'\ne \vec{i}$ and $i'_s\ge i_s$ for all $s=1,\dots,k$.

\begin{lemma}\label{r.upper}
In the following diagrams, the orientations of the strands are arbitrary.
\begin{enuma}
\item The counit $\epsilon$ satisfy
\begin{align}
\epsilon\left(\crossSt{<-}{<-}{p}{}{}{$i'$}{$j'$}{$j$}{$i$}\right)
=\epsilon\left(\crossSt{<-}{<-}{n}{}{}{$j'$}{$i'$}{$i$}{$j$}\right)
&=0&&\text{if }i'<i\text{ or }j'>j.\\
\epsilon\left(\crossSt{<-}{<-}{p}{}{}{$i'$}{$j'$}{$j$}{$i$}\right)
\eqq\epsilon\left(\crossSt{<-}{<-}{n}{}{}{$j'$}{$i'$}{$i$}{$j$}\right)
&\eqq1&&\text{if }i'=i\text{ and }j'=j.
\end{align}
\item For any orientation of the strands in the following tangles, we have
\begin{gather}
\input{height-ex-mult}\label{eq-height-group}\\
\input{height-ex-mult2}
\end{gather}
for some scalars $c_{\vec{i}'\vec{j}'},\bar{c}_{\vec{i}'\vec{j}'}\in R$.
\end{enuma}
\end{lemma}

\begin{proof}
(a) The statements follow from Identities \eqref{eq.epsR}--\eqref{eq.epsRrp}.

(b) Using \eqref{eq.coact}, we cut $\fS$ along ideal arcs parallel to the drawn boundary edge to split off one crossing at a time and apply the counit to the bigons. Using part (a), all terms with a decreased new state on an overpasses or an increased state on an underpass are zero. In addition, if the new states match the old ones, the coefficient is $\eqq1$. Thus we obtain the identities in (b).
\end{proof}

\subsection{Height exchange}\label{ss.height}

For a non-stated $n$-web diagram $\al$ over $\fS$ let $M(\al)\subset \SS$ be the $R$-span of $\al$ with arbitrary states.

\begin{lemma} \label{r.height}
Suppose $\al$ and $\al'$ are $n$-web diagrams over $\fS$ which differ only in the height order. Then $M(\al) = M(\al')$.
\end{lemma}

\begin{proof}
The diagrams of $\al$ and $\al'$ are identical everywhere except near the boundary. Hence the coaction identity \eqref{eq.coact} shows that each stated $\al$ is an $R$-linear combination of stated $\al'$. This shows $M(\al) \subset M(\al')$. This converse inclusion is proved similarly. Thus $M(\al) = M(\al')$.
\end{proof}

\begin{lemma}\label{r.height2}
Assume $1\le i <j \le n$. We have the following
\begin{alignat}{2}
\vertexfour{white}{$i$}{$i$}&=0.\\
\vertexfour{white}{$i$}{$j$}&=(-q)\vertexfour{white}{$j$}{$i$}.&\qquad
\vertexfourleft{white}{$i$}{$j$}&=(-q)\vertexfourleft{white}{$j$}{$i$}.\\
\vertexfour[->]{white}{$j$}{$i$}&=q^{-\frac{1}{n}}\vertexfour{white}{$j$}{$i$}.&
\vertexfourleft[->]{white}{$i$}{$j$}&=q^{\frac{1}{n}}\vertexfourleft{white}{$i$}{$j$}.
\end{alignat}
\end{lemma}

\begin{proof}
The first two lines follow from Lemma~\ref{lemma-vertex-rev}.

For the third line, using \eqref{e.crossp-wall},
\begin{equation}
\vertexfour[->]{white}{$j$}{$i$}
=\vertexfourcross{white}{$i$}{$j$}
=q^{-\frac{1}{n}}\vertexfour{white}{$j$}{$i$}.
\end{equation}
The second equality is obtained by a $180^\circ$ rotation.
\end{proof}

\begin{lemma}\label{r.height5}
Suppose $\vec{i}, \vec{j} \subset \JJ$ are sequences of consecutive numbers, with either $\max \vec{i} \ge \max \vec{j}$ or $\min \vec{i} \ge \min \vec{j}$, then
\begin{equation}\label{eq-comm-consec}
\input{comm-consec-v}
\end{equation}
where the bracket is defined in \eqref{eq-boundary-pairing}, and by abuse of notations, $\vec{i}$ and $\vec{j}$ also denote the corresponding $\dd$-grading.
\end{lemma}

\begin{proof}
This follows from a more detailed calculation of the coefficients in \eqref{eq-height-group}. Since the orientations are consistent near the boundary, \eqref{eq.epsR} restricts the sum further to $\vec{i}\sqcup\vec{j}=\vec{i}'\sqcup\vec{j}'$. In addition, if the strands connected to the same vertex have repeated states, then the diagram is zero. Combined with the condition on the states $\vec{i},\vec{j}$ and the original restrictions in \eqref{eq-height-group}, the sum is always zero, and only the first term remains. To find the exact coefficient of this term, again use \eqref{eq.epsR}. States cannot exchange between $\vec{i}$ and $\vec{j}$, so only the first term in \eqref{eq.epsR} counts. Each pair of states contributes $q^{-1/n}$, and each overlap $\vec{i}\cap\vec{j}$ has an additional factor of $q$. This agrees with the definition of $\langle\vec{i},\vec{j}\rangle$ in \eqref{eq-boundary-pairing}.
\end{proof}

\subsection{Reflection} \label{sec.reflection}

We introduced algebras with reflection in Subsection~\ref{ss.reflection}.

\begin{proposition}[Theorem 4.9 of \cite{LS}]\label{r.reflection}
Assume $\fS$ is a pb surface and $R= \Zq$. There is a unique reflection $\omega: \SS \to \SS$ such that if $\al$ is a stated $n$-web diagram then $\omega(\al)$ is obtained from $\al$ by switching all the crossings and reversing the height order on each boundary edge.
\end{proposition}

A stated web diagram $\al$ over a pb surface $\fS$ is \term{reflection-normalizable} if over the ground ring $\Zq$ we have $\omega(x) = \hq^{2k} x$ for $k\in \BZ$. Clearly such a $k$ is unique. In that case, over any ground ring $R$, we define the \term{reflection-normalization} by
\begin{equation}\label{eq.reflec}
[x]_{\norm}:= \hq^k x,
\end{equation}
Then when $R=\Zq$ we have $\omega([x]_{\norm})= [x]_{\norm}$, i.e. $[x]_{\norm}$ is reflection invariant. Note the Weyl-normalization of a monomial in a quantum torus agrees with the reflection-normalization.

\begin{lemma}\label{r.triad}
For $i+j+k=n$, the following stated web diagram is reflection-normalizable.
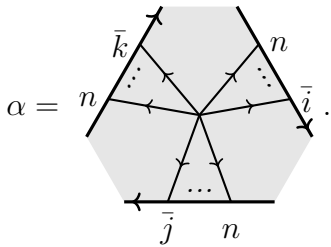
\[\alpha=\input{trigen-hex}.\]
\end{lemma}

\begin{proof}
First we assume that the 3 drawn solid lines are in 3 distinct boundary edges. Then $\alpha$ is reflection-normalizable because from Lemma \ref{r.height2}(c) we have
\[\omega(\alpha)=q^{-\frac{1}{n}\left(\binom{i}{2}+\binom{j}{2}+\binom{k}{2}\right)}\alpha.\]

If two or all three of the solid lines are in the same boundary edge, then we use \eqref{eq-comm-consec} to conclude that $\alpha$ is reflection-normalizable.
\end{proof}

\subsection{Embedding of punctured bordered surfaces} \label{sec.embed}

A proper embedding $f:\fS _1 \embed \fS _2$ of punctured bordered surfaces defines an $R$-linear map $f_\ast: \cS(\fS_1) \to \cS(\fS_2)$ as follows.
Suppose $\al$ is a stated $\pfS_1$-tangle diagram with negative order. Let $[\al]\in \cS(\fS_1)$ be the element determined by $\al$. Define $f_\ast([\al]) = [f(\al)] \in \cS(\fS_2)$, where $f(\al)$ is given the negative boundary order. Clearly $f_\ast$ is a well-defined $R$-linear map, and does not change under ambient isotopies of $f$. In general $f_\ast$ is not an algebra homomorphism.

A proper embedding $f:\fS _1 \embed \fS _2$ of punctured bordered surfaces is \term{strict} if no two boundary edges of $\fS_1$ are mapped under $f$ into one boundary edge of $\fS_2$. Then $f_\ast$ is an algebra homomorphism if and only $f$ is strict.

\subsection{Geometric picture of quantum minor} \label{ss.Weyl}

We will show that a quantum minor of the quantum matrix $\buu$ is expressed by a simple diagram under the identification $\cS(\PP_2)=\Oq$, and show how to cut the quantum minors into smaller pieces.

Let $\binom{\JJ}{k}$ be the set of all $k$-element subsets of $\JJ =\{1,\dots, n\}$. If $I\subset\JJ$, define
\[\bar{I}=\{\bar{i}\mid i\in I\},\quad
I^c=\JJ\setminus I,\quad \bar{I}^c=(\bar{I})^c.\]
For $I,J\in \binom{\JJ}{k}$, let $M^I_J(\buu)\in \Oq$ be the quantum determinant of the $I\times J$ submatrix of $\buu$. We identify $\cS(\PP_2)= \Oq$, so that $u_{ij}\in \Oq$ is identify with the stated arc in Figure \ref{fig-bigon}(c). Assume $a$ is an oriented $\pfS$-arc in a pb surface $\fS$, and $N(a)$ is a small tubular open neighborhood $N(a)$ of $a$ in $\fS$. There is a unique up to isotopy diffeomorphism $f: \PP_2 \to N(a)$ such that the beginning point of $a$ is in the image of the left edge. Let $M^I_J(a)= f_\ast(M^I_J(\buu))$ and depict it by the diagram in Figure~\ref{minor-IJ}.

\begin{figure}
\centering
\input{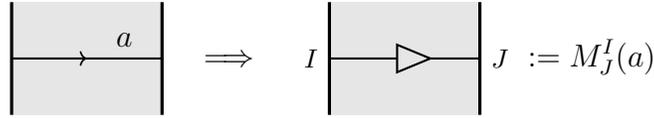}
\caption{Diagrammatic notation for quantum minor}\label{minor-IJ}
\end{figure}

\begin{lemma}\label{r.det}
Assume $I=\{i_1,\dots, i_k\}$ and $J=\{j_1,\dots, j_k\}$ are subsets of $\JJ$. Write $\bar{I}^c=\{ s_1, \dots, s_{n-k}\}$. The following stated $n$-web diagram over $\PP_2$
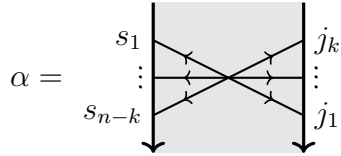
\[\alpha=\input{minor-source}\]
is reflection-normalizable, and its reflection-normalization is $\pm M^I_J(\buu)$. More precisely,
\begin{equation}\label{eq.detu}
M^I_J(\buu)=(-1)^{\binom{n}{2}}(-q)^{\ell(\vec{s})-\ell(\vec{j})}q^{\frac{1}{2n}\left(\binom{n-k}{2}-\binom{k}{2}\right)}\alpha.
\end{equation}
Here for $\vec{i}=(i_1, \dots, i_k)$ we define $\ell(\vec i)$ as the number of inversion in the map $t\to i_t$.

A similar result holds for the diagram with a sink, states $I$ on the left, and states $\bar{J}^c$ on the right.
\end{lemma}

\begin{proof}
From Lemma \ref{r.height2}(b-c) it is easy to see that $\al$ is reflection-normalizable. Using Lemma \ref{r.height2}(b) to permute the states on the boundary, Identity \eqref{eq.detu} is reduced to the case where $j_1 < \dots < j_k$, which we will assume now. Note $\ell(\vec{j})=0$.

Applying Lemma~\ref{lemma-vertex-rev} on the left edge, we get
\[\alpha = \input{minor-src-left}\]
where the sum is over bijections $\sigma_2:[1;k]\to\bar{S}^c=I$. The diagram on the right-hand side is the product $u_{\sigma(1)1}\cdots u_{\sigma(k)k}$. Thus the sum is the determinant $M^I_J(\vec{u})$. This proves the lemma.
\end{proof}

\begin{lemma}\label{r.cutdet}
Assume $I,J \in \binom{\JJ}{k}$. Then
\begin{equation}\label{eq.cutdet}
\input{tri-det-split}
\end{equation}
where $\bar{L} =\{\bar{l}\mid l\in L\}$ and $C_L \in R$ is the unit given by
\[C_L=(-q^{1+\frac{1}{n}})^{\binom{k}{2}}\prod_{l\in L}\returnc_{\bar{l}}^{-1}.\]
\end{lemma}

\begin{proof}
Let $J=\{j_1,\dots,j_k\}$ with $j_1<\cdots<j_k$. Using Lemmas~\ref{r.det} and \ref{r.height2},
\[\input{tri-det-splpf}\]
Here the constants are
\begin{align*}
w_1&=(-1)^{\binom{n}{2}}(-q)^{\ell(\vec{s})+\binom{k}{2}}q^{\frac{1}{2n}\left(\binom{n-k}{2}+\binom{k}{2}\right)},\\
w_2&=w_1\left(\prod_{l\in L}\returnc_{\bar{l}}^{-1}\right)
(-1)^{\binom{n}{2}}(-q)^{-\ell(\vec{s})}q^{\frac{1}{2n}\left(\binom{k}{2}-\binom{n-k}{2}\right)}=C_L.
\end{align*}
The sum can be group by the subset $L=\{l_1,\dots,l_k\}$. Note $\ell(\bar{l}_k,\dots,\bar{l}_1)=\ell(\vec{l})$. Thus for a fixed $L$, the sum in the bottom right is the $\bar{L}\times J$ quantum minor. This proves the lemma.
\end{proof}

\section{Punctured monogon algebra} \label{sec.pMon}

We will study the stated skein algebra $\Bq$ of the once-punctured monogon and a quotient $\bBq$ of it. Later we will show that for any essentially bordered surface $\fS$, the algebra $\SS$ has a tensor product factorization where each factor is either $\Oq$ or $\Bq$.

Recall that the ground ring $R$ is a commutative domain with a distinguished invertible element~$\hq$.

\subsection{Main results of section}

The \term{$m$-punctured $k$-gon} $\PP_{k,m}$ is the result of removing $m$ interior punctures from the $k$-gon $\PP_k$. We will call $\PP_{1,1}$ simply the \term{punctured monogon} and denote $\Bq:= \cS(\PP_{1,1})$. In \cite{LS} it was proved that $\Bq$ is the \term{transmutation} \cite{Majid} of the quantized algebra of regular function $\Oq$ of $SL_n$. As such $\Bq$ was studied in the literature, but mostly for the case when the ground ring is a field. For example when $R=\BC(\hq)$ it is proved \cite{KolbS} that $\Bq$ is a domain and the proof seems to base heavily on the fact that $R$ is a field, as it uses the dual quantum group and decomposition of modules into irreducible submodules. Here we prove that $\Bq$ has a quasimonomial basis, which in particular implies that it is a domain whenever $R$ is a domain. The proof also allows us to show that a quotient $\bBq$ of $\Bq$, later known as the reduced skein algebra of $\PP_{1,1}$, is a domain, that both $\Bq$ and $\bBq$ have uniform GK dimensions. Note the sole fact that $\Bq$ is a domain can be proved using method of the next section.

Let $\chu_{ij}\in \Bq$ be the element represented by arc $a$ of Figure~\ref{fig-mono-corner} with state $i$ on the left and $j$ on the right. Denote
\[ \chG= \{\chu _{ij} \mid i,j \in \JJ \}, \quad \chGm = \{\chu_{ij} \in \chG \mid i< j \}.\]
Define
\[ \bBq = \Bq/(\chGm) = \Bq / \ccI^-,\]
where $\ccI^-\lhd \Bq$ is the 2-sided ideal generated by $\chGm$.

\begin{figure}
\centering
\input{mono-corner}
\caption{Oriented arc $a$ gives linear isomorphism $\kappa: \Oq \to \cB_q$}\label{fig-mono-corner}
\end{figure}

\begin{theorem} \label{r.Bq}
\begin{enuma}
\item The algebra $\Bq=\cS(\PP_{1,1})$ has a quasimonomial basis. Consequently $\Bq$ is
a domain and a free $R$-module. Besides $\Bq$ has uniform GK dimension $n^2-1$ and is orderly finitely generated.
\item The algebra $\bBq$ is a domain and a free $R$-module, and it has uniform GK dimension $(n-1)(n+2)/2$.
\item Let $\overleftarrow{\chGm}=\{\overleftarrow{\chu}_{ij} \mid i <j \in \JJ\}$, where $\overleftarrow{\chu}_{ij}$ is ${\chu}_{ij}$ with reverse orientation. Then
\begin{equation}\label{eq.normal22}
\chGm \Bq = \Bq {\chGm} = \overleftarrow{\chGm} \Bq = \overleftarrow{\chGm} \Bq= \ccI^-.
\end{equation}
\end{enuma}
\end{theorem}

We don't really need the following result for the existence of quantum traces. However it has independent interest.

\begin{theorem}\label{thm.basisBq}
For any linear ord $\ord$ on the set $\JJ^2$, the set
$$ \check B^\ord: = \{ \check b(m):= \prod_{(i,j) \in \JJ^2} \check u_{ij}^{\hat m_{ij} } \mid m \in \Gamma= \Mat_n(\BN)/(\Id) \}$$
where the product is taken with respect to the order $\ord$, is a free basis of $\Bq$. Consequently $\Bq$ is orderly finitely generated.
\end{theorem}

\subsection{From bigon $\PP_2$ to punctured monogon $\PP_{1,1}$}

A tubular neighborhood $N(a)$ of $a$ is diffeomorphic to the based bigon $\PP_2$, where the left edge is defined to be the one containing the beginning point of $a$. A special case of \cite[Theorem 7.13]{LS} states that the embedding $N(a) \embed \PP_{1,1}$ induces a bijective $R$-linear map
\[\kappa:\cF= \cS(N(a)) \to \cS(\PP_{1,1})= \Bq.\]
However $\kappa$ does not preserve the product. In fact, as explained in \cite[Section 7]{LS} the product in $\Bq$ can be obtained from that of $\Oq$ by Majid's transmutation \cite{Majid}, i.e. $\Bq$ is the transmutation of $\Oq$, or the product on $\Bq$ is the covariantized product \cite{Majid}. Below we show that up to elements of lower orders in some filtration, the products in $\Oq$ and $\Bq$ are almost the same.

Let $x$ be an $n$-web diagram over $\PP_2$ having negative (i.e. clockwise) order on both edges of $\PP_2$. By putting states on boundary points of $x$ we get a stated $n$-web $x_{\boi \boj}$ where $\boi$ (respectively $\boj$) is the sequence of states on the left (respectively right) edge, in clockwise order. Denote $\kappa(x_{\boi \boj}) = \check x_{\boi \boj}\in \Bq$. For two sequences $\boi=(i_1,\dots, i_m)$ and $\boi'=(i'_1, \dots, i'_m)$ of the same length we write $\boi \ll \boi'$ if $\boi \neq \boi'$ and $i_k \le i'_k$ for all $k=1, \dots, m$.

\begin{lemma}
Let $x, y$ be $n$-web diagrams over $\PP_2$ having negative order on both edges of $\PP_2$. Consider $x_{\boi \boj}, y_{\bok \bol}$ as elements of $\cS(\PP_2) = \Oq$. Then
\begin{align}
\kappa(x_{\boi \boj} y_{\bok \bol}) &\eqq \check x_{\boi \boj}\check y_{\bok \bol} + \Span\{\check x_{\boi \boj'} \check y_{\bok' \bol} \mid \boj' \gg \boj, \bok' \ll \bok \} \label{eq.M20}\\
\kappa^{-1}( \check x_{\boi \boj} \check y_{\bok \bol}) &\eqq x_{\boi \boj}y_{\bok \bol} + \Span\{x_{\boi \boj'} y_{\bok' \bol} \mid \boj' \gg \boj, \bok' \ll \bok \} \label{eq.M21}\\
\kappa(u_{i_1j_1}\dots u_{i_k j_k}) & \eqq \chu_{i_1j_1}\dots \chu_{i_k j_k} + \Span\{\chu_{i_1' j_1'}\dots \chu_{i_k' j_k'} \mid \boi' \ll \boi, \boj' \gg \boj \} \label{eq.OB}\\
\kappa^{-1}(\chu_{i_1j_1}\dots \chu_{i_k j_k}) & \eqq u_{i_1j_1}\dots u_{i_k j_k} + \Span\{u_{i_1' j_1'}\dots u_{i_k' j_k'} \mid \boi' \ll \boi, \boj' \gg \boj \} \label{eq.BO}
\end{align}
\end{lemma}

\begin{proof}
The diagrams $x_{\vec{i}\vec{j}}y_{\vec{k}\vec{l}}$ and its image under $\kappa$ are shown in the first row of Figure~\ref{fig-mono-prod}. Then we split along the dashed line and use coaction \eqref{e.coact} to obtain
\begin{equation}
\kappa(x_{\vec{i}\vec{j}}y_{\vec{k}\vec{l}})
=\sum_{\vec{i}'\vec{j}'\vec{k}'\vec{l}'} \epsilon_{\vec{i}'\vec{j}'\vec{k}'\vec{l}'} \, \check{x}_{\vec{i}'\vec{j}'} \, \check{y}_{\vec{k}'\vec{l}'},
\end{equation}
where $\epsilon_{\vec{i}'\vec{j}'\vec{k}'\vec{l}'}$ is the counit of the bigon. Since $\epsilon(u_{ij}) = \epsilon(\cev u_{ij})= \delta_{ij}$, we see that $\epsilon_{\vec{i}'\vec{j}'\vec{k}'\vec{l}'}=0$ unless $\vec{i}'=\vec{i}$, $\vec{l}'=\vec{l}$. Assume $\vec{i}'=\vec{i}$, $\vec{l}'=\vec{l}$.
By Lemma \ref{r.upper}(b), the counit value $\epsilon_{\vec{i}'\vec{j}'\vec{k}'\vec{l}'}$ is non-zero only when
\begin{enumerate}
\item [(i)] either $\vec{j}'=\vec{j}$, $\vec{k}'=\vec{k}$, in which case $\epsilon_{\vec{i}\vec{j}\vec{k}\vec{l}}\eqq 1$, or
\item[(ii)]$\vec{j}' \gg \vec{j}$, $\vec{k}' \ll \vec{k}$
\end{enumerate}
Cases (i) and (ii) give respectively the first and the second terms in the right-hand side of \eqref{eq.M20}. This proves \eqref{eq.M20}.

\begin{figure}
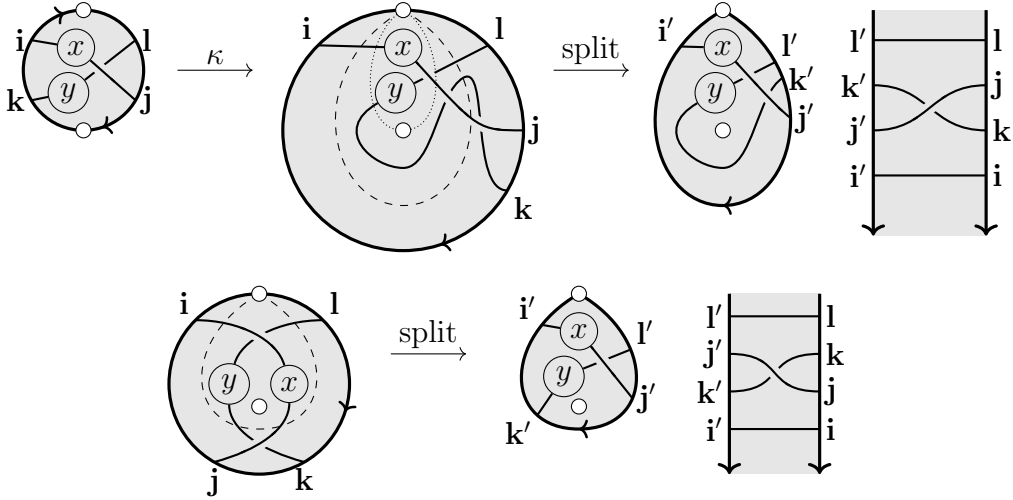

\centering
\input{mono-prod}
\vskip1em
\input{mono-prod-inv}
\caption{Evaluating $\kappa$ and $\kappa^{-1}$}\label{fig-mono-prod}
\end{figure}

Similarly, \eqref{eq.M21} follows from the second row of Figure~\ref{fig-mono-prod}, which reads
\begin{equation}
\kappa^{-1}(\check{x}_{\vec{i}\vec{j}}\check{y}_{\vec{k}\vec{l}})
=\sum_{\vec{j}'\vec{k}'} \bar{\epsilon}_{\vec{i}\vec{j}'\vec{k}'\vec{l}}\, x_{\vec{i}\vec{j}'}\, y_{\vec{k}'\vec{l}},
\end{equation}
where $\bar{\epsilon}_{\vec{i}\vec{j}'\vec{k}'\vec{l}}$ is the counit of the bigon, which is nonzero only if first one has $\vec i = \vec i', \vec l =\vec l'$ and then either $\vec{j}' \gg \vec{j}$, $\vec{k}' \ll \vec{k}$ or $\vec{j}'=\vec{j}$, $\vec{k}'=\vec{k}$. For the latter case $\bar{\epsilon}_{\vec{i}\vec{j}\vec{k}\vec{l}}\eqq 1$. This proves \eqref{eq.M21}.

Then \eqref{eq.OB} and \eqref{eq.BO} follow respectively from \eqref{eq.M20} and \eqref{eq.M21} by induction.
\end{proof}

\subsection{Quasimonomial basis for $\Bq$}

Fix a linear order $\ord$ on the set $\JJ^2$. By Theorem \ref{r.Oq} the set $B^\ord=\{b(m) \mid m \in \Gamma\}$ is a quasimonomial basis of $\Oq$.

\begin{proposition}\label{r.Bq1}
The set $\kappa(B^\ord)$ is a quasimonomial basis of $\Bq$.
\end{proposition}

\begin{proof}
Recall the map $d_{12}: \Gamma \to \BZ^2$ defined in Corollary \ref{r.boo12}, and for $m\in \Gamma$
\[ \Oq( d_{12} < m ) = \Span\{b(m')\mid d_{12}(m') \llex d_{12}(m)\}.\]
For $m\in \Gamma$, let $\chb(m) = \kappa(b(m))$. Clearly,
\[\Bq( d_{12} < m ) = \kappa(\Oq( d_{12} < m )) = \Span\{\chb(m')\mid d_{12}(m') \llex d_{12}(m)\}.\]
Using \eqref{eq.M21} and then \eqref{eq.boo12}, we have
\begin{equation*}
\kappa^{-1}(\chb (m) \chb (m')) \eqq b(m) b(m') + \Oq( d_{12} < m ) \eqq b(m+m') + \Oq( d_{12} < m ).
\end{equation*}
Applying $\kappa$ to both sides, we get
\[ \chb(m) \chb(m') \eqq \chb (m+m') + \Bq( d_{12} < m ).\]
This shows $\{\chb(m) \mid m\in \Gamma\}$ is a quasi-monomial basis of $\Bq$, parameterized by the enhanced monoid $(\Gamma, d_{12})$.
\end{proof}

\subsection{Proof of Theorem \ref{r.Bq} part (a)}\label{sec-Bqa}

Proposition \ref{r.Bq1} shows $\Bq$ has a quasimonomial basis. By Proposition \ref{r.domain9}, $\Bq$ is domain and free as an $R$-module.

The set $G= \{u_{ij} \mid (i,j) \in \JJ^2\}$ is a uniform GK set of generators of $\Oq$, by Proposition \ref{r.Oq}. From \eqref{eq.OB}, we have $\kappa(\Pol_m(G)) \subset \Pol_m(\kappa(G))$. Similarly, from \eqref{eq.BO} we have the converse inclusion. Hence $\kappa(\Pol_m(G)) = \Pol_m(\kappa(G))$. As $\kappa$ is a linear isomorphism, we conclude that $\kappa(G)$ is a uniform GK set of generators of $\Bq$, and that $\GKdim (\Bq)= \GKdim(\Oq) = n^2-1$.

\subsection{From \LOsec to \texorpdfstring{$\bBq$}{F-frak-bar}}

Recall that $G^- = \{u_{ij} \mid i < j \in \JJ \}$, and $\cI^-\lhd \Oq$ is the ideal generated by $G^-$. Let $\overleftarrow{G}^- = \{\cev u _{ij} \mid i < j \in \JJ \}$.

\begin{proposition}
We have $\kappa(\cI^-) = \ccI^-$. Moreover
\begin{align}
\kappa( G^- \Oq) & = \chG^- \Bq \label{eq.13a} \\
\kappa( \overleftarrow{G}^- \Oq) & = \overleftarrow{\chG}^- \Bq \label{eq.13b} \\
\kappa( \Oq G^- ) & = \Bq\chG^- \label{eq.13c} \\
\kappa( \Oq\overleftarrow{G}^- ) & = \Bq \overleftarrow{\chG}^-\label{eq.13d}
\end{align}
\end{proposition}

\begin{proof}
In \eqref{eq.M20} let $x_{\boi \boj}= u_{ij}$ with $ i<j$, we get
\[\kappa(G^- \Oq) \subset\chG^- \Bq.\]
In \eqref{eq.M21} let $x_{\boi \boj}= u_{ij}, i<j$, we get the converse inclusion. Hence we get \eqref{eq.13a}. The identical argument, with $x_{\boi \boj}= \cev u_{ij}$ (with $ i<j$), proves \eqref{eq.13b}.

Similarly let $y_{\boi \boj}= u_{ij}$ with $ i<j$ in \eqref{eq.M20} we get
$\kappa(G^- \Oq) \subset\chG^- \Bq$. Let $y_{\boi \boj}= u_{ij}$ with $ i<j$ in \eqref{eq.M21} we get the converse inclusion. This proves \eqref{eq.13c}. The identical argument, with $y_{\boi \boj}= \cev u_{ij}$ (with $ i<j$), proves \eqref{eq.13d}.

By Proposition \ref{r.redOq0},
\begin{equation} \label{eq.14}
G^- \Oq = \Oq G^-=\cI^-.
\end{equation}
Hence \eqref{eq.13a} and \eqref{eq.13c} imply $\chG^- \Bq = \Bq \chG^- = \ccI^-$, and \eqref{eq.13a} shows that $\kappa(\cI^-) = \ccI^-$.
\end{proof}

\subsection{Proof of Theorem \ref{r.Bq} part (c)}

We have $S(\cI^-)= \cI^-$ by \eqref{eq.HopfS}. By \eqref{eq.Saij} we have $S(u_{ij})\eqq u_{\bar j \bar i}$. Note that $i<j$ if and only if $\bar j < \bar i$. Thus applying $S$ to \eqref{eq.14}, we get
\[ \Oq \overleftarrow{G}^-= \overleftarrow{G}^- \Oq = \cI^-.\]
We conclude that all the left-hand sides of \eqref{eq.13a}-\eqref{eq.13d} are equal to $\cI^-$:
\begin{equation}\label{eq.normal11}
G^- \Oq = \Oq G^- = \Oq \overleftarrow{G}^-= \overleftarrow{G}^- \Oq = \cI^-.
\end{equation}
It follows that all the right-hand sides are equal, and equal to $\ccI^-$:
\begin{equation}\label{eq.normal12}
\chG^- \Bq = \Bq \chG^- = \Bq \overleftarrow{\chG}^-= \overleftarrow{\chG}^- \Bq = \ccI^-.
\end{equation}
This proves \eqref{eq.normal22}.

\subsection{Proof of Theorem \ref{r.Bq} part (b)}

As $\kappa(\cI^-) = \ccI^-$, the bijective $R$-linear map $\kappa$ descends to a bijective $R$-linear map $\bka: \LO \to \bBq$.

By Proposition \ref{r.redOq}, the set
\[\bar B^\ord = \{\pr(b(m)) \mid m\in \bG \},\quad
\text{where }\bG= \{m \in \Gamma \mid \hat m_{ij} = 0 \text{if } i <j \}\subset \Gamma,\]
is a quasimonomial basis of $\LO$. Here $\pr: \Oq\onto \LO$ is the natural projection.

\begin{proposition}\label{r.Bq2}
The set $\bar \kappa(B^\ord)$ is a quasimonomial basis of $\Bq$.
\end{proposition}

\begin{proof}
By Proposition \ref{r.redOq} the set $B^-=\{b(m) \mid m \in \Gamma \setminus \bG\}$ is a free $R$-basis of $\cI^-$. Since $\kappa(\cI^-) = \ccI^-$, the set $\check B^-=\{\chb(m) \mid m \in \Gamma \setminus \bG\}$ is a free $R$-basis of $\ccI^-$. It follows from Lemma \ref{r.domain3} that $\bar \kappa(\bar B^\ord)$ is a quasimonomial basis of $\Bq$.
\end{proof}

\begin{proof}[Proof of Theorem \ref{r.Bq} part (b)]
As $\bBq$ has a quasimonomial basis, it is a domain.

By Proposition \ref{r.redOq}, the set $\bar G= \{\bar u_{ij} \mid i \ge j\}$ is a uniform GK set of generators of $\LO$. From \eqref{eq.OB} we have $\bar{\kappa}(\Pol_m(\bar G )) \subset \Pol_m(\bar \kappa(\bar G))$. Similarly from \eqref{eq.BO} we have the converse inclusion. Hence $\bar{\kappa}(\Pol_m(\bar G )) = \Pol_m(\bar \kappa(\bar G))$. As $\bar \kappa$ is a linear isomorphism, we conclude that $\bar \kappa(S)$ is a uniform GK set of generators of $\bBq$, and that $\GKdim (\bBq)= \GKdim(\LO) = (n-1)(n+2)/2$.
\end{proof}

\subsection{Proof of Theorem \ref{thm.basisBq}} Define the degrees $d_0$ and $d_1$ for letters $u_{ij}$ and $\chu_{ij} $ by
$$ d_0(u_{ij})= d_0(\chu_{ij})=1, \quad d_1(u_{ij})= d_0(\chu_{ij})=i-j .$$
For a word $w$ in the letters $u_{ij}$ or $\chu_{ij}$ we define $d_0(w)$ and $d_1(w)$ additively. If $w$ is a word in $u_{ij}$ (respectively $\chu_{ij}$) let $[w]\in \Oq$ (respectively $\Bq$) be the element it represents. Let $d_{01}(w)= (d_0(w), d_1(w))\in \BN \times \BZ$.

For alphabets $\{u_{ij}\}$ and $\{\chu_{ij}\}$, the sets of all possible values of $d_{01}(w)$ of all words are the same, and are denoted by
$\Lambda\subset \BN \times \BZ$. Then $\Lambda$ is a submonoid of $\BN\times \BZ$, and is well-ordered in the lexicographic order of $\BN \times \BZ$ because for each $k\in \BN$ there is only a finite number of words $w$ with $d_0\le k$.

For $k \in \Lambda$ let $\Oq_k$ (respectively $\Bq_k$) be the $R$-span of $[w]$, where $w$ are words in $u_{ij}$ (respectively $\chu_{ij}$) with $d_{01}(w) \le k$. Then $(\Oq_k)_{k\in \Lambda}$ is a $\Lambda$-filtration of $\Oq$ and $(\Bq_k)_{k\in \Lambda}$ is a $\Lambda$-filtration of $\Bq$.

The second term on the right-hand side of \eqref{eq.OB} has $d_{01}$ less than that of the remaining terms.
Hence \eqref{eq.OB} implies that $\kappa(\Oq_k) \subset \Bq_k$. Similarly \eqref{eq.BO} implies the $\kappa^{-1}(\Bq_k) \subset \Oq_k$. It follows that $\kappa(\Oq_k) = \Bq_k$. Equ. \eqref{eq.OB} implies
\begin{equation}\label{eq.110a}
\kappa(b(m)) \eqq \check b(m) + \Bq_{ < d_{01}( \check b(m)) }.
\end{equation}
Since $B^\ord$ is a free $R$-basis of $\Oq$, Equ. \eqref{eq.boo5a} implies $B^\ord \cap \Oq_k$ is a free $R$-basis of $\Oq_k$. Hence~\eqref{eq.110a} and induction on $k\in \Lambda$ show that $\check B^\ord \cap \Bq_k$ is a free $R$-basis of $\Bq_k$. It follows that $\check B^\ord$ is a free $R$-basis of $\Bq$.
This completes the proof of Theorem \ref{thm.basisBq}.

\section{Integrality and GK dimension}
\label{sec.Int}

Recall that a punctured bordered surface $\fS$ is \term{essentially bordered} if each connected component of $\fS$ has non-empty boundary, and tensor product factorization was introduced in Subsection \ref{ss.tproduct}.

\subsection{Main results of section}
For an essentially bordered pb surface $\fS$ define
\begin{equation}\label{eq.rfS}
r(\fS) = \# \pfS - \chi(\fS),
\end{equation}
where $\# \pfS$ is the number of components of $\pfS$ and $\chi(\fS)$ is the Euler characteristics.

\begin{theorem} \label{thm.domain}
Let $\fS$ be an essentially bordered pb surface, and
the ground ring $R$ is a commutative domain with a distinguished invertible $\hq$.
\begin{enuma}
\item The algebra $\SS$ is a domain, and is free as an $R$-module.
\item The GK dimension of $\fS$ is
\begin{equation}\label{eq.GKfS}
\GKdim(\SS) = (n^2-1)r(\fS).
\end{equation}
\item There is a tensor product factorization
\[ \SS = A_1 \boxtimes A_2 \boxtimes \dots \boxtimes A_r,\]
where $r= r(\fS)$ and each $A_i$ is isomorphic to either $\Oq$ or $\Bq$.

\item The algebra $\SS$ is orderly finitely generated.
\end{enuma}
\end{theorem}

The idea is to cut $\SS$ along ideal arcs to obtain polygons.

\subsection{Arc algebras}

An oriented $\pfS$-arc $a$ defines the \term{arc algebra} $\cS(a)$ as follows. There are two cases.
\begin{enumerate}[{Case} 1]
\item The two endpoints of $a$ are on two different boundary edges. Let $N(a)$ be a small tubular open neighborhood of $a$. The orientation of $a$ identifies $N(a)$ with a based bigon, where the beginning point of $a$ is on the left edge. Define $\cS(a):= \cS(N(a))$, which is identified with $ \Oq$.
\item The two endpoints of $a$ are in the same boundary edge $b$ of $\fS$. Let $N(a)$ be a small tubular neighborhood of $a \cup b$, which is diffeomorphic to the punctured monogon $\PP_{1,1}$. Define $\cS(a) := \cS(N(a)) = \Bq$. Here our $a$ is identified with the arc $a$ of Figure \ref{fig-mono-corner}.
\end{enumerate}
A collection $A=\{a_1,\dots , a_r\}$ of disjoint oriented $\pfS$-arcs is \term{saturated} if
\begin{enumerate}[(i)]
\item each connected component of $\surface \setminus \bigcup_{i=1}^r a_i$ contains exactly one ideal point (interior or boundary) of $\surface$, and
\item $A$ is maximal with respect to the above condition.
\end{enumerate}

\begin{theorem}[Corollary 7.20 of \cite{LS}]\label{thm.saturated}
Assume $\{a_1,\dots , a_r\}$ is a saturated system of oriented $\pfS$-arcs, where $\fS$ is an essentially bordered pb surface.
\begin{enumerate}
\item $r= \#\pfS - \chi(\fS)$.
\item For each $i$, the embedding $N(a_i) \embed \fS$ induces an embedding of algebras $\cS(a_i) = \cS(N(a_i)) \embed \SS$. We identify $\cS(a_i)$ with the image under the embedding.
\item The algebras $\cS(a_1), \dots, \cS(a_r)$ form a tensor product factorization of $\SS$.
\end{enumerate}
\end{theorem}

\begin{proof}
Parts (a), (b), and the fact that $\cS(a_i), \dots, \cS(a_i)$ form a weak tensor product factorization of $\SS$ was proved in \cite[Corollary 7.20]{LS}.

For each $i$ let $G_i$ be the set of stated $\pfS$-arcs which are $a_i$ with all possible states. That is, $G_i$ is the image of $G=\{u_{ij} \mid i,j \in \JJ\}$ under the identification $\cS(a_i) = \Oq$. Then $G_i$ is an algebra generator set for $\cS(a_i)$.

Let $a_i\overleftarrow\sqcup a_j$ be the $n$-web diagram, which is $a_i \sqcup a_j$, with boundary order defined so that on each boundary edge any endpoint of $a_i$ is higher than any endpoint of $a_j$. Then $\Pol_1(G_i) \Pol_1(S_j)= M(a_i\overleftarrow\sqcup a_j)$, where $M(\al)$ is defined in Subsection \ref{ss.height}. Because $a_i\overleftarrow\sqcup a_j$ and $a_j\overleftarrow\sqcup a_i$ differ only in the boundary order, by Lemma \ref{r.height} we have $\Pol_1(G_i) \Pol_1(G_j) = \Pol_1(G_j) \Pol_1(G_i)$. Thus $\cS(a_i), \dots, \cS(a_i)$ form a tensor product factorization of $\SS$.
\end{proof}

Recall the notion of strict embedding in Subsection \ref{sec.embed}.

\begin{corollary}\label{r.embed1}
Suppose $\fS'\embed \fS$ is a strict embedding of essentially bordered pb surfaces. Assume that there is a saturated system of $\fS'$ which is a subset of a saturated system of $\fS$. Then the natural map $\cS(\fS') \to \SS$ is an algebra embedding.
\end{corollary}

\subsection{Integrality for the polygon}

Recall that $\PP_k$ is the $k$-gon with a based vertex. Let $v_1,\dots, v_k$ be the vertices of $\PP_k$ in counterclockwise order, beginning at the based vertex. Let $a_i$ be the oriented corner arc at $v_i$ as depicted in Figure \ref{fig-poly-arc}. Fix a linear order $\ord$ on the set $\JJ^2$. By Proposition \ref{r.Oq} the set $B^\ord =\{b(m) \mid m\in \Gamma\}$ is a quasimonomial basis of $\Oq$, parameterized by the enhanced monoid $(\Gamma, d_2)$. Let $b_i(m)$ is the image of $b(m)$ under the identification $\cS(a_i) = \Oq$.

\begin{figure}
\centering
\input{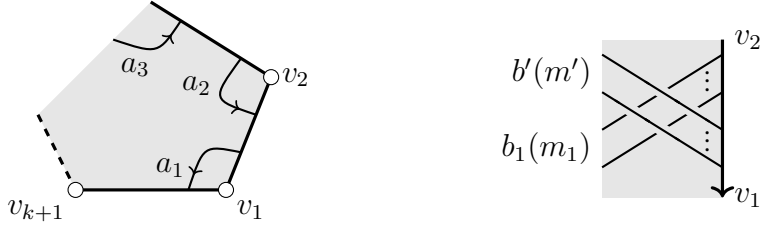}
\caption{Corner arcs of the polygon $\poly_{k+1}$ and the product $b'(m')b_1(m_1)$}\label{fig-poly-arc}
\end{figure}

\begin{proposition}\label{r.polygon}
The algebra $\cS(\PP_k)$ has a quasimonomial basis
\[B:= \{b_1(m_1) \dots b_{k-1}(m_{k-1}) \mid (m_1,\dots, m_{k-1}) \in \Gamma^{k-1} \}\]
parameterized by $(\Gamma^{k-1}, d_2^{k-1})$. Consequently $\cS(\PP_k)$ is a domain.
\end{proposition}

\begin{proof}
We proceed by induction. When $k=2$ this is Proposition \ref{r.Oq}. Assume that the statement is true for $k$.

If we remove the boundary edge connecting $v_1$ and $v_{k+1}$ from $\PP_{k+1}$, the result is a $k$-gon, for which $a_2, \dots, a_k$ form a saturated system. By Theorem \ref{thm.saturated} and Corollary \ref{r.embed1}, we can identify $\cS(\PP_k)$ with the subalgebra of $\cS(\PP_{k+1})$ generated by $\cS(a_2), \dots, \cS(a_k)$. Besides $\skein(a_1)$ and $\skein(\poly_k)$ form a tensor product factorization of $\cS(\PP_{k+1})$.

Let $b'(m')=b_2(m_2) \dots b_k(m_k)$ for $m'=(m_2,\dots, m_k) \in \Gamma^{k-1}$. By the induction hypothesis, $B'=\{b'(m') \mid m' \in \Gamma^{k-1} \}$ is a quasimonomial basis of $\skein(\poly_k) \subset \skein(\poly_{k+1})$. Consider the commutation of an element of the basis $B'$ and an element of the basis $B^\ord$ of $\cS(a_1)$. We have
\[ b'(m') b_1 (m_1) \eqq b_1(m_1) b'(m') + \sum_{d_2(m'_1) < d_2(m_1)} b_1(m'_1) \cS(\PP_{k}),\]
which follows from Lemma~\ref{r.upper}, where $\vec{j}'\ll\vec{j}$ implies $d_2(m'_1)<d_2(m_1)$. See Figure \ref{fig-poly-arc}. By Lemma~\ref{r.domain2}, the set $B$ is a quasimonomial basis for $\cS(\PP_{k+1})$ parameterized by $(\Gamma^k, d_2^k)$.
\end{proof}

\subsection{Proof of Theorem \ref{thm.domain}}

(a) Cut $\fS$ along ideal arcs to get a disjoint union $\fS'$ of polygons $P_1, \dots, P_k$. If $i\neq j$ then each element of $\cS(P_i)$ commutes with each element of $\cS(P_j)$, and each $\cS(P_i)$ has a quasimonomial basis by Proposition \ref{r.polygon}. Hence by Lemma~\ref{r.domain2} the algebra $\cS(\fS')= \ot \cS(P_i)$ has a monomial basis, and is a domain.

By Theorem \ref{t.splitting2}, the cutting homomorphism $\cS(\fS)\to \cS(\fS')$ is an embedding. It follows that $\cS(\fS)$ is a domain.

(b) and (c). Since $\fS$ is essentially bordered, it has a saturated system of $\pfS$-arcs $a_1, \dots, a_r$. By Theorem \ref{thm.saturated}, $\cS(a_1), \dots, \cS(a_k)$ form a tensor product factorization of $\SS$. Each $\cS(a_i)$ is either $\Oq$ or $\Bq$, and both have uniform GK dimension $n^2-1$, by Propositions~\ref{r.Oq} and~\ref{r.Bq}. By Proposition \ref{r.tensorP}, the GK dimension of $\SS$ is $r(n^2-1)$. By Theorem \ref{thm.saturated}(a), we have $r= r(\fS)=\#\pfS -\chi(\fS)$.

(d) By part (c) we have $\SS = A_1 \dots A_r$. Each $A_i$ is orderly finitely generated by Proposition \ref{r.basisOq} and Theorem \ref{thm.basisBq}. Hence by Lemma~\ref{r.ogen} the algebra $\SS$ is orderly finitely generated.

\section{Reduced skein algebra} \label{sec.Rd}

We define the reduced skein algebra $\bSS$ and establish some of its properties.

\subsection{Bad arcs} \label{ss.badarcs}

Let $v$ be a vertex of a pb surface $\fS$. We call $v$ a \term{monogon vertex} if the connected component of $\fS$ containing $v$ is a monogon, having $v$ as its only vertex. The \term{corner arcs} $C(v)_{ij}$ and $\ceC(v)_{ij}$, where $i,j \in \JJ $, are depicted in Figure \ref{fig-corner}. We also denote by $C(v)$ (resp. $\ceC(v)$) the arcs $C(v)_{ij}$ (resp. $\ceC(v)_{ij}$) without states. Note that $C(v)$ is a trivial $\pfS $-arc if and only if $v$ is a monogon vertex.

\begin{figure}
\centering
\input{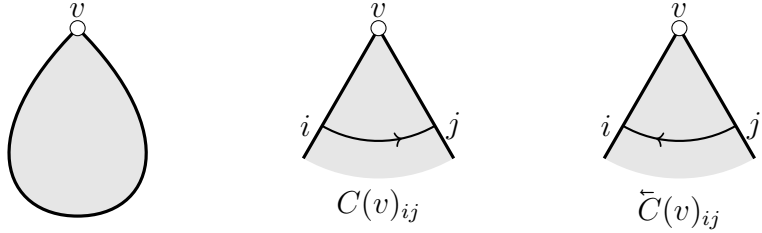}
\caption{Monogon vertex and corner arcs}\label{fig-corner}
\end{figure}

For a non-monogon vertex $v$, let
\[C_v = \{C(v)_{ij} \mid i <j \} , \quad \ceC_v = \{\ceC(v)_{ij} \mid i <j\}.\]
If $v$ is a monogon vertex let $C_v = \ceC_v = \emptyset$.
An element of $C_v$ or $\ceC_v $ is called a \term{bad arc} at $v$. Let $\Ibad_v \lhd\SS$ be the 2-sided ideal generated by $C_v \cup \ceC_v$, and $\Ibad \lhd \SS$ be the two-sided ideal generated by all bad arcs. The quotient algebra
\[\Srd(\fS):= \cS(\fS)/\Ibad\]
is called the \term{reduced $SL_n$-skein algebra} of $\fS$.

\begin{proposition}\label{prop-cut-rd}
Suppose $c$ is an interior ideal arc of a punctured bordered surface $\surface$. Then the cutting homomorphism $\Theta_c:\skein(\surface)\to\skein(\Cut_c(\surface))$ descends to the reduced algebra
\begin{equation}
\rd\Theta_c:\reduceS(\surface)\to\reduceS(\Cut_c(\surface)).
\end{equation}
\end{proposition}

\begin{proof}
We just need to show that the image of bad arcs are in $\Ibad$. Let $\alpha=C(v)_{ij}$ be a bad arc. The case $\ceC(v)_{ij}$ is similar.

If the ideal arc $c$ does not end on $v$, then after isotopy, $\alpha$ is disjoint from $c$. Then the image is clearly a bad arc.

If one of the endpoints of $c$ is $v$, then the image of $\alpha$ has the form
\begin{equation}
\Theta_c(\alpha)=\sum_{s\in\JJ}C(v_1)_{is}C(v_2)_{sj},
\end{equation}
where $v_1,v_2$ are vertices of $\Cut_c(\surface)$ corresponding to $v$. The bad condition $i<j$ implies that at least one of the factors is bad in each term of the sum. Hence the image is in $\Ibad$.

The last case where both endpoints of $c$ are $v$ is similar.
\end{proof}

\subsection{Normality of $C_v$ and \texorpdfstring{$\ceC_v$}{cevCv}}

\begin{theorem}\label{r.normal1}
Let $\al$ be a non-stated $n$-web diagram. Recall that $M(\al)\subset \SS$ is $R$-span of all stated $n$-webs which are $\al$ with arbitrary states. We have
\begin{align}
C_v M(\al)& = M(\al) C_v, \quad \ceC_v M(\al)= M(\al) \ceC_v, \label{eq.normal1}\\
C_v \SS &= \ceC_v \SS. \label{eq.normal2}
\end{align}
Consequently,
\begin{align}
\Ibad_v & =C_v \SS = \ceC_v \SS= \SS C_v = \SS \ceC_v \label{eq.Ibadv}\\
\Ibad &= \sum _{v: \text{vertices}} \Ibad_v. \label{eq.Ibad}
\end{align}
\end{theorem}

\begin{proof}
For a map $s:\partial\alpha\to\{1,\dots,n\}$ let $(\alpha,s)$ be the stated $n$-web diagram which is $\al$ stated by $s$. Define a partial order $\preceq$ on $\JJ^2$ such that $(i',j') \preceq (i,j)$ if $i' \le i$ and $j' \ge j$.

\begin{lemma}\label{lemma-comm-arc}
In $\skein(\surface)$, for $a_{ij} = C(v)_{ij}$ or $\ceC(v)_{ij}$, one has
\begin{align}
a_{ij}(\al, s)& \eqq (\alpha,s)a_{ij} +
\Span\{(\al, s') a_{i'j'}\mid (i',j') \prec (i,j)\},\label{eq.hmono}\\
(\alpha,s)a_{ij}&\eqq a_{ij}(\al, s) + \Span\{a_{i'j'} (\alpha, s')\mid (i',j') \prec (i,j)\}.\label{eq.hmono2}
\end{align}
\end{lemma}

\begin{proof}
The case when $R=\Zq$ implies the general case. Assuming $R=\Zq$, the two statements are related by the reflection of Subsection \ref{sec.reflection}. Let us prove \eqref{eq.hmono2}.

Since $a_{ij}$ is a corner arc, $\alpha$ can be isotoped so that it does not intersect $\al$.

First, suppose $v$ is incident to two different edges. The calculation can be done in a neighborhood of $a_{ij}$, which is identified with the bigon such that $v$ is the top vertex. Then

\begin{align}
(\alpha,s)a_{ij}&=\bubblestrand{b}{$\alpha$}{$s_1$}{$s_2$}{$i$}{$j$}
=\bubblesplit{$\alpha$}{$s_1$}{$s_2$}{$i$}{$j$}\\
&=\sum_{i',j',s'}\epsilon\left(\crossSt{<-}{<-}{p}{}{}{$i$}{$s_1$}{$s'_1$}{$i'$}\right) \bubblestrand{t}{$\alpha$}{$s'_1$}{$s'_2$}{$i'$}{$j'$} \epsilon\left(\crossSt{<-}{<-}{n}{}{}{$s'_2$}{$j'$}{$j$}{$s_2$}\right)\\
&=\sum_{i',j',s'}\epsilon\left(\crossSt{<-}{<-}{p}{}{}{$i$}{$s_1$}{$s'_1$}{$i'$}\right) \epsilon\left(\crossSt{<-}{<-}{n}{}{}{$s'_2$}{$j'$}{$j$}{$s_2$}\right) a_{i'j'}(\alpha,s')
\end{align}
Here, the second line uses the coactions on both edges of the bigon. By Lemma \ref{r.upper}, the counit values are $\eqq 1$ if $(i',j') = (i,j)$, and non-zero only if $i' \le i$, $j' \ge j$ (and $s'_1\gg s_1$, $s'_2\ll s_2$). This implies~\eqref{eq.hmono2}.

Now suppose $v$ is incident to only one edge. The neighborhood of $a_{ij}$ together with the boundary edge is identified with the punctured monogon.
\begin{equation}
(\alpha,s)a_{ij}=\bubblemono
=\sum_{i',j',s'}a_{i'j'}(\alpha,s')\epsilon\left(\bubblemonosplit\right).
\end{equation}
In this picture, the unmarked sides of the square is the vertex $v$, and the unshaded region in the middle is the puncture of the monogon. By Lemma \ref{r.upper}, the counit is $\eqq 1$ if $(i',j') = (i,j)$, and non-zero only if only if $(i',j')\preceq(i,j)$. This implies \eqref{eq.hmono2}.
\end{proof}

Let us prove \eqref{eq.normal1}. If $(i',j') \preceq(i,j)$ and $C(v)_{ij}$ is a bad arc, then $C(v)_{i'j'}$ is also a bad arc since $i'\le i< j \le j'$. Equation \eqref{eq.hmono} shows that $C_v M(\al)\subset M(\al) C_v$ and $\ceC_v M(\al)\subset M(\al) \ceC_v$. The converse inclusions follow from \eqref{eq.hmono2}. This proves \eqref{eq.normal1}.

Let us prove \eqref{eq.normal2}. First assume $v$ is the based vertex of the bigon $\PP_2$. By the identification $\cS(\poly_2) = \Oq$ we have $C(v)_{ij} = u_{ij}$. Thus $C_v \cS(\poly_2) = \cI^-$, and \eqref{eq.normal2} is Identity \eqref{eq.normal11}.

Similarly, when $v$ is the vertex of the punctured monogon $\poly_{1,1}$, Identity \eqref{eq.normal2} follows from~\eqref{eq.normal12}.

Now assume $v$ is an arbitrary non-monogon vertex of a pb surface $\fS$. Then $\cS(C(v))$ is either $\Oq$ or $\Bq$, according as $v$ is incident with two different edges or one edge. In either case, we have
\[ C_v \SS= C_v \cS(C(v)) \SS = \ceC_v \cS(C(v)) \SS = \ceC_v \SS,\]
proving \eqref{eq.normal2}.

Let us prove \eqref{eq.Ibadv}. By \eqref{eq.normal1}, we have $C_v \SS = \SS C_v$. This implies $I_1=C_v \SS$ is a two-sided ideal. Similarly, $I_2= \ceC_v \SS= \SS \ceC_v$ is a two-sided ideal. By \eqref{eq.normal2}, we have $I_1=I_2$, which proves \eqref{eq.Ibadv}.

Let us prove \eqref{eq.Ibadv}. It follows from \eqref{eq.normal1} that $(\cup_{v} C_v)$ is $\SS$-normal. It follows that $I= (\cup_{v} C_v) \SS$ is two-sided ideal. Similarly $I'= (\cup_{v} \ceC_v) \SS$ is a two-sided ideal. From \eqref{eq.normal2} we have $I=I'$, which implies that $I=I'=\Ibad$. Then
\[ \Ibad = (\cup_{v} C_v) \SS = \sum_v C_v \SS = \sum_v \Ibad_v,\]
proving \eqref{eq.Ibadv}.
\end{proof}

\subsection{Top right corner quantum minor} \label{sec.Du}

Recall that $M^{I}_{J}(\buu)\in\Oq$ is the $(I \times J)$ quantum minor of the quantum matrix $\buu$, where $I, J \in \binom{\JJ}k$. Also $[i;j]= \{ k \in \BZ, i \le k \le j\}$.

For each $i\in \JJ$ let $D_i(\buu):= M^{[1;i]}_{[\bar{i};n]}(\buu)\in\Oq$, which is a top right corner quantum minor of size $i$. Note that $D_n(\buu)={\det}_q(\buu)=1$. In the notation of Subsection \ref{ss.Weyl},
\[D_i(\vec{u})=\input{minor-Di}.\]
Let $D(\buu):= D_1(\buu) D_2(\buu) \dots D_{n-1} (\buu)$. By \cite[Theorem 4.3]{JZ}, we have

\begin{proposition}\label{r.73}
Any two $D_i(\buu), D_j(\buu)$ are commuting, and each $D_i(\buu)$ is $q$-commuting with each $u_{kl}$. Consequently $D(\buu)$ is $q$-commuting with each $u_{kl}$.
\end{proposition}

This fact will be generalized in Lemma~\ref{r.commun2} below.

\subsection{Algebra near a vertex} \label{sec.corner0}

For an oriented $\pfS$-arc $a$ whose endpoints are on two different boundary edges, let $D_i(a), D(a), M^I_J(a) \in \SS$ be the images of $D_i(\buu), D(\buu), M^I_J(\buu)$ respectively under the algebra homomorphism $\Oq= \cS(a) \to \SS$.

\begin{lemma}\label{r.commun2}
Assume $v$ is a vertex of a pb surface which is incident with two different boundary edges. Let $a= C(v)$.

For $i, j\in \JJ$ the elements $D_i(a)$ and $D_j(a)$ commute, and $D_i(a)$ is $q$-commuting with any stated $\pfS$-arc $\al$ in $\SS$.

Consequently $D(a)$ is $q$-commuting with every stated $\pfS$-arc in $\SS$.
\end{lemma}

\begin{proof}
As $D_i(\buu) D_j(\buu)= D_j(\buu) D_i(\buu)$, we have $D_i(a)D_j(a)= D_j(a) D_i(a)$.

By mimicking the calculations in Lemma~\ref{lemma-comm-arc} where $a_{ij}$ is replaced with $M^I_J(a)$, we get
\begin{equation}
M^I_J(a)\alpha \eqq \alpha M^I_J(a) + \sum_{\substack{I'\ll I\\ J'\gg J}}M^{I'}_{J'}(a)\skein(\surface).
\end{equation}
Since $D_i(a)=M^I_J(a)$ where $I=[1,i]$ and $J=[\bar{i},n]$, there are no $I'$ or $J'$ satisfying the restriction of the sum. Thus the sum is empty, and the equation reduces to a $q$-commuting relation.
\end{proof}

\subsection{Individual $\Ibad_v$}

In view of \eqref{eq.Ibad} let us study $\SS/\Ibad _v$, where $v$ is a vertex of $\fS$. For example, when $\fS=\PP_2$, the bigon, and $v$ is the based vertex, then $\SS/\Ibad _v$ is exactly $\LO$, if we identify $\cS(\PP_2)= \Oq$.

\begin{lemma}\label{r.Ibad1}
Let $v$ be a vertex of a pb surface $\fS$ incident with two different edges. Let $C_v^{\mathrm{diag}}=\{C(v)_{ii}, \ceC(v)_{ii} \mid i\in \JJ\}$.
\begin{enuma}
\item Any two elements from $C_v^{\mathrm{diag}}$ commute in $\SS/\Ibad_v$. Moreover in $\SS/\Ibad_v$,
\begin{gather}
\prod_{i=1}^nC(v)_{ii} = \prod_{i=1}^n\ceC(v)_{ii} =1,\\
\ceC(v)_{ii} = C(v)_{\bar i\,\bar i}^{-1}. \label{eq.ceCii}
\end{gather}
Consequently every element of $C_v^{\mathrm{diag}}$ is invertible in $\SS/\Ibad _v$.
\item For any $x\in C_v^{\mathrm{diag}}$ and any state $n$-web $y$ over $\fS$, $x y \eqq y x$ in $\SS/\Ibad _v$.
\item Let $a=C(v)$ or $\cev{C}(v)$. Then in $\SS/\Ibad _v$,
\begin{equation}
M^{[1;i]}_J(a)=\begin{cases}
0,&J\ne[1,i],\\
\prod_{k=1}^i a_{kk},&J=[1;i].
\end{cases}
\qquad
M^I_{[\bar{i};n]}(a)=\begin{cases}
0,&I\ne[\bar{i};n],\\
\prod_{k=\bar{i}}^n a_{kk},&I=[\bar{i};n].
\end{cases}
\end{equation}
Equivalently, in terms of diagrams,
\begin{equation}\label{eq-minor-vrd}
\input{minor-vrd2}
\end{equation}
\end{enuma}
\end{lemma}

\begin{proof}
(a) First assume $\fS=\PP_2$ and $v$ is the based vertex. Then $\cS(\fS)= \Oq$ and $\Ibad= \cI^-$. Hence $\SS/\Ibad _v= \LO$. Besides $\LO \ni \bar u_{ij} = C(v)_{ij}\in \SS/\Ibad _v$.

In Subsection \ref{sec.Bq}, we see that any two $\bar u_{ii}$ and $\bar u_{jj}$ commute. By \eqref{eq.diag0}, we have $\prod_{i} \bar u_{ii}=1$. Applying the antipode $S$ we get $\prod_{i} \ceC(v)_{ii}=1$. From \eqref{eq.Saij} and then \eqref{e-Oq-ops} we have
\[ \ceC(v)_{ii} = S(\bar u_{ii}) = \prod_{j\neq \bar i} \bar u_{jj}= \bar u_{\bar i\,\bar i}^{-1}=C(v)_{\bar i\,\bar i}^{-1},\]
which proves \eqref{eq.ceCii}. From here we see that any two elements of $\{C(v)_{ii}, \ceC(v)_{ii} \mid i\in \JJ\}$ commute. This proves (a) for the case when $\fS=\PP_2$.

Consider now the general case. The arc $C(v)$ gives an algebra homomorphism from $\Oq$ onto $\cS(C(v))$ which maps $\cI^-$ onto $\Ibad_v$. Hence it descends to an algebra homomorphism $ \LO \to \SS/\Ibad _v$ which shows that all the statements in (a) are true for $\fS$.

(b) In \eqref{eq.hmono}, let $a_{ii}=x$ and $(\al,s)=y$. If $(i',j') \prec (i,i)$ then $a_{i'j'}$ is a bad arc. Hence \eqref{eq.hmono} implies $xy\eqq yx$ in $\SS/\Ibad _v$.

(c) Consider the first identity. Assume $J= \{j_1 < \dots < j_i \}$. The left-hand side is
\[ \text{LHS}= M^{[1;i]}_{J}= \sum_{\sigma \in \Sym i} (-q)^{\ell(\sigma)} \bar u_{1 j_{\sigma(1)}} \dots \bar u_{i j_{\sigma(i)}}.\]
If $J \neq [1;i]$ then there is $k\in [1,i]$ such that $k < j_{\sigma(k)}$, showing that the each term in the sum is zero. On the other hand when $I= [1;i]$ the only non-zero term is the one with $\sigma= \id$. Hence we have the formula.

The proof the second identity is similar.
\end{proof}

\section{Reduced skein algebra of polygons} \label{sec.Rd2}

Recall that $\PP_k$ is the ideal $k$-gon, with vertices $v_1, \dots, v_k$ in counterclockwise order.
We will show that the reduced skein algebra $\cS(\PP_k)$ of a polygon is a domain, calculate its GK dimension, and give an explicit description for the case $k=3$.

\subsection{Main results of section}

\begin{theorem}\label{thm.domainr}
The algebra $\Srd(\PP_k)$ is an $R$-torsion free domain with GK dimension
\begin{equation}\label{eq.GKrd}
\GKdim(\Srd(\PP_k)) = k\frac{(n-1)(n+2)}{2} - n^2+1.
\end{equation}
\end{theorem}

Actually the proof will give an explicit description of $\bS(\PP_k)$. Let us spell out the details for $\PP_3$, an important case for us.

Recall that $\LO=\Oq/\cI^-$ has set of algebra generators $\{\bar u_{ij}, j \le i \in \JJ\}$ and is a domain of uniform GK dimension $(n-1)(n+2)/2$, see Subsection \ref{sec.Bq}. For $i\in \JJ$ let $\tau_i: \LO \to \LO$ be the diagonal automorphism defined by
\[ \tau_i(\bar u_{jk}) = q^{\delta_{ij} - 1/n}\bar u_{jk}.\]
It is easy to check that $\tau_i$ is a well-defined algebra automorphism of $\LO$, and that $\tau_i \tau_j = \tau_j \tau_i$.

{ For $i\in \JJ$ consider the bottom left quantum minor ${}_i D:=M^{[\bi;n]}_{[1;i]}(\bar \buu)\in \LO$. Let
\[\DD =\ ({}_1D) ({}_2D) \dots ({}_{n-1}D)\in \LO.\]
}

\begin{theorem} \label{thm.PP3}
\begin{enuma}
\item { $\DD$ is a non-zero element $q$-commuting with $\LO$, and is an eigenvector of each automorphism $\tau_i, i=1, \dots, n-1$.} Consequently, one can define the Ore localization $\LO\{\DD\}^{-1}$ and then the iterated skew-Laurent extension
\[\LO\{\DD\}^{-1}[ x_1^{\pm 1}, \dots, x_{n-1}^{\pm 1};\tau_1, \dots, \tau_{n-1}],\]
as in Example \ref{exa.001}.
\item There is a unique algebra isomorphism
\[\LO\{\DD\}^{-1}[ x_1^{\pm 1}, \dots, x_{n-1}^{\pm 1};\tau_1, \dots, \tau_{n-1}] \xrightarrow{\cong} \Srd(\PP_3)\] given by
\[ \bar u _{ij} \to C(v_1)_{ij}\ \text{for } i \ge j \in \JJ, \ x_i \to \ceC({v_2})_{ii}, \ \text{for } i = 1, \dots , n-1 .\]
\end{enuma}
\end{theorem}

{The reason for using a different corner than Section~\ref{sec.Du} is purely conventional. When $\DD$ is defined in Lemma~\ref{r.AkBk} for all polygons, we still use the top right corner. However, the arc $a$ used there, specialized to $\poly_3$, is opposite of the choice $a_1$ used for the identification $\bar{A}_1=\LO$ in Section~\ref{sec-poly-quot}. This transposes the indices.}

\subsection{Quotients of $\cS(\PP_k)$}
\label{sec-poly-quot}

Recall that if $S$ is a subset of a ring $A$ then we denote $A/(S)$ the quotient $A/I$, where $I$ is the ideal generated by $S$.

Let $a_i= C(v_i)$ be the oriented corner arc at $v_i$. Let $A_i:= \cS(a_i)$ which is identified with $ \Oq$. Under the identification, the ideal $I_i:= C_{v_i} A_i \lhd A_i$ is equal to $\cI^-$. Hence we can identify $\bar A _i:=A_i/I_i\equiv \LO$. By Theorem \ref{thm.saturated}, any $k-1$ from the algebras $A_1,A_2, \dots, A_{k}$ form a tensor product factorization of $\cS(\PP_k)$. In particular, $\cS(\PP_k)= A_2 \dots A_{k}$.

By \eqref{eq.normal1} each $I_i$ is $A_j$-normal,
\begin{equation}\label{eq.com4}
I_i A_j = A_j I_i.
\end{equation}
Hence each $I_i$ is $\cS(\PP_k)$-normal, and $\cI^\bad \lhd \cS(\PP_k)$ has the form
\begin{equation}
\cI^\bad = \sum _{i=1}^k I_i \cS(\PP_k).
\end{equation}

For $0\le l<k$ let $A_{k,l}= \cS(\PP_k)/I_{k,l}$, where $ I_{k,l}= I_{{l+1}}\cS(\PP_k) + \dots +I_{k}\cS(\PP_k).$ Note that $A_{k,0} = \Srd(\PP_k)$.

\begin{lemma} \label{r.Akl}
Let $1 \le l <k$. The algebra $A_{k,l}$ is a domain and a free $R$-module, and it has GK dimension
\begin{equation}\label{eq.GKdim7}
\GKdim(A_{k,l})= (l+k-2)\frac{n(n-1)}{2} + (k-1)(n-1).
\end{equation}
\end{lemma}

\begin{proof}
The normality \eqref{eq.com4} and Lemma \ref{r.tensorP} show that $A_2, \dots, A_l, \bar A_{l+1}, \dots, \bar A_{k}$ form a tensor product factorization of $A_{k,l}$. In particular, as $R$-modules $A_{k,l}$ is isomorphic to the tensor product of all the factors. Since each factor is a free $R$-module, so is $A_{k,l}$. Each factor $A_i=\Oq$ has uniform GK dimension $n^2-1$, by Proposition \ref{r.Oq}, and each factor $\bar A_i=\LO$ has uniform GK dimension $(n+2)(n-1)/2$, by Proposition \ref{r.redOq}. By Proposition \ref{r.tensorP}, the GK dimension is additive when each factor has uniform GK dimension. This proves \eqref{eq.GKdim7}.

Let us use the notations of the proof of Proposition \ref{r.redOq}, where it is proved that the set $\{b(m) \mid m \in \Gamma\setminus \bG\}$ is a free $R$-basis of $\cI^-$ and $\{e(m) \mid m \in \Gamma\}$ is a free $R$-basis of $\Oq$. Under the identification $\cS(a_i) = \Oq$ let $b_i(m)$ be the element corresponding to $b(m)$. By Proposition \ref{r.polygon} the set
\[\{b_2(m_2) \dots b_{k}(m_{k}) \mid ( m_2, \dots, m_{k}) \in \Gamma^{k-1} \}\]
is a quasimonomial basis of $\cS(\PP_k)$. The normality \eqref{eq.com4} shows that for $i\ge 2$,
\[I_i \cS(\PP_k)= A_2 \dots A_{i-1} I_i A_{i+1} \dots A_{k},\]
which is spanned by $\{b_2(m_1) \dots b_k(m_k) \mid m_i \in \Gamma \setminus \bG, ( m_2, \dots, m_{k}) \in \Gamma^{k-1} \}$. Hence $I_{k,l}$ is spanned by
\[\{b_2(m_1) \dots b_k(m_k) \mid ( m_2, \dots, m_{k}) \in \Gamma^{k-1} \setminus \Gamma^{l-1}\times \bG^{k-l} \}.\]
By Lemma \ref{r.domain3}, the quotient $A_{k,l}= \cS(\PP_k)/I_{K,l}$ has a quasimonomial basis and is a domain.
\end{proof}

\subsection{A copy of $\PP_{k-1}$ in $\PP_k$}

The result of removing the edge $v_{k-1}v_k$ from $\PP_k$ is $\PP_{k-1}$, giving an embedding $\PP_{k-1} \embed \PP_k$. Note that $a_1, \dots, a_{k-2}$ form a saturated system for $\PP_{k-1}$. By Corollary \ref{r.embed1}, the embedding $\PP_{k-1} \embed \PP_k$ induces an embedding $\cS(\PP_{k-1}) \embed \cS(\PP_k)$, and we identify $\cS(\PP_{k-1})$ with the image of this embedding. Let
\begin{align}
E_k &:=\cS(\PP_k)/(I_{{2}} + \dots +I_{k})= A_{k,1} \label{eq.EkAk} \\
B_{k-1}&:= \cS(\PP_{k-1}) /(I_{{2}}+ \dots +I_{k-2}). \notag
\end{align}
The embedding $\cS(\PP_{k-1}) \embed \cS(\PP_k)$ descends an algebra homomorphism
\begin{equation}\label{eq.h}
h:B_{k-1}\to E_k.
\end{equation}

Let $a,b,c$ be the oriented $\partial \PP_k$-arcs depicted in Figure~\ref{fig-poly-arc-rd}. Note that $b= \ceC(v_{k-1})$.

\begin{figure}
\centering
\input{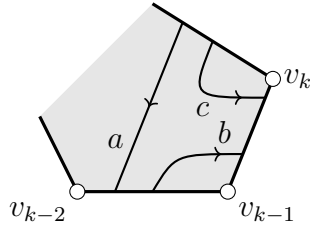}
\caption{Arcs $a,b,c$}\label{fig-poly-arc-rd}
\end{figure}

Recall we define $D(a)\in \cS(\PP_{k-1})$ in Subsection \ref{sec.corner0}.

\begin{lemma} \label{r.AkBk}
Let $\DD\in B_{k-1}$ be the image of $D(a)$ under the projection $\cS(\PP_{k-1})\to B_{k-1}$.
\begin{enuma}
\item $\DD$ is $q$-commuting with $B_{k-1}$.
\item $h(\DD)$ is invertible in $E_k$. Consequently $\DD$ is regular.
\end{enuma}
\end{lemma}

\begin{proof}
(a) As $D(a)$ is $q$-commuting with $\cS(\PP_{k-1})$ by Lemma \ref{r.commun2}, the element $\DD$ is $q$-commuting with $B_{k-1}$.

(b) By Lemma \ref{r.cutdet}, there are units $C_J\in R$ such that
\[ D_i(a) = M^{[1;i]} _{[\bi ; n]} (a) = \sum _{J \subset \binom{\JJ}{i} } C_J\, M^{\bar J }_{ [\bi; n]} (\cev b)\, M^{[1;i]} _{J}(c). \]
By Lemma \ref{r.Ibad1}(c), the element $M^{[1;i]} _{J}(c) $ is non-zero only when $J=[1;i]$. Then $\bar{J}=[\bar{i};n]$, and both $M^{[1;i]} _{[1;i]}(c)$ and $ M^{[\bi ;n] }{ [\bi; n]} (\cev b)$ are invertible because each is a product of diagonal elements at vertices $v_{k-1}$ and $v_{k}$ up to an invertible scalar, see Lemma \ref{r.Ibad1}. It follows that $\DD$ is invertible. Since $E_k \cong A_{k,1}$ is a domain and a non-zero algebra by Lemma \ref{r.Akl}, any invertible element of $E_k$ is regular.
\end{proof}

By Lemma \ref{r.AkBk} the element $\DD$ is $q$-commuting with $B_{k-1}$. Hence by Lemma \ref{r.GKdim} we can define the right Ore localization
$B_{k-1}\{\DD^{-1}\}$ which has the GK dimension of $B_{k-1} \cong A_{k-1,2}$. Using \eqref{eq.GKdim7} we have
\begin{equation}\label{eq.GK9}
\GKdim (B_{k-1}\{\DD^{-1}\}) = \GKdim (A_{k-1,2}) = (k-1) n(n-1)/2 + (k-2)(n-1).
\end{equation}

\subsection{Diagonal automorphisms of $B_{k-1}$}

For each $i\in \JJ$ define
\[ \eta_i: \JJ \to R, \quad \eta_i(j) = q^{\delta_{ij} - 1/n}.\]
It is easy to check that $\prod_{j\in \JJ} \eta_i(j) =1$. For the edge $v_{k-2}v_{k-1}$, we can define the diagonal automorphism
\[\tau_i:= \phi_{v_{k-2}v_{k-1}, \eta_i}: \cS(\PP_{k-1}) \to \cS(\PP_{k-1})\]
using Proposition \ref{r.edgeweight}, and $\tau_i, \tau_j$ commute for any $i,j\in\JJ$. By definition, any stated arc (in particular a bad arc) is an eigenvector of $\tau_i$ with an invertible eigenvalue. This shows
\begin{equation}\label{eq.taui}
\tau_i I_j= I_j.
\end{equation}
Hence $\tau_i$ descends to a diagonal automorphism of the quotient $B_{k-1}$, which is also denoted by $\tau_i$.

\begin{lemma} \label{r.85}
For each $i$, the element $\DD$ is an eigenvector of $\tau_i$.
\end{lemma}

\begin{proof}
It is enough to show that each $D_j(a)$ is an eigenvector of $\tau_i$. By definition
\[\tau_i(a_{mj})= q^{-\delta_{i\bar j} + 1/n} a_{mj}.\]
Thus $a_{mj}$ is an eigenvector with eigenvalue $q^{-\delta_{i\bar j} + 1/n}$, which depends only on the second index $j$. By the determinant formula,
\begin{equation}\label{eq.detDj}
D_j(a) = \sum_{\sigma \in \Sym_j} (-q)^{\ell (\sigma)} a_{\sigma(1), \bar j} a_{\sigma(2),\bar j +1} \dots a_{\sigma(j), n}.
\end{equation}
All the terms of the right-hand side are eigenvectors of $\tau_i$ of the same eigenvalue. Hence $D_j(a)$ is also an eigenvector of that same eigenvalue.
\end{proof}

It follows from the lemma above that $\tau_i$ extends to a diagonal automorphism on the localization $B_{k-1}\{\DD^{-1}\}$, and we denote this extension also by $\tau_i$. Since $\tau_i, \tau_j$ commute on $B_{k-1}$, they also commute on $B_{k-1}\{\DD^{-1}\}$.

As explained in Example \ref{exa.001}, we can define the iterated skew-Laurent extensions
\[B_{k-1}\{\DD^{-1}\}[x_1^{\pm 1}, \dots, x_{n-1}^{\pm 1}; \tau_1, \dots, \tau_{n-1}],\]
which is a domain and a free $R$-module, and it has GK dimension $n-1$ more than that of $B_{k-1}\{D^{-1}\}$. From \eqref{eq.GK9} and Lemma \ref{r.Ore4}, we have that
\begin{equation} \label{eq.Ek}
\GKdim(B_{k-1}\{\DD^{-1}\}[x_1^{\pm 1}, \dots, x_{n-1}^{\pm 1}; \tau_1, \dots, \tau_{n-1}])= \GKdim (E_k).
\end{equation}

\begin{lemma}
The homomorphism $h$ extends to an algebra isomorphism
\begin{equation}\label{eq.Ak}
B_{k-1}\{\DD^{-1}\} [x_1^{\pm1}, \dots , x_{n-1}^{\pm1}; \tau_ 1, \dots \tau_{n-1} ] \cong E_k.
\end{equation}
\end{lemma}

\begin{proof}
Let $b_{ij}= \ceC(v_{k-1})_{ji}\in \cS(\PP_k)$ be the element represented by $b$ stated with $i$ at the beginning point and $j$ at the terminating point. Denote by $\beta_{ij}$ the image of $b_{ij}$ in the quotient $E_k$ of $\cS(\PP_k)$. Note that $\beta_{ij} = 0$ if $j >i$.

Let us show that for $i=1,\dots, n-1$ and $x\in B_{k-1}$,
\begin{equation}\label{eq.bii5}
\beta_{ii} h'(x) = h'(\tau_i(x)) \beta_{ii}.
\end{equation}
It is enough to consider the case when $x$ is a nontrivial stated arc. If $x$ is a stated arc that does not end on the edge $v_{k-2}v_{k-1}$, then $\tau_i(x)=x$, and $x$ does not intersect $\beta_{ii}$, so they commute. Hence \eqref{eq.bii5} holds. If $x$ is a stated arc ending on $v_{k-2}v_{k-1}$, then
\begin{equation}
\input{poly-arc-bi}
\end{equation}
By Lemma~\ref{r.upper}, the counit is zero if $i'<i$. The arc $b_{i'i}$ is bad if $i'>i$. Thus in $B_{k-1}\{\DD^{-1}\}$, only the $i'=i$ term is nonzero, and it matches \eqref{eq.bii5} by comparing the definition of $\tau_i$ with \eqref{eq.epsR} and \eqref{eq.epsRr}.

By Lemma \ref{r.AkBk}, the element $h(\DD)$ is invertible in $E_{k-1}$. The universality of localization implies that $h$ can be extended to an algebra homomorphism $h':B_{k-1}\{\DD ^{-1}\}\to E_k$. In $E_{k}$, the elements $\beta_{11}, \dots, \beta_{n-1, n-1}$ pairwise commute, and together with \eqref{eq.bii5} this implies $h'$ can be extended to an algebra homomorphism
\[ h'' : B_{k-1}\{\DD^{-1}\}[x_1^{\pm 1}, \dots, x_{n-1}^{\pm 1}; \tau_1, \dots, \tau_{n-1}] \to E_k,\]
such that $h''(x_i)= \beta_{ii}$. The domain of $h''$ is a free $R$-module, a domain, and has the same GK dimension as the codomain by \eqref{eq.Ek}. By Lemma~\ref{r.GKdim}, to show that $h''$ is an isomorphism, it is enough to show that $h''$ is surjective.

Note that $A_1,\dots, A_{k-1}$ generate the algebra $\cS(\PP_k)$ while $A_1,\dots, A_{k-2}$ generate the algebra $\cS(\PP_{k-1})$. Hence as an algebra, $E_k$ is generated by $B_{k-1}$ and $\bar A_{k-1}$, the latter being the image of $A_{k-1}$ under the projection $\cS(\PP_k) \onto E_k$.

$\bar A_{k-1} \cong \LO$ is generated by $\beta_{ji}$ with $j\le i$, and by definition $\beta_{ii}=h''(x_i)$ is in the image of $h''$. Thus it is enough to show that $\beta_{ji}$ with $j<i$ is in the image of $h''$. For this we will show:
\begin{equation}\label{eq.bji}
\beta_{ji} = h''\left(q^{-1} D_{i-1}(a)^{-1} D_{j,i-1}(a) x_i\right),
\end{equation}
where $D_{j,i-1}(a)$ is the quantum determinant of the $[1,i-1] \times ([\bar i, n] \setminus \{\bar j\})$-submatrix of the matrix $\mathbf{a}=(a_{ij})_{i,j=1}^n$.

For $j< i$, we have $C(v_k)_{ji}=0$ in $E_k$ as it is a bad arc. Using Equation \eqref{e.capnearwall} to express $C(v_k)_{ji}$ in terms of $a$'s and $b$'s, we get
\begin{equation}\label{eq.15}
\input{poly-arc-split}
\end{equation}
Fix $i$ for the moment and consider $j=1,\dots, i-1$. Let $z_j= (\ccc_{\bar i}/ \ccc_{\bar j} )b_{ji}b_{ii}^{-1}$. After multiplying on the right by $\ccc_{\bar i} b_{ii}^{-1}$, Equations \eqref{eq.15} with $j=1, \dots, i-1$ becomes
\[ \begin{pmatrix}
a_{1,\bar i} & a_{1, \bar i+1} & \dots & a_{1,n} \\
a_{2,\bar i} & a_{2, \bar i+1} & \dots & a_{2,n} \\
\vdots & \vdots & \dots & \vdots \\
a_{i-1,\bar i} & a_{i-1, \bar i+1} & \dots & a_{i-1,n}
\end{pmatrix}
\begin{pmatrix}1 \\ z_{i-1} \\ \vdots \\ z_1\end{pmatrix} =0.\]
(All identities are in $E_k$.) Solving this linear system using Proposition \ref{e.Cramer}, we get $z_j=- (-q)^{i-j-1} D_{i-1}(a)^{-1} D_{j,i-1}(a) $. Hence in $E_k$ we have
\[ b_{ji}= (\ccc_{\bar j}/ \ccc_{\bar i} ) z_j b_{ii} = q^{-1} D_{i-1}(a)^{-1} D_{j,i-1}(a) x_i. \]
This proves \eqref{eq.bji}, and the lemma.
\end{proof}

\subsection{Structure of \texorpdfstring{$\reduceS(\PP_k)$}{S-bar(Pk)}}

By definition,
\begin{align*}
B_{k-1} &= \cS(\PP_{k-1}) / (I_2 + \dots + I_{k-2})\\
E_{k-1} &= \cS(\PP_{k-1}) / (I_1 + I_2 + \dots + I_{k-2}).
\end{align*}
Hence $E_{k-1} = B_{k-1} /(I_1)$, where, by abusing notations, we denote the image of $I_1$ under the projection $\cS(\PP_{k-1}) \onto E_{k-1}$ also by $I_1$.

\begin{lemma} \label{r.87}
Let $p:B_{k-1} \onto E_{k-1}$ be the natural projection. Then the element $p(\DD)$ is non-zero and $q$-commuting with $E_{k-1}$.
\end{lemma}

\begin{proof}
Since $\DD$ is $q$-commuting with $B_{k-1}$ it is clear that $p(\DD)$ is $q$-commuting with $E_{k-1}$.

Since $\rOr$ is an $R$-linear isomorphism, we only need to show that $\rOr(\DD)$ is non-zero in
\[ E_{k-1} = \bar A_1 \boxtimes \dots \boxtimes \bar A_{k-2}.\]

Recall that $d_1(\bar u_{ij}) = i-j$ define a $\BZ$-grading on $\LO$. Any quantum minor $M^I_J$ of the quantum matrix $\bar \buu= (\bar u_{ij})_{i,j=1}^n$ where $I, J \subset \JJ$ have the same cardinality is $d_1$-homogeneous. If we keep the sizes fixed so that $|I|= |J|=i$, then the quantum minor with the largest $d_1$-degree is the one with $I=[\bar i, n]$ and $J= [1,i]$. In this case $M^I_J=\bar v_{\bar i, 1}$ is an element of the quantum torus frame of $\LO$ given in Theorem \ref{thm.qtorus1}, where we proved that it is non-zero.

Since each $\bar A_i =\LO$ is $d_1$-graded, we can equip the $R$-module $E_{k-1}=\bar A_1 \otimes \dots \otimes \bar A_{k-2}$ with a $\BZ^{k-2}$-grading (not compatible with the algebra structure).

\begin{figure}
\centering
\input{poly-arc-Di-split}
\caption{Decomposing $\protect\rOr(D_i(a))$}\label{fig-poly-Di-split}
\end{figure}

Let $I_1= I_{k-1} = [1;i]$. Using Lemma \ref{r.cutdet} repeatedly, we can express $\rOr(D_i(a))$ in $\bar A_2 \dots \bar A_{k-2}$ by
\begin{equation}
\rOr(D_i(a)) = \sum_{I_2, \dots, I_{k-2} \in\binom{\JJ}{i}} c(I_2,\dots, I_{k-2}) M^{\bar I_2}_{I_1}(C(v_1)) M^{\bar I_3}_{I_2}(C(v_2)) \dots M^{\bar I_{k-1}}_{I_{k-2}}(C(v_{k-1})),
\end{equation}
where $c(I_2,\dots, I_{k-2})\in R$ is invertible. Each term in the sum is $(d_1^{k-2})$-homogeneous, and the term with maximal $(d_1^{k-2})$-degree is the one with $I_2 = \dots = I_{k-2}= [1,i]$. In the maximal degree term, each factor is a copy of $\bar v_{\bar i, 1}$, which is non-zero. Therefore, $\rOr(D_i(a))$ is non-zero in $E_{k-1}$, and so is $D_i(a)$. As $E_{k-1}$ is a domain, the product $\DD = \prod D_i(a)$ is non-zero.
\end{proof}

The projection of $\DD$ under $B_{k-1} \onto E_k$ is also denoted by $\DD$. By the above Lemma $\DD$ is a regular element $q$-commuting with $E_{k-1}$. Hence we can define the localization $E_{k-1} \{\DD ^{-1} \}$. The diagonal automorphism $\tau_i= \phi_{e, \eta_i}$ of $B_{k-1}$ is defined as the edge-weight automorphism on edge $e= v_{k-2} v_{k-1}$. Since $e$ is also an edge of $\PP_{k-1}$, the diagonal automorphism $\tau_i$ descends to a diagonal automorphism of $E_{k-1}$.

We also denote by $\DD$ the image of $\DD$ under the projection $B_{k-1} \onto E_{k-1}$. By Lemma~\ref{r.87} $\DD$ is a non-zero element $q$-commuting with $E_{k-1}$. By Lemma \ref{r.GKdim} we can construct the localization $E_{k-1} \{\DD^{-1}\}$. Since $\tau_i(I_1)= I_1$ and $\DD$ is an eigenvector of $\tau_i$ for all $i=1,\dots, n-1$, the diagonal automorphism $\tau_i$ descend to a diagonal automorphism of $E_{k-1} \{\DD^{-1}\}$ denoted by the same notation, and they commute. As in example \ref{exa.001} we can construct the skew-Laurent extension $E_{k-1} \{\DD^{-1}\} [x_1^{\pm 1}, \dots, x_{n-1}^{\pm 1}; \tau_1, \dots, \tau_{n-1} ]$.

\begin{theorem}\label{r.Pkk}
We have an algebra isomorphism
\begin{equation}
\Srd(\PP_k) = E_{k-1} \{\DD^{-1}\} [x_1^{\pm 1}, \dots, x_{n-1}^{\pm 1}; \tau_1, \dots, \tau_{n-1} ].
\end{equation}
\end{theorem}

\begin{proof}
By definition $\Srd(\PP_k)= E_k/(I_{1}$). From \eqref{eq.Ak},
\begin{align*}
\Srd(\PP_k) &= (B_{k-1} \{\DD^{-1}\})[x_1^{\pm 1}, \dots, x_{n-1}^{\pm 1}; \tau_1, \dots, \tau_{n-1} ]/I_{1}\\
&= (B_{k-1} \{\DD^{-1}\}/I_{1} ) [x_1^{\pm 1}, \dots, x_{n-1}^{\pm 1}; \tau_1, \dots, \tau_{n-1} ] \qquad \text{because }\tau_i(I_{1})= I_{1}\\
&= (E_{k-1}\{\DD^{-1}\})[x_1^{\pm 1}, \dots, x_{n-1}^{\pm 1}; \tau_1, \dots, \tau_{n-1} ]\qedhere
\end{align*}
\end{proof}

\subsection{Proof of Theorems \ref{thm.domainr}}

By \eqref{eq.EkAk} we have $E_{k-1}\cong A_{k-1,1}$ which, by Lemma \ref{r.Akl}, is a domain of GK dimension $(k-2)\frac{(n-1)(n+2)}{2}$. As $\DD$ is $q$-commuting with $E_{k-1}$ Lemma \ref{r.GKdim} shows that $\GKdim (E_{k-1} \{\DD^{-1}\}) = \GKdim E_{k-1}$. Each $\tau_i$ is a diagonal automorphism and hence locally algebraic. By Lemma \ref{r.Ore4} the ring $E_{k-1} \{\DD^{-1}\} [x_1^{\pm 1}, \dots, x_{n-1}^{\pm 1}; \tau_1, \dots, \tau_{n-1} ]$ is a an $R$-torsion free domain and has GK dimension
\begin{align*}
\GKdim(E_{k-1} \{\DD^{-1}\} [x_1^{\pm 1}, \dots, x_{n-1}^{\pm 1}; \tau_1, \dots, \tau_{n-1} ])
&= \GKdim (E_{k-1} \{\DD^{-1}\}) + n-1 \\
&= (k-2)\frac{(n-1)(n+2)}{2} + n-1\\
&= k\frac{(n-1)(n+2)}{2} - n^2+1.
\end{align*}
Theorem \ref{thm.domainr} follows, since $\Srd(\PP_k) = E_{k-1} \{\DD^{-1}\} [x_1^{\pm 1}, \dots, x_{n-1}^{\pm 1}; \tau_1, \dots, \tau_{n-1} ]$ by Theorem \ref{r.Pkk}. \qed

\subsection{Proof of Theorems \ref{thm.PP3}}

Theorem \ref{thm.PP3} is a special case of Theorem \ref{r.Pkk}.

\section{Quantum tori associated to the triangle}\label{sec-ntri}

In this section we recall Fock-Goncharov' algebra $\rd{\FG}(\PP_3)$ of the ideal triangle $\PP_3$, and define a new algebra $\bA(\PP_3)$, a quantization of the $A$-space. Both $\rd{\FG}(\PP_3)$ and $\rd{\lenT}(\PP_3)$ are quantum tori and serve as building blocks for the construction of the $A$- and $X$- quantum tori of triangulated surfaces. We show that the matrices of the $\rd{\FG}(\PP_3)$ and $\rd{\lenT}(\PP_3)$ form a compatible pair.

As usual, for any set $S$ let $\BZ^S$ denote the $\BZ$-module of all maps $S \to \BZ$.

\subsection{The $n$-triangulation of the triangle}

Use barycentric coordinates for $\stdT$ so that
\begin{equation}
\stdT=\{(i,j,k)\in\reals^3\mid i,j,k\ge0,i+j+k=n\}\setminus\{(0,0,n),(0,n,0),(n,0,0)\}.
\end{equation}
Here $(i,j,k)$ (or $ijk$ for brevity) are the barycentric coordinates. Let $v_1=n00$, $v_2=0n0$, $v_3=00n$. (This is opposite of the order used in the previous sections.) The edge following $v_i$ in the clockwise orientation is denoted $e_i$. We will draw $\stdT$ in the standard plane as an equilateral triangle with $v_1$ at the top. See Figure~\ref{fig-coords} for an example.

\begin{figure}
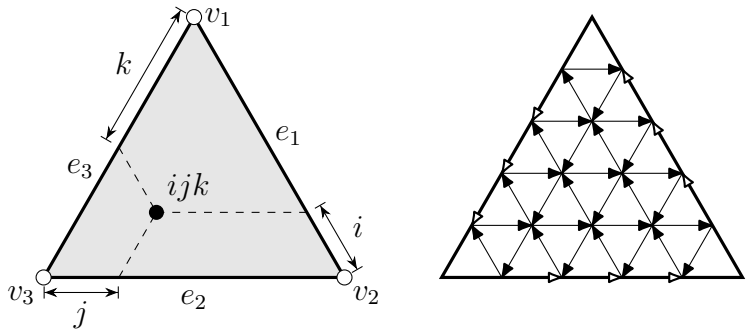

\centering
\input{barycentric}
\quad
\input{def-Q}
\caption{Barycentric coordinates $ijk$ and a $5$-triangulation with its quiver}\label{fig-coords}
\end{figure}

The \term{$n$-triangulation} of $\stdT$ is obtained by subdividing $\stdT$ into $n^2$ small triangles using lines $i,j,k=\text{constant integers}$. An example of a $5$-triangulation is shown in Figure~\ref{fig-coords}.

The vertices and edges of all small triangles, except for the vertices of $\stdT$ and the small edges adjacent to them, form a directed graph (or quiver) $\Gamma_\stdT$.
Here the direction of a small edge, also called an \term{arrow}, is defined as follows. If the small edge $u$ is in the boundary $\partial\stdT$ then $u$ has the positive (or counterclockwise) direction of $\partial \stdT$. If $u$ is interior then the direction of $u$ is the same as that of a boundary edge parallel to $u$. Assign weight $1$ to any boundary arrow and weight $2$ to any interior arrow.

The vertex set $\rd{V}=\rd{V}_\stdT$ of $\Gamma_\stdT$ is the set of points with integer barycentric coordinates:
\begin{equation}
\rd{V}=\{ijk\in\stdT\mid i,j,k\in\ints\}.
\end{equation}
Elements of $\rdV$ are called \term{small vertices}, and small vertices on the boundary of $\stdT$ are called the \term{edge vertices}.

Ignoring the assignment of the based vertex, the triangle $\stdT$ has a $\BZ/3$-symmetry that cyclically permutes the boundary edges. If $\stdT$ is presented as an equilateral triangle as in Figure \ref{fig-coords}, then the symmetry is generated by the rotation by $2\pi/3$.

\subsection{Fock-Goncharov algebra}

We define now the Fock-Goncharov algebra, or the reduced $X$-torus, of the ideal triangle.

Let $\bmQ = \bmQ_\stdT$ be the adjacency matrix of the weighted quiver $\Gamma_\stdT$. In other words,
\[\bmQ: \rd{V} \times \rd{V}\to \BZ\]
is the antisymmetric function defined by
\begin{equation} \label{eq.bmQ}
\bmQ(v,v') = \begin{cases} w, \quad & \text{if there is an arrow from $v$ to $v'$ of weight $w$},\\
0, &\text{if there is no arrow between $v$ and $v'$}.
\end{cases}
\end{equation}
The \term{Fock-Goncharov algebra} $\rd{\FG}(\PP_3)$, also called the \term{reduced $X$-torus}, is defined by
\begin{equation}
\rd{\FG}(\PP_3)= \qtorus(\bmQ) = R \la x_v^{\, \pm 1}, v \in \rd{V}\ra / ( x_v x_{v'} = \hq^{2 \bmQ_\stdT (v,v')} x_{v'} x_v ).
\end{equation}
The set of Weyl-normalized monomials $\{ x^ \bk \mid \bk \in \BZ^{ \rd{V} } \}$ is a free $R$-basis of $\rd{\FG}(\PP_3)$.

\begin{remark}
The original version of the Fock-Goncharov algebra is defined with $\hq$ replaced by $\hq^{n^2}$. Equivalently it is the subalgebra of our $\rd\FG(\PP_3)$ generated by $\{x^{ n \bk}\mid \bk \in \BZ^{\rdV}\}$.
\end{remark}

\subsection{The balanced Fock-Goncharov algebra}

We introduce now an important subalgebra of $\rd{\FG}(\PP_3)$, called the balanced Fock-Goncharov algebra.

Let $\vec{k}_1,\vec{k}_2,\vec{k}_3:\rd{V} \to\ints$ be the functions defined by
\begin{equation}\label{eq-bal-basis}
\vec{k}_1(ijk)=i,\quad \vec{k}_2(ijk)=j,\quad \vec{k}_3(ijk)=k.
\end{equation}
Let $\rd{\Lambda}=\rd{\Lambda}_\stdT\subset\ints^{\rd{V}}$ be the subgroup generated by $\vec{k}_1,\vec{k}_2,\vec{k}_3$ and $(n\ints)^{\rd{V}}$. Vectors in $\rd{\Lambda}$ are called \term{balanced}. Note $\vec{k}_1+\vec{k}_2+\vec{k}_3\in(n\ints)^{\rd{V}}$, so only two of these vectors are necessary in the definition of $\rd{\Lambda}$.

The \term{balanced Fock-Goncharov algebra} is the monomial subalgebra
\[\rdbl(\stdT)=\bT(\rdm{Q};\rd{\Lambda})= R\text{-span of\ } \{ x^\bk \mid \bk \in \rd{\Lambda}\}.\]

\subsection{The $A$-version quantum torus} \label{sec.Atori1}

We define now the reduced $A$-torus $\bA(\PP_3)$ of the triangle, which is a quantum torus $\BT(\bmP)$. The matrix $\bmP= \bmP_\stdT$ comes from the commutations of a set of special elements of $\rd{\cS}(\PP_3)$, see Section \ref{sec.qtr3}. Here we give a purely combinatorial definition of $\bmP$: Define the $\BZ/3$-invariant function
\[\bmP: \rd{V} \times \rd{V} \to n\BZ\]
such that if two small vertices $v=ijk, v'=i'j'k'$ in $\rd{V}$ satisfy
\begin{itemize}
\item[(*)] either $i \le i'$ and $j\le j'$, or $i \ge i'$ and $j\ge j'$,
\end{itemize}
then
\begin{equation}\label{eq-P-tri}
\bmP (v, v')= n \left|\begin{matrix}i &j \\ i'& j'\end{matrix}\right| = n(ij' -ji').
\end{equation}
Here, $\BZ/3$-invariance means for any rotation $\tau\in\ints/3$,
\[\rdm{P} (\tau(x),\tau(x'))=\rdm{P}(x,x').\]
Let us explain why $\bmP$ is well-defined. Condition (*) is equivalent to: The line $vv'$ in the planar picture forms with the horizontal axis an angle $\le 60^\circ$ in the upper half-plane, i.e. the line $vv'$ has slope in $[0, \sqrt 3/2]$. Any pair $v,v'\in \rd{V}_\stdT$ satisfy condition (*) after a rotation in $\BZ/3$ and hence $\bmP(v, v')$ can be defined. The only ambiguous case is when $vv'$ has slop $0$, so that a clockwise rotation by $2\pi/3$ also makes $vv'$ satisfy (*). But one can easily check that $\bmP$ agrees on the original pair and the new pair. Thus $\rdm{P}$ is well-defined. In addition, it is antisymmetric.

The \term{reduced $A$-torus} $\bA(\stdT)$ is the quantum torus $\bT( \bmP)$:
\begin{equation}
\bA(\stdT)= R \la a_v^{\, \pm 1}, v \in \rd{V}\ra / ( a_v a_{v'} = \hq^{2 \bmP(v,v')} a_{v'} a_v ).
\end{equation}
The following set of Weyl-normalized monomials is a free $R$-basis of $\rd{\FG}(\PP_3)$:
\begin{equation}\label{eq.basisA}
\{ a^ \bk \mid \bk \in \BZ^{ \rd{V} } \}
\end{equation}
The \term{positive part} $\rd{\lenT}_+(\stdT)$ is the quantum space $\rd{\lenT}_+(\stdT):=\qplane(\rdm{P})$.

\subsection{Transition between $A$- and $X$-tori} \label{sec-AX-tri}

We will show that there is an algebra isomorphism from the $A$-torus to the balanced the $X$-torus given by a multiplicative linear homomorphism, and that the matrices $\bmP$ and $\bmQ$ form a ``compatible pair".

Define a $\ints/3$-invariant map
\[\rdm{K}=\rdm{K}_\stdT:\rd{V}_\stdT\times\rd{V}_\stdT\to\ints\]
such that if $v=ijk$ and $v'=i'j'k'$ satisfy $i'\le i$ and $j'\ge j$, then
\begin{equation}\label{eq-trigen-exp}
\rdm{K}(v,v')=jk'+ki'+i'j.
\end{equation}
It is easy to see that every pair of $v$ and $v'$ can be rotated into a position where the definition applies.

A special case that will be useful later is
\begin{equation}\label{eq-trigen-0}
\rdm{K}_\stdT((ijk),(i',n-i',0))=n\min\{i,i'\}-ii'
=n\langle\varpi_i,\varpi_{i'}\rangle,
\end{equation}
where the last equality follows from \eqref{eq.weights}. In particular, it is independent of $j,k$, and it is zero if $i=0$.

In what follows we consider a function $ \rdV \times \rdV \to \BZ$ as a $\rdV \times \rdV$ matrix, and a function $\rdV \to \BZ$ as a $1 \times \rdV$ matrix (or a horizontal vector).

\begin{theorem}\label{thm-tdual-tri}
\begin{enuma}
\item The $R$-linear map $\rd{\psi}: \rd{\lenT}(\stdT) \to \rd{\FG}(\stdT)$, given on the basis \eqref{eq.basisA} by
\begin{equation}
\rd{\psi} (a^\vec{k})=x^{\vec{k}\rdm{K}}
\end{equation}
is an $R$-algebra embedding with image equal to balanced subalgebra $\rdbl(\PP_3)$.
\item If $\irdV\subset \rdV$ is the subset of all small vertices in the interior of $\PP_3$, then
\begin{equation} \label{eq.compa}
\bmP \, \bmQ = \left[
\begin{array}{c|c}
-4 n^2 (\Id_{\irdV \times \irdV } )& \ast \\ \hline
0 & \ast
\end{array}\right]
\end{equation}
where the upper left block is the diagonal $(\irdV \times \irdV)$-matrix with $-4n^2$ on the diagonal, and the lower left block is a $0$ matrix.
\end{enuma}
\end{theorem}

\begin{remark}
\begin{enuma}
\item If $\tilde B $ is the $\rdV \times \irdV$-submatrix of $\bmQ$, then Equ. \eqref{eq.compa} shows that the pair $(\bmP, \tilde B)$ is compatible in the theory of quantum cluster algebra \cite{BZ}.
\item The pair $(\bmP, \bmQ)$ is also compatible in the sense of \cite{GS}: Let $\BZ[\rdV]$ be the free $\BZ$-module with basis $\rdV$, equipped with the skew-symmetric bilinear defined by $(v, v')= \bmQ(v,v')$. For $v\in \rdV$ let $f_v = -\frac{1}{2n}\sum_{v'\in \rdV} K(v, v') v' \in \BQ[\rdV]$. Then $(\{ f_v\}, \BZ [\rdV])$ is a compatible pair in the sense of \cite[Section 18]{GS}. The result of \cite[Section 12]{GS} implies that $\bmQ$ has a compatible matrix. However compatible matrix might not be unique, and we don't know if our $\bmP$ is the same compatible matrix obtained in \cite{GS}.
\end{enuma}
\end{remark}

\subsection{The inverse of $\bmK$}

Let $\rdm{H}= \rdm{H}_\stdT:\rd{V} \times\rd{V}\to\ints$ be the map defined as follows:
\begin{itemize}
\item If $v$ and $v'$ are not on the same boundary edge then let $\rdm{H}(v,v')=-\frac{1}{2}\rdm{Q}(v,v')\in\ints$.
\item If $v$ and $v'$ are on the same boundary edge, then let
\[ \bmH(v,v') = \begin{cases} 1  \qquad &\text{when $v=v'$},\\
 -1 &\text{when there is arrow from $v$ to $v'$} \\
0 &\text{otherwise}
\end{cases} \]
\end{itemize}
See Figure~\ref{fig-H} for an illustration of $\rdm{H}$ values.

\begin{lemma}\label{lemma-HK-tri}
The following matrix identities hold.
\begin{enuma}
\item $n(\rdm{K} -\rdt{K} )=\rdm{P} $.
\item $\rdt{H} -\rdm{H} =\rdm{Q} $.
\item $\rdm{H}\, \rdm{K} =n\id$. In particular $\bmK$ is non-degenerate.
\item $\rdm{K}\, \rdm{Q}\, \rdt{K} =\rdm{P} $.
\end{enuma}
\end{lemma}

\begin{proof}
(a) and (b) are straightforward calculations using the definitions. The proof of (c) will be given in Subsection~\ref{sec-dual-tri}.

(d) is equivalent to (a) assuming (b) and (c):
\[\rdm{K}\, \rdm{Q}\, \rdt{K}
=\rdm{K} (\rdt{H} -\rdm{H} )\rdt{K}
=n(\rdm{K} -\rdt{K} )=\rdm{P} .\qedhere\]
\end{proof}

\begin{proposition}\label{prop-bal-tri}
Let $\vec{k}$ be a vector in $\ints^{\rd{V}}$. The following are equivalent.
\begin{enumerate}
\item $\vec{k}$ is balanced.
\item $\vec{k}\rdm{H} \in(n\ints)^{\rd{V} }$.
\item There exists a vector $\vec{c}\in\ints^{\rd{V} }$ such that $\vec{k}=\vec{c}\rdm{K} $.
\end{enumerate}
\end{proposition}

\begin{proof}
(1)$\Rightarrow$(2) can be directly verified on the generators in \eqref{eq-bal-basis}:
\[(\vec{k}_a\rdm{H} )(v)=0\qquad\text{for }a=1,2,3.\]

(2) and (3) are equivalent by Lemma~\ref{lemma-HK-tri} with $\vec{c}=\vec{k}\rdm{H} /n$.

(3)$\Rightarrow$(1) because $\rdm{K} (v,\cdot)\equiv k\vec{k}_1-j\vec{k}_2\pmod{n}$ by \eqref{eq-trigen-exp}.
\end{proof}

\subsection{Proof of Theorem \ref{thm-tdual-tri}}

\begin{proof}
(a) Recall that $\rd{\psi}(a^\bk) = x^{ \bk \bmK}$.
The identity of Lemma~\ref{lemma-HK-tri}(d) shows that
$\rd{\psi} $ is $R$-algebra homomorphism. The image of $\rd{\psi} $ is $\rdbl(\stdT)$ by Proposition~\ref{prop-bal-tri}. The non-degeneracy of $\bmK$, see Lemma~\ref{lemma-HK-tri}(c), shows that $\rd{\psi}$ maps the $R$-basis \eqref{eq.basisA} of $\bA(\PP_3)$ injectively into an $R$-basis of $\rd{\FG}(\PP_3)$, hence $\rd\psi$ is injective.

(b) Using the Identities of Lemma~\ref{lemma-HK-tri} and skew-symmetric property of $\bmQ$ we have
\begin{align} \label{eq.hkh1}
\bmP \, \bmQ = n (\bmK - \bmK ^t ) \bmQ= n \, \bmK \, \bmQ +n \, (\bmQ\, \bmK)^t.
\end{align}
By definition $\bmQ= -2 \bmH$ on the block $\rdV \times \irdV$, which we will focus on. Since $\bmK\, \bmH = n \Id$, we get
\begin{equation} \bmK\, \bmQ =
\left[\begin{array}{c|c}
-2n (\Id_{\irdV \times \irdV } )& \ast \\ \hline
0 & \ast
\end{array}\right]
\label{eq.HKH}
\end{equation}
Similarly, using $ \bmH\, \bmK = n \Id$ and focusing on the block $\irdV \times \rdV$, we see that $(\bmQ\, \bmK)^t$ is also equal to the right-hand side of \eqref{eq.HKH}. Using these values of $\bmK \, \bmQ $ and $(\bmQ\, \bmK)^t$ in \eqref{eq.hkh1}, we get \eqref{eq.compa}.
\end{proof}

\section{Quantum trace maps, triangle case} \label{sec.qtr3}

We show that the $A$-version quantum trace $\btr^A$ exists for the ideal triangle by exhibiting a quantum torus frame for $\rd\cS(\PP_3)$. Then we derive the $X$-version $\btr^X$. We show that $\btr^X$ has a grading on the boundary edge, a fact used later to patch the $\btr^X$ of the triangles to give a global $X$-version quantum trace for general pb surface. We also show how to recover the quantum holonomy results of \cite{CS} using the existence of $\btr^X$. Finally we give an extension of the counit for $\LO$, which will be used later to relate the reduced and non-reduced quantum traces for general surfaces.

In this section we continue to use the notations of the preceding section.

\subsection{Quantum torus frame and quantum traces}

For a small vertex $(ijk) \in \rdV = \rdV_{\PP_3}$ the diagram $\gaa'_{ijk}$ in Figure~\ref{fig-trigen-res} is reflection-normalizable by Lemma \ref{r.triad}. Define
\begin{equation}\label{eq.wijk}
\gaa_{ijk} = (-1)^{\binom{n}{2}+\binom{k}{2}}w_{ijk} \gaa'_{ijk}, \quad \text{where }
w_{ijk}=q^{-\frac{1}{2n}\binom{i}{2}}q^{\frac{1}{2n}\binom{j}{2}}q^{\frac{1}{2n}\binom{k}{2}} q^{\binom{k}{2}}.
\end{equation}
Here $w_{ijk}$ is the reflection-normalization factor so $\gaa_{ijk}$ is reflection invariant. The sign is introduced to simplify Lemma~\ref{lemma-agen-decomp} and ensure $\ints/3$-invariance.

 {Recall that by Theorem \ref{thm.domainr}, $\reduceS(\PP_3)$ is a domain, and we define the quantum torus frame in Definition~\ref{def.qtf}.}

\begin{figure}
\centering
\input{trigen-res}
\caption{Diagram $\gaa'_{ijk}$} \label{fig-trigen-res}
\end{figure}

\begin{theorem}\label{thm.bP3a}
The set $\GGG=\{\gaa_v \mid v \in \rd{V}\}$ is a $\ints/3$-invariant quantum torus frame for $\reduceS(\stdT)$ with the commutation rule
\begin{equation}\label{eq-trigen-comm}
\gaa_v \gaa_{v'} = \hq^{2\rdm{P} (v,v')} \gaa_{v'} \gaa_v.
\end{equation}
Consequently we have a reflection invariant and $\ints/3$-equivariant embedding
\begin{equation}\label{eq.trA01}
\rdtr^A: \reduceS(\stdT) \embed \rd{\lenT}(\stdT), \quad
\btr^A(\gaa_v) = a_v.
\end{equation}
Moreover if $\reduceS(\stdT)$ is identified with its image, then
\begin{equation}\label{eq-sandwich}
\rd{\lenT}_+(\stdT) \subset \reduceS(\stdT) \subset \rd{\lenT}(\stdT).
\end{equation}
\end{theorem}

From the isomorphism $\rd\psi: \rd{\lenT}(\stdT) \xrightarrow\cong \rdbl(\stdT)$, we obtain the $X$-version:

\begin{theorem}\label{thm.Xtr0}
There is a unique $\ints/3$-equivariant algebra embedding
\begin{equation}\label{eq.bXtrT}
\rdtr^X: \reduceS(\stdT) \embed \rdbl(\stdT), \quad \text{given by} \ \rdtr^X(\gaa_v)=x^{\rdm{K}(v,\cdot)}.
\end{equation}
Here $\rdm{K}(v,\cdot): \rdV \to \BZ$ is the map $v' \mapsto \rdm{K}(v,v')$. We have the commutative diagram
\begin{equation}\label{eq.diag00}
\begin{tikzcd}[row sep=tiny]
&\rd{\lenT}(\stdT) \arrow[dd,"\rd \psi","\cong"'] \\
\reduceS(\stdT) \arrow[ru,hook,"\btr^A"]
\arrow[rd,hook,"\btr^X"']&\\
&\rdbl(\stdT)
\end{tikzcd}
\end{equation}
\end{theorem}

\begin{proof}[Proof of Theorem \ref{thm.bP3a}]
A routine calculation using Lemma~\ref{r.height2} shows $\ints/3$-invariance.

 {
Next we show that $\gaa_{ijk}$ is the product of two quantum minors and use the $q$-commutation between quantum minors. Let $I_1=\{(i,k)\in\nats^2\mid i\ge0,k\ge1,i+k\le n\}$. For $(i,k) \in I_1$ and $j\in\{1,\dots,n-1\}$, let $M_1(i,k), M_2(j) \in \reduceS(\stdT)$ be the following quantum minors using notations from Subsection~\ref{ss.Weyl}.
\begin{equation}
\input{trigen-minors}
\end{equation}
In other words,
$M_1(i,k)=M^{[i+1; i+k]}_{[\bar{k};n]}(\cev{C}(v_1))$ and $M_2(j)=M^{[\bar{j};n]}_{[\bar j; n]}(C(v_2))$, where the corner arcs $\ceC(v_1),C(v_2)$ are defined in Subsection \ref{ss.badarcs}. For convenience, let $M_2(0)=M_1(i,0)=1$.
}

\begin{lemma}\label{lemma-agen-decomp}
The elements $M_1(i,k)$ and $M_2(j)$ are $q$-commuting in $\reduceS(\stdT)$, and
\begin{equation}\label{eq-agen-decomp}
\gaa_{ijk}=\left[M_1(i,k)M_2(j)\right]_\Weyl.
\end{equation}
\end{lemma}

\begin{proof}
The $k=0$ or $j=0$ case is a special case of \eqref{eq.detu}. Now assume $j,k\ge1$. From Lemmas~\ref{r.height5} and \ref{r.det}, the elements $M_1(i,k)$ and $M_2(j)$ are $q$-commuting, and
\begin{equation}\label{eq.mqcom}
M_2(j) M_1(i,k) = q^{\frac{jk}{n}} M_1(i,k) M_2(j).
\end{equation}

Applying Lemma~\ref{lemma-vertex-rev} on the right edge, we get
\begin{equation}
\gaa'_{ijk}=(-1)^{\binom{n}{2}}q^{\frac{1}{2n}\left(\binom{i}{2}-\binom{n-i}{2}\right)}(-q)^{-\binom{n-i}{2}}
\sum_{\sigma_2}(-q)^{\ell(\sigma_2)}\input{trigen-g-res}
\end{equation}
where $\sigma_2:[i+1;n]\to[i+1;n]$. For the corner arcs in the bottom right not to be bad arcs, we must have
\[\sigma_2(i+t)=\bar{t},\quad t=1,\dots,j.\]
The product of these corner arcs is equal to $M_2(j)$ by Lemma~\ref{r.Ibad1}.

Let $\sigma_3=\sigma_2|_{[\bar{k};n]}$. Then $\sigma_3:[\bar{k};n]\to[i+1;i+k]$, and $\ell(\sigma_2)=\binom{j}{2}+\ell(\sigma_3)+jk$. Thus
\begin{align*}
\gaa'_{ijk}&=(-1)^{\binom{n}{2}}q^{\frac{1}{2n}\left(\binom{i}{2}-\binom{n-i}{2}\right)}(-q)^{-\binom{n-i}{2}+\binom{j}{2}+jk}
\sum_{\sigma_3}(-q)^{\ell(\sigma_3)}\input{trigen-g-res2}\\
&=(-1)^{\binom{n}{2}}q^{\frac{1}{2n}\left(\binom{i}{2}-\binom{n-i}{2}\right)}(-q)^{\binom{k}{2}}
M_2(j)M_1(i,k).\\
\gaa_{ijk}&=(-1)^{\binom{n}{2}+\binom{k}{2}}w_{ijk}\gaa'_{ijk}
=q^{-\frac{jk}{2n}} M_2(j) M_1(i,k)=\left[M_1(i,k)M_2(j)\right]_\Weyl.
\end{align*}
where for the last identity we use \eqref{eq.mqcom} and the definition of the Weyl-normalization.
\end{proof}

Let us prove \eqref{eq-trigen-comm}. Using the $\ints/3$ symmetry and switching $v$ and $v'$ if necessary, we can assume $v=ijk$ and $v'=i'j'k'$ with $i\ge i'$ and $j\ge j'$.

By Lemma~\ref{r.Ibad1}, $M_2(j)$ and $M_2(j')$ commute. Also $M_1(i,k) $ and $M_1(i', k')$ commute by \eqref{eq-v-comm}. Finally, using Lemma~\ref{r.height5}, we get
\[M_1(i,k) M_2(j')=q^{-\frac{j'k}{n}}M_2 (j')M_1(i,k),\quad
M_1(i', k')M_2(j) =q^{-\frac{jk'}{n}+(j-j')}M_2(j) M_1(i',k').\]
Putting all these together and using Lemma \ref{lemma-agen-decomp}, we have
\begin{align*}
\gaa_v \gaa_{v'}&= [M_1(i,k)M_2(j) ]_\Weyl [M_1(i',k')M'_2(j')]_\Weyl\\
&=q^{-\frac{j'k}{n}}(q^{-\frac{jk'}{n}+(j-j')})^{-1}[M_1(i',k')M'_2(j')]_\Weyl[M_ 1(i,j)M_2(j) ]_\Weyl \\
&=q^{\frac{1}{n}(ij'-ji')}\gaa_{v'}\gaa_v
= \hq^{2\rdm{P}(v,v')}\gaa_{v'}\gaa_v,
\end{align*}
where for the last equality we use the definition of $\rdm{P}(v,v')$ in \eqref{eq-P-tri}. This proves \eqref{eq-trigen-comm}.

Use the isomorphism of Theorem~\ref{thm.PP3} to identify
\begin{equation}\label{eq.iden5}
\bS(\stdT) \equiv \LO\{\DD\}^{-1}[x_1^{\pm 1}, \dots, x_{n-1}^{\pm 1} ;\tau_1, \dots, \tau_{n-1}].
\end{equation}
Under this identification,
\begin{align}\label{eq.401}
M_1(i,k)&=M^{[i+1;i+k]}_{[\bar{k};n]}(\cev{C}(v_1))=\rOr(\bar{v}_{\bar{k},i+1}),
\end{align}
where $\{\bar{v}_{ij} \mid (i,j) \in I_2\}$ for $I_2:= \{ (i,j) \mid 1\le j\le i\le n,i\ne1\}$ is a quantum torus frame of $\LO$; see Theorem ~\ref{thm.qtorus1}. Since the map $(i,k) \to (\bar k, i+1)$ is a bijection between $I_1$ and $I_2$, Identity~\eqref{eq.401} shows that the set $\{ M_1(i,k) \mid (i,k) \in I_1\}$ is a quantum torus frame of $\LO$. By definition $\DD = \DD_1 \dots \DD_{n-1}$ where $\DD_k= M_1( 0,k )$. Hence $\{ M_1(i,k) \mid (i,k) \in I_1\}$ is also a quantum torus frame of $\LO\{\DD\} ^{-1}$. It follows that
\[G'=\{ M_1(i,k) \mid (i,k) \in I_1\} \cup \{ x_s^{-1} \mid s= 1, \dots n-1\}\]
is a quantum torus frame of $\bS(\stdT)$. Under the identification \eqref{eq.iden5} we have $x_s = \ceC(v_2)_{ss}=C(v_2)_{\bar s\bar s}^{-1}$ by \eqref{eq.ceCii}. Thus, if we replace $x_s^{-1} $ in $G'$ with $C(v_2)_{ss}$ we still have a quantum torus frame for $\bS(\stdT)$.

By Lemma \ref{r.Ibad1} we have
\[ M_2(j)=\prod_{s=\bar{j}}^n C(v_2)_{ss}.\]
Together with $\prod_{s=1}^n C(v_2)_{ss}=1$, this shows that all the monomials in $M_2(j), j=1, \dots, n-1$ with integer powers are the same as all the monomials in $C(v_2)_{jj}, j=1, \dots, n-1$. Thus we can replace $C(v_2)_{jj}$ by $M_2(j)$, and the set
\[G''=\{ M_1(i,k) \mid (i,k) \in I_1\} \cup \{ M_2(j) \mid j= 1, \dots n-1\}\]
is a quantum torus frame for $\bS(\stdT)$. As each $M_2(j)$ is invertible, we can modify each element
\[ M_1(i,k)\to q^{-\frac{jk}{2n}} M_2(j) M_1(i,k) = \gaa_{ijk},\] and still have a quantum torus frame. The last modification changes $G''$ to $\GGG$. Thus $\GGG$ is a quantum torus frame for $\bS(\stdT)$.

Proposition~\ref{r.qtframe} shows that the map $\btr^A$ given by \eqref{eq.trA01} is a well defined reflection invariant algebra embedding, with the sandwichness property \eqref{eq-sandwich}. The $\BZ/3$-invariance is clear from the definition.
\end{proof}

\begin{proof}[Proof of Theorem \ref{thm.Xtr0}]
Define $\btr^X = \rd \psi \circ \btr^A$. Clearly, $\btr^X$ is an algebra embedding satisfying \eqref{eq.bXtrT}, and Diagram \eqref{eq.diag00} is commutative. Since $\GGG=\{ \gaa_v \mid v \in \rdV\}$ is a quantum torus frame, it weakly generates $\rd{\cS}(\PP_3)$. Thus the algebra homomorphism satisfying \eqref{eq.bXtrT} is unique.
\end{proof}

\subsection{Boundary terms of $\rdtr^X$}

Suppose $\alpha$ is a stated web over $\PP_3$ and $u$ is a small vertex on the boundary. We now show that $\rdtr^X(\al)$ is homogeneous in any variable $x_u$. This will help to patch together the $\rdtr^X$ of different ideal triangles to give a global quantum trace.

The quantum torus $\rd{\FG}(\PP_3)$ has $R$-basis $\{ x^\bk \mid \bk \in \BZ^{\rdV}\}$. For $v\in \rdV$, an element of $\rd{\FG}(\PP_3)$ is \term{homogeneous in $x_v$ of order $d$} if it is an $R$-linear combination of $x^\bk$ with $\bk(v)=d$.

\begin{proposition}\label{prop-bdry-exp}
Assume $\al$ is a stated web diagram over $\PP_3$ and $e$ is a boundary edge. Let $u_1, \dots , u_{n-1}$ be the small vertices on $e$ listed in the positive order. Then $\btr^X(\al) \in \rd\FG(\PP_3)$ is homogeneous in $x_{u_i}$ of order $n \la \dd_e(\al), \varpi_i\ra$.
\end{proposition}

Here we recall that the degree $\dd_e(\al)\in \LL$ in subsection~\ref{sec-grading}, where $\LL$ is the weight lattice and $\varpi_i$'s are the fundamental weights.

\begin{proof}
Due to the $\BZ/3$ invariance, we can assume $e=e_1$. Then $u_i$ has barycentric coordinate $u_{i}= (i, n-i, 0)$.

First assume $\al= a_{sjk}$, which has $s$ outgoing endpoints on $e_1$ stated by $\bar s, \dots, n$. By definition, $\dd_e(a_{sjk}) = \ww_1 + \dots + \ww_s= \varpi_s$. Therefore,
\begin{equation}\label{eq.si}
n \la \dd_e(a_{sjk}), \varpi_{i} \ra = n \la \varpi_{s}, \varpi_{i} \ra.
\end{equation}
By definition $ \rdtr^X(a_{sjk})) = x^{\bmK(sjk, \cdot) }$. Thus the exponent of $x_{u_i}$ in $\btr^X(a_{sjk}))$ is
\[\bmK((sjk), u_i)= \bmK((sjk),(i, n-i, 0)),\]
which agrees with the right-hand side of \eqref{eq.si} using \eqref{eq-trigen-0}. This proves the statement for $\al= a_{sjk}$. By additivity of the degrees, the statement is true when $\al$ is a monomial in $a_v, v\in \rdV$.

Now assume $\al$ is an arbitrary stated web diagram. The sandwichness \eqref{eq-sandwich} means there exists a monomial $a^\bk$ such that $a^\bk \al$ is an $R$-linear combination of monomials in $a_v$. The additivity of both $\dd_e$ and $x_{u_i}$-degree implies the statement holds for general $\alpha$.
\end{proof}

\subsection{Explicit form of $\btr^X$}\label{ss.CS}

We now relate the $X$-version quantum trace to the quantum transport matrices of \cite{CS,SS2}. Let $\al$ be a stated $\partial \PP_3$-arc, we will show $\btr^X(\al)= \sum _\bk x^{\bk}$ with explicit $\bk$. Due to the $\BZ/3$ symmetry, we assume $\al$ is a corner arc at $v_1$. In other words, $\al= C(v_1)_{ij}$ or $\al= \cev C(v_1)_{ij}$, see Figure \ref{fig-corner}.

\begin{figure}
\centering
\input{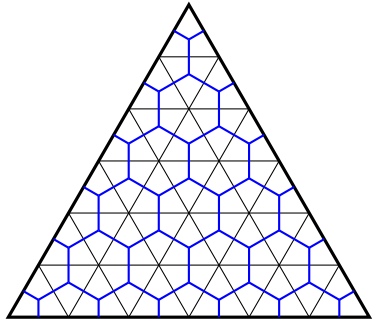}
\caption{Dual graph of the $n$-triangulation}\label{fig-trigen-dual}
\end{figure}

Consider the dual graph of the $n$-triangulation, shown in Figure~\ref{fig-trigen-dual}. We define the set $P(\al)$ of paths \term{compatible} with $\al$ as follows.

Suppose $\al= C(v_1)_{ij}$. On each edge $v_1 v_2$ and $v_1 v_3$ number the vertices of the dual graph from $1$ to $n$, starting at the vertex nearest to $v_1$. A directed path in the dual graph is \term{compatible} with $\alpha$ if
\begin{itemize}
\item it goes from the $i$-th point the left edge to the $j$-th point on the right edge, and
\item the vertical segments of the path must be upwards, and all other segments must be from left to right.
\end{itemize}

Now suppose $\al = \cev C(v_1)_{ij}$. On each edge $v_1 v_2$ and $v_1 v_3$ number the vertices of the dual graph decreasingly from $n$ to $1$, starting at the vertex nearest to $v_1$. A directed path in the dual graph is \term{compatible} with $\alpha$ if
\begin{itemize}
\item it goes from the $j$-th point the right edge to the $i$-th point on the left edge, and
\item the vertical segments of the path must be upwards, and all other segments must be from right to left.
\end{itemize}

Examples of paths compatible with $\al$ are given in Figure~\ref{fig-trigen-path}. If $\al$ is a bad arc, i.e. if $i<j$, then $P(\al)=\emptyset$ due to the condition on the vertical segments. If $i=j$, then there is a unique compatible path, which has no vertical segment.

\begin{figure}
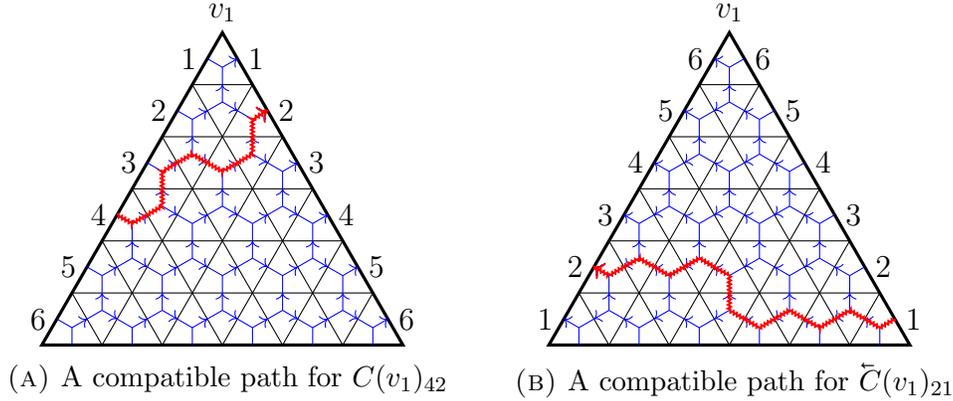

\centering
\begin{subfigure}[t]{0.4\linewidth}
\centering
\input{trigen-Cv1}
\subcaption{A compatible path for $C(v_1)_{42}$}
\end{subfigure}
\begin{subfigure}[t]{0.4\linewidth}
\centering
\input{trigen-ceCv1}
\subcaption{A compatible path for $\ceC(v_1)_{21}$}
\end{subfigure}
\caption{Examples of paths}\label{fig-trigen-path}
\end{figure}

For each directed path $p\in P(\al)$ let $\vec{k}'_p\in\ints^{\rd{V}}$ be the vector with value $n$ for all small vertices lying on the left of $p$ and $0$ otherwise. For $v=(i_1i_2i_3)\in \rdV$ define
\begin{equation}\label{eq-path-val}
\begin{aligned}
\vec{k}_p&:=\vec{k}'_p-\vec{k}_1,&
\vec{k}_p(v)&=\vec{k}'_p(v)-i_1,&&\text{if }\alpha=C(v_1)_{ij},\\
\vec{k}_p&:=\vec{k}'_p-\vec{k}_2-\vec{k}_3,&
\vec{k}_p(v)&=\vec{k}'_p(v)-i_2-i_3,&&\text{if }\alpha=\ceC(v_1)_{ij},
\end{aligned}
\end{equation}
where $\vec{k}_1,\vec{k}_2,\vec{k}_3$ are the generators of the balanced subgroup defined in \eqref{eq-bal-basis}.

The connection between $\btr^X$ and the quantum transports \cite{CS,SS2} is given the the following theorem whose proof is given in Appendix \ref{sec-trX-CS}.

 {
\begin{theorem}\label{thm-trX-CS}
For every simple stated $\partial\stdT$-arc $\al$ we have
\begin{equation}\label{eq.301}
\btr^X(\al) = \sum_{p\in P(\alpha)}x^{\vec{k}_p}.
\end{equation}
\end{theorem}

Note that the sum is empty for bad arcs since there are no compatible paths. The theorem is trivial in this case.

Let us comment on the connection to Chekhov and Shapiro's work \cite{CS}. For each $m=1,2,3$, let $\bM_m,\cev\bM_m\in\Mat_n(\rd\FG(\PP_3))$ be the $n\times n$ matrices with entries in $\rd\FG(\PP_3)$ defined by
\[(\bM_m)_{ij} = \sum_{p\in P( C(v_m)_{ij})}x^{\vec{k}_p}, \quad
(\cev \bM_m)_{ij} = \sum_{p\in P( \ceC(v_m)_{ji})}x^{\vec{k}_p}.\]
Then our $\bM_1$ and $\cev \bM_2$ are equal respectively to $\mathcal M_1 D_1^{-1}$ and $\mathcal M_2 [D_1 D_2]_\Weyl^{-1}$ of \cite{CS}. A main result of \cite{CS} can be formulated as follows.

\begin{theorem}[Chekhov and Shapiro \cite{CS}] \label{thm.CS}
\begin{enuma}
\item Each $\bM_m$ and $\cev \bM_m$ is a quantum $q$-matrix, and
\begin{equation} \label{eq.btensor}
(\bM_1)_{ij} \, (\cev \bM_2)_{kl} = \sum_{j',l'} \cR_{jl}^{j'l'} (\cev \bM_2)_{kl'} \, (\bM_1)_{ij'}.
\end{equation}
\item If $\bC$ is the $n\times n$ anti-diagonal matrix defined by $\bC_{ij} = \delta_{\bar i j} \ccc_j$, then
\begin{equation}\label{eq.200}
\cev \bM_2 = \bM_3 \bC \bM_1.
\end{equation}
\end{enuma}
\end{theorem}

Part (a) is \cite[Theorem~2.5]{CS}, and the proof there is quite involved. Part (b) is \cite[Theorem~2.6]{CS}, see \cite[Remark~2.7]{CS}.
}

Using Theorem \ref{thm-trX-CS}, we can get an alternative proof of Theorem \ref{thm.CS} and a new perspective of it as follows. Many identities proved in \cite{CS} can be derived from the relations in stated skein algebra of surfaces. Thus, let $A_m, \cev A_m \subset \cS(\PP_3)$ be the arc algebras of the arcs $C(v_m)$ and $\cev C (v_m)$ respectively. Each is a quantum matrix $\Oq$. Since $\bM_m = \rdtr^X(A_m ) $ and $\cev \bM_m = \rdtr^X(\cev A_m)$ by Theorem~\ref{thm-trX-CS}, both $\bM_m$ and $\cev \bM_m$ are quantum matrices. By \cite[Example~7.8]{LS}, the algebra $\cS(\PP_3)$ is the braided tensor product of $A_1$ and $\cev A_2$, and Equ. \eqref{eq.btensor} expresses exactly the multiplication in the braided tensor product. Finally, Equ. \eqref{eq.200} follows from the defining relation \eqref{e.capnearwall} when we push the arc $\cev C(v_2)$ to near the edge $v_1 v_3$.

\begin{remark}
Conversely, from Theorem \ref{thm.CS}, with a little work using skein $SL_n$-theory, we can construct $\btr^X$ for the ideal triangle by setting $\btr^X(C(v_m)_{ij}) = (\bM_m)_{ij}$. This was our original approach to constructing the $X$-version of quantum trace. However this approach does not explain why $\btr^X$ is injective, nor does it give the $A$-version with its geometric picture. While trying to show that $\rd \cS(\PP_3)$ is a domain we found a quantum torus frame for it, and from there we get the $A$-version of the quantum trace, and then recover the $X$-version.
\end{remark}

\subsection{Extension of the counit of \LOsec} \label{sec-attach-counit}
We will show that under the natural embedding
\[i_2:\LO\hookrightarrow\reduceS(\stdT)\xhookrightarrow{\btr^X}\rdbl(\stdT).\]
through the corner arc $a= C(v_2)$, the counit of $\LO$ can be extended to a subalgebra of $\rdbl(\stdT)$ containing the image of $i_2$. The result is used later to relate the reduced and non-reduced traces for general surfaces.

Let $B$ be the submonoid of the balanced subgroup $\rd{\Lambda}$ consisting of vectors $\vec{k}\in\rd{\Lambda}$ such that
\begin{itemize}
\item $\vec{k}(ijk)=0$ if $j=0$.
\item $\vec{k}(i'j'k')\le\vec{k}(ijk)$ whenever $j'=j$ and $k'\ge k$.
\end{itemize}
Define $\rd{B}$ as the subgroup where the vectors satisfy the equality in the second condition.

Next we define generators for $B$ and $\rd{B}$. Let
\[V_2=\{ijk\mid j\ne0\},\qquad
\rd{V}_2=\{(n-j,j,0)\mid j\ne0\}.\]
For each $ijk\in V_2$, define a vector $\vec{b}_{ijk}\in(n\ints)^{\rd{V}_\stdT}$ by
\begin{equation}
\vec{b}_{ijk}(i'j'k')=n\delta_{jj'}\delta_{k'\ge k}.
\end{equation}
Illustrations of $\vec{b}_{ijk}$ can be found in Figure~\ref{fig-attach-counit}. Then $\vec{k}_2$ and $\vec{b}_{ij0}$ are in $\rd{B}$, and $-\vec{b}_{ijk}$ is in $B$, but the positive multiples of $\vec{b}_{ijk}$ with $ijk\notin\rd{V}_2$ are not in $B$.

\begin{lemma}\label{lemma-attach-counit}
$\rd{B}$ is a free abelian group of rank $n-1$. It is generated as a group by the $n-1$ elements $\vec{k}_2$ and $\vec{b}_{n-j,j,0}$, $j\ge2$.

$B\cong \rd{B}\oplus\nats^{\binom{n}{2}}$, where the generators of the second part correspond to $-\vec{b}_{ijk}$, $ijk\in V_2\setminus\rd{V}_2$.
\end{lemma}

\begin{proof}
By definition, a vector $\vec{k}\in\rd{\Lambda}_\stdT$ can be written as
\begin{equation}
\vec{k}=a\vec{k}_1+b\vec{k}_2+n\vec{k}'
\end{equation}
for some $a,b\in\ints$ and $\vec{k}'\in\ints^{\rd{V}_\stdT}$. The condition $\vec{k}(1,0,n-1)=a+n\vec{k'}(1,0,n-1)=0$ shows that $a$ is a multiple of $n$. Thus the term $a\vec{k}_1$ can be absorbed into $n\vec{k}'$, so we can assume $a=0$.

If $\vec{k}\in\rd{B}$, then $n\vec{k}'=\vec{k}-b\vec{k}_2$ is also in $\rd{B}$. Then we can directly verify that
\begin{equation}\label{eq-attach-extend}
n\vec{k}'=\sum_{j=1}^{n-1}\vec{k}'(n-j,j,0)\vec{b}_{n-j,j,0}.
\end{equation}
Thus $\rd{B}$ is generated as a group by $\vec{k}_2$ and $\vec{b}_{n-j,j,0}$. By definition, $n\vec{k}_2=\sum_{j=1}^{n-1}j\vec{b}_{n-j,j,0}$, so $\vec{b}_{n-1,1,0}$ is redundant as a generator. It is easy to show that the remaining generators are independent.

If $\vec{k}\in B$, then similarly $n\vec{k}'\in B$. By subtracting the right-hand side of \eqref{eq-attach-extend}, which is an element of $\rd{B}$, we can assume $\vec{k}'(ij0)=0$. Let $V_{\vec{k}}\subset V_2\setminus\rd{V}_2$ be the subset of small vertices $ijk$ such that $\vec{k}'(ijk)<\vec{k}'(i+1,j,k-1)$. Then we can directly verify that
\begin{equation}
n\vec{k}'=\sum_{ijk\in V_{\vec{k}}}(\vec{k}'(ijk)-\vec{k}'(i+1,j,k-1))\vec{b}_{ijk}.
\end{equation}
The coefficients are all negative. This shows that $B$ is generated by $\rd{B}$ and $-\vec{b}_{ijk}$, $ijk\in V_2\setminus\rd{V}_2$. It is easy to see that the generators are independent.
\end{proof}

\begin{theorem}\label{thm-attach-counit}
The image of $i_2$ is contained in the monomial subalgebra $\qtorus(\rdm{Q}_\stdT;B)$. The $R$-linear map $\epsilon_X:\qtorus(\rdm{Q}_\stdT;B)\to R$ defined by
\begin{equation}\label{eq-attach-counit}
\epsilon_X(x^{\vec{k}})=1,\quad \vec{k}\in\rd{B},\qquad \epsilon_X(x^{\vec{k}})=0,\quad \vec{k}\notin\rd{B}
\end{equation}
is an algebra homomorphism such that $\epsilon_X\circ i_2$ is the counit $\epsilon(\bar{u}_{st})=\delta_{st}$ of $\LO$.
\end{theorem}

\begin{proof}
To find the image of $i_2$, we start with the generators $\bar{u}_{st}$ with $s\ge t$.

When $s=t$, the image of $\bar{u}_{ss}$ is given by the unique compatible path $p_s$ that only has segments pointing downward or toward the left. The value of the path is
\begin{equation}\label{eq-attach-diag}
i_2(\bar{u}_{ss})=x^{\vec{k}_{p_s}},\qquad
\vec{k}_{p_s}=-\vec{k}_2+\sum_{j=\bar{s}}^{n-1}\vec{b}_{n-j,j,0}\in\rd{B}.
\end{equation}
The sum is understood to be $0$ if $s=1$. See Figure~\ref{fig-attach-counit} left. The blue shade represents the sum, and the dots correspond to $(n-j,j,0)$. Thus
\[i_2(\bar{u}_{ss})\in\qtorus(\rdm{Q}_\stdT;\rd{B})\subset\qtorus(\rdm{Q}_\stdT;B).\]

\begin{figure}
\centering
\input{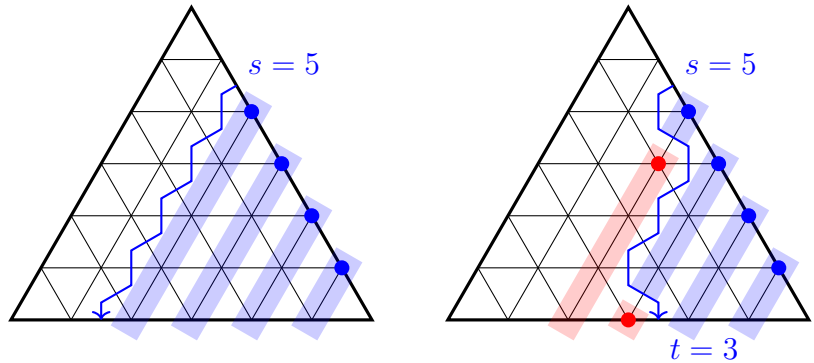}
\caption{Image of $i_2$}\label{fig-attach-counit}
\end{figure}

As a consequence, the restriction of $\rdm{Q}_\stdT$ to $\rd{B}$ is zero. This can be proved directly, but it is much easier to use the result above. Using \eqref{eq-attach-diag},
\[x^{n\vec{b}_{n-j,j,0}}=i_2(\bar{u}_{\bar{j}\bar{j}})i_2(\bar{u}_{\bar{j}-1,\bar{j}-1})^{-1}.\]
Since the diagonal elements $\bar{u}_{ss}$ commute with each other \eqref{eq.iijj}, so do $x^{\vec{b}_{n-j,j,0}}$. Thus
\[\rdm{Q}_\stdT(\vec{b}_{n-j,j,0},\vec{b}_{n-j',j',0})=0\qquad
\text{for all }j,j'\ne0.\]
It is easy to see that $\rd{B}$ is contained in the subgroup generated by $\frac{1}{n}\vec{b}_{n-j,j,0}$. This proves the claim.

Next consider $i_2(\bar{u}_{st})=\rdtr^X(C(v_2)_{st})$ with $s>t$. Any compatible path $p_{st}$ will have $s-t$ segments going toward the right. These segments occur at different $j$ coordinates with $\bar{s}\le j<\bar{t}$. Let $(i_j,j,k_j)$ be the small vertex to the immediate right of such a segment (as viewed from the path). By definition, $k_j\ne0$. Then the exponents for the path $p_{st}$ is
\begin{equation}
\vec{k}_{p_{st}}=\vec{k}_{p_s}-\sum_{j=\bar{s}}^{\bar{t}-1}\vec{b}_{i_j,j,k_j}\in B\setminus\rd{B}.
\end{equation}
See Figure~\ref{fig-attach-counit} right. The combined shade represents $\vec{k}_{p_s}+\vec{k}_2$, which is the same as the previous picture. The blue shade represents $\vec{k}_{p_{st}}+\vec{k}_2$. The red shade represents the sum, and the dots correspond to $(i_j,j,k_j)$. Thus
\begin{equation}\label{eq-attach-nondiag}
i_2(\bar{u}_{st})=\sum_{p_{st}}x^{\vec{k}_{p_{st}}}
\in\bigoplus_{\vec{k}\in B\setminus\rd{B}} R x^{\vec{k}}
\subset\qtorus(\rdm{Q}_\stdT;B).
\end{equation}

Since the images of the generators $\bar{u}_{st}$ are in $\qtorus(\rdm{Q}_\stdT;B)$, the image of $i_2$ is also in it. $\qtorus(\rdm{Q}_\stdT;B)$ has a presentation with all monomials as generators and relations given by \eqref{eq.prod}. It is easy to see that $\epsilon_X$ respects the relations. Thus $\epsilon_X$ is well-defined. Finally, by \eqref{eq-attach-diag} and \eqref{eq-attach-nondiag}, $\epsilon_X\circ i_2$ matches the counit of $\LO$.
\end{proof}

\section{Quantum tori associated to ideal triangulations} \label{sec.qtori}

This section is devoted to the combinatorics of ideal triangulations of surfaces. For an ideal triangulation $\lambda$ of a punctured bordered surface $\fS$
we will recall the Fock-Goncharov algebra $\bXS$ and introduce its extension $\XS$, which is a quantum torus having the GK dimension of the stated skein algebra $\SS$, and is the target space of the extended quantum trace.

When $\fS$ has no interior ideal point we will introduce the $A$-version quantum tori $\bAS$ and $\AS$ and prove a compatibility between the $A$-tori and the $X$-tori. The algebra $\bAS$ can be thought of as the quantization of $A$-moduli space of Fock and Goncharov.

\subsection{Ideal triangulation and Fock-Goncharov algebra} \label{ss.rdFG}

\begin{definition}
Let $\fS$ be a punctured bordered surface.
\begin{enuma}
\item $\fS$ is \term{exceptional} if it is the once- or twice-punctured sphere, the monogon, or the bigon.
\item $\fS$ is \term{triangulable} if every connected component of it has at least one ideal point and is not exceptional.
\item An \term{(ideal) triangulation} of a triangulable surface $\surface$ is a maximal collection $\lambda$ of non-trivial ideal arcs which are pairwise disjoint and pairwise non-isotopic. We consider ideal triangulations up to isotopy.
\end{enuma}
\end{definition}

The triangle $\stdT$ has a unique triangulation consisting of the 3 boundary edges up to isotopy. By abuse of notation, the triangulation is also denoted $\stdT$.

Fix an ideal triangulation $\lambda$ of $\fS$. An element of $\lambda$ is called \term{boundary} if it is isotopic to a boundary edge. By cutting $\fS$ along all non-boundary edges we get a disjoint union of ideal triangles, each is called a \term{face} of the triangulation. Let $\face_\lambda$ denote the set of faces. Then
\begin{equation}\label{eq.glue}
\fS = \Big( \bigsqcup_{\tau\in\face_\lambda} \tau \Big) /\sim,
\end{equation}
where each face $\tau$ is a copy of $\stdT$, and $\sim$ is the identification of certain pairs of edges of the faces. Note that one might glue two edges of the same face. Each face $\tau$ comes with a \term{characteristic map} $f_\tau: \tau \to \fS$, which is a homeomorphism when restricted to the interior of $\tau$ or the interior of each edge of $\tau$.

An \term{$n$-triangulation} of $\lambda$ is a collection of $n$-triangulations of the faces $\tau$ which are compatible with the gluing $\sim$. \term{Compatibility} means whenever an edge $b$ is glued to another edge $b'$, the edge-vertices on $b$ are glued to the edge-vertices on $b'$. Then define the \term{reduced vertex set}
\[\rd{V}_\lambda=\bigcup_{\tau\in\face_\lambda}\rd{V}_\tau, \quad \rd{V}_\tau=f_\tau(\rd{V}_\stdT).\]
The images of the weighted quivers $\Gamma_\tau$ under $f_\tau$ together form a quiver $\Gamma_\lambda$ on $\fS$.
Note that when edges $b$ and $b'$ are glued, a small edge on $b$ is then glued to a small edge of $b'$ with opposite direction, resulting an arrow of weight $0$.

Let $\bmQ_\lambda: \rd{V}_\lambda\times \rd{V}_\lambda \to \BZ$ be the signed adjacency matrix of the weighted quiver $\Gamma_\lambda$.
The ($n$-th root version) \term{Fock-Goncharov algebra} is the quantum torus of $\bmQ_\lambda$:
\begin{equation}
\bXS= \bT(\bmQ_\lambda) = R \la x_v^{\pm 1}, v \in \rd{V}_\lambda \ra / (x_v x_{v'}= \hq^{\, 2 \bmQ_\lambda(v,v')} x_{v'}x_v \ \text{for } v,v'\in \rd{V}_\lambda ).
\end{equation}

Another way to define $\bXS$ is as follows. Consider the tensor product algebra
\begin{equation}\label{eq.Xlambda}
\rd{\cX}_\lambda:= \bigotimes_{\tau\in \cF_\lambda} \rd{\cX}(\tau)
= \bigotimes_{\tau\in \cF_\lambda} \BT(\bmQ_\tau)
= \BT\Big( \bigoplus_{\tau\in \cF_\lambda} \bmQ_\tau \Big),
\end{equation}
where the last identity is the natural identification. Then $\bXS$ is the $R$-submodule of $\rd{\cX}_\lambda$ spanned by $x^{\bk}$ with $\bk({v'}) = \bk({v''})$ whenever $v'$ is glued to $v''$ in the identification \eqref{eq.glue}.

We define the \term{extension by zero} for matrices to simplify some definitions. Let $M_\stdT:\rd{V}_\stdT\times\rd{V}_\stdT\to\ints$ be a matrix associated to the standard triangle $\stdT$, and $f_\tau:\rd{V}_\stdT\to\rd{V}_\lambda$ be the map of small vertices induced by the characteristic map. Define the extension of $M_\stdT$ by zero, denoted $M_\tau:\rd{V}_\lambda\times\rd{V}_\lambda\to\ints$, by
\begin{equation}
M_\tau(u,v)=\sum_{u'\in f_\tau^{-1}(u)}\sum_{v'\in f_\tau^{-1}(v)} M_\stdT(u',v').
\end{equation}

Since $\Gamma_\lambda$ is obtained by gluing copies of $\Gamma_\stdT$, $\bmQ_\lambda$ can be written as
\begin{equation}\label{eq-Q-def}
\bmQ_\lambda = \sum_{\tau\in\face_\lambda}\bmQ_\tau,
\end{equation}
where $\rdm{Q}_\tau$ is the extension of $\rdm{Q}_\stdT$ by zero.

\begin{remark}
The original Fock-Goncharov algebra can be embedded as the subalgebra of $\bXS$ generated by $x^{\pm n}, x\in \rd{V}_\lambda$, which is isomorphic to the quantum torus $\qtorus(n^2\rdm{Q}_\lambda)$. Thus $\bXS$ can be considered as the $n$-th root version of the original Fock-Goncharov algebra.
\end{remark}

\subsection{The extended Fock-Goncharov algebra}

Given a punctured bordered surface $\surface$, attach a copy of $\stdT$ to each boundary edge of $\surface$. The resulting surface is denoted $\ext{\surface}$. We adopt the convention that in an attached triangle, the attaching edge is the $e_1$ edge. See Figure~\ref{fig-attach-T}.

\begin{figure}
\centering
\input{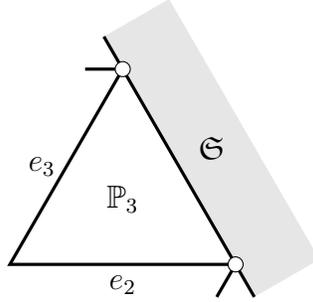}
\caption{Attaching triangles}\label{fig-attach-T}
\end{figure}

If the surface $\surface$ has an ideal triangulation $\lambda$, then there is an unique extension $\ext{\lambda}\supset\lambda$ to an ideal triangulation of $\ext{\surface}$ by adding all the new boundary edges. The new faces are exactly the glued triangles. Let $\rd{V}_{\ext{\lambda}}$ be the reduced vertex set of the extended $n$-triangulation. Define the \term{$X$-vertex set} $V_\lambda\subset\rd{V}_{\ext{\lambda}}$ as the subset of all small vertices not on the $e_3$ edge in the attached triangles. Let \term{$A$-vertex set} $\lv{V}_\lambda\subset\rd{V}_{\ext{\lambda}}$ be the subset of all small vertices not on the $e_2$ edge in the attached triangles. Note $\rd{V}_\lambda$ is naturally a subset of both $V_\lambda$ and $\lv{V}_\lambda$.

Let $\mQ_\lambda: V_\lambda \times V_\lambda \to \BZ$ be the restriction of $ \bmQ_{\ext{\lambda}}: \rd{V}_{\ext{\lambda}} \times \rd{V}_{\ext{\lambda}} \to \BZ$. The \term{extended $X$-algebra} is defined as
\[\XS = \bT(\mQ_\lambda).\]
There is a natural identification of subalgebras $\rd{\FG}(\surface,\lambda)\subset\FG(\surface,\lambda)\subset\rd{\FG}(\ext{\surface},\ext{\lambda})$.

\begin{lemma}\label{lemma-vset-size}
Suppose $\lambda$ is any ideal triangulation of a triangulable surface $\surface$. Recall that $\#\partial\surface$ is the number of boundary edges, and $r(\surface)=\#\partial\surface-\chi(\surface)$ is defined in \eqref{eq.rfS}. Then
\begin{align}
\abs{V_\lambda}=\abs{\lv{V}_\lambda}&=(n^2-1)r(\surface)=\GKdim\skein(\surface).\label{eq-vset-size1}\\
\abs{\rd{V}_\lambda}&=\abs{V_\lambda}-\binom{n}{2}\#\partial\surface.\label{eq-vset-size2}
\end{align}
In particular, if $\surface=\poly_k$ is a polygon, then $\abs{\rd{V}_\lambda}=\GKdim\reduceS(\surface)$ by \eqref{eq.GKrd}.
\end{lemma}

\begin{proof}
$V_\lambda\setminus\rd{V}_\lambda$ consists of small vertices in the attached triangles (excluding the attaching edges and the unused edges). There are $\#\partial\surface$ such triangles, and there are $\binom{n}{2}$ extra small vertices in each. This shows \eqref{eq-vset-size2}.

Let $v=\#\partial\surface$. A standard Euler characteristic argument shows that $\lambda$ has $e=2v-3\chi$ edges and $v-2\chi$ faces. Each edge has $n-1$ small vertices, and each face has $f=(n-1)(n-2)/2$ small vertices in the interior. Thus
\[\abs{V_\lambda}=e(n-1)+f\frac{(n-1)(n-2)}{2}+v\frac{n(n-1)}{2}=(n^2-1)r(\surface).\qedhere\]
\end{proof}

\begin{remark}
Even though the bigon $\PP_2$ is exceptional, we can still use the definition of an ideal triangulation and its extension. By attaching triangles to the boundary of the bigon, we obtain a quadrilateral $\poly_4$, in which the bigon embeds as a neighborhood of a diagonal $e$. $\lambda=\{e\}$ can be considered as an ideal triangulation of $\poly_2$, and $\ext{\lambda}$ can be defined as before. Then Lemma~\ref{lemma-vset-size} can be directly verified, as well as many results in the rest of the paper.
\end{remark}

\subsection{Skeletons of small vertices} \label{ss.skeleton}

Suppose $\surface$ does not have interior punctures, and $\lambda$ is an ideal triangulation. Since there is no interior ideal point, each characteristic map $f_\tau : \tau \to \fS$ is an embedding, and we will identify $f_\tau(\tau)$ with $\tau$, which is a copy of $\PP_3$.

For a small vertex $v\in \rdV_\lambda$ and an ideal triangle $\tau \in \cF_\lambda$, we now define its \term{skeleton} $ \sk_\tau(v)\in \BZ [ \rdV_\tau]$ and a graphical representation.

Choose a face $\nu\in\face_\lambda$ which contains $v$. There are two such $\nu$ when $v$ is on an interior edge of the triangulation. Otherwise, $\nu$ is unique. Assume $v=(ijk)\in V_\nu$. Draw a weighted directed graph $Y_v$ properly embedded into $\nu$ as in Figure \ref{fig-skel-main}. Here an edge of $Y_v$ has weight $i$, $j$ or $k$ according as the endpoint lands on the edge $e_1$, $e_2$ or $e_3$ respectively. The directed weighted graph $Y_v$ is unique up to ambient isotopy of the ideal triangle $\nu$.

\begin{figure}
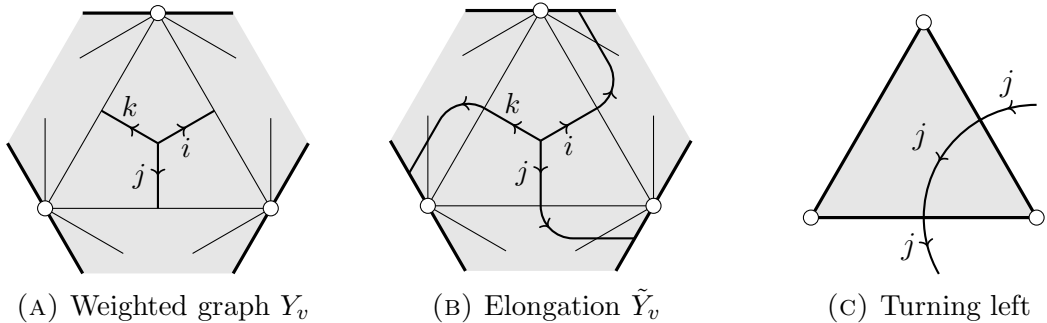

\centering
\begin{subfigure}[b]{0.3\linewidth}
\centering
\input{skel-main}
\subcaption{Weighted graph $Y_v$}\label{fig-skel-main}
\end{subfigure}
\begin{subfigure}[b]{0.3\linewidth}
\centering
\input{skel-longed}
\subcaption{Elongation $\tY_v$}\label{fig-skel-long}
\end{subfigure}
\begin{subfigure}[b]{0.3\linewidth}
\centering
\input{skel-turn-left}
\subcaption{Turning left}\label{fig-skel-left}
\end{subfigure}
\caption{Graphs associated to a small vertex $v$}\label{fig-skel}
\end{figure}

Elongate the nonzero-weighted edges of $Y_v$ to get an embedded weighted directed graph $\tY_v$ as in Figure~\ref{fig-skel-long}. Here the edge is elongated by using left turn whenever it enters a triangle, see Figure~\ref{fig-skel-left} for left turn. The portion of the elongated edge between the entering point and the exiting point in a triangle $\tau$ is called a \term{segment} of $\tY_v$ in $\tau$. In addition, we also consider $Y_v$ as a segment of $\tY_v$, called the \term{main segment}.

For the main segment $s=Y_v$ define $Y(s) = v \in \rdV_\nu$. For an arc segment $s$ in a triangle $\tau$ define $Y(s)\in \rdV_\tau$ to be the small vertex of the following weighted graph
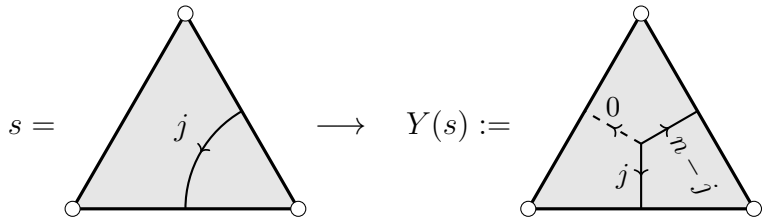
\[\input{skel-arc}\]

For example, if in the above picture the top ideal vertex is $v_1$, then $Y(s) = (n-j, j, 0) \in \rdV_\tau$. Define $\sk_\tau(v)$ by
\begin{equation}\label{eq-skel-def}
\sk_\tau(v) = \sum_{ s \subset \tau\cap\tilde{Y}_v} Y(s) \in \BZ [\rdV_\tau]
\end{equation}
where the sum is over all segments of $\tY_v$ in $\tau$.

\begin{lemma}\label{lemma-skel-welldef}
The skeleton $\sk_\tau(v)$ is well-defined, i.e., it does not depend on the choice of $\nu$.
\end{lemma}

\begin{proof}
The only ambiguous case is when $v\in\rd{V}_{\tau_1}\cap\rd{V}_{\tau_2}$ for two faces $\tau_1,\tau_2$ sharing a common edge $e$. Choose one end of the common edge as the top vertex in both $\tau_1$ and $\tau_2$ as in Figure~\ref{fig-edge-share}.

\begin{figure}
\centering
\input{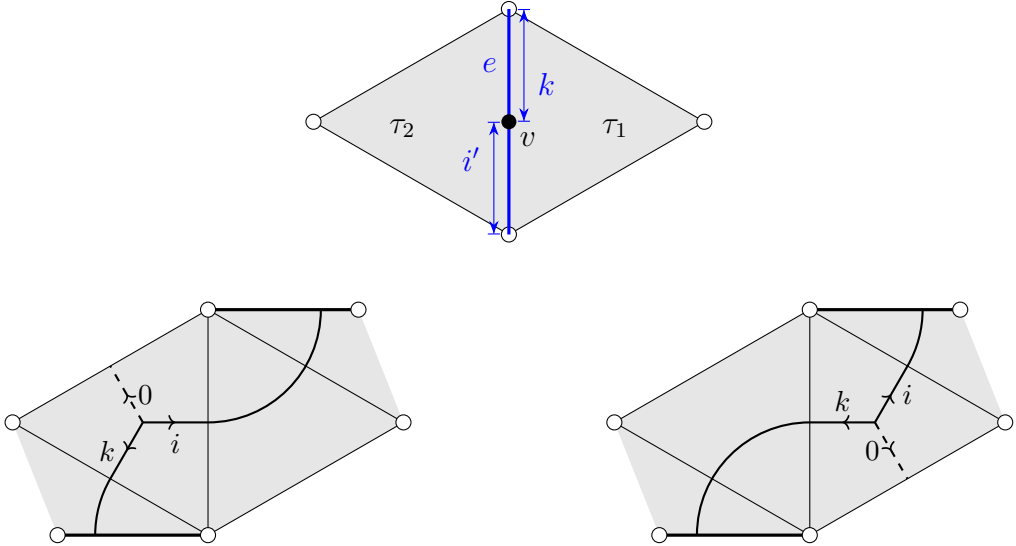}
\caption{Faces adjacent to $v$}\label{fig-edge-share}
\end{figure}

Let $v=(ijk)$ in $\tau_1$ and $v=(i'j'k')$ in $\tau_2$. Then the choice of $v_1$ implies
\[j=0=k',\qquad i'=n-k=i,\qquad j'=n-i'-k'=k.\]
Then the weighted graphs $\tY_v$ for the two choices are given in Figure~\ref{fig-edge-share}. Note the dashed line is the $0$ weighted edge, which is not elongated. There is a clear one-to-one correspondence between the segments.
\end{proof}

\subsection{The $A$-version quantum tori} \label{sec.Atori}

Continue to assume $\surface$ does not have interior punctures, and $\lambda$ is an ideal triangulation.

Define the matrix $\rdm{P}_\lambda:\rd{V}_\lambda\times\rd{V}_\lambda\to n\ints$ by
\begin{equation}
\rdm{P}_\lambda(v,v')=\sum_{\tau\in\face_\lambda}\rdm{P}_\tau(\sk_\tau(v),\sk_\tau(v')),
\end{equation}
where we also denote by $\rdm{P}_\tau$ the $\BZ$-bilinear extension of $\rdm{P}_\tau: \rdV _\tau \times\rdV _\tau \to n \BZ $ .

The extended matrix $\mat{P}_\lambda$ is obtained from the extended surface with a change of basis. Define a map $p:\rd{V}_{\ext{\lambda}}\setminus\rd{V}_\lambda\to\rd{V}_{\ext{\lambda}}\setminus\lv{V}_\lambda$ as follows. Every $v \in\rd{V}_{\ext{\lambda}}\setminus\rd{V}_\lambda$ has coordinates $ijk$ in an attached triangle with $k\ne0$, and $\rd{V}_{\ext{\lambda}}\setminus\lv{V}_\lambda$ consists of vertices $ijk$ in attached triangles with $i=0$. Then
\begin{equation}\label{eq-cov-pdef}
p(v)=(0,n-k,k)\quad \text{in the same triangle}.
\end{equation}
The change-of-variable matrix $\mat{C}:\lv{V}_\lambda\times\rd{V}_{\ext{\lambda}}\to\ints$ is defined by
\begin{equation}
\begin{aligned}
\mat{C}(v,v)&=1,&&v\in\lv{V}_\lambda,\\
\mat{C}(v,p(v))&=-1,&&v\in\lv{V}_\lambda\setminus\rd{V}_\lambda,\\
\mat{C}(v,v')&=0,&&\text{otherwise}.
\end{aligned}
\end{equation}
The nontrivial matrix elements are shown in Figure~\ref{fig-cov}, where $\pm$ denotes the values $\pm1$. The extended matrix $\mat{P}_\lambda:\lv{V}_\lambda\times\lv{V}_\lambda\to n\ints$ is given by
\begin{equation}
\mat{P}_\lambda=\mat{C}\exm{P}\mat{C}^t.
\end{equation}
Clearly, the restriction of $\mat{P}_\lambda$ to $\rd{V}_\lambda\times\rd{V}_\lambda$ is $\rdm{P}_\lambda$.

\begin{figure}
\centering
\input{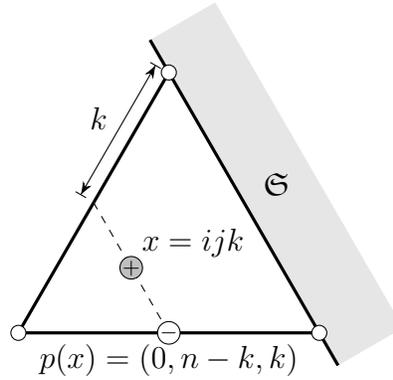}
\caption{Change-of-variable matrix $\mat{C}$}\label{fig-cov}
\end{figure}

Define the following \term{$A$-version quantum tori and quantum spaces}.
\begin{alignat*}{2}
\rd{\lenT}(\surface,\lambda)& =\qtorus(\rdm{P}_\lambda),&\qquad
\rd{\lenT}_+(\surface,\lambda)&=\qplane(\rdm{P}_\lambda),\\
\lenT(\surface,\lambda)&=\qtorus(\mat{P}_\lambda),&
\lenT_+(\surface,\lambda)&=\qplane(\mat{P}_\lambda).
\end{alignat*}

\subsection{Balanced parts of $X$-tori}

The notion of balanced vectors can be generalized to any triangulable surface $\surface$ with an ideal triangulation $\lambda$. A vector $\vec{k}\in\ints^{\rd{V}_\lambda}$ is \term{balanced} if its pullback to $\stdT$ is balanced for every triangle of $\lambda$. Here, for every face $\tau$ and its characteristic map $f_\tau:\stdT\to\surface$, the pullback $f_\tau^\ast\vec{k}$ is a vector $\rd{V}_\stdT\to\ints$ given by $f_\tau^\ast\vec{k}(v)=\vec{k}(f_\tau(v))$. The subgroup of balanced vectors is denoted $\rd{\Lambda}_\lambda$.

The \term{balanced Fock-Goncharov algebra} is the monomial subalgebra
\[\rd{\FG}^\bal(\surface,\lambda)=\bT(\rdm{Q}_\lambda;\rd{\Lambda}_\lambda).\]
The extended version is defined by \[\FGbl(\surface,\lambda)=\qtorus(\mat{Q}_\lambda)\cap\rd{\FG}(\ext{\surface},\lambda)=\qtorus(\mat{Q}_\lambda;\Lambda_\lambda).\]
Here, the intersection is taken in $\qtorus(\exm{Q})$, where $\qtorus(\mat{Q}_\lambda)$ is considered a subalgebra by the natural embedding, and $\Lambda_\lambda=\rd{\Lambda}_{\ext{\lambda}}\cap\ints^{V_\lambda}$ is the subgroup of balanced vectors.

As in the triangle case, the balanced condition has a few equivalent statements. See Proposition~\ref{prop-bal}.

\subsection{Transitions between $A$- and $X$-tori}\label{sec-AX}

We generalize the properties from Subsection~\ref{sec-AX-tri} to more general surfaces. Again assume that $\surface$ does not have interior punctures, and that $\lambda$ is an ideal triangulation.

Define $\rdm{K}_\lambda:\rd{V}_\lambda\times\rd{V}_\lambda\to\ints$ as follows. Let $u,v\in \rd{V}_\lambda$. Choose a face $\tau\in\face_\lambda$ containing $v$ and let
\begin{equation}\label{eq-surgen-exp}
\rdm{K}_\lambda(u,v)=\rdm{K}_\tau(\sk_\tau(u),v)=\sum_{s\subset\tau\cap\tilde{Y}_u}\rdm{K}_\tau(Y(s),v).
\end{equation}

\begin{lemma}\label{lemma-K-welldef}
The matrix $\rdm{K}_\lambda$ is well-defined, that is, it is independent of the choices of $\tau$.
\end{lemma}

\begin{proof}
The only ambiguous case is when $v$ is on an edge $e$ shared by faces $\tau_1,\tau_2$. If the segment $s$ does not intersect the edge $e$ (or has zero weight on $e$), the special case \eqref{eq-trigen-0} shows that $\rdm{K}(Y(s),v)=0$. Any segment that does intersect $e$ does so exactly once by the assumption that there are no interior punctures. For an intersection $a\in e\cap\tilde{Y}_u$, let $s_{r,a}$ be the segment of $\tilde{Y}_u\cap\tau_r$ incident to $a$ for $r=1,2$. Then
\[\rdm{K}_\lambda(u,v)=\sum_{a\in e\cap\tilde{Y}_u} \rdm{K}_{\tau_r}(Y(s_{r,a}),v).\]
We prove the lemma by showing $\rdm{K}_{\tau_r}(Y(s_{r,a}),v)$ is independent of $r$.

\begin{figure}
\centering
\input{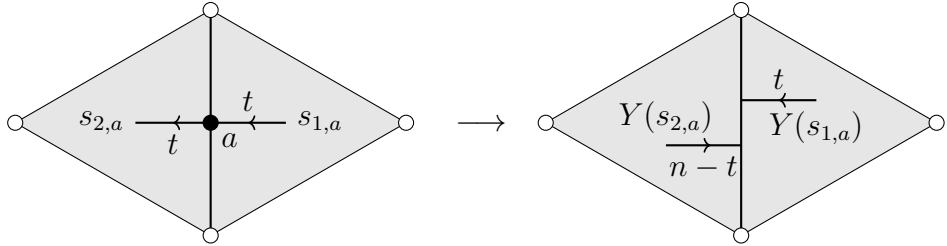}
\caption{Segments $s_{r,a}$}\label{fig-esh-seg}
\end{figure}

Draw the local picture in the same way as Lemma~\ref{lemma-skel-welldef} such that the coordinates of $v$ are $(i,0,n-i)$ in $\tau_1$ and $(i,n-i,0)$ in $\tau_2$. Since $s_{1,a}$ and $s_{2,a}$ are connected through $a$, one of them is elongated from the other. Thus they have the same weight $t$ and consistent directions near $a$. If the segments go to the left, we obtain Figure~\ref{fig-esh-seg}. Then using \eqref{eq-trigen-0}, we get
\[\rdm{K}_{\tau_1}(Y(s_{1,a}),v)=n\langle\varpi_t,\varpi_{n-i}\rangle,\qquad
\rdm{K}_{\tau_2}(Y(s_{2,a}),v)=n\langle\varpi_{n-t},\varpi_i\rangle.\]
They agree by \eqref{eq.rorbra}. The other segment direction is similar.
\end{proof}


To define the extended matrix $\mat{K}_\lambda:\lv{V}_\lambda\times V_\lambda\to\ints$, start with the reduced matrix of the extended surface, $\exm{K}:\rd{V}_{\ext{\lambda}}\times\rd{V}_{\ext{\lambda}}\to\ints$. The product $\mat{C}\,\exm{K}$ is a matrix on $\lv{V}\times\rd{V}_{\ext{\lambda}}$. $\mat{K}_\lambda$ is defined as the restriction of $\mat{C}\, \exm{K}$, that is,
\begin{equation}
\mat{K}_\lambda=(\mat{C}\, \exm{K})|_{\lv{V}_\lambda\times V_\lambda}.
\end{equation}

\begin{lemma}\label{lemma-CK0}
The restriction of $\mat{C}\exm{K}$ to $\lv{V}_\lambda\times(\rd{V}_{\ext{\lambda}}\setminus V_\lambda)$ is $0$.
\end{lemma}

\begin{proof}
Let $u\in\lv{V}_\lambda$ and $v\in\rd{V}_{\ext{\lambda}}\setminus V_\lambda$. Suppose $v$ is in the (attached) triangle $\tau$. If $u$ is not in $\tau$, then $p(u)$ (if defined) is not in $\tau$ either. Thus $(\mat{C}\exm{K})(u,v)=0$.

\begin{figure}
\centering
\input{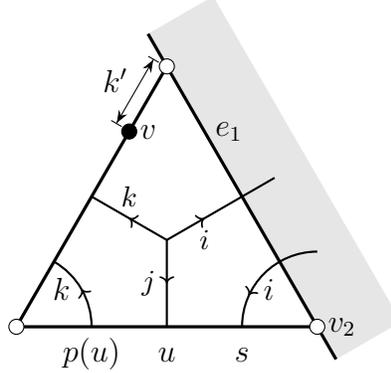}
\caption{Segments in an attached triangle}\label{fig-ck-seg}
\end{figure}

Now suppose $u$ is in $\tau$ as well. Recall the attaching edge is $e_1$ by convention. Let $u=(ijk)$ and $v=(i',0,k')$. To find $\exm{K}(u,v)$, we use the expanded definition in \eqref{eq-surgen-exp}. The segments $\tau\cap\tilde{Y}_u$ consist of the main segment $Y_u$ and possibly an elongated segment $s$ with weight $i$ around $v_2$. Then $Y(s)$ has coordinates $(n-i,i,0)$, and \eqref{eq-trigen-0} shows that
\begin{equation}
\rdm{K}_\tau(Y(s),v)=0,\qquad
\exm{K}(u,v)=\rdm{K}_\tau(u,v)=n\langle\varpi_k,\varpi_{k'}\rangle.
\end{equation}
Similarly, $\exm{K}(p(u),v)=n\langle\varpi_k,\varpi_{k'}\rangle$ as well. Thus
\[(\mat{C}\exm{K})(u,v)=\exm{K}(u,v)-\exm{K}(p(u),v)=0.\qedhere\]
\end{proof}

Thus $\mat{K}_\lambda$ contains all the information in $\mat{C}\exm{K}$. It is also easy to check that the restriction of $\mat{K}_\lambda$ to $\rd{V}_\lambda\times\rd{V}_\lambda$ is $\rdm{K}_\lambda$.

\begin{theorem}\label{thm.dual}
Assume the pb surface $\fS$ does not have interior puncture and has a triangulation $\lambda$.
\begin{enuma}
\item The $R$-linear maps
\begin{align}
&\rd{\psi}_\lambda: \rd{\lenT}(\fS,\lambda) \to \rd{\FG}(\fS,\lambda), \quad \text{given by} \ \rd{\psi}_\lambda (a^\vec{k})=x^{\vec{k}\rdm{K}_\lambda}, \bk \in \BZ^{\rd{V }_\lambda }\\
&{\psi_\lambda}: {\lenT}(\fS,\lambda) \to {\FG}(\fS,\lambda), \quad\text{given by} \ {\psi_\lambda} (a^\vec{k})=x^{\vec{k}\mK_\lambda}, \ \bk \in \BZ^{{V }_\lambda }
\end{align}
are $R$-algebra embeddings with images equal to the balanced subalgebras $\rdbl(\fS, \lambda)$ and $\bl{\FG}(\fS, \lambda)$ respectively.
\item Let $\irdV_\lambda\subset \rdV_\lambda$ be the subset of all small vertices in the interior of $\fS$. Then
\begin{equation}
\bmP_\lambda \, \bmQ_\lambda = \left[\begin{array}{c|c}
-4 n^2 (\Id_{\irdV_\lambda \times \irdV_\lambda } )& \ast \\ \hline
0 & \ast \end{array}\right] \label{eq.compa1}
\end{equation}
\end{enuma}
\end{theorem}

The proof is exactly the same as in the triangle case of Theorem~\ref{thm-tdual-tri}, with Lemma \ref{lemma-HK-tri} and Proposition \ref{prop-bal-tri} replaced respectively by Lemma \ref{lemma-HK} and Proposition \ref{prop-bal}.

\begin{remark}
\begin{enuma}
\item If $\tilde B $ is the $\rdV_\lambda \times \irdV_\lambda$-submatrix of $\bmQ_\lambda$, then Equ. \eqref{eq.compa1} shows that the pair $(\bmP_\lambda, \tilde B)$ is compatible in the theory of quantum cluster algebra \cite{BZ}. Similar statement holds in the extended case.
\item The pair $(\bmP_\lambda, \bmQ_\lambda)$ is also compatible in the sense of \cite{GS}. The result of \cite[Section 12]{GS} implies that $\bmQ_\lambda$ has a compatible matrix. However compatible matrix might not be unique, and we don't know if our $\bmP_\lambda$ is the same compatible matrix obtained in \cite{GS}.
\end{enuma}
\end{remark}

\subsection{Inverses of $\bmK_\lambda$ and $\mK_\lambda$}\label{ss.KH2}

Define the reduced matrix $\rdm{H}_\lambda:\rd{V}_\lambda\times\rd{V}_\lambda\to\ints$ exactly as the triangle case. That is,
\begin{itemize}
\item If $v$ and $v'$ are not on the same boundary edge then $\rdm{H}_\lambda(v,v')=-\frac{1}{2}\rdm{Q}_\lambda(v,v')\in\ints$.
\item If $v$ and $v'$ are on the same boundary edge, then
\[ \bmH(v,v') = \begin{cases} 1  \qquad &\text{when $v=v'$},\\
 -1 &\text{when there is arrow from $v$ to $v'$} \\
0 &\text{otherwise}
\end{cases} \]
\end{itemize}
See Figure~\ref{fig-H} for an illustration of $\rdm{H}_\lambda$ values. Define $\mat{H}_\lambda$ as the restriction of $\exm{H}$ to $V_\lambda\times\lv{V}_\lambda$, which agrees with the restriction of $-\frac{1}{2}\exm{Q}$ since the domain of $\mat{H}_\lambda$ does not contain pairs $(v,v')$ on the same boundary edge of $\ext{\surface}$.

\begin{figure}
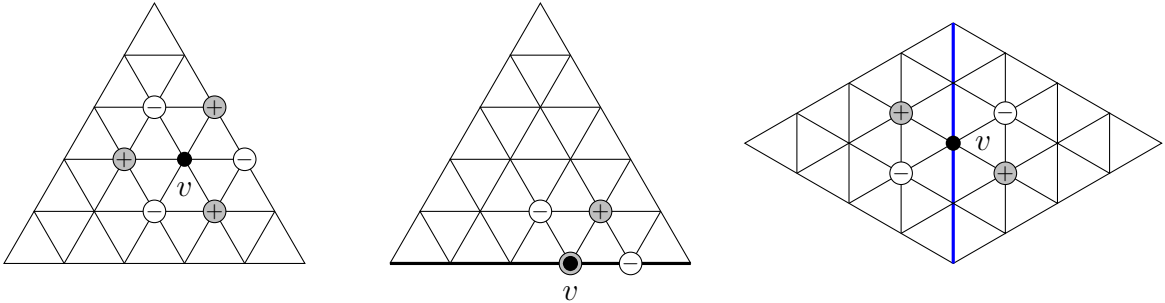

\centering
\input{H-int2}
\qquad
\input{H-edge-d}
\quad
\input{H-edge-int2}
\caption{$\rdm{H}_\lambda(v,\cdot)$ values for (Left) $v$ in the interior of a triangle, (Middle) $v$ on a boundary edge, (Right) $v$ on an interior edge}\label{fig-H}
\end{figure}

\begin{lemma}\label{lemma-HK}
The following matrix identities holds.
\begin{enuma}
\item $n(\rdm{K}_\lambda-\rdt{K}_\lambda)=\rdm{P}_\lambda$.
\item $\rdt{H}_\lambda-\rdm{H}_\lambda=\rdm{Q}_\lambda$.
\item $\rdm{H}_\lambda\rdm{K}_\lambda=n\id$ and $\mat{H}_\lambda\mat{K}_\lambda=n\id$.
\item $\rdm{K}_\lambda\rdm{Q}_\lambda\rdt{K}_\lambda=\rdm{P}_\lambda$ and $\mat{K}_\lambda\mat{Q}_\lambda\mat{K}^t_\lambda=\mat{P}_\lambda$.
\end{enuma}
\end{lemma}

\begin{proof}
First consider the reduced case. (b) is exactly the same as the triangle case. The calculation in (c) is given in Subsection~\ref{sec-dual-sur}. (a) and (d) are equivalent by the same proof in Lemma~\ref{lemma-HK-tri}. Here we prove (d). For $u,v\in\rd{V}_\lambda$,
\begin{align*}
\rdm{P}_\lambda(u,v)&=\sum_{\tau\in\face_\lambda}\rdm{P}_\tau(\sk_\tau(u) ,\sk_\tau(v))\\
&=\sum_{\tau\in\face_\lambda}\sum_{z,z'\in\rd{V}_\tau}\rdm{K}_\tau(\sk_\tau(u),z)\rdm{Q}_\tau(z,z')\rdm{K}_\tau(\sk_\tau(v),z')
&&(\text{(d) for the triangle})\\
&=\sum_{\tau\in\face_\lambda}\sum_{z,z'\in\rd{V}_\tau}\rdm{K}_\lambda(u,z)\rdm{Q}_\tau(z,z')\rdm{K}_\lambda(v,z')
&&(\text{Definition of }\rdm{K}_\lambda)\\
&=\sum_{z,z'\in\rd{V}_\lambda}\sum_{\tau\in\face_\lambda}\rdm{K}_\lambda(u,z)\rdm{Q}_\tau(z,z')\rdm{K}_\lambda(v,z')\\
&=(\rdm{K}_\lambda\rdm{Q}_\lambda\rdt{K}_\lambda)(u,v).
\end{align*}
In the fourth line, there are extra terms where $z$ or $z'$ is not in $\tau$, but these terms are zero since $\rdm{Q}_\tau$ is an extension by zero. The last line used the definition of $\rdm{Q}_\lambda$ for the sum over $\tau$.

Now consider the non-reduced case. (c) is proved in Subsection~\ref{sec-dual-ext}. For (d),
\begin{equation}
\mat{P}_\lambda=\mat{C}\exm{P}\mat{C}^t
=\mat{C}\exm{K}\exm{Q}\exm{K}^t\mat{C}^t
=(\mat{C}\exm{K})\exm{Q}(\mat{C}\exm{K})^t.
\end{equation}
To proceed, we write out the matrix multiplication.
\begin{align*}
\mat{P}_\lambda(u,v)
&=\sum_{z,z'\in\rd{V}_{\ext{\lambda}}} (\mat{C}\exm{K})(u,z)\exm{Q}(z,z')(\mat{C}\exm{K})(v,z')\\
&=\sum_{z,z'\in V_\lambda} (\mat{C}\exm{K})(u,z)\exm{Q}(z,z')(\mat{C}\exm{K})(v,z')
\qquad(\text{by Lemma~\ref{lemma-CK0}})\\
&=(\mat{K}_\lambda\mat{Q}_\lambda\mat{K}_\lambda^t)(u,v).\qedhere
\end{align*}
\end{proof}

\begin{proposition}\label{prop-bal}
Let $\vec{k}$ be a vector in $\ints^{\rd{V}_\lambda}$. Then the following are equivalent
\begin{enumerate}
\item $\vec{k}$ is balanced.
\item $\vec{k}\rdm{H}_\lambda\in(n\ints)^{\rd{V}_\lambda}$.
\item There exists a vector $\vec{c}\in\ints^{\rd{V}_\lambda}$ such that $\vec{k}=\vec{c}\rdm{K}_\lambda$.
\end{enumerate}
The same results hold for the non-reduced case, i.e., when $\rd{V}_\lambda, \rdm{H}_\lambda, \rdm{K}_\lambda$ are replaced respectively by ${V}_\lambda, \mH_\lambda, \mK_\lambda$.
\end{proposition}

\begin{proof}
First consider the reduced case. (2) and (3) are equivalent by Lemma~\ref{lemma-HK} with $\vec{c}=\vec{k}\rdm{H}_\lambda/n$. (3)$\Rightarrow$(1) by the definition of $\rdm{K}_\lambda$ and the triangle case Proposition~\ref{prop-bal-tri}.

Next we show (1)$\Rightarrow$(2). Write
\begin{equation}
(\vec{k}\rdm{H}_\lambda)(v)=\sum_{u\in\rd{V}_\lambda}\vec{k}(u)\rdm{H}_\lambda(u,v).
\end{equation}

If $v$ is in the interior of a triangle or on a boundary edge, then $\rdm{H}_\lambda(u,v)$ is nontrivial only if $u$ is in the same triangle as $v$. Then $(\vec{k}\rdm{H}_\lambda)(v)\in n\ints$ by the triangle case.

If $v$ is on an interior edge, then there are two triangles adjacent to the edge. See Figure~\ref{fig-bal-edge-int}. Label the edges so that the $e_3$ edge in the right triangle $\tau_1$ is identified with the $e_1$ edge in the left triangle $\tau_2$. Then $v=f_{\tau_1}(n-k,0,k)=f_{\tau_2}(i',n-i',0)$ with $k+i'=n$. To proceed further, write $\vec{k}\bmod{n}$ as a linear combination of the generators in each triangle. Let
\begin{equation}\label{eq-k-decomp}
\vec{k}|_{\tau_1}\equiv a\vec{k}_1+b\vec{k}_2,\qquad
\vec{k}|_{\tau_2}\equiv c\vec{k}'_1+d\vec{k}'_3\pmod{n},
\end{equation}
where $\vec{k}_1,\vec{k}_2$ are the generators in $\tau_1$, and $\vec{k}'_1,\vec{k}'_3$ are the generators in $\tau_2$. Note $\vec{k}_2$ and $\vec{k}'_3$ vanish on the common edge, and $\vec{k}_1=\vec{k}'_1$ on the common edge. Thus consistency implies that in \eqref{eq-k-decomp}, $a\equiv c\pmod{n}$. Therefore, $\vec{k}|_{\tau_1\cup\tau_2}\bmod{n}$ is a combination of three vectors: $\vec{k}_1\cup\vec{k}'_1$, $\vec{k}_2$, and $\vec{k}'_3$ (extended by $0$). Each of these vectors satisfies the equation
\[(\vec{k}\rdm{H}_\lambda)(v)=-\vec{k}(v_1)+\vec{k}(v_2)-\vec{k}(v_3)+\vec{k}(v_4)=0.\]
Therefore, $(\vec{k}\rdm{H}_\lambda)(v)\in n\ints$ for any balanced $\vec{k}$.

\begin{figure}
\centering
\input{bal-edge-int}
\caption{$(\vec{k}\rdm{H}_\lambda)(v)$ calculation for $v$ on an interior edge}
\label{fig-bal-edge-int}
\end{figure}

This proves the proposition for the reduced case. The non-reduced case follows using the fact that $\mat{H}_\lambda$ is the restriction of $\exm{H}$.
\end{proof}

\section{Quantum trace maps, the $X$-version} \label{sec.Xtrace}

In this section we prove the existence of $X$-version quantum traces for all triangulable surfaces. The reduced version $\btr^X_\lambda$ is constructed by patching together $\btr^X$ of the triangles. The extended version is constructed using the reduced version of the extended surface.

\subsection{Cutting for Fock-Goncharov algebra}

The $X$-version quantum trace is compatible with cutting homomorphisms. The cutting for skein algebras is given by Theorem~\ref{t.splitting2} and Proposition~\ref{prop-cut-rd}. Here we explain the cutting for Fock-Goncharov algebras.

Let $\lambda$ be an ideal triangulation of a pb surface $\surface$. For an interior edge $c\in\lambda$, the cut surface $\Cut_c\surface$ has a triangulation $\Cut_c\lambda$, which is $\lambda$ with $c$ replaced by the two copies of $c$ in the cut surface.

Construct an algebra embedding $\qtorus(\rdm{Q}_\lambda)\to\qtorus(\rdm{Q}_{\Cut_c\lambda})$ as follows. If $v\in\rd{V}_\lambda$ is not on the edge $c$, then $v$ is naturally identified with a unique $v\in\rd{V}_{\Cut_c\lambda}$. In this case, let $x_v\mapsto x_v$. If $v$ is on $c$, then it is cut into two copies $v_1,v_2\in\rd{V}_{\Cut_c\lambda}$. In this case, let $x_v\mapsto[x_{v_1}x_{v_2}]_\Weyl$. This extends to a well-defined algebra homomorphism by the sum-of-faces definition \eqref{eq-Q-def}. The image is a monomial subalgebra characterized by the matching exponents of vertices cut from the same one.

Recall $\rdbl\subset\qtorus$ is the monomial subalgebra corresponding to the balanced subgroup. Since the balanced condition is defined using faces of the triangulation, it behaves well with cutting, as the faces are unaffected. Thus the embedding above restricts to the cutting homomorphism
\begin{equation}
\Theta_c:\rdbl(\surface,\lambda)\to\rdbl(\Cut_c\surface,\Cut_c\lambda),
\end{equation}
and the image is characterized by the matching condition above.

\subsection{Quantum trace, the reduced case}

\begin{theorem}\label{thm-rdtrX}
Assume $\fS$ is a triangulable punctured bordered surface with an ideal triangulation $\lambda$, and the ground ring $R$ is a commutative domain with a distinguished invertible $\hq$. There exists an algebra homomorphism
\begin{equation}\label{eq.bXtr}
\rdtr_\lambda^X: \reduceS(\surface) \to \rdbl(\fS,\lambda)
\end{equation}
with the following properties:
\begin{enumerate}
\item $\rdtr_\lambda^X$ is compatible with cutting along an edge of $\lambda$.
\item When $\fS=\stdT$, the map $\rdtr_\lambda^X$ is the one given in Theorem~\ref{thm.Xtr0}.
\item When $R=\BC$, $\hq=1$, and $\al$ is a simple closed curve on $\fS$, one has
\begin{equation}
\rdtr_\lambda^X(\alpha) = \tTr_\lambda(\alpha)
\end{equation}
where $\tTr_\lambda(\al)$ the Fock-Goncharov classical trace, which is denoted by $\widetilde{\operatorname{Tr}}_\alpha$ in \cite{Douglas}.
\end{enumerate}
\end{theorem}

\begin{proof}
For each triangle $\tau$ we have the trace $\rdtr_\tau^X:\reduceS(\tau)\to\rd{\FG}(\tau)$. Consider the composition
\begin{equation}\label{eq-rdtrX-def}
\begin{tikzcd}
\reduceS(\surface)\arrow[r,"\rd\Theta_\lambda"] &
\displaystyle\bigotimes_{\tau\in\face_\lambda}\reduceS(\tau) \arrow[r,"\otimes\rdtr_\tau^X"] &
\displaystyle\rd{\FG}_\lambda:=\bigotimes_{\tau\in\face_\lambda}\rd{\FG}(\tau).
\end{tikzcd}
\end{equation}
Here $\Theta_\lambda$ is the composition of all cutting homomorphisms on the interior edges of $\lambda$.

Next, we show that the image of \eqref{eq-rdtrX-def} is contained in $\rd{\FG}(\surface,\lambda)$, which, by Subsection \ref{ss.rdFG}, is identified with the $R$-submodule of $\rd{\FG}_\lambda$ spanned by $x^\bk$ with $\bk(v')= \bk (v'')$ whenever $v'$ and $v''$ are identified under the gluing $\fS = \sqcup \tau /\sim$.

This is a corollary of Proposition~\ref{prop-bdry-exp}. Assume we glue edge $e'$ of triangle $\tau'$ to edge $e''$ of $\tau''$, giving edge $e$ of triangulation $\lambda$. Let $u'_1, \dots u'_{n-1}$ (respectively $u''_1, \dots u''_{n-1}$) be the small vertices on $e'$ (respectively $e''$) in the positive direction. Then $u'_i$ will be identified with $u''_{n-i}$. 

Let $\al$ be a stated web diagram transverse to every edge of $\lambda$. By definition \eqref{eq.cut00},
\begin{equation}
\Theta_\lambda(\al) = \sum_s \bigotimes_{\tau \in \cF_\lambda} (\al \cap \tau, s),
\end{equation}
where the sum is over all states $s: \al \cap c \to \JJ$ for all interior edges $c$. Let fix one $s$ and focus on the edge $e$. Every point $z\in \alpha\cap e$ is cut into two endpoints with the same state $s(z)$ but opposite orientations. Hence the weight of one is the obtained by applying the involution $\ror$ to the other; see Subsection \ref{sec-grading}. It follows that $\dd_e( (\al\cap \tau'', s)) = \ror (\dd_e( (\al\cap \tau', s)))$. Hence from \eqref{eq.rorbra} we have
\begin{equation}\label{eq.300}
\la \dd_e( (\al\cap \tau', s)), \varpi_i\ra = \la \dd_e( (\al\cap \tau'', s)), \varpi_{n-i}\ra.
\end{equation}

By Proposition~\ref{prop-bdry-exp}, the element $\rdtr^X( (\al\cap \tau', s) )$ is homogeneous in $x_{u'_i}$ of degree equal to the left-hand side of \eqref{eq.300}, while $\btr^X( (\al\cap \tau'', s) )$ is homogeneous in $x_{u''_{n-i}}$ of degree equal to the right-hand side.
This shows that the image of \eqref{eq-rdtrX-def} is contained in $\rd{\FG}(\surface,\lambda)$.

Let $\rdtr_\lambda^X$ be the map \eqref{eq-rdtrX-def} with codomain restricted to $\rd{\FG}(\surface,\lambda)$. The image is clearly balanced. Properties (1) and (2) are obvious from the definition. To relate to the classical case, use Theorem~\ref{thm-trX-CS} to replace $\rdtr^X(\al)$ by the right-hand side of \eqref{eq.301}, which is equal to $\tTr_\lambda(\al)$, see \cite[Section 2]{CS}.
\end{proof}

A corollary of the proof is the following.

\begin{corollary}\label{cor-bdry-exp}
Let $v$ be a small vertex on the boundary of a triangulable surface $\fS$ and $\al$ is stated web diagram. Then for any triangulation $\lambda$ of $\fS$ the image $\btr^X_\lambda(\al)$ is homogeneous in $x_{v}$ of degree $ n \la \dd_e(\al), \varpi_i\ra$. Here $v$ is the $i$-th small vertex on the boundary edge containing $v$ if we list the boundary small vertices in the positive direction.
\end{corollary}

\subsection{The non-reduced case}

Recall the extended surface $\ext{\surface}$ defined by attaching triangles to each boundary edge of $\surface$. Let $e$ be a boundary edge of $\surface$. By convention $e=e_1$ in the attached triangle. There is an embedding $\iota:\surface\to\ext{\surface}$ so that $\iota(e)=e_2$. See Figure~\ref{fig-attach-embed}.

\begin{figure}
\centering
\input{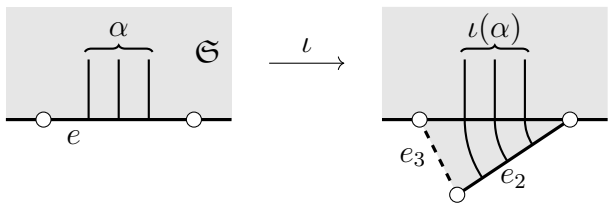}
\caption{The embedding $\iota$}\label{fig-attach-embed}
\end{figure}

Recall the subgroup $\rd{B}$ and submonoid $B$ of the balanced group $\rd{\Lambda}_\stdT$ defined in Subsection~\ref{sec-attach-counit}. Let $f_\tau:\tau=\stdT\to\surface$ be the characteristic map of an attached triangle $\tau$. Let $\rd{B}_\lambda\subset\Lambda_\lambda$ be the subgroup consisting of vectors $\vec{k}$ such that the pullback (or restriction) $f_\tau^\ast\vec{k}$ is in $\rd{B}\subset\rd{\Lambda}_\stdT$ for every attached triangle $\tau$. The submonoid $B_\lambda\subset\Lambda_\lambda$ is similarly defined.

\begin{theorem}\label{thm.full}
Assume $\fS$ is a triangulable surface with an ideal triangulation $\lambda$. There exists an algebra homomorphism
\begin{align}\label{eq.Xtr}
\tr_\lambda^X: \SS \to \qtorus(\mat{Q}_\lambda;B_\lambda) \subset \FGbl(\surface,\lambda)
\end{align}
and an algebra projection $\pr:\qtorus(\mat{Q}_\lambda;B_\lambda)\onto\rdbl(\surface,\lambda)$ such that $\tr_\lambda^X$ is a lift of the reduced trace $\rdtr_\lambda^X$. In other words, the following diagram commutes.
\begin{equation}\label{eq-tr-lift}
\begin{tikzcd}
\skein(\surface)\arrow[r,"\tr_\lambda^X"]\arrow[d,two heads,"\pr"]
&\qtorus(\mat{Q}_\lambda;B_\lambda)\arrow[d,two heads,"\pr"]\\
\reduceS(\surface)\arrow[r,"\rdtr_\lambda^X"]&\rdbl(\surface,\lambda)
\end{tikzcd}
\end{equation}
\end{theorem}

\begin{remark}
The domain $\SS$ and codomain $\cX(\fS,\lambda)$ of $\tr^X$ both have the same GK dimension, given by $(n^2-1)r(\surface)$. See Lemma~\ref{lemma-vset-size}. In Corollary~\ref{cor-trX-embed}, we use this fact to show that $\tr^X$ is an embedding when $\surface$ has no interior punctures. We conjecture that $\tr^X$ is an embedding for all triangulable surfaces.
\end{remark}

\begin{proof}
First define the trace $\tr_\lambda^X$. Consider the composition
\begin{equation}\label{eq-trx-extdef}
\begin{tikzcd}
\skein(\surface) \arrow[r,"\iota_\ast"] &
\reduceS(\ext{\surface}) \arrow[r,"\rdtr_{\ext{\lambda}}^X"] &
\rdbl(\ext{\surface},\ext{\lambda}),
\end{tikzcd}
\end{equation}
where $\ext{\lambda}$ is the triangulation extending $\lambda$. To restrict the codomain, apply the compatibility of $\rdtr_{\ext{\lambda}}^X$ with cutting. Given a diagram $\alpha$ on $\surface$, cutting $\iota(\alpha)$ along $e$ produces parallel corner arcs connecting $e=e_1$ and $e_2$. After applying the trace, Theorem~\ref{thm-attach-counit} implies that the image of the attached triangle part is in $\qtorus(\rdm{Q}_\stdT;B)$. By definition, this means $\tr_\lambda^X(\alpha)=\rdtr_{\ext{\lambda}}^X(\iota(\alpha))$ is in $\qtorus(\mat{Q}_\lambda;B_\lambda)$.

Define $\pr$ by the composition
\begin{equation}\label{eq-trpr-def}
\begin{tikzcd}
\qtorus(\mat{Q}_\lambda;B_\lambda) \arrow[r,"\Theta"]&
\displaystyle\rdbl(\surface,\lambda)\otimes\bigotimes_{e\in\lambda_\partial}\qtorus(\rdm{Q}_\stdT;B) \arrow[r,"\id\otimes\bigotimes\epsilon_X"]&
\rdbl(\surface,\lambda).
\end{tikzcd}
\end{equation}
Here $\lambda_\partial$ is the set of boundary edges, $\Theta$ is the cutting homomorphism along $\lambda_\partial$ so that it cuts off all attached triangles, and $\epsilon_X$ is the extended counit in Theorem~\ref{thm-attach-counit}. This is clearly an algebra homomorphism. Working through the definition, we have the formula
\begin{equation}
\pr(x^{\vec{k}})=
\begin{cases}
x^{\iota^\ast\vec{k}},&\vec{k}\in\rd{B}_\lambda,\\
0,&\vec{k}\notin\rd{B}_\lambda.
\end{cases}
\end{equation}
Here $\iota^\ast:\rd{B}_\lambda\to\rd{\Lambda}_\lambda$ is the pullback (restriction) by $\iota$. It is easy to see that $\iota^\ast$ is an isomorphism. Thus $\pr$ is surjective.

Finally, we prove the lifting property. Again by compatibility with cutting, we can exchange the order of $\Theta$ and $\rdtr_{\ext{\lambda}}^X$ in the composition $\pr\circ\tr_\lambda^X$. For a diagram $\alpha$ on $\surface$, the cutting of $\iota(\alpha)$ has the form
\begin{equation}
\Theta(\iota(\alpha))=\sum_{\text{states }s}(\alpha,s)\otimes C_s,
\end{equation}
where $(\alpha,s)$ is $\alpha$ with a different state $s$, and $C_s$ denotes the arcs in all the attached triangles. Then
\begin{align}
(\pr\circ\tr_\lambda^X)(\alpha)&=(\id\otimes\bigotimes\epsilon_X)(\rdtr_\lambda^X\otimes\bigotimes\rdtr^X)(\Theta(\iota(\alpha))) \notag \\
&=\sum_{\text{states }s}\epsilon_X(\rdtr^X(C_s))\rdtr_\lambda^X(\alpha,s)
\end{align}
Again by Theorem~\ref{thm-attach-counit}, $\epsilon_X(\rdtr^X(C_s))=1$ if $s$ is the same as the original states of $\alpha$, and $\epsilon_X(\rdtr^X(C_s))=0$ otherwise. This proves the commutativity of \eqref{eq-tr-lift}.
\end{proof}

\begin{remark}
The theorem can apply formally to $\surface=\poly_2$. In this case, $\ext{\poly_2}$ is the quadrilateral $\poly_4$ where $\poly_2$ is a neighborhood of a diagonal $e$. The triangulation $\ext{\lambda}$ consists of $e$ and the boundary edges, and both faces are ``attached". This defines $\mat{Q}_\lambda$ and $B_\lambda$. $\rd{\FG}(\poly_2,\lambda)$ is defined as the subalgebra that only involves the small vertices on $e$, which is the (commutative) Laurent polynomial algebra $R[x_1,\dots,x_{n-1}]$, and the balanced subalgebra is generated by $x_1x_2^2\dots x_{n-1}^{n-1}$ and $x_i^{\pm n}$. With $\skein(\poly_2)\cong\Oq$, $\reduceS(\poly_2)$ is identified with $R[u_{11},\dots,u_{nn}]/(u_{11}\dots u_{nn}=1)$, and $\rdtr_\lambda^X(u_{ss})=\prod_{i=1}^{n-1}x_i^{n\langle\weight_{\bar{s}},\varpi_i\rangle}$. The proof goes through with little change.
\end{remark}

\section{Quantum trace maps, the $A$-version}
\label{sec.Atrace}

Throughout this section $\lambda$ is a triangulation of a pb surface $\fS$ which has no interior puncture. We construct $A$-version quantum traces $\btr^A_\lambda$ and $\tr^A_\lambda$ and show that their images are sandwiched between the quantum $A$-tori and their quantum spaces.

Unlike the $X$-version case, one cannot patch the $\btr^A$ of the triangles together to get a global $A$-version quantum trace. This is because the $A$-tori do not admit a cut like the $X$-version. However it is straightforward to construct the $A$-version quantum traces once the $X$-versions have been defined, via the isomorphism $\FGbl(\fS,\lambda)\cong \cA(\fS,\lambda)$. The real task is to prove the images of the $A$-version quantum traces are sandwiched between the quantum $A$-tori and the quantum $A$-spaces.

Recall that we had the extended surface $\fS^\ast$, the vertex sets $\rdV_\lambda, V_\lambda$, and $V'_\lambda$ in Section \ref{sec.qtori}.

\subsection{Results}

\begin{theorem}\label{thm-Atr}
Assume $\surface$ is a triangulable punctured bordered surface with no interior puncture, and $\lambda$ is an ideal triangulation of $\surface$. Assume the ground ring $R$ is a commutative domain with a distinguished invertible element $\hq$.
\begin{enuma}
\item There is a unique algebra embedding
\begin{equation}\label{eq.Atr}
\tr_\lambda^A: \SS \to \cA(\fS,\lambda)
\end{equation}
such that
\begin{equation}\label{eq.TrAX}
\tr_\lambda^X = \psi_\lambda \circ\tr_\lambda^A.
\end{equation}
In addition,
\begin{equation}\label{eq-Aincl}
\cA_+(\fS,\lambda) \subset \tr_\lambda^A(\SS) \subset \cA(\fS,\lambda).
\end{equation}
\item There is a unique algebra homomorphism
\begin{equation}\label{eq.bAtr}
\rdtr_\lambda^A: \bSS \to \bA(\fS,\lambda)
\end{equation}
such that
\begin{equation}\label{eq.bTrAX}
\rdtr_\lambda^X = \rd\psi_\lambda \circ\rdtr_\lambda^A.
\end{equation}
In addition,
\begin{equation}\label{eq-Aincl-rd}
\bA_+(\fS,\lambda) \subset \rdtr_\lambda^A(\bSS) \subset \bA(\fS,\lambda).
\end{equation}
If $\fS$ is a polygon, then $\rdtr_\lambda^A$ is injective.
\end{enuma}
\end{theorem}

We get the following important corollary.

\begin{corollary}\label{cor-trX-embed}
Under the assumption of Theorem \ref{thm-Atr}, the $X$-version quantum trace $\tr_\lambda^X$ is injective, and its reduced version $\rdtr_\lambda^X$ is injective if $\surface$ is a polygon.
\end{corollary}

\subsection{Quantum frames}

We construct a quantum torus frame $\{\ag_v\mid v\in\lv{V}_\lambda\}$ for $\SS$.

First assume $v\in\rd{V}_\lambda$. Then $v= (ijk)\in \rdV_\nu$ for an ideal triangle $\nu$ of $\lambda$. We constructed the $Y$-graph $\tY_v$ in Figure~\ref{fig-skel}. Turn $\tY_v$ into the stated $n$-web $\ag''_v$ by replacing a $k$-labeled edge of $\tY_v$ with $k$ parallel edges of $\gaa''_v$, adjusted by a sign. See Figure~\ref{fig-skel-web} top. By Lemma~\ref{r.triad}, the element $\gaa''_v$ is reflection-normalizable.

\begin{figure}
\centering
\input{skel-web_new}
\caption{Definition of $\gaa''_v$}\label{fig-skel-web}
\end{figure}

Now assume $v\in \lv{V}_\lambda\setminus\rd{V}_\lambda$. Then $v=(ijk)$ is in an attached triangle $\nu\equiv \PP_3$ (of $\fS^\ast$), whose edge $e_1$ is glued to a boundary edge $e$ of $\fS$. Let $c$ be the oriented corner arc of $\fS$ starting on $e$ and going counterclockwise, i.e. turning left all the time. Then the element $\gaa''_v:=M^{[j+1,j+i]}_{[\bar{i},n]}(c)$ is reflection-normalizable by Lemmas~\ref{r.det} and \ref{r.height5}. See Figure~\ref{fig-skel-web} bottom for the diagram of $\gaa''_v$.

Define $\gaa_v$ to be the reflection normalization of $\gaa''_v$ for all $v\in\lv{V}_\lambda$. Let $\bar \ag_v$ be the image of $\ag_v$ in $\reduceS(\surface)$. Note that $\bar \ag_v = 0$ if $v \in \lv{V}_\lambda\setminus\rd{V}_\lambda$.

\begin{lemma}\label{lemma-trX-av}
We have
\begin{align}
\rdtr_\lambda^X(\rd{\ag}_v)& =x^{\rdm{K}_\lambda(v,\cdot)}= \rd{\psi}_\lambda(a_v), \qquad \text{for } \ v\in\rd{V}_\lambda \label{eq.bgav}\\
\tr_\lambda^X(\ag_v)& =x^{\mat{K}_\lambda(v,\cdot)}= {\psi}_\lambda(a_v), \qquad \text{for } \ v\in\lv{V}_\lambda . \label{eq.gav}
\end{align}
\end{lemma}

\begin{proof}
The second equalities in \eqref{eq.bgav} and \eqref{eq.gav} follow from the definition of $\rd\psi_\lambda$ and $\psi_\lambda$, respectively. Let us prove the first identities.

To calculate $\rdtr_\lambda^X(\rd{\ag}_v)$, we use the cutting homomorphism \eqref{eq-rdtrX-def}. Suppose $v=ijk$ in the triangle $\nu$. After an isotopy of the heights, the cutting of $\gaa''_v$ consists of $\gaa'_{ijk}$ (of Figure~\ref{fig-trigen-res}, but with a priori different states) in $\nu$ and corner arcs in various triangles. If the states assigned in the cut contain bad arcs or repeated states connecting to the vertex, then the corresponding term is zero. It is easy to see that to avoid these configurations, there is only one possible state for the cut, where the corner arcs extending from the same endpoint must have the same states throughout. See Figure~\ref{fig-sgcut-state}, where the notation $\bar{2}=n-1$ is used. It follows that
\[ \Theta_\lambda(\gaa''_v) = \bigotimes_{\tau \in \face(\lambda)} \gaa''_v \cap \tau, \]
where each $\gaa''_v \cap \tau$ is stated by the above unique state.

\begin{figure}
\centering
\input{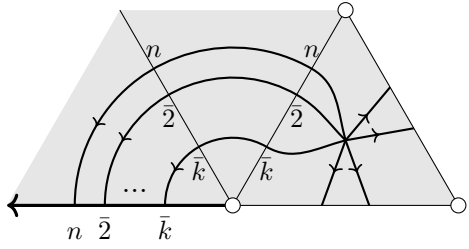}
\caption{The unique nontrivial state assignment}\label{fig-sgcut-state}
\end{figure}

Recall that the edges of the triangulation cut $\tY_v$ into segments, consisting of the main segment $Y_v$ and labeled directed arc segments. As a $k$-labeled edge of $\tY_v$ generates $k$-parallel edges of $\gaa''_v$, an arc segment $s$ in a triangle $\tau$ corresponds to a group of $k$ stated parallel arcs in $\tau$ denoted by $\gaa''(s)$. More precisely,
\[\input{surgen-corner-seg}=M^{[\bar{k};n]}_{[\bar{k};n]}(s),\]
where the last identity is from Lemma \ref{r.Ibad1}. Also define $\gaa''(s)= \gaa_{ijk}$ if $s$ is the main segment.

For each segment $s$ of $\tY_v$ in triangle $\tau$, we defined the small vertex $Y(s)\in \rdV_\tau$ in Subsection~\ref{ss.skeleton}. By definition, $\btr^A_\tau(\gaa''(s))\eqq a_{Y(s)}$. Hence
\[\btr^X_\tau(\gaa''(s))\eqq x^{\bmK_\tau ( Y(s), \cdot) }.\]
By Lemma \ref{r.height5}, in the triangle $\tau$ all the $\gaa(s)$ are $q$-commuting. Hence
\begin{equation}
\btr^X_\tau(\gaa''_v \cap \tau )
\eqq \prod_{s \subset \tau\cap\tilde{Y}_v} x^{\bmK_\tau( Y(s), \cdot)}
\eqq x ^{ \bmK_\tau ( \sk_\tau(v), \cdot) },
\end{equation}
where for the last identity we use the definition \eqref{eq-skel-def}.

By definition \eqref{eq-surgen-exp}, we have $\bmK_\lambda ( v,v')= \bmK_\tau ( \sk_\tau(v), v') $ for all $v' \in \rdV_\tau$. It follows that
\[\btr^X( \gaa''_v) \eqq x^{\bmK_\lambda(v, \cdot) }.\]
Then by reflection invariance, we have $\rdtr_\lambda^X(\gaa_v)=x^{\rdm{K}_\lambda(v,\cdot)}$, proving \eqref{eq.bgav}

Now consider the full trace $\tr_\lambda^X$, which is defined by \eqref{eq-trx-extdef}. Each $v\in\lv{V}_\lambda$ is also a vertex in $\rd{V}_{\ext{\lambda}}$. The corresponding element in $\reduceS(\ext{\surface})$ is denoted $\ext{\gaa}_v$.

If $v\in\rd{V}_\lambda$, then $\tr_\lambda^X(\gaa_v)=\rdtr_{\ext{\lambda}}^X(\ext{\gaa}_v)$. By the first part of the Lemma, $\rdtr_{\lambda^\ast}^X(\ext{\gaa}_v)=x^{\rdm{K}_{\ext{\lambda}}(v,\cdot)}$, which restricts to $x^{\mat{K}_\lambda(v,\cdot)}$.

If $v\in\lv{V}_\lambda\setminus\rd{V}_\lambda$ is $ijk$ in an attached triangle, then by the same calculation as Lemma~\ref{lemma-agen-decomp},
\begin{equation}
\ext{\gaa}_v=\left[\ext{\gaa}_{p(v)} \gaa_v\right]_\Weyl,
\end{equation}
where $p(v)$ is defined by \eqref{eq-cov-pdef}. Thus
\begin{equation}
\tr_\lambda^X(\gaa_v)
=\rdtr_{\ext{\lambda}}^X\left[\ext{\gaa}_v(\ext{\gaa}_{p(v)})^{-1}\right]_\Weyl
=x^{\rdm{K}_{\ext{\lambda}}(v,\cdot)-\rdm{K}_{\ext{\lambda}}(p(v),\cdot)}
=x^{(\mat{C}\rdm{K}_{\ext{\lambda}})(v,\cdot)},
\end{equation}
which restricts to $x^{\mat{K}_\lambda(v,\cdot)}$ by definition.
\end{proof}

\subsection{Proof of Theorem \ref{thm-Atr}}

\begin{proof}
By Theorem \ref{thm.dual} we have the algebra isomorphisms
\[ \rd{\psi}_\lambda: \rd{\cA}(\fS,\lambda) \xrightarrow{\cong} \rd\FG^\bal(\fS,\lambda), \quad
{\psi}_\lambda: {\cA}(\fS,\lambda) \xrightarrow{\cong} \FG^\bal(\fS,\lambda). \]
We define the
the $A$-version quantum traces by
\[ \btr^A_\lambda = (\rd{\psi}_\lambda)^{-1} \circ \btr^X, \quad
\tr^A_\lambda = ({\psi}_\lambda)^{-1} \circ \tr^X. \]
Clearly \eqref{eq.TrAX} and \eqref{eq.bTrAX} are satisfied.

(a) From \eqref{eq.gav} we have, for all $v\in V_\lambda$,
\begin{equation}
\tr_\lambda^X(\gaa_v) = a_v.
\label{eq.gaaa}
\end{equation}
This proves the inclusion ${\lenT}_+(\surface,\lambda)\subset\tr_\lambda^A(\SS)$. Hence
\[\GKdim ({\lenT}_+(\surface,\lambda) ) \le \GKdim (\tr_\lambda^A(\SS) ).\]
Since $\GKdim({\lenT}_+(\surface,\lambda))=|V_\lambda|$, which is equal to $\GKdim(\SS)$ by Lemma \ref{lemma-vset-size}, we have
\[\GKdim \tr_\lambda^A(\SS) \ge \GKdim(\SS).\]
By Theorem \ref{thm.domain} $\SS$ is an torsion-free $R$-domain. Hence by Lemma~\ref{r.GKdim}, $\tr_\lambda^A$ is injective.

(b) Exactly the same proof as in (a) gives the inclusion $\rd{\lenT}_+(\surface,\lambda)\subset\rdtr_\lambda^A(\reduceS(\surface))$ and
\begin{equation}
\btr_\lambda^X(\bar \gaa_v) = a_v \quad \text{for all } \ v\in \rdV_\lambda.
\end{equation}

Assume $\surface=\poly_k$. By Theorem \ref{thm.domainr} the algebra $\reduceS(\surface)$ is an $R$-torsion free domain with GK dimension given by the right-hand side of \eqref{eq.GKrd}, which, by Lemma~\ref{lemma-vset-size}, is equal to
$\abs{\rd{V}_\lambda}$, or the GK dimension of $\rd{\lenT}_+(\surface,\lambda)$. Again Lemma~\ref{r.GKdim} implies that $\rdtr_\lambda^A$ is injective.
\end{proof}

The injectivity of $\tr^A_\lambda: \SS \embed {\lenT}_+(\surface,\lambda)$ and \eqref{eq.gaaa} implies the following.

\begin{corollary}
With the assumption of Theorem \ref{thm-Atr}, the set $\{ \gaa_v \mid v\in V_\lambda\}$ is a quantum torus frame for $\SS$.

If $\fS$ is the polygon $\PP_k$, then $\{ \bar \gaa_v \mid v\in \rdV_\lambda\}$ is a quantum torus frame for $\bSS$.
\end{corollary}

If Conjecture \ref{conj.inj2} is true, then for any surface $\fS$ of Theorem \ref{thm-Atr}, the set $\{ \bar \gaa_v \mid v\in \rdV_\lambda\}$ is a quantum torus frame for $\bSS$.

\section{Coordinate change of quantum trace maps}
\label{sec.Nat}

We establish the naturality of the quantum traces with respect to the change of triangulations. For the $A$-version quantum trace this follows easily from the sandwichness \eqref{eq-Aincl}. For the $X$-version the proof is much more difficult. We first apply the $A$-version case to quadrilaterals, composed with the isomorphism $\rd\psi_\lambda$, to define the transition isomorphism for the $X$-version quantum trace. Then we use the $A$-version case for pentagons to prove the well-definedness.

\subsection{Statements of the results}

\begin{theorem}\label{thm-co-chg-A}
Suppose $\surface$ is a triangulable surface with no interior punctures. Given two ideal triangulations $\lambda,\lambda'$, there exists a unique coordinate change isomorphism
\begin{equation}
\Psi^A_{\lambda'\lambda}:\Fr(\lenT(\surface,\lambda))\to\Fr(\lenT(\surface,\lambda'))
\end{equation}
such that
\begin{equation}\label{eq-co-chg-trA}
\Psi^A_{\lambda'\lambda}\circ\tr_\lambda^A=\tr_{\lambda'}^A.
\end{equation}
The coordinate change isomorphism is functorial in the sense that for ideal triangulations $\lambda,\lambda',\lambda''$,
\begin{equation}
\Psi^A_{\lambda\lambda}=\id,\qquad
\Psi^A_{\lambda''\lambda'}\circ\Psi^A_{\lambda'\lambda}=\Psi^A_{\lambda''\lambda}.
\end{equation}

Analogous results for the reduced algebras holds when $\surface$ is a polygon (or more generally when $\rdtr_\lambda^A$ is injective). The map is denoted by $\rd{\Psi}^A_{\lambda'\lambda}$.
\end{theorem}

\begin{theorem}\label{thm-co-chg-X}
Suppose $\surface$ is a triangulable surface. Given two ideal triangulations $\lambda,\lambda'$, there exists a coordinate change isomorphism
\begin{equation}
\Psi^X_{\lambda'\lambda}:\Fr(\FGbl(\surface,\lambda))\to\Fr(\FGbl(\surface,\lambda'))
\end{equation}
such that
\begin{equation}\label{eq-co-chg-trX}
\Psi^X_{\lambda'\lambda}\circ\tr_\lambda^X=\tr_{\lambda'}^X.
\end{equation}
The coordinate change isomorphism is functorial in the sense that for ideal triangulations $\lambda,\lambda',\lambda''$,
\begin{equation}\label{eq-co-chg-funcX}
\Psi^X_{\lambda\lambda}=\id,\qquad
\Psi^X_{\lambda''\lambda'}\circ\Psi^X_{\lambda'\lambda}=\Psi^X_{\lambda''\lambda}.
\end{equation}

Analogous results holds for the reduced algebras. The map is denoted by $\rd{\Psi}^X_{\lambda'\lambda}$.
\end{theorem}

Note Theorem~\ref{thm-co-chg-A} has uniqueness compared to Theorem~\ref{thm-co-chg-X}. An easy corollary of the theorems is the following.

\begin{corollary}
Suppose $\surface$ is a triangulable surface with no interior punctures.
\begin{enuma}
\item $\Psi^X_{\lambda'\lambda}$ is uniquely determined by
\begin{equation}\label{eq-co-chg-AX}
\Psi^X_{\lambda'\lambda}=\Fr(\psi_{\lambda'})\circ\Psi^A_{\lambda'\lambda}\circ\Fr(\psi_\lambda^{-1}).
\end{equation}
Analogous results hold for the reduced algebras when $\surface$ is a polygon.
\item $A$-version coordinate change maps can be defined for the reduced algebra.
\end{enuma}
\end{corollary}

\begin{proof}
(a) Let $f_{\lambda'\lambda}=\Fr(\psi_{\lambda'}^{-1})\circ\Psi^X_{\lambda'\lambda}\circ\Fr(\psi_\lambda)$. Then $f_{\lambda'\lambda}$ satisfy the defining property \eqref{eq-co-chg-trA} of the $A$-version coordinate change. By uniqueness, $f_{\lambda'\lambda}=\Psi_{\lambda'\lambda}^A$.

(b) $\rd{\Psi}^A_{\lambda'\lambda}=\Fr(\rd{\psi}_{\lambda'}^{-1})\circ\rd{\Psi}^X_{\lambda'\lambda}\circ\Fr(\rd{\psi}_\lambda)$ is a coordinate change map using the properties of the $X$-version.
\end{proof}

\subsection{Proof of the $A$-version Theorem~\ref{thm-co-chg-A}}

Fix an ideal triangulation $\lambda$. Clearly $\lenT(\surface,\lambda)$ is a localization of $\lenT_+(\surface,\lambda)$, so $\Fr(\lenT_+(\surface,\lambda)) = \Fr(\lenT(\surface,\lambda))$. Hence from \eqref{eq-Aincl},
\begin{equation}
\Fr(\lenT_+(\surface,\lambda)) \cong \Fr(\skein(\surface)) \cong \Fr(\lenT(\surface,\lambda)),
\end{equation}
where the second isomorphism is $\Fr(\tr_\lambda^A)$. Then $\Psi^A_{\lambda'\lambda}$ is uniquely defined as $\Fr(\tr_{\lambda'}^A)\circ\Fr(\tr_\lambda^A)^{-1}$. The properties of $\Psi^A_{\lambda'\lambda}$ are trivial to verify.

The reduced case follows from the same argument using \eqref{eq-Aincl-rd}.

\subsection{Proof of the $X$-version Theorem~\ref{thm-co-chg-X}}

We first define the coordinate change for the reduced case, and the non-reduced case can be obtained by a restriction of the reduced case for the extended surface $\ext{\surface}$.

We state the behavior of the coordinate change at boundary edges here. The proof will be given at each step of the construction.

\begin{corollary}\label{cor-co-chg-bdry}
Suppose $\lambda_1,\lambda_2$ are ideal triangulations of the surface $\surface$. Let $\rd{\FG}'_i$ be the subalgebra of $\rd{\FG}(\surface,\lambda_i)$ generated by $x_v$ with $v$ not on a boundary edge. Given a monomial $\mon{m}\in\rdbl(\surface,\lambda)$, write $\mon{m}=\mon{m}_\partial\mon{m}'$ where $\mon{m}_\partial$ is a monomial on the variables $x_v$ with $v$ on a boundary edge, and $\mon{m}'\in\rd{\FG}'_1$. Then
\begin{equation}
\rd{\Psi}_{\lambda_2\lambda_1}^X(\mon{m})\in\mon{m}_\partial\Fr(\rd{\FG}'_2).
\end{equation}
\end{corollary}

\textbf{Step 1}: For the reduced algebra of a polygon, define the map $\rd{\Psi}^X_{\lambda'\lambda}$ by \eqref{eq-co-chg-AX}. Then \eqref{eq-co-chg-trX} and \eqref{eq-co-chg-funcX} are trivial to verify using the $A$-version theorem.

To prove Corollary~\ref{cor-co-chg-bdry} for a polygon, first consider $\mon{m}=\rdtr_{\lambda_1}^X(a_v)$ for any $v\in\rd{V}_{\lambda_1}$. Then $\rd{\Psi}_{\lambda_2\lambda_1}^X(\mon{m})=\rdtr_{\lambda_2}^X(a_v)$, and the statement holds by Corollary~\ref{cor-bdry-exp}. Since these monomials weakly generates $\rdbl(\poly_k;\lambda_1)$, the result holds for all monomials.

\textbf{Step 2}: We define the coordinate change for a flip with the reduced algebra. Consider the flip at an edge $e$, shown in Figure~\ref{fig-flip}. The edges of the quadrilateral need not be distinct.

\begin{figure}
\centering
\input{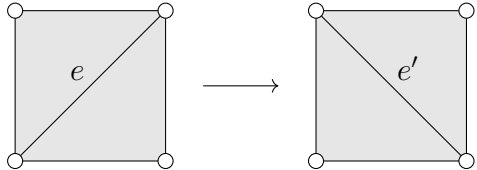}
\caption{A flip at the edge $e$}\label{fig-flip}
\end{figure}

Let $\surface_{\hat{e}}$ be the surface $\surface$ cut along all interior edges of $\lambda$ except $e$, and let $\face_{\hat{e}}$ be the set of triangle components of $\surface_{\hat{e}}$. The one remaining component of $\surface_{\hat{e}}$ is the quadrilateral $\poly_4$ containing $e$. Let
\begin{equation}
\Theta_{\hat{e}}:\reduceS(\surface)\to\reduceS_{\hat{e}}
:=\Bigg(\bigotimes_{\tau\in\face_{\hat{e}}}\reduceS(\tau)\Bigg)\otimes\reduceS(\poly_4)
\end{equation}
be the corresponding cutting homomorphism. Let
\begin{equation}
\rdbl_{\hat{e}}=\Bigg(\bigotimes_{\tau\in\face_{\hat{e}}}\rdbl(\tau)\Bigg)\otimes\rdbl(\poly_4;e),
\end{equation}
where by abuse of notation, $e$ also denotes the ideal triangulation of $\poly_4$ containing the edge $e$. Recall that by cutting, $\rdbl(\surface,\lambda)$ is embedded in $\rdbl_{\hat{e}}$ as the subalgebra satisfying the matching condition. By the compatibility of the $X$-version trace with cutting, $\rdtr_\lambda^X$ is the restriction of the composition $((\bigotimes\rdtr_\tau^X)\otimes\rdtr_e^X)\circ\Theta_{\hat{e}}$.

Define
\begin{equation}
\rd{\Psi}_e^X=\Fr((\bigotimes\id)\otimes\rd{\Psi}_{e'e}^X):\Fr(\rdbl_{\hat{e}})\to\Fr(\rdbl_{\hat{e}'}).
\end{equation}
Restricted to $\rdbl(\surface,\lambda)$, Corollary~\ref{cor-co-chg-bdry} shows that the image of $\rd{\Psi}_e^X$ satisfies the matching condition. Thus $\rd{\Psi}_e^X$ restricts to a map
\[\rd{\Psi}^X_{\lambda'\lambda}:\Fr(\rdbl(\surface,\lambda))\to\Fr(\rdbl(\surface,\lambda')).\]
This is an isomorphism. The inverse is defined by the same flip construction on $e'$. Then \eqref{eq-co-chg-trX} follows from the case of $\poly_4$ and the splitting definition of $\rdtr^X$.

Corollary~\ref{cor-co-chg-bdry} for a flip is an easy consequence of the construction, since the flip uses the coordinate change of the polygon $\poly_4$.

\textbf{Step 3}: Any two triangulations $\lambda,\lambda'$ are connected by a sequence of flips, see \cite{Penner},
\[\lambda=\lambda_0\to\lambda_1\to\cdots\to\lambda_k=\lambda'.\]
Define
\begin{equation}
\rd{\Psi}^X_{\lambda'\lambda}=\rd{\Psi}^X_{\lambda_k\lambda_{k-1}}\circ\cdots\circ\rd{\Psi}^X_{\lambda_1\lambda_0}.
\end{equation}
We need to show that this is well-defined.

It is known (see e.g. \cite[Chapter 5]{Penner}) that two sequences of flips connecting $\lambda\to\lambda'$ are related by the following moves and their inverses:
\begin{enumerate}
\item (Reflexivity Relation) Suppose the new edge in the flip at $e$ is $e'$, then the flips at $e$ then $e'$ can be canceled.
\item (Distant Commutativity Relation) If $e$ and $f$ are edges in different triangles, then flips at $e$ then $f$ is the same as flips at $f$ then $e$.
\item (Pentagon Relation) Suppose 5 (not necessary distinct) edges of the triangulation bound a pentagon $\poly_5$. The 5 triangulations of $\poly_5$ are connected by flips shown in Figure~\ref{fig-pentagon}. Then the sequence of flips at $e$, $f$, $e'$, $f'$, and $e''$ can be canceled.
\end{enumerate}

\begin{figure}
\centering
\input{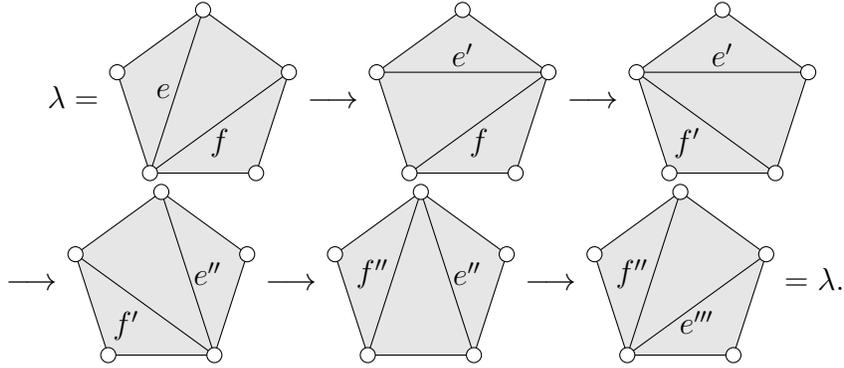}
\caption{The Pentagon Relation}\label{fig-pentagon}
\end{figure}

The corresponding relations holds for the coordinate change isomorphisms using the following cutting argument. Let $b,c\in\lambda$ be disjoint arcs and $\lambda'$ be the flip of $\lambda$ at $b$. By the definition of the coordinate change for a flip, we obtain the following commutative diagram.
\begin{equation}
\begin{tikzcd}[column sep=small]
\Fr(\rdbl(\surface,\lambda)) \arrow[rr,hook] \arrow[rd,hook] \arrow[ddd] & & \Fr(\rdbl(\Cut_c\surface,\Cut_c\lambda)) \arrow[ld,hook] \arrow[ddd] \\
& \Fr(\rdbl_{\hat{b}}) \arrow[d] & \\
& \Fr(\rdbl_{\hat{b}'}) & \\
\Fr(\rdbl(\surface,\lambda')) \arrow[rr,hook] \arrow[ru,hook] & & \Fr(\rdbl(\Cut_c\surface,\Cut_c\lambda')) \arrow[lu,hook]
\end{tikzcd}
\end{equation}
Here, the hook arrows are all cutting homomorphisms, and the vertical arrows are the coordinate changes. The two quadrilateral commute by definition, and the triangles commute since they are just various stages of cutting. Using the outside square of the diagram, we can cut along all edges unchanged in the flips when we verify the above relations for coordinate changes.
\begin{enumerate}
\item This follows from the case of $\reduceS(\poly_4)$, as mentioned after the construction of a flip.
\item $\surface$ cut along unchanged edges becomes a collection of triangles and two quadrilaterals whose diagonals are $e$ and $f$. Each flip is identity on the quadrilaterals not containing the flipped edge. Clearly these two flips commute.
\item $\surface$ cut along unchanged edges becomes a collection of triangles and a pentagon $\poly_5$ containing $e$ and $f$. Since the triangulation goes back to the original after the sequence of flips, the composition of the coordinate changes is identity by the result of $\reduceS(\poly_5)$.
\end{enumerate}
This shows that $\rd{\Psi}_{\lambda'\lambda}^X$ is well-defined.

The properties of the coordinate change isomorphism are easy. Then \eqref{eq-co-chg-trX} and Corollary~\ref{cor-co-chg-bdry} follow from the properties of the flip. The functorial properties \eqref{eq-co-chg-funcX} follow directly from definition. This completes the proof for the reduced case.

\textbf{Step 4}: To obtain the coordinate change for the non-reduced algebra, consider the extended surface $\ext{\surface}$. $\FGbl(\surface,\lambda)$ is the subalgebra of $\rdbl(\ext{\surface},\ext{\lambda})$ characterized by certain boundary generators having zero exponents. By Corollary~\ref{cor-co-chg-bdry}, the coordinate change $\rd{\Psi}_{\ext{(\lambda')}\ext{\lambda}}$ preserves this property. Thus it restricts to a map $\Psi^X_{\lambda'\lambda}:\Fr(\FGbl(\surface,\lambda))\to\Fr(\FGbl(\surface,\lambda'))$. The properties of the non-reduced case follow easily from the reduced case. This complete the proof of Theorem \ref{thm-co-chg-X}.

{
\subsection{Comments}
We defined the coordinate change isomorphisms for the balanced subalgebras $\FG^\bal(\fS,\lambda)$ and $\rd\FG^\bal(\fS,\lambda)$ of the full algebras $\FG(\fS,\lambda)$ and $\rd\FG(\fS,\lambda)$.
It should be noted that there are no extensions of the coordinate change isomorphisms to the full algebras $\FG(\fS,\lambda)$ and $\rd\FG(\fS,\lambda)$. This is one reason why the proof of Theorem Theorem \ref{thm-co-chg-X} is difficult, even for the case when $n=2$, see \cite{BW}. Our approach, which uses the $A$-version quantum trace and avoids the complications arising when the triangulation has self-glued edges, is new even for the case $n=2$. The original Fock-Goncharov algebra is a subspace of $\rd\FG^\bal(\fS,\lambda)$, and one can check that our coordinate change isomorphism restricts to a coordinate change isomorphism of the original Fock-Goncharov algebra, and the restriction is equal to the composition of a sequence of quantum mutations. The details will appear elsewhere. Note that even if we have already had the coordinate change isomorphism for the original Fock-Goncharov algebra, it is non-trivial to extend it to $\rd\FG^\bal(\fS,\lambda)$, as we see that we cannot extend it to $\rd\FG(\fS,\lambda)$.
}

\section{The $SL_3$ case}
\label{sec.sl3}

In the case $n=3$, we will show that the reduced quantum trace $\btr^X_\lambda$ is injective.

\subsection{Positively stated subalgebra}

A stated web $\alpha$ is \term{positively stated} if the state of every endpoint of $\alpha$ is $>(n+1)/2$. The $R$-submodule of $\skein(\surface)$ spanned by positively stated webs form a subalgebra, denoted by $\skein_+(\surface)$. Similarly, $\rdSplus(\surface)\subset\reduceS(\surface)$ is the reduced version.

For convenience, let $m=\lfloor(n+1)/2\rfloor+1\in\JJ$ be the smallest positive state, and define $P=[m;n]$ be the set of positive states. Let $M\subset\rdSplus(\surface)$ be the multiplicatively closed subset generated by $M^P_P(a)$ where $a=C(v),\cev{C}(v)$ for all vertices $v$. By Lemma~\ref{r.height5}, elements of $M$ $q$-commute. Positively stated corner arcs $q$-commute with every diagram, so it is an Ore set.

\begin{lemma}\label{lemma-loc-sur}
For every $\alpha\in\reduceS(\surface)$, there exists an element $m\in M$, such that $m\alpha\in\rdSplus(\surface)$.
\end{lemma}

\begin{proof}
We just need to prove the lemma for diagrams. Let $\alpha$ be a web diagram on $\surface$ with state $s:\partial\alpha\to\JJ$. Define the deficit of $\alpha$ as
\[s_-(\alpha)=\sum_{\substack{x\in\partial\alpha\\s(x)<m}}(m-s(x)).\]
We induct on the deficit. If the deficit is $0$, then all states are positive, and the lemma is trivial. Now assume $\alpha$ has positive deficit.

Among the endpoints with non-positive states, choose the endpoint $x$ with maximal height. Let $i=s(x)<m$. Define a new diagram $\alpha'$ by adding a vertex close to this endpoint, replace the small segment near the boundary by $m-2$ parallel strands with reverse orientation that connect to the vertex, assign the states $I=[\bar{m}+1,n]\setminus\{\bar{i}\}$, adding strands close to the boundary $\partial\surface$ with constant height that extends to the boundary edge counterclockwise to the current one, and assign positive states $P=[m,n]$ to the new strand. The diagrams $\alpha$ and $\alpha'$ are shown in Figure~\ref{fig-st-switch}.

The only state in $I$ that is potentially non-positive is $\bar{m}+1$, which has deficit $1$ if $m$ is odd. Thus the deficit of $I$ is
\[s_-(I)=\begin{cases}1,&i<m-1,\\0,&i=m-1.\end{cases}.\]
This is strictly less than the deficit of the state $i$.

\begin{figure}
\centering
\input{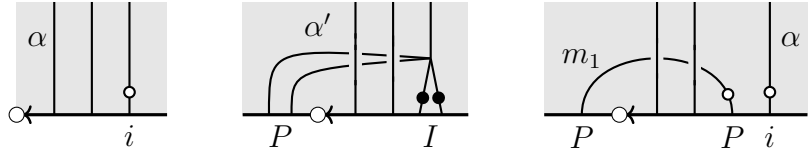}
\caption{The diagrams $\alpha$, $\alpha'$, and the resolution of the vertex}\label{fig-st-switch}
\end{figure}

Using Lemma~\ref{lemma-vertex-rev}, we can resolve the new vertex of $\alpha'$ on the boundary. The only permutation of states giving a nonzero diagram is the last one in Figure~\ref{fig-st-switch}. The states on the returning arcs are $I$ and $\bar{I}=[1,m-1]\setminus\{i\}$, and the states connected to the other boundary edge must be $P$ to avoid a bad arc. The remaining state $i$ is assigned to $\alpha$. The coefficient of this term is $\pm\hq^l$ for some $l\in\ints$.

The corner arcs in the resolution have the highest consecutive states. Let $m_1$ be these corner arcs. We can applying height exchange to the new endpoints stated with $P$ and write the resolved diagram as a product $m_1\alpha$. This results in an addition power of $\hq$. Thus we found a product of positively stated corner arc $m_1$ such that
\[m_1\alpha \eqq \pm\alpha',\]
and the deficit of $\alpha'$ is less than $\alpha$. Then by induction, the lemma is true.
\end{proof}

When $n=3$, The only state allowed is the highest state $3$. In Theorem~\ref{thm-Splus-iso}, we will show that $\skein_+(\surface)$ and $\rdSplus(\surface)$ both agree with the ``reduced $SU_3$-skein algebra" $\RSthree(\surface)$ of \cite{FrohmanSikora} if we set $a=1$ and replace $q^{1/3}$ by $\hq^6=q^{-1/3}$. One direction is easy. We can construct an algebra homomorphism $i:\RSthree(\surface)\to\skein_+(\surface)$ as follows. $\RSthree(\surface)$ is spanned by $3$-webs with no states or height order on the boundary. For such a web $\alpha$, $i(\alpha)\in\skein_+(\surface)$ is obtained by assigning the state $3$ to all endpoints. It is easy to check that the defining relations of $\RSthree(\surface)$ are preserved. Clearly, $i$ is surjective. The projection $\skein(\surface)\to\reduceS(\surface)$ also restricts to $\pr:\skein_+(\surface)\to\rdSplus(\surface)$. In Theorem~\ref{thm-Splus-iso}, we will show that $i$ and the restricted $\pr$ are both isomorphisms.

\subsection{Split triangulation}\label{sec-split}

Given a triangulation $\lambda$ of the surface $\surface$, let $\mathring{\lambda}$ denote the set of interior edges of $\lambda$. The corresponding \term{split triangulation} $\hat{\lambda}=\lambda\sqcup\mathring{\lambda}$ is a collection of disjoint arcs containing $\lambda$ such that each interior edge of $\lambda$ has two isotopic copies in $\hat{\lambda}$. If $\surface$ is cut along the interior edges of $\hat{\lambda}$, the components are triangles and bigons. The triangles are in bijection with the faces $\face_\lambda$, and the bigons are in bijection with the interior edges $\mathring{\lambda}$.

In the definition of $\rdtr^X$, we can cut along the interior edges of $\hat{\lambda}$ instead, and then apply $\rdtr_\tau^X$ to each face $\tau\in\face_\lambda$ and the counit $\epsilon$ to each bigon. The counit property implies that the composition is the same as the previous definition.

\subsection{Basis elements in canonical position}

First we define the basic components of webs in canonical position.

A \term{crossbar web} is a $3$-web in the bigon $\poly_2$ whose underlying graph consists of parallel lines connecting the two sides of the bigon and at most one line (crossbar) connecting each pair of the adjacent parallel lines. An example is given in Figure~\ref{fig-crossbar}.

The \term{honeycomb} of degree $h\in\ints$, denoted $H_h$, is defined to be a $3$-web in the triangle $\poly_3$ whose underlying graph is dual to the $\abs{h}$-triangulation of $\poly_3$. The orientation on the web near the boundary points out of $\poly_3$ if $h>0$ and into $\poly_3$ if $h<0$. By convention, the web is empty when $h=0$. The honeycomb of degree $1$ is the generator $a_{111}$. More examples are shown in Figure~\ref{fig-honeycomb}.

\begin{figure}
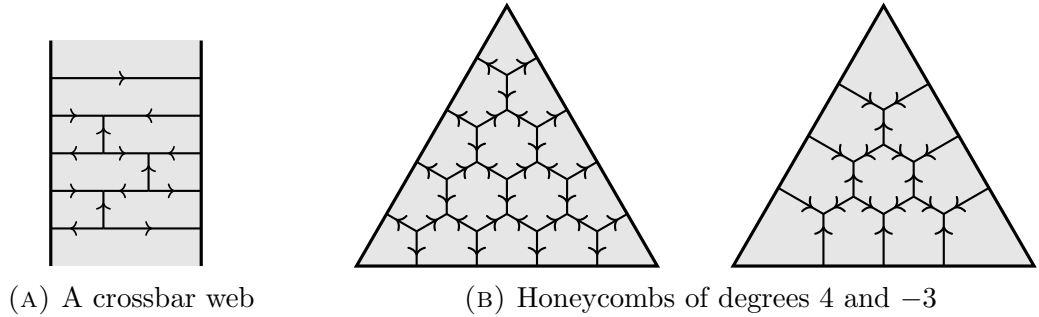

\centering
\begin{subfigure}[b]{0.3\linewidth}
\centering
\input{crossbar}
\subcaption{A crossbar web}\label{fig-crossbar}
\end{subfigure}
\begin{subfigure}[b]{0.6\linewidth}
\centering
\input{honeycomb}
\subcaption{Honeycombs of degrees $4$ and $-3$}\label{fig-honeycomb}
\end{subfigure}
\caption{Basic components of webs in canonical position}
\end{figure}

Suppose $\surface$ is a triangulable surface with an ideal triangulation $\lambda$. Let $B=B(\fS)$ be the basis of $\RSthree(\surface)$ in \cite{FrohmanSikora}. A basis element $\alpha\in B$ is in \term{canonical position} if
\begin{itemize}
\item in each bigon of the split triangulation, the web $\alpha$ is a crossbar web, and
\item in each face of the split triangulation, the web is a disjoint union of corner arcs and a honeycomb.
\end{itemize}

\subsection{Coordinates of basis}

Suppose $\surface$ is a triangulable surface with an ideal triangulation $\lambda$. Following \cite{FrohmanSikora} and \cite{DouglasSun}, the basis can be parameterized by $\nats^{\rd{V}_\lambda}$ as follows.

Every edge of $\lambda$ has two orientations. Let $\avec{\lambda}$ be the set of all oriented edges of $\lambda$. We identify $\rd{V}_\lambda=\avec{\lambda}\sqcup\face_\lambda$ such that
\begin{itemize}
\item an oriented edge $e\in\avec{\lambda}$ is identified with the vertex on $e$ closer to its tail, and that
\item a triangle $\tau\in\face_\lambda$ is identified with the vertex at the center of $\tau$.
\end{itemize}

Given a basis element $\alpha\in B$ in canonical position, we introduce the intersection and rotation numbers from \cite{FrohmanSikora}. For an oriented edge $e\in\avec{\lambda}$, the \term{intersection number} $e(\alpha)$ is the number of intersection points $\alpha\cap e$ where the tangent to $e$ is counterclockwise with respect to the tangent of $\alpha$. For a face $\tau\in\face_\lambda$, define the \term{rotation numbers} $\tau_+(\alpha)$ and $\tau_-(\alpha)$ as the numbers of counterclockwise and clockwise corner arcs of $\al$ in $\tau$, respectively. The net rotation number is defined as $r_\tau(\alpha)=\tau_-(\alpha)-\tau_+(\alpha)$.

The \term{Fock-Goncharov coordinates} $\vec{k}_\alpha\in\nats^{\rd{V}_\lambda}$ are defined by
\begin{equation}\label{eq-FG-coord}
\vec{k}_\alpha(e)=2e(\alpha)+\cev{e}(\alpha),\quad
e\in\avec{\lambda},\qquad
\vec{k}_\alpha(\tau)=\sum_{i=1}^3\left(e_i(\alpha)+\cev{e}_i(\alpha)\right)-\tau_-(\alpha).\quad
\tau\in\face_\lambda.
\end{equation}
Here, $e_1,e_2,e_3$ are the edges of $\tau$, and $\cev{e}$ is the edge $e$ with the opposite orientation. \cite{DouglasSun} showed that the Fock-Goncharov coordinate map
\[\kappa:B\to\nats^{\rd{V}_\lambda},\qquad \alpha\mapsto\vec{k}_\alpha\]
is injective, and that the image $\Gamma=\kappa(B)$ is a submonoid (with an explicit description by linear inequalities).

\subsection{Leading term}

Choose a total order $\preceq$ on $\nats^{\rd{V}_\lambda}$ such that $\vec{k}(v)\le\vec{k}'(v)$ for all $v\in\rd{V}_\lambda$ implies $\vec{k}\preceq\vec{k}'$. This order defines a filtration on $\rdbl(\surface,\lambda)$ by
\[F_\vec{n}=\Span\{x^\vec{k}\mid\vec{k}\preceq\vec{n}\}.\]
Similarly, $\RSthree(\surface)$ has a filtration
\[F_\vec{n}=\Span\{\alpha\in B\mid \kappa(\alpha)\preceq\vec{n}\}.\]

\begin{theorem}[Proposition~5.80 of \cite{Kim1}]\label{thm-lt-trX}
For a basis element $\alpha\in B$, the leading term of $\rdtr_\lambda^X(\pr(i(\alpha)))$ with respect to the filtration $F_\cdot$ is a monomial $\hq^m x^{\vec{k}_\alpha}$ for some $m\in\ints$. Here $\pr:\skein(\surface)\to\reduceS(\surface)$ is the natural projection.
\end{theorem}

\begin{remark}
In \cite{Kim1}, the ``stated skein algebra" is bigger than ours, and the ``reduced" skein algebra is our stated skein algebra. In addition, the isomorphism requires $\hq\leftrightarrow\omega^{-1/2}$ (so $q\leftrightarrow q^{-1}$) and the states are in reverse order $s\leftrightarrow 4-s$. Since Kim's quantum trace $\operatorname{Tr}$ do not have extra attached triangles, it is our reduced trace composed with the projection $\rdtr_\lambda^X\circ\pr$.

Kim calculated the leading term when the surface is a triangle and proved the result only for webs that do not end on the boundary. The proof is exactly the same if the webs end on the boundary with the highest states (our $s=3$ and their $s=1$).
\end{remark}

\begin{theorem}\label{thm-Splus-iso}
Suppose $n=3$.
\begin{enuma}
\item $i:\RSthree(\surface)\to\skein_+(\surface)$ an isomorphism.
\item The restricted projection $\pr:\skein_+(\surface)\to\reduceS(\surface)$ is injective.
\item The restricted trace $\rdtr_\lambda^X:\skein_+(\surface)\to\rdbl(\surface,\lambda)$ is injective.
\end{enuma}
We will identify $\RSthree(\surface)$ with $\skein_+(\surface)$, and consider it as subalgebras of both $\skein(\surface)$ and $\reduceS(\surface)$.
\end{theorem}

\begin{proof}
Since the filtration of $\RSthree(\surface)$ is defined as the span of basis elements, $B$ is also a basis for the associated graded algebra $\Gr(\RSthree(\surface))$. Similarly, the monomials form a basis for $\Gr(\rdbl(\surface,\lambda))=\rdbl(\surface,\lambda)$.

By Theorem~\ref{thm-lt-trX}, the composition
\begin{equation}
\rdtr_\lambda^+:
\begin{tikzcd}
\RSthree(\surface)\arrow[r,"i"] &
\skein_+(\surface)\arrow[r,"\pr"] &
\reduceS(\surface)\arrow[r,"\rdtr_\lambda^X"] &
\rdbl(\surface,\lambda)
\end{tikzcd}
\end{equation}
preserves the filtrations defined above. The associated graded map
\[\Gr(\rdtr_\lambda^+):\Gr(\RSthree(\surface))\to
\Gr(\rdbl(\surface,\lambda))=\rdbl(\surface,\lambda)\]
is injective since the image of the basis $B$ is a subset of the monomial basis of $\rdbl(\surface,\lambda)$.

Therefore, $i$ is injective. We already know $i$ is surjective. Thus $i$ is an isomorphism, and the restricted projection $\pr:\skein_+(\surface)\to\reduceS(\surface)$ is injective, too. With the identification $\RSthree(\surface)=\skein_+(\surface)$ by $i$, $\rdtr_\lambda^X=\rdtr_\lambda^+$ is injective.
\end{proof}

\begin{theorem} \label{thm.n3}
Suppose $n=3$. Let $M\subset\skein_+(\surface)$ be the multiplicative closed subset generated by positively stated corner arcs.
\begin{enuma}
\item $\reduceS(\surface)$ is the Ore localization $\skein_+(\surface)\{M^{-1}\}$, which is a domain.
\item $\rdtr_\lambda^X:\reduceS(\surface)\to\rdbl(\surface,\lambda)$ is the localization of the map $\skein_+(\surface)\to\rdbl(\surface,\lambda)$, and it is injective.
\end{enuma}
\end{theorem}

\begin{proof}
Since $\skein_+(\surface)$ is a subalgebra of $\skein(\surface)$, it is a domain. Thus the Ore localization $\skein_+(\surface)\{M^{-1}\}$ is a domain such that $\skein_+(\surface)$ is an embedded subalgebra.

Elements of $M$ are invertible in $\reduceS(\surface)$ by Lemma~\ref{r.Ibad1}. Thus the inclusion (restricted projection) $\pr:\skein_+(\surface)\hookrightarrow\reduceS(\surface)$ induces a map
\[f:\skein_+(\surface)\{M^{-1}\}\hookrightarrow\reduceS(\surface).\]
It is surjective by Lemma~\ref{lemma-loc-sur}. Thus $f$ is an isomorphism. This proves (a).

If $m$ is a positively stated corner arc, then by Theorem~\ref{thm-trX-CS}, $\rdtr_\lambda^X(m)$ is a monomial, which is invertible. Then (b) follows from Theorem~\ref{thm-Splus-iso} and the universal property of localization.
\end{proof}


\appendix

\section{Proofs of the matrix identities}

\subsection{Proof of Lemma~\ref{lemma-HK-tri}(c)} \label{sec-dual-tri}

It is helpful to rewrite the definition of $\rdm{H}$ more explicitly. Using our coordinate conventions, if $x=ijk$ is an interior vertex, then
\begin{equation}\label{eq-H-int}
\begin{split}
\rdm{H}(x;i+1,j,k-1)=\rdm{H}(x;i-1,j+1,k)=\rdm{H}(x;i,j-1,k+1)&=1,\\
\rdm{H}(x;i-1,j,k+1)=\rdm{H}(x;i+1,j-1,k)=\rdm{H}(x;i,j+1,k-1)&=-1.
\end{split}
\end{equation}

To obtain a formula when $x=ijk$ an edge vertex, we may assume $i=0$ using the rotational symmetry. Then the nonzero components are
\begin{equation}\label{eq-H-edge-d}
\begin{split}
\rdm{H}(x;0jk)=\rdm{H}(x;1,j,k-1)&=1,\\
\rdm{H}(x;0,j+1,k-1)=\rdm{H}(x;1,j-1,k)&=-1.
\end{split}
\end{equation}

First we consider when $x=ijk$ is interior, and let $y=i'j'k'$. By \eqref{eq-H-int},
\begin{equation}\label{eq-tri-dual-int}
\begin{split}
(\rdm{H}\rdm{K})(x,y)&=\rdm{K}(i+1,j,k-1;y)+\rdm{K}(i-1,j+1,k;y)+\rdm{K}(i,j-1,k+1;y)\\
&\quad-\rdm{K}(i-1,j,k+1;y)-\rdm{K}(i+1,j-1,k;y)-\rdm{K}(i,j+1,k-1;y).
\end{split}
\end{equation}

There are 7 cases depending on the relation between $ijk$ and $i'j'k'$, but the expression \eqref{eq-tri-dual-int} is symmetric under rotation. This reduces the calculation to only $3$ cases. See Figure~\ref{fig-tri-dual-int}, where the blue dots are Case 1, the white dots are Case 2, and the black dot is Case 3.

Case 1: $k'\le k-1$, $i'\ge i+1$. Then \eqref{eq-tri-dual-int} becomes
\begin{equation}
\begin{split}
&\quad[(i+1)j'+jk'+k'(i+1)]+[(i-1)j'+(j+1)k'+k'(i-1)]\\
&+[ij'+(j-1)k'+k'i]-[(i-1)j'+jk'+k'(i-1)]\\
&-[(i+1)j'+(j-1)k'+k'(i+1)]-[ij'+(j+1)k'+k'i].
\end{split}
\end{equation}
A careful cancellation shows the result is $0$.

Case 2: $k'\le k-1$, $i'=i$, so $j'\ge j+1$. Then \eqref{eq-tri-dual-int} becomes
\begin{equation}
\begin{split}
&\quad[jk'+(k-1)i'+i'j]+[(i-1)j'+(j+1)k'+k'(i-1)]\\
&+[ij'+(j-1)k'+k'i]-[(i-1)j'+jk'+k'(i-1)]\\
&-[(j-1)k'+ki'+i'(j-1)]-[ij'+(j+1)k'+k'i].
\end{split}
\end{equation}
A careful cancellation shows the result is $0$.

Case 3: $k'=k$, $i'=i$, so $i'j'k'=ijk$. Then \eqref{eq-tri-dual-int} becomes
\begin{equation}
\begin{split}
&\quad[jk+(k-1)i+ij]+[(i-1)j+(j+1)k+k(i-1)]\\
&+[ij+(j-1)k+ki]-[(i-1)j+jk+k(i-1)]\\
&-[(j-1)k+ki+i(j-1)]-[(k-1)i+ij+j(k-1)].
\end{split}
\end{equation}
This simplifies to $i+j+k=n$.

\begin{figure}
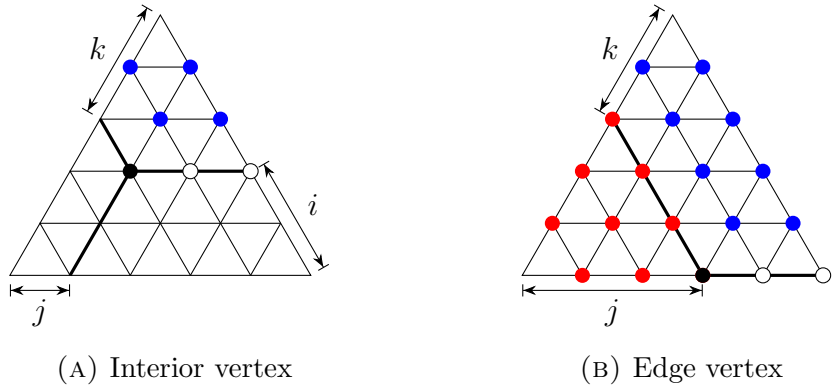

\centering
\begin{subfigure}{0.4\linewidth}
\centering
\input{tri-dual-int}
\subcaption{Interior vertex}\label{fig-tri-dual-int}
\end{subfigure}
\begin{subfigure}{0.4\linewidth}
\centering
\input{tri-dual-edge}
\subcaption{Edge vertex}\label{fig-tri-dual-edge}
\end{subfigure}
\caption{Cases to check in the triangle}
\end{figure}

Next we consider an edge vertex $x=ijk$. Recall we can assume $i=0$ using the rotational symmetry. By \eqref{eq-H-edge-d},
\begin{equation}\label{eq-tri-dual-edge}
(\rdm{H}\rdm{K})(x,y)=\rdm{K}(0jk,y)+\rdm{K}(1,j,k-1;y)-\rdm{K}(0,j+1,k-1;y)-\rdm{K}(1,j-1,k;y).
\end{equation}

This time there are 4 cases to check. See Figure~\ref{fig-tri-dual-edge}, where the blue dots are Case 1, the white dots are Case 2, the red dots are Case 3, and the black dot is Case 4.

Case 1: $k'\le k-1$, $i'\ge 1$. Then \eqref{eq-tri-dual-edge} becomes
\begin{equation}
[jk']+[j'+jk'+k']-[(j+1)k']-[j'+(j-1)k'+k']=0.
\end{equation}

Case 2: $k'\le k-1$, $i'=0$, so $j'\ge j+1$. Then \eqref{eq-tri-dual-edge} becomes
\begin{equation}
[jk']+[jk'+(k-1)i'+i'j]-[(j+1)k']-[(j-1)k'+ki'+i'(j-1)]=0.
\end{equation}

Case 3: $k'\ge k$, $j'\le j-1$. Then \eqref{eq-tri-dual-edge} becomes
\begin{equation}
[ki'+j'k]+[(k-1)i'+j'+j'(k-1)]-[(k-1)i'+j'(k-1)]-[ki'+j'+j'k]=0.
\end{equation}

Case 4: $i'j'k'=0jk$. Then \eqref{eq-tri-dual-edge} becomes
\begin{equation}
[jk]+[jk]-[j(k-1)]-[(j-1)k]=j+k=n.
\end{equation}

Therefore, we have $\rdm{H}\rdm{K}=n\id_{\rd{V}}$ for the triangle.

\subsection{Proof of Lemma~\ref{lemma-HK}(c) for the reduced case} \label{sec-dual-sur}

Let $\tau$ be a face containing $v\in\rd{V}_\lambda$. Let $W=\{w\in\rd{V}_\lambda\mid\rdm{H}_\lambda(u,w)\ne0\}$. Then by definition \eqref{eq-surgen-exp},
\begin{equation}\label{eq-HK-sur}
\rdm{H}_\lambda\rdm{K}_\lambda(u,v)
=\rdm{K}_\tau\left(\sum_{w\in W}\rdm{H}(u,w)\sk_\tau(w),v\right).
\end{equation}
We will show that the sum of skeletons can be reduced to $\tau$. The calculation depends on whether $u$ is in the interior of a face, on a boundary edge, or on an interior edge.

First consider when $u$ is in the interior of the face $\nu$. There are six vertices in $W$, and they are in $\nu$ as well. See Figure~\ref{fig-H-j}. The corresponding $\tilde{Y}_w$ have the same underlying graph $\tilde{Y}$ but with different weights. For any segment $s\subset\tau\cap\tilde{Y}$ extended from, say, $e_1$, the corresponding segment for $\tilde{Y}_w$ form 3 pairs with weights $i-1,i,i+1$, and the $\rdm{H}_\lambda(u,w)$ values have the opposite signs in each pair. Thus all such terms cancel, and
\begin{equation}
\sum_{w\in W}\rdm{H}(u,w)\sk_\tau(w)=\sum_{w\in W}\rdm{H}(u,w)w\in\ints[\rd{V}_\nu].
\end{equation}
Thus $\rdm{H}_\lambda\rdm{K}_\lambda=n\id$ follows from the triangle case. The case when $u$ is on a boundary edge is analogous.

\begin{figure}
\centering
\input{H-int-j2}
\input{H-edge-d-j2}
\input{H-edge-int-j2}
\caption{The calculation of $(\rdm{H}_\lambda\rdm{K}_\lambda)(u,v)$}\label{fig-H-j}
\end{figure}

Finally, there is the case when $u$ is on an interior edge $e$. See Figure~\ref{fig-H-j}. As in the proof of Lemma~\ref{lemma-K-welldef}, segments that do not intersect $e$ do not contribute to $\rdm{K}_\lambda$. Segments extended from an edge of the quadrilateral cancel as before. In particular, if $v$ is not in the quadrilateral, $\rdm{H}_\lambda\rdm{K}_\lambda(u,v)=0$. If $v$ is in the quadrilateral, by a possible rotation of the picture, we assume both $u,v$ are in the left triangle $\tau$. The segments in $\tau$ extending from $a,b$ are $u,f$, respectively. Thus the remaining terms in the skeleton sum is
\[\sum_{w\in W}\rdm{H}(u,w)\sk_\tau(w)=u-f+c-d\in\ints[\rd{V}_\tau],\]
which has the same pattern as the boundary case of the triangle. Thus $\rdm{H}_\lambda\rdm{K}_\lambda(u,v)=n\id(u,v)$ follow from the triangle case.

\subsection{Proof of Lemma~\ref{lemma-HK}(c), the non-reduced case}
\label{sec-dual-ext}

We first show that $\mat{H}_\lambda\mat{C}$ is the restriction of $\exm{H}$ to $V_\lambda\times\rd{V}_{\ext{\lambda}}$. Recall the only nonzero entries of $\mat{C}$ are $\mat{C}(v,v)=1$ when $v\in\lv{V}_\lambda$ and $\mat{C}(w,v)=-1$ when $w\in\lv{V}_\lambda\setminus\rd{V}_\lambda$ and $v=p(w)\in\rd{V}_{\ext{\lambda}}\setminus\lv{V}$. Thus
\begin{equation}\label{eq-HC-sum}
(\mat{H}_\lambda\mat{C})(u,v)=\sum_{w\in\lv{V}_\lambda}\mat{H}_\lambda(u,w)\mat{C}(w,v)
=\begin{cases}
\mat{H}_\lambda(u,v),&v\in\lv{V},\\
-\sum_{w\in p^{-1}(v)}\mat{H}_\lambda(u,w),&\text{otherwise}.
\end{cases}
\end{equation}

In the first case, $\mat{H}_\lambda(u,v)=\exm{H}(u,v)$. In the second case, if $u$ is not in the attached triangle containing $v$, then $\mat{H}_\lambda(u,w)=0$ for all $w\in p^{-1}(v)$, so $(\mat{H}_\lambda\mat{C})(u,v)=0=\exm{H}(u,v)$.

The remaining case is when $u\in V_\lambda$ is in the attached triangle containing $v\in\rd{V}_{\ext{\lambda}}\setminus\lv{V}_\lambda$. The calculation can be divided into 5 cases shown in Figure~\ref{fig-HC}, where the black dot is $u$. The first two diagrams are when $u$ is on the attaching edge. The next one is when $u\in V_\lambda\setminus\lv{V}_\lambda$. The last two are when $u\in (V_\lambda\cap\lv{V}_\lambda)\setminus\rd{V}_\lambda$. The white and gray circles are $\mat{H}_\lambda(u,\cdot)$ values. The empty circles in each diagram are $v$ positions where \eqref{eq-HC-sum} is not trivially zero, and the thick lines indicate preimages of $p$. The sum is also indicated in the figure. Comparing these results with $\exm{H}$ (see Figure~\ref{fig-H}), we see that the equality holds.

\begin{figure}
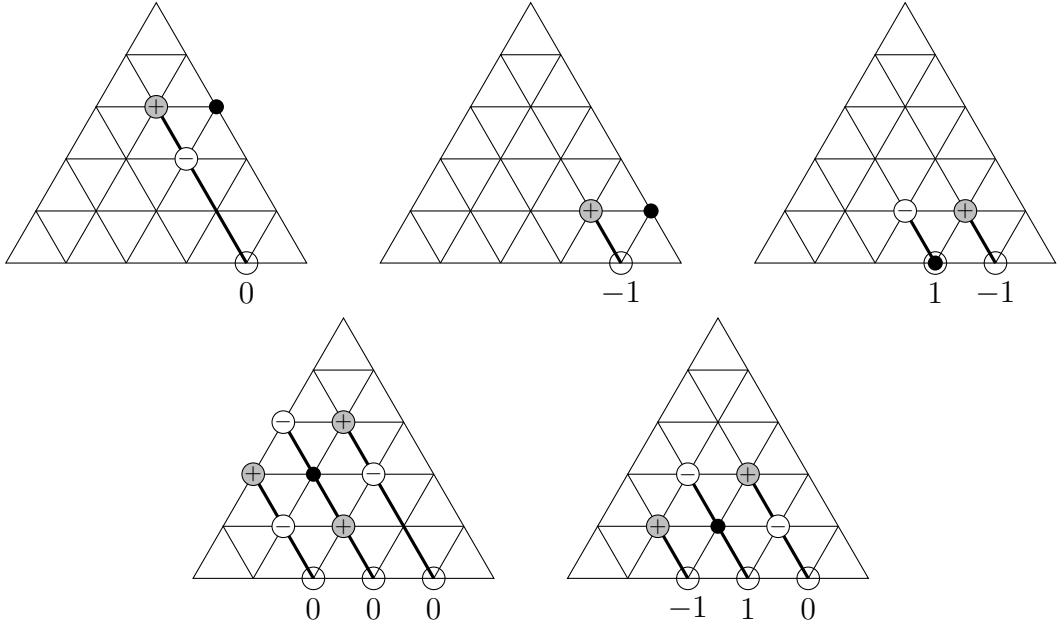

\centering
\input{QC-attach-1}\qquad
\input{QC-attach-2}\qquad
\input{QC-boundary}

\input{QC-int-1}\qquad
\input{QC-int-2}\qquad
\caption{Nontrivial calculations for $\mat{H}_\lambda\mat{C}$}\label{fig-HC}
\end{figure}

Then it follows that for $u,v\in V_\lambda$,
\begin{align*}
(\mat{H}_\lambda\mat{K}_\lambda)(u,v)&=(\mat{H}_\lambda\mat{C}\exm{K})(u,v)
=\sum_{w\in\rd{V}_{\ext{\lambda}}}(\mat{H}_\lambda\mat{C})(u,w)\exm{K}(w,v)\\
&=(\exm{H}\exm{K})(u,v)=n\id_{\rd{V}_{\ext{\lambda}}}(u,v)=n\id_{V_\lambda}(u,v).
\end{align*}

\section{Proof of Theorem~\ref{thm-trX-CS}}
\label{sec-trX-CS}

Recall by Lemma~\ref{lemma-agen-decomp}, for $v=(ijk)\in\rd{V}$,
\[\gaa_{ijk}=\left[M_1(i,k)M_2(j)\right]_\Weyl,\]
where $M_1(i,k)=M^{[i+1; i+k]}_{[\bar{k};n]}(\ceC(v_1))$ and $M_2(j)=M^{[\bar{j};n]}_{[\bar{j};n]}(C(v_2))$.

The following special cases are useful.
\begin{equation}\label{eq-agen-decomp-sp}
M_2(j)=\gaa_{n-j,j,0},\qquad M_1(i,k)\eqq \gaa_{ijk}M_2(j)^{-1}=\gaa_{ijk}\gaa_{n-j,j,0}^{-1}.
\end{equation}
In particular, $\rdtr^X(M_1(i,k))$ is a monomial. Hence it is invertible.

\begin{lemma}\label{lemma-trX-CS-diag}
Theorem~\ref{thm-trX-CS} holds for $\alpha=C(v_2)_{ss}$.
\end{lemma}

\begin{proof}
Write
\[C(v_2)_{ss}=M_2(\bar{s})M_2(\bar{s}-1)^{-1}
=\gaa_{s-1,\bar{s},0}\gaa_{s,n-s,0}^{-1}\]
using Lemmas~\ref{lemma-agen-decomp} and \ref{r.Ibad1}. The image under $\rdtr^X$ is a normalized monomial since the factors commute. The exponent is
\[\rdm{K}((s-1,\bar{s},0),\cdot)-\rdm{K}((s,n-s,0),\cdot)=\vec{k}_p\]
for the unique compatible path $p$ by direct calculation.
\end{proof}

\begin{lemma}\label{lemma-trX-CS-uni}
Recall $I_1=\{(i,k)\in\nats^2\mid i\ge0,k\ge1,i+k\le n\}$. The system of equations
\begin{align}
&\sum_{\sigma\in\Sym_k}(-q)^{\ell(\sigma)}\cev{z}_{\bar{k},i+\sigma(1)}\cdots\cev{z}_{n,i+\sigma(k)}=\rdtr^X(M_1(i,k)),\quad(i,k)\in I_1,\label{eq-T-eq1}\\
&\cev{z}_{ij}=0,\quad i<j.
\end{align}
has a unique solution $\cev{z}_{ij}=\rdtr^X(\ceC(v_1)_{ij})$ in $\rd\FG(\stdT)$.
\end{lemma}

\begin{proof}
$\cev{z}_{ij}=\rdtr^X(\ceC(v_1)_{ij})$ is a solution essentially by definition. To show uniqueness, we solve these equations inductively. Let $P_d$ be the statement
\begin{center}
$\cev{z}_{st}$ is uniquely determined to be $\rdtr^X(\ceC(v_1)_{st})$ for all $s+t\ge d$.
\end{center}
The base case is $d=2n$ or $(s,t)=(n,n)$. When $(i,k)=(n-1,1)$, the left-hand side of \eqref{eq-T-eq1} is simply $\cev{z}_{nn}$. Then
\[\cev{z}_{nn}=\rdtr^X(M_1(n-1,1))=\rdtr^X(M^{[n;n]}_{[n;n]}(\ceC(v_1)))
=\rdtr^X(\ceC(v_1)_{nn}).\]
This proves the base case.

Let $Z^{[i+1;i+k]}_{[\bar{k};n]}$ denote the left-hand side of \eqref{eq-T-eq1}. Now consider an arbitrary $\cev{z}_{st}$, $s\ge t$. Perform a ``cofactor expansion" on \eqref{eq-T-eq1} with $(i,k)=(t-1,\bar{s})$. Let $\Sym_{k-1}\subset\Sym_k$ be the embedded subgroup consisting of permutations $\sigma$ with $\sigma(1)=1$. Such a permutation is identified with $\sigma_1\in\Sym_{k-1}$, $\sigma_1(i)=\sigma(i+1)-1$. Then $\ell(\sigma)=\ell(\sigma_1)$. The left-hand side of \eqref{eq-T-eq1} becomes
\begin{align*}
Z^{[t;n-s+t]}_{[s;n]}
&=\sum_{\sigma_1\in\Sym_{k-1}}(-q)^{\ell(\sigma_1)}\cev{z}_{st}\cev{z}_{s+1,t+\sigma_1(1)}\cdots\cev{z}_{n,t+\sigma_1(k-1)}
+\sum_{\sigma\in\Sym_k\setminus\Sym_{k-1}}(\cdots)\\
&=\cev{z}_{st}Z^{[t+1;n-s+t]}_{[s+1;n]}+\sum_{\sigma\in\Sym_k\setminus\Sym_{k-1}}(\cdots).
\end{align*}
All $\cev{z}_{s't'}$ other than $\cev{z}_{st}$ satisfy $s'+t'>s+t$, so they are determined by the induction hypothesis $P_{s+t+1}$. In addition, $Z^{[t+1;n-s+t]}_{[s+1;n]}=\rdtr^X(M_1(t,n-s))$ is invertible by \eqref{eq-agen-decomp-sp}. Therefore, $\cev{z}_{st}$ can be uniquely solved. Since we already verified that $\cev{z}_{st}=\rdtr^X(\ceC(v_1)_{st})$ is a solution, this completes the inductive step.
\end{proof}

For this section only, we modify the definitions in Section~\ref{sec-ntri} to include the vertices $v_1,v_2,v_3$. Now $\rd{V}$ has 3 more vertices. The definitions of the quiver $\Gamma$, the matrix $\rdm{Q}$, and the vectors $\vec{k}_1,\vec{k}_2,\vec{k}_3$ (from \eqref{eq-bal-basis}) are copied verbatim. For example, $\Gamma$ has 6 extra boundary arrows determined by the positive direction of the boundary.

When the extra vertices are included, $\vec{k}_p$ defined by \eqref{eq-path-val} satisfies $\vec{k}_p(0jk)=0$ and $\vec{k}_p(n00)=0$. In particular, $\vec{k}_p$ always vanishes at the vertices $v_1,v_2,v_3$. Thus the extra vertices do not appear in the final result.

Here are two lemmas useful for calculating the Weyl-normalization.

\begin{lemma}\label{lemma-Q-k-normed}
If $\vec{l}\in\ints^{\rd{V}}$ vanishes at $v_1,v_2,v_3$, then $\rdm{Q}(\vec{k}_p,\vec{l})=\rdm{Q}(\vec{k}'_p,\vec{l})$.
\end{lemma}

\begin{proof}
This follows from \cite[Lemma~6.4]{CS}. In our notations, the lemma cited implies that $\rdm{Q}(\vec{k}_1,\vec{l})=0$ if $\vec{l}(v_2)=\vec{l}(v_3)=0$. By the rotational symmetry, analogous statements for $\vec{k}_2$ and $\vec{k}_3$ also hold. This clearly implies our lemma.
\end{proof}

\begin{lemma}\label{lemma-path-comm}
Let $p\in P(\ceC(v_1)_{ij})$ and $p'\in P(\ceC(v_1)_{i'j'})$ be two compatible paths such that $j<j'$. Let $\vec{s}=\vec{k}_{p'}-\vec{k}_p$.
\begin{enumerate}
\item If $p$ and $p'$ are disjoint, then $x^{\vec{k}_p}$ and $x^{\vec{k}_{p'}}$ commute.
\item Suppose $p$ and $p'$ merge exactly once and do not separate. (In particular $i=i'$.) Then
\begin{equation}\label{eq-path-sep}
x^{\vec{k}_{p'}}=q^{-1/2}x^{\vec{k}_p}x^{\vec{s}}=q^{1/2}x^{\vec{s}}x^{\vec{k}_p}.
\end{equation}
\end{enumerate}
\end{lemma}

\begin{proof}
In both cases, the path $p$ never cross to the right of $p'$. Let $S\subset\rd{V}$ be the set of generators in the region bounded by $p,p'$ and the boundary edges. In addition, let $K$ be the set of vertices to the left of $p$. Then $\vec{s}(v)=n$ for $v\in S$ and $\vec{s}(v)=0$ otherwise. Similarly, $\vec{k}'_p(v)=n$ for $v\in K$ and $\vec{k}'_p(v)=0$ otherwise.

The key of the proof is the calculation of $\rdm{Q}(\vec{k}_p,\vec{s})$, which is the same as $\rdm{Q}(\vec{k}'_p,\vec{s})$ by Lemma~\ref{lemma-Q-k-normed}. This means we are counting the arrows between $K$ and $S$ (multiplied by $n^2$ because $\vec{k}'_p$ and $\vec{s}$ have value $n$ on their respective sets). Due to the adjacency nature of $\rdm{Q}$, the relevant arrows are the ones that intersects $p$.

\begin{figure}
\begin{subfigure}[b]{0.4\linewidth}
\centering
\input{trigen-between3}
\subcaption{Case (1)}
\end{subfigure}
\begin{subfigure}[b]{0.4\linewidth}
\centering
\input{trigen-between2}
\subcaption{Case (2)}
\end{subfigure}
\caption{Relevant arrows for $\rdm{Q}(\vec{k}'_p,\vec{s})$}\label{fig-trigen-between}
\end{figure}

These arrows are illustrated in Figure~\ref{fig-trigen-between}. Recall that interior arrows count double. We can split each interior arrow into two and assign one each to the adjacent triangles. Except for the arrow in the triangle containing the merging point, these split arrows (together with the boundary ones) can be grouped according to the triangles they are in. If two arrows are grouped, then one points towards $S$, and the other points away from $S$. Thus their contributions cancel for $\rdm{Q}$.

In case (1), all arrows are paired. Thus $\rdm{Q}(\vec{k}_p,\vec{s})=0$. It follows that $\rdm{Q}(\vec{k}_p,\vec{k}_{p'})=0$ as well, so the corresponding monomials commute.

In case (2), the arrow in the merging triangle is the only one not in a pair, and it always points towards $S$. Thus $\rdm{Q}\big(\vec{k}_p,\vec{s}\big)=n^2$. Then \eqref{eq-path-sep} follows from the definition of Weyl-normalization.
\end{proof}

We extend the path description to a product of arcs. Suppose $\alpha_1,\dots,\alpha_k$ are simple stated arcs in the triangle $\stdT$, and let $\alpha=\alpha_1\cdots\alpha_k$. Define $P(\alpha)=P(\alpha_1)\times\cdots\times P(\alpha_k)$. Each element $p=(p_1,\dots,p_k)\in P(\alpha)$ can be represented by a path diagram. Define
\begin{equation}\label{eq-path-def}
T(\alpha)=\sum_{p\in P(\alpha)}x^{\vec{k}_{p_1}}\cdots x^{\vec{k}_{p_k}}\in\rdbl(\stdT)
\end{equation}
and extend linearly to formal linear combinations of such diagrams. A priori, $T$ may not preserve the defining relations of the skein algebra.

With this definition, Lemma~\ref{lemma-trX-CS-diag} implies that
\begin{equation}
\rdtr^X(M_2(j))=T(M_2(j)),
\end{equation}
and it is given by the unique path diagram compatible with $M_2(j)$.

\begin{lemma}\label{lemma-trigen-qminor}
$T(M_1(i,k))$ is a Weyl-normalized monomial given by the unique path diagram with disjoint paths.
\end{lemma}

\begin{proof}
For each permutation $\sigma\in\Sym_k$, let $\alpha_\sigma=\cev{C}(v_1)_{\bar{k},i+\sigma(1)}\cdots\cev{C}(v_1)_{n,i+\sigma(k)}$. Then
\begin{equation}\label{eq-TM1}
T(M_1(i,k))=\sum_{\sigma\in\Sym_k}(-q)^{\ell(\sigma)}T(\alpha_\sigma)
=\sum_{\sigma\in\Sym_k}\sum_{p\in P(\alpha_\sigma)} (-q)^{\ell(\sigma)}x^{\vec{k}_{p_1}}\cdots x^{\vec{k}_{p_k}}.
\end{equation}
Let $P$ be the union of all $P(\alpha_\sigma)$, and let $p^{(0)}\in P$ be the unique path diagram where the paths are disjoint. It is easy to check that the $p^{(0)}$-term is the lowest degree term. Let $P_0=P\setminus\{p^{(0)}\}$.

For any path diagram $p\in P_0$, there are overlapping segments between paths. Each overlapping segment has a merging point, shown in Figure~\ref{fig-trigen-merge} by the circles. Here $a_\sigma=\cev{C}(v_1)_{42}\cev{C}(v_1)_{54}\cev{C}(v_1)_{63}$.

\begin{figure}
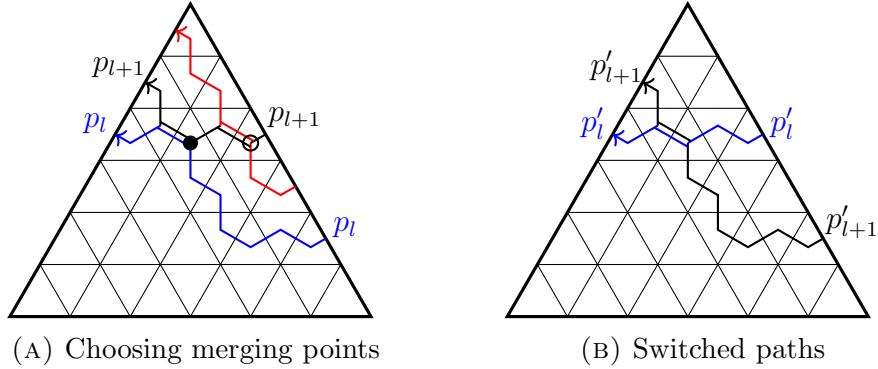

\centering
\begin{subfigure}[b]{0.4\linewidth}
\centering
\input{trigen-cancel2}
\subcaption{Choosing merging points}\label{fig-trigen-merge}
\end{subfigure}
\begin{subfigure}[b]{0.4\linewidth}
\centering
\input{trigen-switched}
\subcaption{Switched paths}\label{fig-trigen-switched}
\end{subfigure}
\caption{Cancellation of path diagrams with overlapping paths}
\label{fig-trigen-cancel}
\end{figure}

There is at least one pair of adjacent paths
\[p_l\in P(\cev{C}(v_1)_{\bar{k}+l-1,i+\sigma(l)})\qquad\text{and}\qquad
p_{l+1}\in P(\cev{C}(v_1)_{\bar{k}+l,i+\sigma(l+1)})\]
in the product that merges. We choose the pair $p_l,p_{l+1}$ with minimum $l$. Such a pair of paths may also merge at multiple points. In this case, we pick out the first merging point as we follow the direction of the paths. By switching the segments of the paths before this point, we produce a new path diagram $p'\in P_0$ with the new paths
\[p'_l\in P(\cev{C}(v_1)_{\bar{k}+l-1,i+\sigma(l+1)})
\qquad\text{and}\qquad
p'_{l+1}\in P(\cev{C}(v_1)_{\bar{k}+l,i+\sigma(l)})\]
which correspond to the permutation $\sigma'\in\Sym_k$ obtained from $\sigma$ by swapping the values at $l$ and $l+1$.

This operation is clearly involutive. Thus $P_0$ is partitioned into pairs. In each pair $(p,p')$, we can assume $\sigma(l)<\sigma(l+1)$, so $\ell(\sigma')=\ell(\sigma)+1$. Then the $p$- and $p'$-terms in the sum are of the form
\begin{gather}
(-q)^{\ell(\sigma)}(\cdots x^{\vec{k}_{p_l}}x^{\vec{k}_{p_{l+1}}}\cdots)
+(-q)^{\ell(\sigma')}(\cdots x^{\vec{k}_{p'_l}}x^{\vec{k}_{p'_{l+1}}}\cdots)\notag\\
=(-q)^{\ell(\sigma)}(\cdots)(x^{\vec{k}_{p_l}}x^{\vec{k}_{p_{l+1}}}-qx^{\vec{k}_{p'_l}}x^{\vec{k}_{p'_{l+1}}})(\cdots).\label{eq-pl-term}
\end{gather}
Note the region bounded by $p_l$ and $p'_l$ is the same as the one bounded by $p'_{l+1}$ and $p_{l+1}$. By Lemma~\ref{lemma-path-comm},
\begin{align}
x^{\vec{k}_{p'_l}}x^{\vec{k}_{p'_{l+1}}}
&=\Big(q^{-1/2}x^{\vec{k}_{p_l}}x^{\vec{s}}\Big)\Big(q^{-1/2}x^{-\vec{s}}x^{\vec{k}_{p_{l+1}}}\Big)
=q^{-1}x^{\vec{k}_{p_l}}x^{\vec{k}_{p_{l+1}}}.
\end{align}
This shows that \eqref{eq-pl-term} is zero.

So far, we showed that in \eqref{eq-TM1}, the only term that does not cancel is the $p^{(0)}$-term. The permutation correspond to $p^{(0)}$ is $\sigma=\id$. Thus $(-q)^{\ell(\sigma)}=0$. This shows that $T(M_1(i,k))$ is a product of normalized monomials. By Lemma~\ref{lemma-path-comm}, the monomials in the product commute. Thus $T(M_1(i,k))$ is normalized as well.
\end{proof}

\begin{lemma}\label{lemma-trX-CS-M1}
$T(M_1(i,k))=\rdtr^X(M_1(i,k))$.
\end{lemma}

\begin{proof}
Let $v=(i,n-i-k,j)\in\rd{V}$. Since $T(M_1(i,k))$ is a Weyl-normalized monomial, the lemma is equivalent to
\[T_v:=[T(M_1(i,k))\rdtr^X(M_2(j))]_\Weyl=\rdtr^X(\gaa_v).\]
Now we just need to evaluate the exponents of the monomial $T_v$. This is given by the lowest degree path diagram of $M_1(i,k)$ and the unique path diagram of $M_2(j)$. An example is shown in Figure~\ref{fig-trigen-lodeg}.

\begin{figure}
\centering
\input{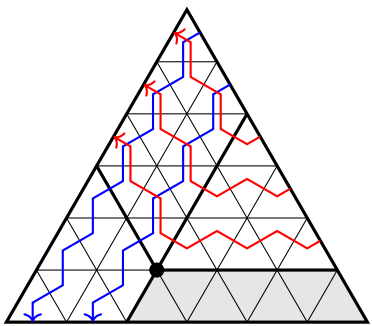}
\caption{Lowest degree term of $T_{123}$}\label{fig-trigen-lodeg}
\end{figure}

The calculation is different in the four regions separated by the thick lines in Figure~\ref{fig-trigen-lodeg}. Here we show the calculation of the exponents in the bottom right region ($i'\le i,j'\ge j$), where $\rdm{K}(v,\cdot)$ is given by \eqref{eq-trigen-exp}. The other cases are similar.

The exponent of $x_{i'j'k'}$ in $T_v$ is as follows. The chosen region is to the left of every path. Each $v_1$ corner arc contributes $n-\vec{k}_2(i'j'k')-\vec{k}_3(i'j'k')=n-j'-k'=i'$, and each $v_2$ corner arc contributes $n-\vec{k}_2(i'j'k')=n-j'=i'+k'$. Thus the total is
\[ki'+j(i'+k')=\rdm{K}(ijk,i'j'k').\qedhere\]
\end{proof}

\begin{corollary}
Theorem~\ref{thm-trX-CS} holds for $\alpha=\ceC(v_m)_{ij}$.
\end{corollary}

\begin{proof}
By expanding the definitions of $T$ in Lemma~\ref{lemma-trX-CS-M1}, we see $\cev{z}_{ij}=T(\cev{C}(v_1)_{ij})$ solves the equations in Lemma~\ref{lemma-trX-CS-uni}. By uniqueness, Theorem~\ref{thm-trX-CS} holds for $\alpha=\ceC(v_1)_{ij}$. The cases $m=2,3$ are obtained by rotation.
\end{proof}

The proof of Theorem~\ref{thm-trX-CS} is complete once the following lemma is proved.

\begin{lemma}
Theorem~\ref{thm-trX-CS} holds for $\alpha=C(v_m)_{ij}$.
\end{lemma}

\begin{proof}
Again by rotation, we can assume $m=1$. Modify the proof of \cite[Theorem~2.6]{CS} to obtain the matrix identity
\begin{equation}
\vec{M}_1=\cev{\vec{M}}_3\vec{C}^{-1}\cev{\vec{M}}_2.
\end{equation}
See Theorem~\ref{thm.CS} for the notations. On the other hand, the following holds in the skein algebra by \eqref{e.capnearwall}.
\[\input{T-eq3}\]
After applying $\rdtr^X$ and using the case $\ceC(v_m)_{ij}$ of Theorem~\ref{thm-trX-CS}, we obtain
\begin{align*}
\rdtr^X(C(v_1)_{ij})&=\sum_{k=1}^n\rdtr^X(\ceC(v_3)_{ki})\returnc_k^{-1}\rdtr^X(\ceC(v_2)_{\bar{k}j})
=\sum_{k=1}^n(\cev{\vec{M}}_3)_{ik}\returnc_k^{-1}(\cev{\vec{M}}_2)_{\bar{k}j}\\
&=(\cev{\vec{M}}_3\vec{C}^{-1}\cev{\vec{M}}_2)_{ij}
=(\vec{M}_1)_{ij}.
\end{align*}
This is the desired result by the definition of $\vec{M}_1$.
\end{proof}

\bibliographystyle{hamsalpha}
\bibliography{biblio}

\end{document}

%% file: defsv4.tex
\usepackage[headings]{fullpage}
\usepackage{amssymb,amsmath,amscd,amsthm}
\usepackage{ifthen,comment}
\usepackage{graphicx,subcaption}
\usepackage{mathrsfs,scalerel}

\usepackage[shortlabels]{enumitem}
\newlist{enuma}{enumerate}{1}
\setlist[enuma]{label=(\alph*)}
\setlist{font=\normalfont}

\usepackage{standalone,tikz}
\usetikzlibrary{cd,positioning,math,shapes,calc}
\usetikzlibrary{decorations.markings,decorations.pathreplacing,decorations.pathmorphing}
\usetikzlibrary{knots,intersections}
\usepgfmodule{decorations}
\usepackage{bbm}
\usepackage{mathtools}
\usepackage[bookmarks=true,colorlinks=true,linkcolor=blue,citecolor=blue,filecolor=blue,menucolor=blue,urlcolor=blue]{hyperref}

\input{tikz-r}
\allowdisplaybreaks

\theoremstyle{plain}
\newtheorem{theorem}{Theorem}[section]
\newtheorem{thm}{Theorem}
\newtheorem{lemma}[theorem]{Lemma}
\newtheorem{corollary}[theorem]{Corollary}
\newtheorem{proposition}[theorem]{Proposition}
\newtheorem{conjecture}{Conjecture}

\newtheorem{definition}{Definition}[section]

\theoremstyle{definition}
\newtheorem{remark}[theorem]{Remark}
\newtheorem{example}[theorem]{Example}

\makeatletter
\renewcommand*{\fps@figure}{htb}
\makeatother

\newcommand{\term}[1]{\textbf{#1}}

\DeclareMathOperator{\tr}{tr}
\DeclareMathOperator{\id}{id}
\DeclareMathOperator{\pr}{pr}

\DeclareMathOperator{\Span}{span}

\DeclareMathOperator{\Fr}{Fr}
\DeclareMathOperator{\Gr}{Gr}
\DeclareMathOperator{\GKdim}{GKdim}
\DeclareMathOperator{\Mat}{Mat}
\DeclareMathOperator{\LT}{LT}
\DeclareMathOperator{\Tr}{Tr}
\DeclareMathOperator{\tTr}{\widetilde{\Tr}}
\newcommand{\detq}{\mathop{\operatorname{det}_q}}
\newcommand{\Id}{\mathrm{Id}}
\newcommand{\llex}{\mathbin{<_{\mathrm{lex}}}}
\newcommand{\lelex}{\mathbin{\le_{\mathrm{lex}}}}

\newcommand{\eqq}{\mathrel{\overset{(q)}=}}
\newcommand{\embed}{\hookrightarrow}
\newcommand{\onto}{\twoheadrightarrow}

\newcommand{\otst}{\mathbin{\overset{\mathrm{st}}{\otimes}}}

\newcommand{\nats}{\mathbb{N}}
\newcommand{\ints}{\mathbb{Z}}

\newcommand{\reals}{\mathbb{R}}

\def\BN{\mathbb N}
\def\BZ{\mathbb Z}
\def\BQ{\mathbb Q}
\def\BR{\mathbb R}
\def\BC{\mathbb C}
\def\BT{\mathbb T}
\def\bT{\mathbb T}

\newcommand{\hq}{\hat{q}}

\def\Zq{{ \BZ[\hq^{\pm 1}]}}

\def\cS{\mathscr S}

\def\cR{\mathcal R}

\def\cF{\mathcal F}
\def\cV{\mathcal V}
\def\cX{\mathcal X}

\def\cA{\mathcal A}
\def\cB{\mathcal B}
\def\cI{\mathcal I}

\def\Srd{\rd{\cS}}

\newcommand{\Mqn}{\mathcal{M}_q(n)}
\newcommand{\Oq}{\mathcal{F}}
\newcommand{\Bq}{\mathfrak{F}}

\newcommand{\marked}{\mathcal{P}}

\def\Ibad{\mathcal{I}^{\mathrm{bad}}}

\def\Sym{{\mathrm{Sym}}}
\def\Weyl{{\mathrm{Weyl}}}
\def\ord{{\mathsf{ord}}}
\def\parti{{\mathrm{par}}}

\def\fS{\mathfrak{S}}
\newcommand{\mon}{\mathfrak}
\def\fm{\mon{m}}

\def\DD{{\mathsf D}}

\newcommand{\Cut}{\mathsf{Cut}}
\newcommand{\Mon}{\mathsf{Mon}}
\newcommand{\Pol}{\mathsf{Pol}}
\newcommand{\LPol}{\mathsf{LPol}}
\newcommand{\sign}{\mathsf{sign}}

\def\JJ{{\mathbb{J}}}
\def\oo{{\mathsf d}}

\def\bk{\mathbf k}

\def\buu{{\mathbf u}}

\def\boo{{\mathbf{o}}}
\def\boi{{ \mathbf {i}}}
\def\boj{{ \mathbf {j}}}
\def\bok{{ \mathbf {k}}}
\def\bol{{ \mathbf {l}}}

\def\bi{{\bar i}}
\def\bj{{\bar j}}

\def\bV{{\rd V}}
\def\bfS{\rd{\fS}}
\def\bS{\rd \cS}

\def\btr{\rdtr}
\def\bA{\rd{\cA}}

\def\LO{\rd{\Oq}}
\newcommand{\LOsec}{\texorpdfstring{$\LO$}{F-bar} }
\def\bBq{ {\rd { \Bq }}}
\def\bG{{\bar \Gamma}}
\def\bka{{\bar \kappa}}

\def\rOr{\overleftarrow{\mathsf{or}}}

\def\tfS{{\widetilde \fS}}

\def\chu{{\check u}}
\def\chb{{\check b}}
\def\chG{{\check G}}
\def\chGm{\chG^-}
\def\ccI{ \check{\cI}}

\def\TT{{\mathsf T}}

\def\pfS{\partial \fS}

\def\al{\alpha}
\def\ve{\varepsilon}

\def\pal{{\partial \al}}

\def\la{\langle}
\def\ra{\rangle}
\def\ot{\otimes}

\def\ttt{{\mathbbm t}}
\def\aaa{{\mathbbm a}}
\def\ccc{{\mathbbm c}}
\newcommand{\vertexa}{\mathbbm{a}}
\newcommand{\returnc}{\mathbbm{c}}

\newcommand{\weight}{\mathsf{w}}
\def\ww{\weight}
\def\wwr{\cev {\ww}}
\newcommand{\lattice}{\mathsf{L}}
\def\LL{\lattice}
\newcommand{\dgrade}{\mathbbm{d}}
\def\dd{\dgrade}

\def\gaa{\mathsf g}
\def\ag{{ \gaa }}
\def\GGG{\mathsf G}

\providecommand{\abs}[1]{\lvert #1\rvert}

\newcommand{\surface}{\fS}
\newcommand{\face}{\cF}

\newcommand{\rdo}{\overline}
\DeclareMathOperator{\rdtr}{\rdo{tr}}
\newcommand{\mat}{\mathsf}
\newcommand{\rdm}[1]{\rdo{\mat{#1}}}
\newcommand{\rdt}[1]{\rdm{#1}^t}
\newcommand{\exm}[1]{\rdm{#1}_{\ext{\lambda}}}
\newcommand{\matsp}{\hspace{0.1em}}

\newcommand{\rd}[1]{\protect\ThisStyle{\makebox[0pt][l]{\ensuremath{\protect\SavedStyle\overline{\phantom{#1}}}}}#1}

\newcommand{\lv}[1]{#1'}
\newcommand{\ext}[1]{#1^\ast}
\newcommand{\bl}[1]{#1^\mathrm{bl}}

\newcommand{\qplane}{\mathbb{T}_+}
\newcommand{\qtorus}{\mathbb{T}}
\newcommand{\FG}{\mathcal{X}}
\newcommand{\lenT}{\mathcal{A}}

\newcommand{\bad}{{\mathrm{bad}}}
\newcommand{\bal}{{\mathrm{bl}}}
\newcommand{\FGbl}{{\FG^\bal}}
\newcommand{\rdbl}{\rd{\FG}^\bal}

\newcommand{\skein}{\mathscr{S}}

\newcommand{\reduceS}{\rd{\skein}}
\newcommand{\rdSplus}{\reduceS_+}
\newcommand{\RSthree}{\mathcal{RS}}

\newcommand{\stdT}{{\poly_3}}
\newcommand{\poly}{\mathbb{P}}
\def\PP{\poly}
\newcommand{\rdV}{\rd{V}}
\newcommand{\sk}{\mathsf{sk}}

\let\avec=\vec
\renewcommand{\vec}{\mathbf}
\newcommand{\cev}[1]{\protect\ThisStyle{\reflectbox{\ensuremath{\protect\SavedStyle\avec{\reflectbox{\ensuremath{\protect\SavedStyle#1}}}}}}}
\newcommand{\ceC}{\cev{C}}

\def\XS{\FG(\fS,\lambda)}
\def\AS{{\lenT}(\fS,\lambda)}
\def\bXS{\rd{\FG}(\fS,\lambda)}
\def\bAS{\rd{\lenT}(\fS,\lambda)}
\def\bmQ{\rdm{Q}}
\def\mQ{\mat{Q}}
\def\mP{\mat{P}}
\def\bmP{\rdm{P}}
\def\bmK{\rdm{K}}
\def\bmK{\rdm{K}}
\def\bmH{\rdm{H}}
\def\mK{\mat{K}}

\def\bsX{{\rd{\FG}}}

\def\tI{{\tilde I}}
\def\bbuu{\bar \buu}
\def\ror{\cev{\mathsf{or}}}
\def\norm{{ \mathrm{norm}}}

\def\irdV{{\mathring{\rdV } } }

\def\bM{{\mathbf M}}
\def\bC{{\mathbf C}}
\def\tY{\tilde Y}

\def\mH{\mathsf H}
\def\bQ{{\overline Q}}

\def\SS{\cS(\fS)}
\def\bSS{\reduceS(\fS)}

%% file: tikz-r.tex
\tikzset{knot diagram/every knot diagram/.style={background color=gray!20,clip width=6,end tolerance=5pt,clip radius=0.2cm}}

\tikzset{edge/.style={line width=0.8}}
\tikzset{wall/.style={very thick}}
\tikzset{det/.style={edge,decoration={markings,mark=at position .5 with
	{\node[draw,fill=gray!20,isosceles triangle,sharp corners,transform shape,inner sep=0.1cm] (det){};}},postaction={decorate}}}
\tikzset{-o-/.code 2 args={
\ifthenelse{\equal{#2}{}}{
}{
	\ifthenelse{\equal{#2}{>}\OR\equal{#2}{<}}{
		\pgfkeysalso{decoration={markings,mark=at position #1 with {\arrow[scale=0.8]{#2}}},postaction={decorate}}
	}{
		\pgfkeysalso{decoration={markings,mark=at position #1 with {\draw[black, fill={#2}] circle[radius=2pt];}},postaction={decorate}}
	}
}
}}

\newcommand{\picmargin}{\mathop{}\!}
\newcommand{\stsize}{\footnotesize}

\newcommand{\TanglePic}[5]{
\picmargin
\begin{tikzpicture}[baseline=(ref.base)]
\tikzmath{\xw=#1; \yh=#2; \yd=0; \yu=0;}
\ifthenelse{\equal{#3}{<-}\OR\equal{#4}{<-}}{
	\tikzmath{\yd=-0.1;}
}{}
\ifthenelse{\equal{#3}{->}\OR\equal{#4}{->}}{
	\tikzmath{\yu=0.1;}
}{}
\tikzmath{\yt=\yh+\yu;}
\fill[gray!20] (0,\yd)rectangle(\xw,\yt);
\node(ref) at ({\xw/2},{\yh/2}) {\phantom{$-$}};
\begin{scope}[wall]
\ifthenelse{\equal{#3}{w}}{
	\draw (0,\yd) --(0,\yt);
}{\ifthenelse{\equal{#3}{}}{}{
	\draw[#3] (0,\yd) --(0,\yt);
}}
\ifthenelse{\equal{#4}{w}}{
	\draw (\xw,\yd) --(\xw,\yt);
}{\ifthenelse{\equal{#4}{}}{}{
	\draw[#4] (\xw,\yd) --(\xw,\yt);
}}
\end{scope}
#5
\end{tikzpicture}
\picmargin
}

\newcommand{\HorizontalTangle}[7]{
\TanglePic{0.9}{0.9}{#1}{#2}{
\tikzmath{\ya=\yh/2+0.2; \yb=\yh/2-0.2;}
\path (0,\ya) coordinate (C) (0,\yb) coordinate (D);
\ifthenelse{\equal{#3}{}\AND\equal{#4}{}}{}{
	\draw[left,inner sep=2pt] (C)node{\stsize #3} (D)node{\stsize #4};
}
\path (\xw,\ya) coordinate (A) (\xw,\yb) coordinate (B);
\ifthenelse{\equal{#5}{}\AND\equal{#6}{}}{}{
	\draw[right,inner sep=2pt] (A)node{\stsize #5} (B)node{\stsize #6};
}
#7
}}

\newcommand{\crossSt}[9]{
\HorizontalTangle{#1}{#2}{#6}{#7}{#8}{#9}{
\ifthenelse{\equal{#3}{a}}{
	\draw[edge] (C) ..controls ({\xw/2},\ya) and ({\xw/2},\yb).. (B);
	\draw[edge] (D) ..controls ({\xw/2},\yb) and ({\xw/2},\ya).. (A);
}{
	\begin{knot}
	\strand[edge] (C) ..controls ({\xw/2},\ya) and ({\xw/2},\yb).. (B);
	\strand[edge] (D) ..controls ({\xw/2},\yb) and ({\xw/2},\ya).. (A);
	\ifthenelse{\(\equal{#3}{n}\AND\equal{#4}{#5}\)\OR\(\equal{#3}{p}\AND\NOT\equal{#4}{#5}\)}{
		\flipcrossings{1}
	}{}
	\end{knot}
}
\ifthenelse{\equal{#4}{>}\OR\equal{#4}{<}}{
	\tikzmath{\pos=0.9;}
}{
	\tikzmath{\pos=0.8;}
}
\path[edge,-o-={\pos}{#4}] (C) ..controls ({\xw/2},\ya) and ({\xw/2},\yb).. (B);
\path[edge,-o-={\pos}{#5}] (D) ..controls ({\xw/2},\yb) and ({\xw/2},\ya).. (A);
}}

\newcommand{\cross}[5]{\crossSt{#1}{#2}{#3}{#4}{#5}{}{}{}{}}

\newcommand{\crosswall}[6]{\crossSt{}{#1}{#2}{#3}{#4}{}{}{#5}{#6}}

\newcommand{\walltwowallSt}[8]{
\HorizontalTangle{#1}{#2}{#5}{#6}{#7}{#8}{
\draw[edge,-o-={0.5}{#3}] (C) -- (A);
\draw[edge,-o-={0.5}{#4}] (D) -- (B);
}}

\newcommand{\walltwowall}[4]{\walltwowallSt{#1}{#2}{#3}{#4}{}{}{}{}}

\newcommand{\twowall}[5]{\walltwowallSt{}{#1}{#2}{#3}{}{}{#4}{#5}}

\newcommand{\kink}{
\TanglePic{0.9}{0.9}{}{}{
\begin{knot}
\strand[edge] (\xw,0.3) -- (0.6,0.3)
	..controls (0.2,0.3) and (0.2,0.7).. (0.45,0.7);
\strand[edge] (0,0.3) -- (0.2,0.3)
	..controls (0.7,0.3) and (0.7,0.7).. (0.45,0.7);
\end{knot}
\path[edge,-o-={0.5}{>}] (0.6,0.3) -- (\xw,0.3);
}}

\newcommand{\horizontaledge}[1]{
\TanglePic{0.9}{0.9}{}{}{
\draw[edge,-o-={0.8}{#1}] (0,{\yh/2}) --(\xw,{\yh/2});
}}

\newcommand{\circlediag}[2][]{
\TanglePic{0.9}{0.9}{}{}{
\draw[edge,-o-={0.1}{#2}] (0.45,0.45) circle (0.25);
}}

\newcommand{\ThreeStates}[8]{
\ifthenelse{\equal{#6}{}\AND\equal{#7}{}\AND\equal{#8}{}}{}{
	\draw (#2,#3)node[#1,inner sep=2pt]{{\stsize #6}};
	\draw (#2,#4)node[#1,inner sep=2pt]{{\stsize #7}};
	\draw (#2,#5)node[#1,inner sep=2pt]{{\stsize #8}};
}}

\newcommand{\MultiStrand}[7][b]{
\TanglePic{#2}{#3}{#4}{#5}{
\tikzmath{\xw=#2; \y1=0.2; \y4=\yh-\y1; 
\s=0.35; \y2=\y1+\s; \y3=\y4-\s; \v=(\y2+\y3)/2;}
\ifthenelse{\equal{#1}{b}}{
	\tikzmath{\ym=\y2; \yds=(\y2+\y4)/2;}
}{
	\tikzmath{\ym=\y3; \yds=(\y1+\y3)/2;}
}
\ifthenelse{\equal{#6}{1}\OR\equal{#6}{3}}{
	\draw ({\xw-0.15},\yds)node[rotate=90]{...};
}{}
\ifthenelse{\equal{#6}{2}\OR\equal{#6}{3}}{
	\draw (0.15,\yds)node[rotate=90]{...};
}{}
#7
}}

\newcommand{\sinksourcethree}[1]{
\MultiStrand{1.4}{1.3}{}{}{3}{
\coordinate (V0) at (0.6,\v);
\coordinate (V1) at (0.8,\v);
\draw[edge,-o-={.50}{#1}] (0,\y1) to [out=0, in=-120] (V0);
\draw[edge,-o-={.57}{#1}] (0,\y2) to [out=0, in=-150] (V0);
\draw[edge,-o-={.50}{#1}] (0,\y4) to [out=0, in=120] (V0);
\draw[edge,-o-={.65}{#1}] (V1) to [out=-60, in=180] (\xw,\y1);
\draw[edge,-o-={.58}{#1}] (V1) to [out=-30, in=180] (\xw,\y2);
\draw[edge,-o-={.65}{#1}] (V1) to [out=60, in=180] (\xw,\y4);
}}

\newcommand{\coupon}[2]{
\MultiStrand{1.4}{1.3}{}{}{3}{
\node (V) at (\xw/2,\yh/2) [ellipse, minimum height=0.8cm,inner sep=0pt, draw]
	{\small{#1}};
\draw[edge,-o-={.35}{#2}] (0,\y1) ..controls (\xw*0.2,\y1).. (V.-120);
\draw[edge,-o-={.50}{#2}] (0,\y2) ..controls (\xw*0.2,\y2).. (V.-170);
\draw[edge,-o-={.35}{#2}] (0,\y4) ..controls (\xw*0.2,\y4).. (V.120);
\draw[edge,-o-={.70}{#2}] (V.-60) ..controls (\xw*0.8,\y1).. (\xw,\y1);
\draw[edge,-o-={.60}{#2}] (V.-10) ..controls (\xw*0.8,\y2).. (\xw,\y2);
\draw[edge,-o-={.70}{#2}] (V.60) ..controls (\xw*0.8,\y4).. (\xw,\y4);
}}

\newcommand{\widecoupon}[8]{
\MultiStrand[#1]{1.5}{1.4}{<-}{<-}{3}{
\ThreeStates{left}{0}{\y4}{\ym}{\y1}{#3}{#4}{#5}
\ThreeStates{right}{\xw}{\y4}{\ym}{\y1}{#6}{#7}{#8}
\node (V) at (\xw/2,\v) [circle,minimum height=1cm,
	inner sep=0pt,draw]{#2};
\begin{scope}[edge]
\draw (0,\y1) ..controls (\xw*0.2,\y1).. (V.-120);
\draw (0,\y4) ..controls (\xw*0.2,\y4).. (V.120);
\draw (V.-60) ..controls (\xw*0.8,\y1).. (\xw,\y1);
\draw (V.60) ..controls (\xw*0.8,\y4).. (\xw,\y4);
\ifthenelse{\equal{#1}{b}}{
	\draw (0,\ym) ..controls (\xw*0.1,\ym).. (V.-165);
	\draw (V.-15) ..controls (\xw*0.9,\ym).. (\xw,\ym);
}{
	\draw (0,\ym) ..controls (\xw*0.1,\ym).. (V.165);
	\draw (V.15) ..controls (\xw*0.9,\ym).. (\xw,\ym);
}
\end{scope}
}}

\newcommand{\vertexnearwall}[2][b]{
\MultiStrand[#1]{1}{1.3}{}{w}{2}{
\coordinate (V) at (0.7,\v);
\draw[edge,-o-={.40}{#2}] (0,\y4) to [out=0, in=120] (V);
\draw[edge,-o-={.40}{#2}] (0,\y1) to [out=0, in=-120] (V);
\draw[edge,-o-={.45}{#2}] (0,\ym) ..controls (0.3,\ym).. (V);
}}

\newcommand{\nedgewall}[6][b]{
\MultiStrand[#1]{1}{1.3}{}{#2}{1}{
\ThreeStates{right}{\xw}{\y4}{\ym}{\y1}{#4}{#5}{#6}
\draw[edge, -o-={.5}{#3}] (0,\y4) -- (\xw,\y4);
\draw[edge, -o-={.5}{#3}] (0,\ym) -- (\xw,\ym);
\draw[edge, -o-={.5}{#3}] (0,\y1) -- (\xw,\y1);
}}

\newcommand{\capwall}[6][]{
\HorizontalTangle{}{#2}{}{}{#5}{#6}{
\draw[edge,-o-={0.8}{#4}] (\xw,\yb) ..controls (0.1,\yb) and (0.1,\ya).. (\xw,\ya);
}}

\newcommand{\capnearwall}[1]{
\HorizontalTangle{}{w}{}{}{}{}{
\draw[edge,-o-={0.8}{#1}] (0,\yb) ..controls (0.7,\yb) and (0.7,\ya).. (0,\ya);
}}

\newcommand{\vertexfour}[4][<-]{
\TanglePic{1}{1.2}{}{#1}{
\tikzmath{\s=0.35; \y2=(\yh-\s)/2; \y3=\y2+\s;}
\coordinate (V) at (0.4,{\yh/2});
\draw[edge] (V) edge +(150:\s) edge +(180:\s) edge +(-150:\s);
\draw[edge,-o-={.7}{#2}] (V) to [out=25, in=180]
	(\xw,\y3)node[right,inner sep=2pt]{{\stsize #3}};
\draw[edge,-o-={.7}{#2}] (V) to [out=-25, in=180]
	(\xw,\y2)node[right,inner sep=2pt]{{\stsize #4}};
}}

\newcommand{\vertexfourcross}[3]{
\TanglePic{1}{1.2}{}{<-}{
\tikzmath{\s=0.35; \y2=(\yh-\s)/2; \y3=\y2+\s;}
\coordinate (V) at (0.4,{\yh/2});
\draw[edge] (V) edge +(150:\s) edge +(180:\s) edge +(-150:\s);
\begin{knot}
\strand[edge] (V) to [out=45, in=175, looseness=1.5] (\xw,\y2);
\strand[edge] (V) to [out=-45, in=-175, looseness=1.5] (\xw,\y3);
\end{knot}
\path[edge,-o-={.8}{#1}] (V) to [out=30, in=175, looseness=1.5]
	(\xw,\y2)node[right,inner sep=2pt]{{\stsize #3}};
\path[edge,-o-={.8}{#1}] (V) to [out=-30, in=-175, looseness=1.5]
	(\xw,\y3)node[right,inner sep=2pt]{{\stsize #2}};
}}

\newcommand{\vertexfourleft}[4][<-]{
\TanglePic{1}{1.2}{#1}{}{
\tikzmath{\s=0.35; \y2=(\yh-\s)/2; \y3=\y2+\s;}
\coordinate (V) at (0.6,{\yh/2});
\draw[edge] (V) edge +(30:\s) edge +(0:\s) edge +(-30:\s);
\draw[edge,-o-={.7}{#2}] (V) to [out=155, in=0]
	(0,\y3)node[left,inner sep=2pt]{{\stsize #3}};
\draw[edge,-o-={.7}{#2}] (V) to [out=-155, in=0]
	(0,\y2)node[left,inner sep=2pt]{{\stsize #4}};
}}

\newcommand{\bubblestrand}[6]{
\TanglePic{1}{1}{<-}{<-}{
\ifthenelse{\equal{#1}{b}}{\tikzmath{\b=1;}}{\tikzmath{\b=-1;}}
\tikzmath{\yb=\yh/2-\b*0.2; \ys=\yh/2+\b*0.2;}
\node (B) at ({\xw/2},\yb) [circle,draw,inner sep=1pt] {\stsize #2};
\begin{scope}[every node/.style={inner sep=2pt}]
\begin{knot}
\strand[edge] (0,\ys)node[left]{\stsize #5}
	-- (\xw,\ys)node[right]{\stsize #6};
\strand[edge] (B) -- ({\xw/2},\yd);
\end{knot}
\draw[edge] (0,\yb)node[left]{\stsize #3} -- (B)
	-- (\xw,\yb)node[right]{\stsize #4};
\end{scope}
}}

\newcommand{\bubblesplit}[5]{
\TanglePic{2}{1}{<-}{<-}{
\tikzmath{\y1=\yh/2-0.2; \y2=\yh/2+0.2;}
\node (B) at ({\xw/2},\y2) [circle,draw,inner sep=1pt] {\stsize #1};
\draw[dashed] (0.6,\yd) -- (0.6,\yt) (1.4,\yd) -- (1.4,\yt);
\begin{knot}
\strand[edge,inner sep=2pt] (0,\y2)node[left]{\stsize #4}
	..controls (0.3,\y2) and (0.3,\y1).. (0.6,\y1)
	-- (1.4,\y1) ..controls (1.7,\y1) and (1.7,\y2)..
	(\xw,\y2)node[right]{\stsize #5};
\strand[edge,inner sep=2pt] (0,\y1)node[left]{\stsize #2}
	..controls (0.3,\y1) and (0.3,\y2).. (0.6,\y2)
	-- (B) -- (1.4,\y2) ..controls (1.7,\y2) and (1.7,\y1)..
	(\xw,\y1)node[right]{\stsize #3};
\strand[edge] (B) -- ({\xw/2},\yd);
\end{knot}
}}

\newcommand{\bubblemono}{
\TanglePic{2}{1.6}{}{<-}{
\tikzmath{\s=0.5; \y2=\yh/2; \y3=\y2+\s; \y1=\y2-\s;}
\fill[white] (0.5,\y2) circle[radius=0.2];
\node (B) at (1,\y2) [circle,draw,inner sep=1pt] {\stsize $\alpha$};
\begin{knot}
\strand[edge,inner sep=2pt] (\xw,\y3)node[right]{\stsize $j$}
	..controls +(-0.25,0) and +(0.25,0).. (1.5,\y2)
	..controls +(-0.25,0) and +(0.25,0).. (1,\y3)
	..controls +(-0.6,0) and +(0,0.25).. (0.2,\y2)
	..controls +(0,-0.25) and +(-0.6,0).. (1,\y1)
	..controls +(0.5,0).. (\xw,\y2)node[right]{\stsize $i$};
\strand[edge,inner sep=2pt] (0.7,\y2) -- (B)
	..controls (1.4,\y2) and (1.4,\y3).. (1.6,\y3)
	..controls +(0.2,0) and +(-0.2,0).. (\xw,\y1)node[right]{$s$};
\flipcrossings{3}
\end{knot}
\path[edge,-o-={0.1}{white}] (1,\y3)
	..controls +(-0.6,0) and +(0,0.25).. (0.2,\y2);
\draw[dashed] (1.6,\yd) -- (1.6,\yt);
}}

\newcommand{\bubblemonosplit}{
\TanglePic{0.8}{1.6}{<-}{<-}{
\tikzmath{\s=0.5; \y2=\yh/2; \y3=\y2+\s; \y1=\y2-\s;}
\begin{knot}
\strand[edge,inner sep=2pt] (0,\y2)node[left]{\stsize $j'$}
	-- (\xw,\y3)node[right]{\stsize $j$};
\strand[edge,inner sep=2pt] (0,\y3)node[left]{\stsize $s'$}
	-- (\xw,\y1)node[right]{$s$};
\strand[edge,inner sep=2pt] (0,\y1)node[left]{\stsize $i'$}
	-- (\xw,\y2)node[right]{\stsize $i$};
\end{knot}
\path[edge,inner sep=2pt,-o-={0.7}{white}] (0,\y2) -- (\xw,\y3);
\path[edge,inner sep=2pt,-o-={0.3}{black}] (0,\y1) -- (\xw,\y2);
}}



%% file: minor-Di.tex
\begin{tikzpicture}[baseline=(ref.base)]
\fill[gray!20] (-1,-0.75) rectangle (1,0.75);
\draw[wall] (-1,-0.75) -- +(0,1.5) (1,-0.75) -- +(0,1.5);
\draw[det] (-1,0)node[left]{\stsize$[1;i]$}
	-- (1,0)node[right]{\stsize$[\bar{i};n]$};
\node(ref) at (0,0) {\phantom{$-$}};
\end{tikzpicture}

%% file: ms.bbl
\providecommand{\bysame}{\leavevmode\hbox to3em{\hrulefill}\thinspace}
\providecommand{\href}[2]{#2}
\providecommand{\eprint}{\begingroup \urlstyle{rm}\Url}
\begin{thebibliography}{BZBJ18}

\bibitem[AGS95]{AGS}
Anton~Yu. Alekseev, Harald Grosse, and Volker Schomerus, \emph{Combinatorial quantization of the {H}amiltonian {C}hern-{S}imons theory. {I}}, Comm. Math. Phys. \textbf{172} (1995), no.~2, 317--358.

\bibitem[BFR23]{BFR}
Stéphane Baseilhac, Matthieu Faitg, and Philippe Roche, \emph{Unrestricted quantum moduli algebras, iii: Surfaces of arbitrary genus and skein algebras}, 2023, Preprint arXiv:2302.00396.

\bibitem[BW11]{BW}
Francis Bonahon and Helen Wong, \emph{Quantum traces for representations of surface groups in {${\rm SL}_2(\mathbb{C})$}}, Geom. Topol. \textbf{15} (2011), no.~3, 1569--1615.

\bibitem[BZ05]{BZ}
Arkady Berenstein and Andrei Zelevinsky, \emph{Quantum cluster algebras}, Adv. Math. \textbf{195} (2005), no.~2, 405--455.

\bibitem[BZBJ18]{BBJ}
David Ben-Zvi, Adrien Brochier, and David Jordan, \emph{Integrating quantum groups over surfaces}, J. Topol. \textbf{11} (2018), no.~4, 874--917.

\bibitem[CF00]{CF}
L.~O. Chekhov and V.~V. Fock, \emph{Observables in 3{D} gravity and geodesic algebras}, vol.~50, 2000, Quantum groups and integrable systems (Prague, 2000), pp.~1201--1208.

\bibitem[CKL23]{CKL}
Francois Costantino, Julien Korinman, and Thang T.~Q. L\^e, \emph{To appear}, 2023.

\bibitem[CKM14]{CKM}
Sabin Cautis, Joel Kamnitzer, and Scott Morrison, \emph{Webs and quantum skew {H}owe duality}, Math. Ann. \textbf{360} (2014), no.~1-2, 351--390.

\bibitem[CL22]{CL}
Francesco Costantino and Thang T.~Q. L\^{e}, \emph{Stated skein algebras of surfaces}, J. Eur. Math. Soc. (JEMS) \textbf{24} (2022), no.~12, 4063--4142.

\bibitem[CS20]{CS}
L.~Chekhov and M.~Shapiro, \emph{Darboux coordinates for symplectic groupoid and cluster algebras}, 2020, Preprint arXiv:2003.07499.

\bibitem[CSV95]{CSV}
Andreas Cap, Hermann Schichl, and Ji\v{r}\'{\i} Van\v{z}ura, \emph{On twisted tensor products of algebras}, Comm. Algebra \textbf{23} (1995), no.~12, 4701--4735.

\bibitem[Dou21]{Douglas}
Daniel~C. Douglas, \emph{Quantum traces for $\mathrm{SL}_n(\mathbb{C})$: the case $n=3$}, 2021, Preprint arXiv:2101.06817.

\bibitem[DS20]{DouglasSun}
Daniel~C. Douglas and Zhe Sun, \emph{Tropical fock-goncharov coordinates for $\mathrm{SL}_3$-webs on surfaces i: construction}, 2020, Preprint arXiv:2011.01768.

\bibitem[FG06]{FG}
Vladimir Fock and Alexander Goncharov, \emph{Moduli spaces of local systems and higher {T}eichm\"{u}ller theory}, Publ. Math. Inst. Hautes \'{E}tudes Sci. (2006), no.~103, 1--211.

\bibitem[FG09]{FG2}
Vladimir~V. Fock and Alexander~B. Goncharov, \emph{Cluster ensembles, quantization and the dilogarithm}, Ann. Sci. \'{E}c. Norm. Sup\'{e}r. (4) \textbf{42} (2009), no.~6, 865--930.

\bibitem[FS22]{FrohmanSikora}
Charles Frohman and Adam~S. Sikora, \emph{{$SU(3)$}-skein algebras and webs on surfaces}, Math. Z. \textbf{300} (2022), no.~1, 33--56.

\bibitem[Gav07]{Gavarini}
Fabio Gavarini, \emph{P{BW} theorems and {F}robenius structures for quantum matrices}, Glasg. Math. J. \textbf{49} (2007), no.~3, 479--488.

\bibitem[Goo06]{Goodearl}
Kenneth~R. Goodearl, \emph{Commutation relations for arbitrary quantum minors}, Pacific J. Math. \textbf{228} (2006), no.~1, 63--102.

\bibitem[GS15]{GS}
Alexander Goncharov and Linhui Shen, \emph{Geometry of canonical bases and mirror symmetry}, Invent. Math. \textbf{202} (2015), no.~2, 487--633.

\bibitem[GW04]{GW}
K.~R. Goodearl and R.~B. Warfield, Jr., \emph{An introduction to noncommutative {N}oetherian rings}, second ed., London Mathematical Society Student Texts, vol.~61, Cambridge University Press, Cambridge, 2004.

\bibitem[Hig20]{Higgins}
Vijay Higgins, \emph{Triangular decomposition of $sl_3$ skein algebras}, 2020, Preprint arXiv:2008.09419.

\bibitem[IY21]{IY}
Tsukasa Ishibashi and Wataru Yuasa, \emph{Skein and cluster algebras of marked surfaces without punctures for $\mathfrak{sl}_3$}, 2021, Preprint arXiv:2101.00643.

\bibitem[JZ97]{JZ}
Hans~Plesner Jakobsen and Hechun Zhang, \emph{The center of the quantized matrix algebra}, J. Algebra \textbf{196} (1997), no.~2, 458--474.

\bibitem[Kas95]{Kass}
Christian Kassel, \emph{Quantum groups}, Graduate Texts in Mathematics, vol. 155, Springer-Verlag, New York, 1995.

\bibitem[Kim20]{Kim1}
Hyun~Kyu Kim, \emph{${\rm sl}_3$-laminations as bases for ${\rm pgl}_3$ cluster varieties for surfaces}, 2020, Preprint arXiv:2011.14765.

\bibitem[Kim21]{Kim2}
\bysame, \emph{The mutation compatibility of the ${\rm sl}_3$ quantum trace maps for surfaces}, 2021, Preprint arXiv:2104.06286.

\bibitem[KL85]{KL}
G.~R. Krause and T.~H. Lenagan, \emph{Growth of algebras and {G}efand-{K}irillov dimension}, Research Notes in Mathematics, vol. 116, Pitman (Advanced Publishing Program), Boston, MA, 1985.

\bibitem[KLS18]{KLS}
Hyun~Kyu Kim, Thang T.~Q. Lê, and Miri Son, \emph{${\rm sl}_2$ quantum trace in quantum teichmüller theory via writhe}, 2018, Algebr. Geom. Topol, to appear.

\bibitem[KS97]{KS}
Anatoli Klimyk and Konrad Schm\"{u}dgen, \emph{Quantum groups and their representations}, Texts and Monographs in Physics, Springer-Verlag, Berlin, 1997.

\bibitem[KS09]{KolbS}
Stefan Kolb and Jasper~V. Stokman, \emph{Reflection equation algebras, coideal subalgebras, and their centres}, Selecta Math. (N.S.) \textbf{15} (2009), no.~4, 621--664.

\bibitem[Kup96]{Kuperberg}
Greg Kuperberg, \emph{Spiders for rank {$2$} {L}ie algebras}, Comm. Math. Phys. \textbf{180} (1996), no.~1, 109--151.

\bibitem[Le18]{Le:triangulation}
Thang T.~Q. Le, \emph{Triangular decomposition of skein algebras}, Quantum Topol. \textbf{9} (2018), no.~3, 591--632.

\bibitem[L{e}19]{Le:qtrace}
Thang T.~Q. L{e}, \emph{Quantum {T}eichm\"{u}ller spaces and quantum trace map}, J. Inst. Math. Jussieu \textbf{18} (2019), no.~2, 249--291.

\bibitem[LMO88]{Leroy}
A.~Leroy, J.~Matczuk, and J.~Okninski, \emph{On the {G}elfand-{K}irillov dimension of normal localizations and twisted polynomial rings}, Perspectives in ring theory ({A}ntwerp, 1987), NATO Adv. Sci. Inst. Ser. C: Math. Phys. Sci., vol. 233, Kluwer Acad. Publ., Dordrecht, 1988, pp.~205--214.

\bibitem[LS93]{Levasseur}
T.~Levasseur and J.~T. Stafford, \emph{The quantum coordinate ring of the special linear group}, J. Pure Appl. Algebra \textbf{86} (1993), no.~2, 181--186.

\bibitem[LS21]{LS}
Thang T.~Q. L\^e and Adam Sikora, \emph{Stated $sl(n$)-skein modules and algebras}, 2021, Preprint arXiv:2201.00045.

\bibitem[LS23]{LS2}
\bysame, \emph{To appear}, 2023, Preprint 2023.

\bibitem[LY22]{LY2}
Thang T.~Q. Le and Tao Yu, \emph{Quantum traces and embeddings of stated skein algebras into quantum tori}, Selecta Math. (N.S.) \textbf{28} (2022), no.~4, Paper No. 66, 48.

\bibitem[Maj95]{Majid}
Shahn Majid, \emph{Foundations of quantum group theory}, Cambridge University Press, Cambridge, 1995.

\bibitem[MOY98]{MOY}
Hitoshi Murakami, Tomotada Ohtsuki, and Shuji Yamada, \emph{Homfly polynomial via an invariant of colored plane graphs}, Enseign. Math. (2) \textbf{44} (1998), no.~3-4, 325--360.

\bibitem[MR01]{MR}
J.~C. McConnell and J.~C. Robson, \emph{Noncommutative {N}oetherian rings}, revised ed., Graduate Studies in Mathematics, vol.~30, American Mathematical Society, Providence, RI, 2001, With the cooperation of L. W. Small.

\bibitem[Mul16]{Muller}
Greg Muller, \emph{Skein and cluster algebras of marked surfaces}, Quantum Topol. \textbf{7} (2016), no.~3, 435--503.

\bibitem[Pen12]{Penner}
Robert~C. Penner, \emph{Decorated {T}eichm\"{u}ller theory}, QGM Master Class Series, European Mathematical Society (EMS), Z\"{u}rich, 2012, With a foreword by Yuri I. Manin.

\bibitem[Prz91]{Prz}
J\'{o}zef~H. Przytycki, \emph{Skein modules of {$3$}-manifolds}, Bull. Polish Acad. Sci. Math. \textbf{39} (1991), no.~1-2, 91--100.

\bibitem[PS19]{PS2}
J\'{o}zef~H. Przytycki and Adam~S. Sikora, \emph{Skein algebras of surfaces}, Trans. Amer. Math. Soc. \textbf{371} (2019), no.~2, 1309--1332.

\bibitem[RT90]{RT}
N.~Yu. Reshetikhin and V.~G. Turaev, \emph{Ribbon graphs and their invariants derived from quantum groups}, Comm. Math. Phys. \textbf{127} (1990), no.~1, 1--26.

\bibitem[She22]{Shen}
Linhui Shen, \emph{Cluster nature of quantum groups}, 2022, Preprint arXiv:2209.06258.

\bibitem[Sik05]{Sikora}
Adam~S. Sikora, \emph{Skein theory for {${\rm SU}(n)$}-quantum invariants}, Algebr. Geom. Topol. \textbf{5} (2005), 865--897.

\bibitem[SS17]{SS2}
Gus Schrader and Alexander Shapiro, \emph{Continuous tensor categories from quantum groups i: algebraic aspects}, 2017, Preprint arXiv:1708.08107.

\bibitem[SS19]{SS0}
\bysame, \emph{A cluster realization of {$U_q( {sl}_{ n})$} from quantum character varieties}, Invent. Math. \textbf{216} (2019), no.~3, 799--846.

\bibitem[Tur88]{Turaev}
V.~G. Turaev, \emph{The {C}onway and {K}auffman modules of a solid torus}, Zap. Nauchn. Sem. Leningrad. Otdel. Mat. Inst. Steklov. (LOMI) \textbf{167} (1988), no.~Issled. Topol. 6, 79--89, 190.

\end{thebibliography}
